\documentclass[12pt,english]{smfbook}

\usepackage{smfthm}
\usepackage{amsmath,leftidx,amsthm,sfheaders}
\usepackage[all]{xy}
\usepackage{xspace}
\usepackage[psamsfonts]{amssymb}
\usepackage[latin1]{inputenc}
\usepackage{graphicx,color}
\usepackage{thmtools}
\usepackage{thm-restate}
\usepackage{subcaption}

\newtheorem{theorem}{\textbf{Theorem}}
\newtheorem{proposition}{\textbf{Proposition}}[chapter]
\newtheorem{definition}{\textbf{Definition}}[chapter]

\newtheorem*{claim}{\textbf{Claim}}
\newtheorem{lemma}[proposition]{\textbf{Lemma}}
\newtheorem{corollary}[proposition]{\textbf{Corollary}}
\newtheorem{propdef}[proposition]{\textbf{Proposition-Definition}}

\theoremstyle{definition}
\newtheorem{remark}[proposition]{\textbf{Remark}}
\newtheorem{example}[proposition]{\textbf{Example}}

\newtheorem*{question}{\textbf{Question}}

\def\B{\mathbb{B}}
\def\D{\mathbb{D}}
\def\C{{\mathbb C}}
\def\N{{\mathbb N}}
\def\R{{\mathbb R}}
\def\A{{\mathbb A}}
\def\Z{{\mathbb Z}}
\def\KK{{\mathbb K}}
\def\G{{\mathbb G}}
\def\U{{\mathbb U}}
\def\H{{\mathbb H}}
\def\Q{{\mathbb Q}}

\def\p{\mathbb{P}}
\def\G{{\mathbb G}}

\def\cO{\mathcal{O}}

\def\cM{\mathcal{M}}
\def\cF{\mathcal{F}}
\def\cD{\mathcal{D}}

\def\fs{\mathfrak{s}}

\def\aut{\mathrm{Aut}}
\def\id{\mathrm{id}}
\def\pe{\textup{ := }}

\def\preper{\mathrm{Preper}}
\def\spec{\mathrm{Spec}}
\def\bif{\textup{bif}}

\def\Poly{\textup{Poly}}
\def\xg{x_g}

\def\an{\mathrm{an}}
\def\wdeg{\widetilde{\deg}}

\def\supp{\textup{supp}}

\def\fin{\mathrm{fin}}
\def\esc{\mathrm{esc}}

\DeclareMathOperator{\Gal}{Gal}
\DeclareMathOperator{\Aff}{Aff}
\DeclareMathOperator{\Card}{Card}
\DeclareMathOperator{\dv}{div}
\DeclareMathOperator{\ord}{ord}
\DeclareMathOperator{\poly}{Poly}
\DeclareMathOperator{\mpoly}{MPoly}
\DeclareMathOperator{\pair}{Pair}
\DeclareMathOperator{\mpair}{MPair}
\DeclareMathOperator{\mpcrit}{MPcrit}

\DeclareMathOperator{\st}{st}
\DeclareMathOperator{\crit}{Crit}
\DeclareMathOperator{\stab}{Stab}
\DeclareMathOperator{\lot}{l.o.t.}
\DeclareMathOperator{\red}{red}
\DeclareMathOperator{\reg}{Reg}
\DeclareMathOperator{\diam}{diam}
\DeclareMathOperator{\spf}{Spf}
\DeclareMathOperator{\PL}{PL}
\DeclareMathOperator{\sh}{SH}
\DeclareMathOperator{\proj}{Proj}
\DeclareMathOperator{\length}{Length}
\DeclareMathOperator{\mass}{Mass}

\DeclareMathOperator{\inter}{Inter}
\DeclareMathOperator{\lyap}{Lyap}

\DeclareMathOperator{\codim}{codim}

\def\cal{\mathcal}

\def\and{{\quad\text{and}\quad}}

\usepackage{color}		

\author{Charles Favre}
\address{CMLS, \'Ecole polytechnique, CNRS, Institut Polytechnique de Paris, 91128 Palaiseau Cedex, France}
\address{PIMS, University of British Columbia, Department of Mathematics, Vancouver, BC, V6T 1Z2, Canada}
\email{charles.favre@polytechnique.edu}

\author{Thomas Gauthier}
\address{CMLS, \'Ecole polytechnique, Institut Polytechnique de Paris, 91128 Palaiseau Cedex, France}
\email{thomas.gauthier@polytechnique.edu}

\title{The arithmetic of polynomial dynamical pairs}
\alttitle{Etude arithm\'etique des paires dynamiques polynomiales}

\makeindex

\begin{document}

\frontmatter 

\begin{abstract}
We study one-dimensional algebraic families of pairs given by a polynomial with a marked point. We prove an "unlikely intersection" statement for such pairs 
thereby exhibiting strong rigidity features for these pairs. We infer from this result the dynamical Andr\'e-Oort conjecture for curves in the moduli space of polynomials, by describing one-dimensional families in this parameter space containing infinitely many post-critically finite parameters. 
\end{abstract}

\begin{altabstract}
Nous \'etudions les familles alg\'ebriques param\'etr\'ees par une courbe des paires de polyn\^omes munis d'un point marqu\'e. Nous d\'emontrons un \'enonc\'e
"d'intersection improbable" pour ces paires, qui se traduit par une forme forte de rigidit\'e pour ces paires. Nous en d\'eduisons  la conjecture d'Andr\'e-Oort dynamique pour les courbes dans l'espace des modules des polyn\^omes, en d\'ecrivant les courbes de cet espace contenant une infinit\'e de param\`etres post-critiquement finis.
\end{altabstract}


\altkeywords{}

\thanks{The second author is partially supported by ANR project ``Fatou'' ANR-17-CE40-0002-01.}

\dedicatory{Aux math\'ematiciens qui nous ont tant inspir\'es Adrien Douady, Tan Lei, Jean-Christophe Yoccoz.}

\maketitle

%
%
%
%
%
%

\tableofcontents

%
%
%
%
%
%

\mainmatter


\chapter*{Introduction}

This book is intented as an exploration of the moduli space  $\Poly_d$ of complex polynomials of degree $d\ge2$ in one variable
using tools primarily coming from arithmetic geometry. 

The Mandelbrot set in $\Poly_2$ has undoubtedly been the focus of the most comprehensive set
of  studies, and its local geometry is still an active research field in connection with the Fatou conjecture, see \cite{Benini-MLC} and the references therein. 
In their seminal work,  Branner and Hubbard
\cite{BH,MR1194004} gave a topological description of the space of cubic polynomials with disconnected Julia sets
using  combinatorial tools. In any degree, $\Poly_d$
is a complex orbifold of dimension $d-1$,  and is therefore naturally amenable to complex analysis and in particular to pluripotential theory.
This observation has been particularly fruitful to describe the  locus of unstability, and to investigate the boundary of the connectedness locus. DeMarco \cite{DeMarco1} constructed a positive closed $(1,1)$ current whose support is precisely the set of unstable parameters. Dujardin and the first author~\cite{favredujardin} then noticed that the Monge-Amp\`ere measure of this current
defines a probability measure $\mu_\bif$ whose support
is in a way the right generalization of the Mandelbrot set in higher degree, capturing the part of the moduli space where the dynamics is the most unstable (see also \cite{BB1} for the case of rational maps). The support of $\mu_\bif$ has a very intricate structure:  it was proved by Shishikura~\cite{Shishikura2} in degree $2$ and later generalized in higher degree by the second author~\cite{Article1} that the Hausdorff dimension of the support of $\mu_\bif$ was maximal
equal to $2(d-1)$.

\bigskip

A polynomial is said to be post-critically finite (or PCF) if all its critical points have a finite orbit. The Julia set of a PCF polynomial is connected, of measure zero, and the dynamics on it is hyperbolic off the post-critical set. PCF polynomials form a countable subset of larger classes of polynomials (such as Misiurewicz, or Collet-Eckmann) for which the thermodynamical formalism is well understood, \cite{MR2332354,MR2784276}. 
They  also play a pivotal role in the study of the
connectedness locus of $\Poly_d$: their distribution was described in a series of papers \cite{favregauthier,Gauthier-Vigny1,Gauthier-Vigny2} and proved to represent the bifurcation measure $\mu_\bif$. 

PCF polynomials are naturally defined by $d-1$ equations of the form $P^n(c) = P^m (c)$ where $c$ denotes a critical point and $n,m$ are two distinct integers. In the moduli space, these equations are algebraic
with integral coefficients, so that any PCF polynomial is in fact defined over a number field. 
Ingram \cite{MR2885981} has pushed this remark further and has built a natural height  $ h_\bif \colon \Poly_d(\bar{\Q}) \to \R_+$ for which the set of PCF polynomials coincides with $\{h_\bif =0\}$. 

Height theory yields interesting new perspectives on the geometry of $ \Poly_d$, and more 
specifically on the distribution of PCF polynomials. 
We will be mostly interested here in the so-called dynamical Andr\'e-Oort conjecture which appeared in \cite{BD}, see also \cite{silvermanmoduli}. 

This remarkable conjecture was set out by Baker and DeMarco who were motivated by deep analogies between PCF dynamics and CM points in Shimura varieties, 
and more specifically by works by Masser-Zannier~\cite{MR2457263,Masser-Zannier,MR2918151} on torsion points in elliptic curves.
An historical account of the introduction of these ideas in arithmetic dynamics is given in~\cite[\S 1.2]{BD}, and~\cite[\S 1.2]{BDM}, see also~\cite{Ghioca-exposition}.
We note that this analogy is going far beyond the problems considered in this book, and applies to various conjectures described in~\cite{Silverman,DeMarco-ICM}. We refer to the book by Zannier~\cite{MR2918151}
for a beautiful discussion of unlikely intersection problems in arithmetic geometry.

\smallskip

Baker and DeMarco proposed to characterize irreducible subvarieties of $\Poly_d$ (or more generally of the moduli space of rational maps) containing a Zariski dense subset of PCF polynomials, and conjectured that such varieties were defined by critical relations. This conjecture was proven in degree $3$ in~\cite{specialcubic} and~\cite{Ghioca-Ye-Cubic}, and for unicritical polynomials in \cite{GKH-unicritical} and \cite{GKHY}. 

It is our aim to give a proof of that conjecture for \emph{curves} in $\Poly_d$ for any $d\ge 2$, and based on this result to attempt a classification of these curves 
in terms of combinatorial data encoding critical relations. 

Our proof roughly follows  the line of arguments devised in the original paper of Baker and DeMarco, and relies
on equidistribution theorems of points of small height by Thuillier~\cite{thuillier} and Yuan~\cite{yuan}; on the expansion of the B\"ottcher coordinates; and on Ritt's theory 
characterizing equalities of composition of polynomials.

We needed, though, to overcome several important technical difficulties, such as proving the continuity of metrics naturally attached to families of polynomials. We also 
had to inject new ingredients, most notably some dynamical rigidity results concerning families of  polynomials with a marked point
whose bifurcation locus is real-analytic. 

\bigskip

For the most part of the memoir, we shall work in the general context of polynomial dynamical pairs $(P,a)$
parameterized by a complex affine curve $C$, postponing the proof of the dynamical Andr\'e-Oort conjecture 
to the last chapter. We investigate quite generally the problem of unlikely intersection
that was promoted in the context of torsion points on elliptic curves by Zannier and his co-authors \cite{Masser-Zannier,MR2918151}, and
later studied by Baker and DeMarco \cite{BDM,BD}
in our context. This problem amounts to understanding
when two polynomial dynamical pairs $(P,a)$ and $(Q,b)$ parameterized by the same curve $C$ have an infinite set of common parameters
for which the marked points are preperiodic. We obtain quite definite answers for polynomial pairs, and we prove finiteness theorems
that we feel are of some interest for further exploration.

We have tried to review all the necessary material for the proof of the dynamical Andr\'e-Oort conjecture,
but we have omitted some technical proofs that are already available in the literature in an optimal form. 
On the other hand, we have made some efforts to clarify some proofs which we felt too sketchy in the literature.
The group of dynamical symmetries of a polynomial play a very important role in unlikely intersection problems, and we have thus included 
a detailed discussion of this notion.

Let us now describe in more detail the content of the book.

\paragraph*{Polynomial dynamical pairs}
In this paragraph we present the main players of our memoir. 
The central notion is the one of {\sc polynomial dynamical pair} parameterized by a curve. 
Such a pair $(P,a)$ is by definition an algebraic family of polynomials $P_t$ parameterized by an irreducible affine curve $C$ defined over a field $K$, 
accompanied by a regular function $a \in K[C]$ which defines an algebraically varying marked point. 
Most of the time, our objects will be defined over the field of complex numbers $K= \C$, but it will also be important to consider
dynamical pairs over other fields like number fields, $p$-adic fields, or finite fields.

Any polynomial dynamical pair leaves a "trace" on the parameter space $C$, which may take different forms.
Suppose first that $K$ is an arbitrary field, and let $\bar{K}$ be an algebraic closure of $K$. The first basic object to consider is  the set $\preper(P,a)$ of (closed) points $t\in C(\bar{K})$ such that 
$a(t)$ is preperiodic under $P_t$. This set is either equal to $C$ or is at most countable. 

A slightly more complicated  but equally important object one can attach to $(P,a)$ is the following divisor. 
Let $\bar{C}$ be the completion of $C$, that is the unique projective algebraic curve containing $C$ as a  Zariski dense open subset, 
and smooth at all points $\bar{C} \setminus C$.
Points in $\bar{C} \setminus C$ are called branches at infinity of $C$.
Any pair $(P,a)$ induces an effective divisor $\mathsf{D}_{P,a}$ on $\bar{C}$, which is obtained by setting
\begin{equation}\label{eq:def-div}
\ord_{\mathfrak{c}} \left(\mathsf{D}_{P,a}\right)
:= \lim_{n\to\infty} -\frac1{d^n} \min \{ 0, \ord_{\mathfrak{c}} (P^n(a))\}
,\end{equation} 
for any branch $\mathfrak{c}$ at infinity.
The limit is known to exist and is always a rational number, see \S \ref{sec:div dyn pair}.

When $K =\C$, one can associate more topological objects to a dynamical pair. 
One can consider the locus of stability of the pair $(P,a)$ which consists of the open set over which
the family of holomorphic maps $\{ P^n (a)\}_{n\ge0}$ is normal. Its complement is the {\sc bifurcation locus} which we denote by $\mathrm{Bif}(P,a)$.
This set can be characterized using potential theory as follows.
Recall the definition of the Green function of a polynomial $P$ of degree $d$: 
\[
g_P(z):= 
\lim_{n\to\infty} \frac1{d^n} 
\max \{ \log |P^n (z)| , 0\}
, \]
so that $\{ g_P =0\}$ is the filled-in Julia set of $P$ consisting of those points having bounded orbits.
On the parameter space $C$, we then define the function 
\[g_{P,a} (t) = g_{P_t} (a(t)). \] 
It is a non-negative continuous
subharmonic function on $C$, and the support of the measure $\mu_\bif = \Delta g_{P,a}$  is precisely equal to $\mathrm{Bif}(P,a)$.
Of crucial technical importance is the following result from \cite{continuity} which relates the function $g_{P,a}$ to the divisor defined above. 
\begin{theorem}\label{thm:continuity0}
In a neighborhood of any  branch at infinity $\mathfrak{c} 
\in \bar{C}$, one has the expansion
\[ g_{P,a} (t) = \ord_{\mathfrak{c}}\left(\mathsf{D}_{P,a}\right)\, \log |t|^{-1} + \tilde{g}(t)
\]
where $t$ is a local parameter centered at $\mathfrak{c}$ and  $\tilde{g}$ is continuous at $0$.
\end{theorem}
This result can be interpreted in the langage of complex geometry by saying that $g_{P,a}$ induces a continuous semi-positive metrization 
on the $\Q$-line bundle $\mathcal{O}_{\bar{C}} (\mathsf{D}_{P,a})$. This technical fact will be of crucial importance when we shall apply techniques from arithmetic geometry.

\smallskip

Let us now suppose that $K = \KK$ is a number field. For any place $v$ of $\KK$, denote by $\KK_v$
the completion of $\KK$, and by $\C_v$ the completion of its algebraic closure.
It is then possible to mimic the previous constructions at any (finite or infinite) place
$v$ of $\KK$ so as to obtain functions $g_{P,a,v} \colon C^{\an}_v \to \R_+$ 
on the analytification (in the sense of Berkovich) $C^{\an}_v$  of the curve $C$ over $\C_v$.
Combining these functions yield a height function 
$h_{P,a}\colon C(\bar{\KK}) \to \R_+$. 

Alternatively, we may start from the standard Weil height $h_{\st}\colon \mathbb{P}^1(\bar{\KK})\to \R_+$, see e.g.~\cite{Silvermandiophantine}. 
Then for any polynomial with algebraic coefficients, we 
define its canonical height~\cite{call-silverman}: 
\[
h_P(z):= 
\lim_{n\to\infty} \frac1{d^n} h_{\st} (P^n(z)), \]
and finally set $h_{P,a}(t) := h_{P_t}(a(t))$.
Using the Northcott theorem, one obtains that $\{ h_{P,a} =0\}$ coincides with the set $\preper(P,a)$ of parameters $t\in C(\bar{\KK})$ for which $a(t)$ is a preperiodic point of $P_t$.

\smallskip

It is an amazing fact that all the objects attached to a dynamical pair $(P,a)$ we have seen so far are tightly interrelated, as the 
next theorem due to DeMarco \cite{demarco} shows. 

An isotrivial pair $(P,a)$ is a pair which is conjugated to a constant polynomial and a constant marked point
possibly after a base change. A marked point is {\sc stably preperiodic} when there exist two integers $n>m$ such that 
$P_t^n(a(t))= P_t^m(a(t))$.

 \begin{theorem}\label{thm:demarco-stab}
Let $(P,a)$ be a dynamical pair of degree $d\ge 2$ parametrized by an affine irreducible curve $C$ 
defined over a number field $\KK$.
If the pair is not isotrivial, then the following assertions are equivalent:
\begin{enumerate}
\item the set $\preper(P,a)$ is equal to $C(\bar{\KK})$;
\item the marked point is stably preperiodic;
\item the divisor $\mathsf{D}_{P,a}$ of the pair $(P,a)$ vanishes;
\item for any Archimedean place $v$, the bifurcation measure $\mu_{P,a,v} := \Delta g_{P,a,v}$ vanishes;
\item the height $h_{P,a}$ is identically zero.
\end{enumerate}
\end{theorem}

A pair $(P,a)$ which satisfies either one of the previous conditions is said to be passive, otherwise it is called 
an {\sc active pair}. For an active pair, $\preper(P,a)$ is countable, 
the bifurcation measure  $\mu_{P,a}$ is non trivial, and the height $h_{P,a}$ is non zero.

\medskip

\paragraph*{Holomorphic rigidity for dynamical pairs}
Rigidity results are pervasive in (holomorphic) dynamics. 
One of the most famous rigidity result was obtained by Zdunik \cite{zdunik} and states the following.
The measure of maximal entropy of  a polynomial $P$ is 
 absolutely continuous to the Hausdorff measure of its Julia set iff $P$ is conjugated by an affine transformation
 to either a monomial map $M_d(z) = z^d$, or to a Chebyshev polynomial $\pm T_d$ where $T_{d}(z+z^{-1})= z^d+z^{-d}$.
In particular, these two families of examples are the only ones having a smooth Julia set, a theorem due to Fatou \cite{Fatou}.

The following analog of Zdunik's result for polynomial dynamical pairs is our first main result. 
\begin{restatable}{Theorem}{rigidity}\label{tm:rigidaffine}
Let $(P,a)$ be a dynamical pair of degree $d\ge 2$ parametrized by a connected Riemann surface $C$. 
Assume that $\mathrm{Bif}(P,a)$ is non-empty and included in a smooth real curve. 
Then one of the following holds:
\begin{itemize}
\item 
either  $P_t$ is conjugated to $M_d$ or $\pm T_d$ for all $t\in C$;
\item
or  there exists a univalent map $\imath\colon \D \to C$ such that 
$\imath^{-1}(\mathrm{Bif}(P,a))$ is a non-empty closed and totally disconnected perfect subset of the real line
and the pair $(P \circ \imath, a \circ \imath)$ is conjugated to a real family over $\D$. 
\end{itemize}
\end{restatable}
We say that a polynomial dynamical pair $(P,a)$ parameterized by the unit disk is a real family whenever the power series defining the coefficients of $P$ and the marked point
have all real coefficients. 

The previous theorem is a crucial ingredient for handling the unlikely intersection problem that we will describe later. Its proof builds on a transfer principle from the parameter space to the dynamical plane which can be decomposed into two parts. 

The first step is to find a parameter $t_0$ at which $a(t_0)$ is preperiodic to a repelling orbit of $P_{t_0}$ and such that 
$t \mapsto a(t)$ is transversal at $t_0$ to the preperiodic orbit degenerating to $a(t_0)$. This step builds on an argument of Dujardin \cite{dujardin-higher}.
The second step relies on Tan Lei's similarity theorem \cite{similarity} which shows that the 
bifurcation locus $\mathrm{Bif}(P,a)$ near $t_0$ is conformally equivalent at small scales to 
the Julia set of $P_{t_0}$.

Combining these two ingredients, we see that 
if $\mathrm{Bif}(P,a)$ is connected, then Zdunik's theorem implies that $P_t$ is isotrivial
conjugated to $M_d$ or $\pm T_d$ for all $t\in C$. When $\mathrm{Bif}(P,a)$ is disconnected, then we prove that all multipliers of $P_{t_0}$ are real
and we conclude that $P_t$ is real for all nearby parameters using an argument of Eremenko and Van Strien \cite{Eremenko-vanStrien}.

\smallskip 

In many results that we present below, we shall exclude all polynomials that are affinely conjugated to either $M_d$ or $\pm T_d$. These dynamical systems carry  different names in the literature: Zdunik~\cite{zdunik} name them maps with parabolic orbifolds; they are called special in~\cite{boundedheight,pakovich};
and Medvedev and Scanlon call  them non-disintegrated polynomials, see the discussion on~\cite[p.16]{medvedev-scanlon}. We shall refer them to as {\sc integrable} polynomials by analogy with the notion of integrable systems in hamiltonian dynamics (see \cite{MR2325017,MR1098340}). A family of polynomials $\{P_t\}_{t\in C}$ will be called non-integrable whenever there exists a dense open set $U\subset C$ such that $P_t$ is not integrable for any $t\in U$.

\medskip

\paragraph*{Unlikely intersections for polynomial dynamical pairs}
Our next objective is to investigate the problem of characterizing when two dynamical pairs $(P,a)$ and $(Q,b)$ parameterized by the same algebraic curve $C$
 leave the same "trace" on $C$. 
 
Analogies with arithmetic geometry suggested that the quite weak condition of $\mathrm{Preper}(P,a)\cap\mathrm{Preper}(Q,b)$ being infinite in fact implies very strong relations between the two pairs. 
This phenomenon was first observed for Latt\`es maps by Masser and Zannier \cite{Masser-Zannier}, and later for 
unicritical polynomials by Baker and DeMarco \cite{BDM}, and for more general families of polynomials parameterized by the affine line by Ghioca, Hsia and Tucker~\cite{Ghioca-Hsia-Tucker2}.
We refer to the surveys~\cite{Ghioca-exposition}, \cite{DeMarco-ICM} and \cite{Current-trends-arithmetic} where this problem is also addressed.

A precise conjecture was formulated by DeMarco in~\cite[Conjecture 4.8]{DeMarco-Kawa}:  up to symmetries and taking iterates the two families $P$ and $Q$ are actually equal, and the marked points belong to the same grand orbit. 
 In other words, the existence of unlikely intersections forces some algebraic rigidity between the dynamical pairs.

We prove here DeMarco's conjecture for polynomial dynamical pairs defined over a number field.

\begin{restatable}{Theorem}{unlikely}\label{tm:unlikely}
Let $(P,a)$ and $(Q,b)$ be active non-integrable dynamical pairs parametrized by an irreducible algebraic curve $C$ of respective degree $d,\delta\geq2$.
Assume that the two pairs are defined over a number field $\KK$.
Then, the following are equivalent:
\begin{enumerate}
\item the set $\mathrm{Preper}(P,a)\cap\mathrm{Preper}(Q,b)$ is an infinite subset of $C(\bar{\KK})$; 
\item
the two height functions $h_{P,a}, h_{Q,b} \colon C(\bar{\KK}) \to \R_+$ are proportional;
\item  there exist integers $N,M\geq1$, $r,s\ge0$, and families $R,\tau$ and $\pi$ of polynomials of degree $\ge 1$ parametrized by $C$ such that
\begin{equation}\label{eq:Rit}
\tau\circ P^N= R\circ\tau \ \text{ and } \ \pi\circ Q^M= R\circ\pi, \tag{\dag}
\end{equation}
and $\tau (P^{r}(a))= \pi(Q^{s}(b))$.
\end{enumerate}
\end{restatable}

It is not difficult to see that (3)$\Rightarrow$(2)$\Rightarrow$(1) so that the main content of the theorem is the implication 
(1)$\Rightarrow$(3).
To obtain (1)$\Rightarrow$(2), we first apply Yuan-Thuillier's equidistribution result \cite{thuillier,yuan} of points of small height: it is precisely 
at this step that the continuity of $\tilde{g}$ in Theorem \ref{thm:continuity0} is crucial. This allows one to prove that 
the bifurcation measures $\mu_{P,a,v}$ and $\mu_{Q,b,v}$ are proportional at any place $v$ of $\KK$. 
From there, one infers the proportionality of height functions i.e. (2) using our above rigidity result (Theorem \ref{tm:rigidaffine}).

The implication (2)$\Rightarrow$(3) is more involved. We first prove that 
 $\deg(P)$ and $\deg(Q)$ are multiplicatively dependent using an argument  taken from \cite{Favre-Dujardin}
 which consists of computing the H\"older constants of continuity of the potentials of the bifurcation measures at a complex place. 
From this, we obtain (3) by combining in a quite subtle way several ingredients including: 
\begin{itemize}
\item
a precise understanding of the expansion at infinity of the B\"ottcher coordinate; 
\item
an algebraization result of germs of curves defined by adelic series due to Xie \cite{Xie}; 
\item
and the classification of invariant curves by product maps $(z,w) \mapsto (R(z),R(w))$.
\end{itemize}
The latter result is due to Medvedev and Scanlon \cite{medvedev-scanlon} whose proof elaborates on Ritt's theory~\cite{Ritt}. This theory aims at describing all possible ways
a polynomial can be written as the composition of lower degree polynomials. It is very combinatorial in nature and was treated by several authors, by Zannier~\cite{MR1244972}, by M\"uller-Zieve~\cite{zieve-muller}, see also the references therein. 
Of particular relevance for us are the series of papers by Pakovich \cite{pako:preimages,pako:rational,pakovich}, and by Ghioca, Nguyen and their co-authors~\cite{MR3632102,GNY}.

\smallskip

As mentioned above, the line of arguments for proving Theorem~\ref{tm:unlikely} is mostly taken from the seminal paper of Baker and DeMarco, but with considerably more technical issues. The core of the proof takes about 8 pages and is the content of \S \ref{sec:biggest proof}.

\smallskip

It would be desirable to extend Theorem \ref{tm:unlikely} to families defined over an arbitrary field of characteristic zero. 
Reducing to the  case over a number field typically uses a specialization argument.
We faced an essential difficulty in the course of this argument, and thus had to require an additional assumption.

\begin{restatable}{Theorem}{unlikelyspecial}\label{tm:unlikely-car0}
Pick any irreducible algebraic curve $C$ defined over a field of characteristic $0$. 
Let $(P,a)$ and $(Q,b)$ be active non-integrable dynamical pairs parametrized by $C$ of respective degree $d,\delta\geq2$.
Assume that
\[\text{ any branch at infinity } \mathfrak{c} \text{ of } C  \text{ belongs to the support of the divisor } \mathsf{D}_{P,a}.
\tag{$\vartriangle$}\]
Then, the following are equivalent:
\begin{enumerate}
\item the set $\mathrm{Preper}(P,a)\cap\mathrm{Preper}(Q,b)$ is an infinite subset of $C$; 
\item  there exist integers $N,M\geq1$, $r,s\ge0$, and families $R,\tau$ and $\pi$ of polynomials parametrized by $C$ such that
\[\tau\circ P^N= R\circ\tau \ \text{ and } \ \pi\circ Q^M= R\circ\pi,\]
and $\tau (P^{r}(a))= \pi(Q^{s}(b))$.
\end{enumerate}
\end{restatable}
Note that although $(\vartriangle)$ may not hold in general, it is always satisfied when $C$ admits a unique branch at infinity, e.g. when $C$ is the affine line.
In particular, our result yields a far-reaching generalization of \cite[Theorem~1.1]{BDM}.

In the sequel, we call two active dynamical pairs $(P,a)$ and $(Q,b)$ {\sc entangled} when $\mathrm{Preper}(P,a)\cap\mathrm{Preper}(Q,b)$ is infinite. This terminology inspired by quantum theory reflects the fact the two pairs are dynamically strongly  correlated.

\medskip

\paragraph*{Description of all pairs entangled to a fixed pair}
Let us fix a polynomial dynamical pair $(P,a)$ parameterized by an algebraic curve $C$ and 
for which the previous theorems apply (i.e. either the field of definition of the pair is a number field, or condition $(\vartriangle)$ holds). 
We would like now to determine \emph{all} pairs that are entangled to $(P,a)$. 

In principle this problem is solvable by Ritt's theory.  Given a polynomial $P$, it is, however, very delicate to describe all 
polynomials $Q$ for which \eqref{eq:Rit} holds, in particular because there is no a priori bounds on the degrees of $\tau$ and $\pi$. Much progress have been made by Pakovich \cite{pakovich}  but it remains unclear whether
one can design an algorithm to solve this problem. 

To get around this, we consider a more restrictive question which is to determine all pairs $(P,b)$ that are entangled with $(P,a)$. 
In this problem, the notion of symmetries of a polynomial plays a crucial role, and most of Chapter \ref{chapter:symmetries} is devoted to the study of  this notion from the algebraic, topological and adelic perspectives.
Suffice it to say here that the group  $\Sigma(P)$ of symmetries of a complex polynomial $P$ is the group of affine transformations preserving its Julia set. 
Of importance in the latter discussion is  the subgroup $\Sigma_0(P)$ of affine maps $g\in \Sigma(P)$ such that $P^n (g \cdot z) = P^n(z)$ for some $n\in \N^*$. 

We also introduce the notion of  {\sc primitive} polynomials. A polynomial $P$ is primitive if any equality $P = g \cdot Q^n$ with $g \in \Sigma(P)$ implies
$n =1$. 

These notions of symmetries and primitivity allow us to obtain the following neat statement. 

\begin{restatable}{Theorem}{unlikelymarked}\label{tm:unlikely-marked}
Let $(P,a)$ be any active primitive non-integrable dynamical pair
 parameterized by an algebraic curve defined over  a field $K$ of characteristic $0$.
Assume that $K$ is a number field, or that  ($\vartriangle$) is satisfied. 
 
 For any marked point $b\in K[C]$ such that $(P,b)$ is active, 
the following assertions are equivalent:
\begin{enumerate}
\item the set $\mathrm{Preper}(P,a)\cap\mathrm{Preper}(P,b)$ is infinite,
\item there exist $g\in\Sigma(P)$ and integers $r,s\geq0$ such that $P^r(b)=g\cdot P^s(a)$.
\end{enumerate}
\end{restatable}

Note that this gives a positive answer to~\cite[Question~1.3]{Ghioca-Hsia-Tucker2} for polynomials.

Beware that the previous result does not quite describe the marked points parameterized by $C$ which are entangled with $(P,a)$. 
Indeed if $s =0$ and $r$ is sufficiently large, the solutions to the equation $P^r(b) = a$
are not necessarily regular functions on $C$: $b$ belongs to a finite extension of $K(C)$ of degree at most $d^r$.
In fact, we have:

\begin{restatable}{Theorem}{unlikelyfiniteness}\label{tm:unlikely-finiteness}
Let $(P,a)$ be any active non-integrable primitive dynamical pair parameterized by an irreducible affine curve $C$
 defined over  $\bar{\Q}$. 

Then, any marked point $b \in \bar{\Q}[C]$ that is entangled with $a$
belongs to the set $\{g\cdot P^n(a)\, ; n\geq0 \ \text{and} \ g\in\Sigma_0(P)\}$ except for finitely many exceptions.
\end{restatable}

This result seems to be new even for the unicritical family.

It would obviously be more natural to assume the pair to be defined over an algebraically closed field of characteristic $0$, but
we use at a crucial step the assumption that $(P,a)$ is defined over $\bar{\Q}$. Note that the theorem is not hard when
$\bar{\Q}$ is replaced by a number field.

Interestingly enough, the proof of this finiteness theorem relies on the same ingredients as Theorem \ref{tm:unlikely-car0}, namely 
 the expansion of the B\"ottcher coordinate, an algebraization result of adelic curves, and Ritt's theory. 
The proof in fact shows that one may only suppose $b \in \bar{\Q}(C)$.

\medskip

\paragraph*{Unicritical polynomials.} 
In the short Chapter \ref{chapter:unicritical}, we discuss in more depth some aspects of unlikely intersection problems  for
unicritical polynomials.

Recall that in their seminal paper, Baker and DeMarco obtained the following striking result: 
for any $d\ge2$, and any $a,b \in \C$, the pairs $\preper(z^d+t,a)$, $\preper(z^d+t,b)$
are entangled iff $a^d = b^d$. This result was further expanded by Ghioca, Hsia, and Tucker to 
more general families of polynomials and not necessarily constant marked points, see~\cite[Theorem~2.3]{Ghioca-Hsia-Tucker2}.

Building on our previous results, we obtain the following statement.

\begin{restatable}{Theorem}{samemandel}\label{thm:same-mandel}
Let $d,\delta \ge2$.
If $a, b$ are polynomials of the same degree and such that $\preper(z^d+t,a) \cap \preper(z^\delta+t,b)$ is infinite, 
then $d =\delta$ and $a^d = b^d$.
\end{restatable}

After proving this theorem, we make some preliminary exploration of the set $\mathbb{M}$ of complex numbers $\lambda\in \C^*$ such that 
the bifurcation locus $\partial M_\lambda$ is connected, where we have set $M_\lambda :=  \{ t, \, \lambda^{-1} t \in K(z^d + t) \}$. We observe
that $\lambda \in \mathbb{M}$ iff $M_\lambda \subset \cal{M}(d,0)$, and
 prove that $\mathbb{M}$ is a closed set of $\C^*$ included in the unit disk, and containing the punctured disk $\D^*(0,1/8)$.

\medskip

\paragraph*{Special curves in the parameter space of polynomials}
We finally come back in Chapter~\ref{chapter:special} to our original objective, namely the classification of curves in $\Poly_d$
which contain an infinite subset of PCF polynomials, and the proof of 
Baker and DeMarco's conjecture claiming that these curves are cut out by critical relations.

A first answer to Baker and DeMarco's question is given by the next result. 

\begin{restatable}{Theorem}{specialclas}\label{thm:specialclas}
Pick any non-isotrivial complex family $P$ of polynomials of degree $d\ge2$ with marked critical points which is parameterized by an algebraic curve $C$
containing infinitely many PCF parameters. Then we are in one of the following two situations.
\begin{enumerate}
\item
There exists a symmetry $\sigma \in \Sigma(P)$ and a family of polynomials $Q$ parameterized by some finite ramified cover $C' \to C$ such that 
$P = \sigma \cdot Q^n$ for some $n\ge2$, and the set of $t'\in C'$ such that $Q_{t'}$ is PCF is infinite.
\item
There exists a non-empty subset $\mathsf{A}$ of the set of critical points of $P$ such that
for any pair $c_i, c_j \in \mathsf{A}$, there exists a symmetry $\sigma\in\Sigma(P)$ and integers 
$n,m\geq0$ such that
\begin{equation}\label{eq:relat-crit1}
P^n(c_i)=\sigma \cdot P^m(c_j)~;
\end{equation}
and for any $c_i \notin\mathsf{A}$ there exist integers $n_i>m_i\ge0$ such that 
\begin{equation}\label{eq:relat-crit2}
P^{n_i}(c_i) = P^{m_i}(c_i)~.
\end{equation}
\end{enumerate}
\end{restatable}
Following the terminology of \cite[\S 1.4]{BD} inspired from arithmetic geometry, we call {\sc special} any curve in $\Poly_d$ containing infinitely many PCF polynomials.

Our theorem says that  a special curve in the moduli space of polynomials of degree $d$ either arises as the image under the composition map of a special curve in
a  lower degree moduli space, or is defined by critical relations (including symmetries) such that all active critical points are entangled.

This result opens up the possibility to give a combinatorial classification of all special curves in the moduli space of polynomials of a fixed degree $\Poly_d$. 
Recall that a combinatorial classification of PCF polynomials in terms of Hubbard trees has been developed by Douady and Hubbard~\cite{orsay1,orsay2}, Bielefeld-Fisher-Hubbard~\cite{BFH} and further expanded by Poirier~\cite{Poirier}, and Kiwi~\cite{kiwi-portrait}. 
We make here a first step towards the ambitious goal of classification of special curves
using a combinatorial gadget: {\sc the critically marked dynamical graph}.

We refer to \S \ref{sec:classif special} for a precise definition of critically marked dynamical graph. It is a (possibly infinite) graph $\Gamma(P)$ with a dynamics
that encodes precisely all dynamical critical relations (up to symmetry) of a given polynomial $P$. We show that to any irreducible curve $C$ in the moduli space of (critically marked) polynomials, one can attach a marked dynamical graph $\Gamma(C)$ such that $\Gamma(P) = \Gamma(C)$ for all but countably many $P \in C$.
We then identify a class of marked graphs that we call special which arise from special curves. 

Under the assumption that the special graph $\Gamma$ has no symmetry and that its marked points are not periodic, we conversely prove that the set of polynomials such that $\Gamma(P)= \Gamma$ defines a (possibly reducible) special curve. Our precise statement is quite technical, see Theorem \ref{thm:correspondence}, but it provides an interesting correspondence between a wide family of special curves and combinatorial objects of finite type.

The proof of this result builds on two constructions of polynomials with special combinatorics, one by Bielefeld-Fisher-Hubbard of strictly PCF combinatorics, and one 
by McMullen and DeMarco of dynamical trees \cite{DeMarco-McMullen}. Binding together these two results was quite challenging.
We have been able to prove only a partial correspondence under simplifying additional assumptions.
 
 \medskip

\paragraph*{Some technical details that we have worked out and hopefully clarified!}
Beside presenting a set of new results, we have made special efforts to clarify some technical aspects of the standard approach
to the unlikely intersection problem for polynomials. We emphasize some of them below.
\begin{itemize}
\item
We include a self-contained proof (by J. Xie) of his algebraization result for adelic curves (Theorem \ref{thm:Junyi}).
\item
We give the full expansion of the B\"ottcher coordinates for polynomials over a field of characteristic $0$ without assuming it to be
centered or monic (\S \ref{sec:bottcher}).
\item
We study over an arbitrary field the group of symmetries of a polynomial. 
In particular, we give a purely algebraic characterization of this group (Theorem \ref{tm:algebraiccharacterization}).
\item 
We introduce the notion of primitivity in \S \ref{sec:primitivity}, which seems appropriate to exclude tricky examples of entangled pairs.
\item
We give a detailed proof that the height $h_{P,a}(t) = h_{P_t}(a(t))$ attached to any polynomial dynamical pair is
adelic (Proposition \ref{prop:adelic}).
\item
For a family of polynomials $\{P_t\}$ parameterized by an algebraic variety $\Lambda$, we 
consider the preperiodic locus in $\Lambda \times \A^1$ which is a union of countably many algebraic subvarieties.
We study the set of points which are included in an infinite collection of irreducible components of the preperiodic locus (Theorem \ref{thm:finite branches}). 
This result is crucial to our specialization argument to obtain Theorem \ref{tm:unlikely-car0} and clarifies some arguments used in \cite{GNY2}.
\end{itemize}

\medskip

\paragraph*{Open questions and perspectives.}

There are many directions in which our results could find natural generalizations. 

\smallskip

Let us indicate first why the restriction to families of polynomials is crucial for us. Given a family of rational maps $R_t$ parameterized by an algebraic curve $C$, 
and given any marked point $a$, then one can attach to the pair $(R,a)$ a divisor at infinity $\mathsf{D}_{R,a}$ generalizing the definition \eqref{eq:def-div} above. It is not completely clear, however, whether $\mathsf{D}_{R,a}$ has rational coefficients. Some cases have been worked out by DeMarco and Ghioca \cite{demarco_ghioca_2019} but the general case remains elusive. One can next build a natural height by setting $h_{R,a} = h_{R_t}(a(t))$, but it is not completely clear if this height is a Weil height associated to 
$\mathsf{D}_{R,a}$ (in the sense of Moriwaki \cite{MR3498148}). There are instances (see \cite{DWY}) where this height is not adelic. Although a notion of quasi-adelicity has been proposed by Mavraki and Ye \cite{Mavraki-Ye} to get around this problem, a continuity of the metrizations involved in the definition of the canonical height (Theorem \ref{thm:continuity0} in the polynomial case) turns out to be crucial for applications and remains completely open up to now. We also note that Ritt's theory is much less powerful for rational maps leading to weaker classification of curves left invariant by product maps (see \cite{pako-invariant-curve-rational}). 
We also refer to~\cite{Ye} for a characterization of rational maps having the same maximal entropy measure; and to~\cite{Mimar-13} for a version of Theorem~\ref{tm:unlikely}  for constant families of rational maps (but varying marked point).

\smallskip

It would be extremely interesting to look at polynomial dynamical tuples parameterized by higher dimensional algebraic variety $\Lambda$ and prove unlikely intersection statements. The obstacles to surmount are also formidable. It is already unclear whether the canonical height is a Weil height on a
suitable compactification of $\Lambda$. Also in this case, the bifurcation measure is naturally defined as a Monge-Amp\`ere measure of some psh function on $\Lambda$, and dealing with a non-linear operator makes things much more intricate. We refer, though, to the papers by Ghioca, Hsia \& Tucker \cite{MR3632092} and by Ghioca, Hsia \& Nguyen \cite{GHN-higher}
for tentatives to handle higher dimensional parameter spaces using one-dimensional slices. 
\smallskip

Let us list a couple of questions that are directly connected to our work.
\begin{itemize}
\item[(Q1)]
Prove the following purely archimedean rigidity statement:
two \emph{complex} polynomial dynamical pairs $(P,a)$ and $(Q,b)$ 
having identical  bifurcation measures are necessarily entangled. 
One of the problem to get such a statement is to prove the multiplicative dependence of the degrees in this context.
Observe that Theorem \ref{tm:multiplicative} yields this dependence but at the cost of a much stronger assumption.
\item[(Q2)]
Is it possible to remove  condition ($\vartriangle$) and obtain Theorem \ref{tm:unlikely-car0} over any field of characteristic $0$?
\item[(Q3)]
Can one extend Theorem~\ref{tm:unlikely-finiteness} to any field of characteristic $0$?
\item[(Q4)]
Give a classification of special (irreducible) curves in the moduli space of critically marked polynomials in terms of suitable combinatorial data. Ideally, one would like to attach to each special irreducible curve a combinatorial object (like a family of decorated graphs) and prove a one-to-one correspondence between special curves and these objects.
It would also be interesting to study the distribution (as currents) of special curves whose associated combinatorics has  complexity increasing to infinity.
\end{itemize}

Further more specific open problems can be found in the three sections \S\ref{sec:Ritt-open} (related to Ritt's theory), \S\ref{sec:further-open} (on extensions of Theorem~\ref{tm:unlikely}) and \S\ref{sec:classic} (on special curves).

\medskip

\paragraph*{Acknowledgements}

We  warmly thank Junyi Xie and Khoa Dang Nguyen for sharing their thoughtful comments after reading a first version of this work. We express our gratitude to Dragos Ghioca and Mattias Jonsson for their very constructive remarks, and to Gabriel Vigny for many exchanges on the rigidity of complex polynomials that greatly helped improving our discussion on this matter.


\section*{Notations}
~
\begin{itemize}
\item
$\bar{K}$ the algebraic closure of a field $K$
\item
$K^\circ$ the ring of integers of a non-Archimedean metrized field
\item
$\tilde{K}$ the residue field of a non-Archimedean metrized field
\item
$\C_K$ the completion of an algebraic closure of $K$
\item
$\KK$ a number field
\item
$\mathsf{O}(z)$ the set of Galois conjugates over an algebraic closure of $\KK$ of a point $z\in\KK$
\item
$M_\KK$ the set of places of a number field, i.e. of norms extending the usual $p$-adic and real norms on $\Q$
\item
$\KK_v$ the completion of $\KK$ w.r.t $v \in M_\KK$
\item
$\C_v$ the completion of an algebraic closure of $\KK_v$
\end{itemize}
\medskip

\hrule

\medskip
\begin{itemize}
\item
$X$ an algebraic variety
\item
$\mathcal{O}_X$ the structure sheaf of $X$
\item
$X^{\an}$ the analytification in the sense of Berkovich of  $X$
\item
$X^{\an}_v$ the analytification of $X_{\C_v}$ when $X$ is defined over a number field $\KK$ and $v\in M_\KK$
\item
$C$ an algebraic curve
\item
$\bar{C}$ a projective compactification of a curve $C$ such that $\bar{C}\setminus C$ is smooth
\item
$\hat{C}$ the normalization of $\bar{C}$, and $\mathsf{n} \colon \hat{C}\to \bar{C}$ the normalization map
\end{itemize}

\medskip

\hrule

\medskip
\begin{itemize}
\item
$\log^+ = \max \{ 0, \log \}$
\item
$\D=\{ z \in \C, \, |z| < 1\}$ the complex open unit disk 
\item
$S^1$ the unit circle in $\C$
\item
$\U$ the group of complex numbers of modulus $1$
\item
$\U_m = \{ z \in \C, z^m =1\}$ the group of complex $m$-th root of unity
\item
$\U_\infty$ the group of all roots of unity
\item
$\D_K(z,r) = \{ |w-z| <r \}$ the open disk of center $z$ and radius $r$ (as a subset of  either $K$ or the Berkovich affine line $\A^{1,\an}_K$)
\item
$\D^*_K(z,r) = \{ 0<|w-z| <r \}$ the punctured open disk of center $z$ and radius $r$ (as a subset of  either $K$ or the Berkovich affine line $\A^{1,\an}_K$)
\item
$\overline{\D_K(z,r)}= \{ |w-z| \le r \}$ the closed disk of center $z$ and radius $r$ (as a subset of  either $K$ or the Berkovich affine line $\A^{1,\an}_K$)
\item
$\D^N_K(z,r)$ the open polydisk of center $z$ and polyradius $r = (r_1, \cdots, r_N)$ (as a subset of  either $K^N$ or the Berkovich affine space $\A^{N,\an}_K$)
\item
$\overline{\D^N_K(z,r)}$ the closed polydisk of center $z$ and polyradius $r = (r_1, \cdots, r_N)$ (as a subset of  either $K^N$ or the Berkovich affine space $\A^{N,\an}_K$)
\item
$\D_v(z,r), \overline{\D_v(z,r)}, \D^N_v(z,r), \overline{\D^N_v(z,r)}$ the respective open and closed disks in $\KK_v$ if $v$ is a place on a number field $\KK$
\item
$\mathbb{M}_K$: the ring of analytic functions on the punctured unit disk $\D_K^*(0,1)$ that are meromorphic at $0$
\end{itemize}

\medskip

\hrule

\medskip
\begin{itemize}
\item
$\mathcal{C}_c^0(X)$ the space of compactly supported continuous functions on $X$
\item
$\cD(U)$ the space of smooth (resp. model) functions on $U$
\item
$\Delta u$ the Laplacian of $u$
\item
$u,g$ subharmonic functions
\item
$h$ harmonic functions
\item
$o(1), O(1)$: Landau notations
\end{itemize}

\medskip

\hrule

\medskip
\begin{itemize}
\item
$g_P$ the Green function associated to a polynomial $P$ 
\item
$G(P)$ the critical local height of a polynomial (the maximum of $g_P$ at all critical points)
\item
$\varphi_P$ the B\"ottcher coordinate of a  polynomial  $P$ at infinity
\item
$J(P)$ the Julia set  of a polynomial $P$
\item
$K(P)$ the filled-in Julia set of a polynomial $P$
\item
$\mu_P$ the equilibrium measure of a polynomial $P$
\item
$\crit(P)$ the critical set of a polynomial $P$
\item
$\Sigma(P)$ the group of dynamical symmetries of a polynomial $P$
\item
$\aut(P)$ the group of affine transformations commuting with a polynomial $P$
\item
$\aut(J)$ the group of affine transformations fixing a compact subset $J$ of the complex plane
\item
$\preper(P,K)$ the set of preperiodic points \emph{lying in $K$} of a polynomial $P\in K[z]$

\end{itemize}
\medskip

\hrule

\medskip
\begin{itemize}
\item
$\poly^d$ the space of polynomials of degree $d$
\item
$\poly_{\rm mc}^d$ the space of monic and centered polynomials of degree $d$
\item
$\mpoly^d$ the space of polynomials of degree $d$ modulo conjugacy
\item
$\mpoly_{\rm crit}^d$ the moduli space of critically marked polynomials of degree $d$ modulo conjugacy
\item
$\mpair^d$ the moduli space of dynamical pairs of degree $d$ modulo conjugacy
\item
$\stab(P)$ the stability locus of a holomorphic family of polynomials
\item
$T_d$ the Chebyshev polynomial of degree $d$
\item
$M_d$ the monomial map of degree $d$
\end{itemize}

\medskip

\hrule

\medskip
\begin{itemize}
\item
$(P,a)$ a dynamical pair (either holomorphic or algebraic)
\item
$\mathrm{Bif}(P,a)$ the bifurcation locus of a holomorphic polynomial  pair
\item
$g_{P,a}= g_{P_t}(a(t))$ the Green function associated to a holomorphic polynomial pair
\item
$\mu_{P,a}$ the bifurcation measure of a dynamical pair defined over a metrized field
\item
$\preper(P,a)$ the set of parameters $t$ such that the marked point $a(t)$ is preperiodic for $P_t$

\end{itemize}

%
%
%
%
%
%

\chapter{Geometric background} \label{chapter:geom}

We briefly review  some material from analytic and arithmetic geometry.  This includes
the notion of subharmonic functions on analytic curves defined over a non-Archimedean or an Archimedean field; 
the construction of the Laplace operator on the space of subharmonic functions; a discussion of the notion of semi-positive and adelic
metrics on a line bundle over a curve; and the definition of heights attached to an adelic  metrized line bundle.

We then define various moduli spaces of polynomials of interest, and give a proof of an algebraization theorem
of Xie \cite{Xie} for germs of curves defined by adelic series.

This section will be mainly used as a reference for the rest of the book. 


\section{Analytic geometry}\label{sec-analyticgeo}

\subsection{Analytic varieties}
Let $(K, |\cdot|)$ be any complete metrized field. When the norm is non-Archimedean, we let $K^\circ = \{ |z| \le 1\}$ be its ring of integers with maximal ideal  $K^{\circ\circ} = \{ |z| < 1\}$. We write $\tilde{K} = K^{\circ}/ K^{\circ\circ}$ for its residue field, and $|K^*| = \{ |z|, z \in K^*\} \subset \R_+^*$ for its value group.

If  $X = \spec (A)$ is an affine $K$-variety, we denote by $X^{\an}$ its Berkovich analytification. \index{variety!analytic}
As a set, it is given by the Berkovich spectrum of $A$, i.e. the set of all multiplicative semi-norms on  $A$
whose restriction to $K$ is equal to the field norm. We endow it with the topology of the pointwise convergence for which it becomes
a locally compact and locally arcwise connected space. 
The space  $X^{\an}$ is also endowed with a structural sheaf of analytic functions. 

When $K = \C$, we recover the 
complex analytification of $X$ with its standard euclidean topology and the structural sheaf is the sheaf of complex analytic functions.

When $K$ is non-Archimedean, we refer to \cite{berko-book} for a proper definition of the structural sheaf. 
Any point $x \in X^{\an}$ defines a semi-norm $|\cdot|_x$ on $A$, and it is customary to write $|f(x)| = |f|_x$ for any $f\in A$.
The kernel $\ker(x)$ is a prime ideal of $A$ which may or may not be trivial. 
The quotient ring $A/\ker(x)$ is a field and one denotes by $\mathcal{H}(x)$ its completion with respect to the norm induced by $|\cdot|_x$.
It is the complete residue field of $x$. When $\mathcal{H}(x)$ is a finite extension of $K$, then we say that $x$ is a rigid point. 

One can naturally extend the analytification functor to the category of algebraic varieties over $K$ using a standard patching procedure. 
In this context, the GAGA principle remains valid, see \cite[Chapter 3]{berko-book}.

\subsection{The non-Archimedean affine and projective lines}\label{sec:non-archline}
Suppose $(K, |\cdot|)$ is a complete metrized non-Archimedean field. The analytification of the affine line $\A^1$ is the space of multiplicative semi-norms on $K[T]$
restricting to $|\cdot|$ on $K$.
It is topologically an $\R$-tree in the sense that it is uniquely pathwise connected, see~\cite[\S 2]{jonsson-notes} for precise definitions. In particular for any pair of points $x,y\in  \A^{1,\an}_K$ there is a well-defined
segment $[x,y]$. 

We define closed balls as usual $\bar{B}(z_0,r) = \{ z\in K, |z-z_0| \le r \}$. To any closed ball is attached a point $x_B \in \A^{1,\an}_K$ defined
by the relation $|P(x_B)| = \sup_B |P| $ for any $P\in K[T]$. The Gau{\ss} point $\xg$ is the point associated to the closed unit ball $\xg = x_{\bar{B}(0,1)}$. 
It induces the Gau{\ss} norm on the ring of polynomials: $|P(\xg)| = \max\{ |a_i|\}$ with $P(T) = \sum a_i T^i$.

\smallskip

When $K$ is algebraically closed, points in $ \A^{1,\an}_K$ fall into one of the following four categories. 
If the kernel of $x$ is non-trivial, then $x$ is rigid and $x = x_{B(z,0)}$ for some $z \in K$. We also say that $x$ is of type $1$.
When  $x = x_{B(z,r)}$ with $r\in |K^*|$ (resp. $r\notin |K^*|$), we say that $x$ is of type $2$ (resp. $3$).
If $x$ is not of one of the preceding types, then it is of type $4$: one can then show that there exists a decreasing sequence of balls $B_n$ with empty intersection
such that $|P(x)| = \lim_{n\to \infty} |P(x_{B_n})|$ for all $P\in K[z]$, see \cite[\S 1.2]{baker-rumely-book}.

When $K$ is not algebraically closed, and $K'/K$ is any complete field extension, the inclusion $K[Z] \subset K'[Z]$ yields by restriction a canonical surjective and continuous map  $\pi_{K/k}\colon \A^{1,\an}_{K'} \to \A^{1,\an}_{K}$, and the Galois group $\Gal(K'/K)$ acts continuously on $\A^{1,\an}_{K'}$. 

Let  $\C_K$ be the completion of an algebraic closure of $K$.  Then  $\A^{1,\an}_{K}$ is homeomorphic to the quotient of 
$\A^{1,\an}_{\C_K}$ by $\Gal(\C_K/K)$, see~\cite[Corollary 1.3.6]{berko-book}. The group $\Gal(\C_K/K)$ 
preserves the types of points in  $\A^{1,\an}_{\C_K}$, so that we may define the type of a point $x \in  \A^{1,\an}_K$ as the type of any of its
preimage by $\pi_{\C_K/K}$ in $\A^{1,\an}_{\C_K}$. Note that since the field extension $\C_K/K$ is not algebraic in general, it may happen 
that some type $1$ points in $\A^{1,\an}_K$ have trivial kernel hence are not rigid.

\smallskip

Any open subset of the affine line carries a canonical analytic structure in the sense of Berkovich. We shall refrain from defining this notion precisely and
discuss only the case of balls and annuli. 

The Berkovich analytic open unit ball $\D_K(0,1)$ is defined as the space of semi-norms $|\cdot|_x \in \A^{1,\an}_K$
such that  $|x| :=  |T|_x  <1$. Its structure sheaf is the restriction of the analytic sheaf of $ \A^{1,\an}_K$.
Any analytic isomorphism of $\D_K(0,1)$ is given by a power series of the form $\sum_{n\ge 0} a_n T^n$ with $|a_0| <1$, $|a_1| =1$, and $|a_n|<1$ for all $n\ge 2$.

For any $\rho>1$, the standard open annulus $A= A(\rho)$ of modulus $\rho$ is the analytic subset of $\A^{1,\an}_K$ defined by $\{ 1<|x| <e^\rho\}$. 
Any analytic isomorphism of $A$ is given by a Laurent series of the form $\sum_{n\in \Z} a_n T^n$ with $|a_1| =1$
and $|a_1| r > |a_n| r^n$ for all $n\neq 1$ and all $1 < r < e^\rho$, possibly composed with the inversion $a/T$ with 
$e^\rho = |a|$\footnote{The latter automorphism exists only if $e^\rho\in |K^*|$}. 

The skeleton on $A$ is the set of points $\Sigma(A) = \{ x_{B(0,e^t)}, \, 0 < t < \rho\}$.
Any automorphism of $A$ leaves $\Sigma(A)$ invariant so that the skeleton only depends on the analytic structure of $A$.

The projective line $\p^{1,\an}_K$ is homeomorphic to the one-point compactification of $\A^{1,\an}_K$. The point at infinity in $\p^{1,\an}_K$ is rigid/of type $1$.

\subsection{Non-Archimedean Berkovich curves}\label{sec:nonArchcurves}
Let $(K, |\cdot|)$ be any algebraically closed and complete non-Archimedean metrized field. Let $C$ be any smooth connected projective curve defined over $K$. Berkovich~\cite[Chapter 4]{berko-book} proved that
the geometry of the Berkovich analytification $C^{\an}$\index{curve!analytic} can be completely understood using the semi-stable reduction theorem. 
Berkovich's results were further expanded by Baker, Payne and Rabinoff in \cite{MR3204269,MR3455421}. We refer the interested reader to the unpublished monograph of Ducros~\cite{ducros} fo a detailed account on the geometry of any analytic Berkovich curve (not necessarily algebraic). 
Our presentation follows \cite{MR3204269}.

\paragraph*{Models}
A model of $C$\index{curve!model of a Berkovich} is a normal $K^{\circ}$-scheme $\mathfrak{C}$ that is projective and flat over $\spec(K^{\circ})$, 
together with an isomorphism of its generic fiber with $C$. Denote by $\mathfrak{C}_s$ its special fiber: it is a proper scheme defined over the residue field $\tilde{K}$.
The valuative criterion of properness implies the existence of a reduction map $\mathrm{red} \colon C^{\an} \to \mathfrak{C}_s$ which is anti-continuous, see \cite[1.3]{MR3204269}.  

By a theorem of Berkovich \cite[Proposition 2.4.4]{berko-book}, the preimage by $\red$ of the generic point $\eta_E$ of any irreducible component $E$ of $\mathfrak{C}_s$ is a single point in $C^{\an}$, which we denote by $x_E$. Such a point is called of type $2$. This terminology is compatible with the case $C = \p^1$. Indeed any point of the form $\mathrm{red}^{-1}(\eta_E)\in \p^{1,\an}_K$  is of type $2$. Conversely, for any type $2$ point $x\in\p^{1,\an}_K$, there exists a model $\mathfrak{P}$ of $\p^1$ and  a  component $E$ of $\mathfrak{P}_s$ such that $x = x_E$.

If $\bar{x}$ is a closed point in $\mathfrak{C}_s$, then $\mathrm{red}^{-1}(\bar{x})$ is an open subset of $C^{\an}$ whose boundary is finite and consists of those points $x_E$ where $E$ is an irreducible component of the central fiber containing $\bar{x}$, see~\cite[Theorem~4.3.1]{berko-book}.

\smallskip

A model with simple normal singularities (or simply an snc model)\index{curve!snc model of a Berkovich} is a smooth model for which the special fiber is a curve with only ordinary double point singularities, i.e. $\mathfrak{C}$
admits a covering formally isomorphic to $\spf (K^{\circ}\langle x,y\rangle /(xy-a) )$ for some $a \in K^{\circ} - \{0\}$ near any of its closed point, see \cite[Proposition~4.3]{MR3204269}. 
A fundamental theorem of Bosch and L\"utkebohmert \cite[Propositions 3.2 \& 3.3]{MR774362} implies the following result.
\begin{theorem}\label{thm:boschlu}
\begin{enumerate}
\item 
If $\bar{x}$ is a smooth point of $\mathfrak{C}_s$ lying in a component $E$, then $\mathrm{red}^{-1}(\bar{x})$ is analytically isomorphic to the (Berkovich) open unit ball
and its boundary in $C^{\an}$ is equal to $x_E$.
\item
If $\bar{x}$ is an ordinary double singularity of $\mathfrak{C}_s$ and belongs to the two components $E$ and $E'$, then $\mathrm{red}^{-1}(\bar{x})$ is analytically isomorphic to a (Berkovich) open annulus of the form $\{ 1 < |z| < r\}$ for some $r\in |K^*|$, and its boundary is equal to $\{ x_E, x_{E'}\}$.
\end{enumerate}
\end{theorem}

\paragraph*{Skeleta}
The skeleton\index{curve!skeleton of a Berkovich} $\Sigma (\mathfrak{C})$ of an snc model is the union of all points $x_E$ for all components $E$ of the special fiber, together with the union of all skeleta
of the annuli  $\mathrm{red}^{-1}(\bar{x})$ for all singular points of $\mathfrak{C}_s$. The skeleton contains no rigid points.

Since $\mathfrak{C}_s$ is a curve with only ordinary singularities, we may define its dual graph $\Delta(\mathfrak{C})$ whose vertices (resp. of edges) are in bijection with the irreducible components of  $\mathfrak{C}_s$ (resp. with the singular points of $\mathfrak{C}_s$).
The skeleton $\Sigma (\mathfrak{C})$ is a geometric realization of the graph $\Delta(\mathfrak{C})$ in $C^{\an}$.

There a canonical continuous map $\tau_{\mathfrak{C}} \colon C^{\an} \to \Sigma(\mathfrak{C})$, \cite[Definition 3.7]{MR3204269}.
For any irreducible component $E$ of $\mathfrak{C}_s$,  $\tau_{\mathfrak{C}} (x_E)$ is equal to the vertex $ [E]$ associated to $E$. More generally when $\red (x)$ is a smooth point lying in $E$ then 
$\tau_{\mathfrak{C}} (x) = [E]$. Finally when $\red(x)$ is the intersection of two components, then $\tau_{\mathfrak{C}} (x)$ belongs to the edge
joining these two components.

Since open subsets of the affine line are retractible, it follows that  $\tau_{\mathfrak{C}}$ is a retraction, and that $C^{\an}$ is locally modeled on an $\R$-tree.
In particular, for any two points $x\neq y$ there exists a continuous injective map $\gamma \colon [0,1] \to \C^{\an}$ such that 
$\gamma(0) =x$ and $\gamma(1) =y$. Up to reparameterization, there are only finitely many such maps, and when $C^{\an}$ is a tree
this map is unique. The latter occurs iff the dual of some (or any) snc model of $C$  is a graph having no loop.

\paragraph*{Metrics}
Using the $\R$-tree structure of the affine line, one can endow the complement of its set of rigid point 
 $\H(\A^1) = \A^{1,\an} \setminus \A^1(K)$ with a complete metric $d_\H$ as follows.

Suppose first that $x_0, x_1 \in \A^{1,\an}$  are associated to closed balls $x_i= x_{B_i}$ as in the previous section.
If $B_0$ and $B_1$ are not disjoint, then one is contained into the other say $B_0 \subset B_1$, and one sets $d_\H (x_0,x_1) = \log (\diam(B_1)/\diam(B_0) )$.
When the balls are disjoint, we consider  the closed ball of smallest radius $B$ containing $B_0$ and $B_1$, and set 
 \[ d_\H (x_0,x_1) = \log (\diam(B)/\diam(B_0) ) + \log (\diam(B)/\diam(B_1) ) . \]
 It is a fact that this distance extends to $\H(\p^{1}) = \H(\A^{1})$ as a complete metric \cite[\S 2.7]{baker-rumely-book}. 
\begin{lemma}\label{lem:autball}
Any injective analytic map from the open unit ball or from an annulus
to the affine line induces an isometry for $d_\H$.
\end{lemma}
\begin{proof}[Sketch of proof]
Let us treat the case of the open unit ball and take any injective analytic map $f\colon \D_K(0,1) \to \A^{1,\an}_K$.
The image of $f$ is an open ball, since it has at most one boundary point. 
Since any affine automorphism is an isometry for $d_\H$, we may suppose that $f(\D_K(0,1)) = \D_K(0,1)$.
From the explicit description of the automorphism groups of the ball given in \S\ref{sec:non-archline}, we get that 
$f(T) = \sum_n a_n T^n$ with $|a_0| < |a_1| =1$ and $|a_n| <1$ for all $n\ge2$.
The lemma then follows from the following estimation of the diameter of a ball: 
\[ 
\diam (f(B(0,r))) =\max_{n\ge 1} |a_n| r^n~. \]
The case of the annulus is treated analogously. 
\end{proof}
Observe that the metric spaces $(\D_K(0,1)\cap \H(\A^1) , d_\H)$ and $(A(\rho)\cap \H(\A^1),d_\H)$ are not complete. Their completions are respectively equal to 
$(\D_K(0,1) \cup \{ \xg\}, d_\H)$ and $(A(\rho)\cup \{ \xg, x_{\bar{B}(0,e^\rho)}\}, d_\H)$.

\begin{remark}\label{rem:canonicamet}
The preceding observation combined with Lemma~\ref{lem:autball} show that 
$\D_K(0,1) \cap \H(\A^1) \cup \{ \xg\}$ and $(A(\rho)\cap \H(\A^1) \cup \{ \xg, x_{\bar{B}(0,e^\rho)}\})$
are canonically endowed with a complete metric that we shall again denote by $d_\H$.
\end{remark}

\medskip

Let now $C$ be any connected projective smooth curve, and denote by $\H(C) = C^{\an} \setminus C(K)$
the complement of the set of rigid points.
\begin{propdef} 
There exists a unique complete metric $d_{\H,C}$ on $\H(C)$ such that the following holds. There exists $\epsilon >0$ such that 
any analytic embedding $f \colon  U \to C$ where $U$ is either the open unit ball or an annulus of modulus $\le \epsilon$
induces an isometry from $(U,d_\H)$ onto its
image $(f(U), d_{\H,C})$. 
\end{propdef}
When the context is clear we shall drop the index $C$ and simply write $d_\H$ for the metric on $\H(C)$.
\begin{proof}[Sketch of proof]
Choose an snc model $\mathfrak{C}$. We shall build a metric $d_\mathfrak{C}$ on $\H(C)$ and latter show that this metric does not depend on the choice of the model. 

By Bosch and L\"utkebohmert's theorem (Theorem~\ref{thm:boschlu}) the Berkovich curve  $C^{\an}$ is the disjoint union of 
finitely many type $2$ points (those of the form $x_E$ for some irreducible component $E$ of the special fiber), 
of finitely many annuli $A_i$ (the preimages under the reduction map of the singular points of $\mathfrak{C}_s$), 
and (possibly  infinitely many) disjoints open balls $B_j$.

Observe that the closures $\bar{A_i}, \bar{B_j}$ forms a compact covering of $C^{\an}$ and that the intersection of any two
of these pieces is either empty or equal to a point $x_E$ for some $E$ as before. 
By Remark~\ref{rem:canonicamet}, we may endow each piece $\bar{A}_i \cap \H(C)$, $\bar{B}_j\cap \H$
with a canonical metric $d_\H$ for which they become an $\R$-tree.
We then extend the metric to $\H(C)$ by setting
for any two points $x, y \in \H(C)$
\[
d_\mathfrak{C} (x,y) = \inf \{ \length(\gamma)\}\]
where $\gamma$ ranges over all injective continuous
maps $\gamma \colon [o,t] \to \H(C)$ such that $\gamma(o) =x$ and $\gamma(t) =y$. Since each piece of the covering is
an $\R$-tree, one can take the infimum over paths which are parameterized by length, hence this infimum is taken over finitely
many maps and is thus attained.  The metric $d_\mathfrak{C}$ is complete since each piece is. 

Let us now verify  that $d_\mathfrak{C}$ satisfies the expected property that any embedding of a ball or an annulus is isometric. 
Note that this will prove that distance does actually not depend on the choice of the model. 
Take any embedding of the open unit ball $f \colon \D_K(0,1) \to C$. For any $s,t \in \D_K(0,1))$ note that 
any path $\gamma$ minimizing $d_\mathfrak{C} (f(s),f(t))$ is actually included in $ f(\D_K(0,1))$ since the latter has only one boundary point.  
And it follows from Lemma~\ref{lem:autball} that $d_\mathfrak{C} (f(s),f(t))= d_\H(s,t)$ as required.  

Let $\epsilon$ be half of the minimum of the length of loops included in the skeleton $\Delta (\mathfrak{C})$.
If $f \colon A \to C$ is an injective analytic map, and $A$ is an annulus of modulus $\le \epsilon$, then again
any path $\gamma$ minimizing $d_\mathfrak{C} (f(s),f(t))$ has to be included in $A$. We conclude the proof as before.
\end{proof}

\begin{remark}
Another construction of this metric is given in~\cite[\S 5.3]{MR3204269} for details. We refer also to~\cite{ducros} for a construction of the metric in the case of an arbitrary analytic curve.
\end{remark}

The metric topology $(\H(C), d_\H)$ is stronger than the restriction of the compact topology $C^{\an}$ to $\H(C)$. However for any snc model $\mathfrak{C}$, 
the restriction of both topologies to $\Delta(\mathfrak{C})$ coincide.


\section{Potential theory}\label{sec-potential}

\subsection{Pluripotential theory on complex manifolds}
We shall mainly use potential theory on Riemann surfaces. However we need to pass to higher dimensional complex varieties at a few places (most notably to prove 
Theorem \ref{tm:suppT} in Chapter \ref{chapter-pairs}). We review also basic properties of plurisubharmonic (psh) functions, see~\cite{MR2311920} for a general reference. 

\smallskip

Let $C$ be any Riemann surface. We let $\Delta$ be the usual Laplace operator defined on $\mathcal{C}^2$ functions in any holomorphic chart
by $\Delta \varphi = \frac{i}{\pi} \partial_z \partial_{\bar z} \varphi(z,\bar{z})$\index{laplacian}. This operator extends to any $L^1_{\mathrm{loc}}$ function in which case $\Delta \varphi$ is merely
a distribution on $C$. Harmonic functions are those $\mathcal{C}^2$ functions $h \colon C \to \R$ such that $\Delta h =0$.
A subharmonic function\index{function!subharmonic} $u \colon C \to \R \cup\{ -\infty\}$ is a upper semicontinuous function
such that for any harmonic function $h$ on a holomorphic disk $D \subset C$ such that $u \le h$ on $\partial D$ then $u\le h$ on $D$.
For any subharmonic function, $\Delta u$ is a positive measure. Conversely, any $L^1_{\mathrm{loc}}$ function whose Laplacian is a positive measure
is equal a.e. to a (unique) subharmonic function, see \cite[\S 1.6]{MR1045639}. 

For any holomorphic function $f\colon C \to \C$, the function $\log |f|$ is subharmonic. Subharmonic functions are stable by sum, multiplication by a positive constant, 
by taking maximum. Any decreasing limit of a sequence of subharmonic functions remains subharmonic (or is identically $-\infty$).

Let $M$ be any connected complex manifold. To simplify the discussion we shall assume that the dimension of $M$ is $2$ (it covers all our needs for this book). 
A function $u \colon M \to \R \cup\{ -\infty\}$ is said to be psh\index{function!psh} if it is 
upper semicontinuous and for any holomorphic map $\imath \colon \D \to M$ the function $u \circ \imath$  is subharmonic. Again for any holomorphic function $f\colon M \to \C$, the function $\log |f|$ is psh, and psh functions are stable by sum, multiplication by a positive constant, by taking maximum, by composition by holomorphic functions and by decreasing limits.

In a holomorphic chart $(z_1, z_2)$, define the operator 
\[dd^c \varphi = \sum_{\alpha, \beta}  \frac{\partial^2 \varphi}{\partial z_{\alpha} \partial \bar{z}_{\beta}} \, \frac{i}\pi dz_{\alpha} \wedge d\bar{z}_{\beta}
\]
on $\mathcal{C}^2$ functions $\varphi \colon M \to \R$. The form $dd^c \varphi$ is then a real closed form of type $(1,1)$.
We may extend this operator to any $L^1_{\mathrm{loc}}$ functions in which case $dd^c \varphi$ becomes a current\index{current!closed positive} of bidegree $(1,1)$, i.e.
a linear form on the space of $(1,1)$ forms with compact support. For any smooth form $\omega$, we thus have a pairing 
$\langle dd^c \varphi, \omega \rangle \in \C$.

Recall that a smooth real $(1,1)$ form $\omega =  \sum_{\alpha, \beta}  \omega_{ij}  \frac{i}2 dz_{\alpha} \wedge d\bar{z}_{\beta}$ is positive
if the hermitian matrix $(\omega_{\alpha, \beta})_{\alpha, \beta}$ is positive at any point. 
When $u$ is psh, then $dd^c u$ is a positive current in the sense that for any smooth positive real $(1,1)$ form $\omega$, we have
$\langle dd^c u , \omega \rangle \ge 0$. Conversely, any $L^1_{\mathrm{loc}}$ function whose $dd^c$ is a positive current is equal a.e. to a subharmonic function, see \cite[Theorem~4.3.5.2]{MR2311920}.
In a local chart $dd^c u$ can be expressed as
$dd^c u = \sum_{\alpha, \beta} T_{\alpha, \beta}\, \frac{i}2 dz_{\alpha} \wedge d\bar{z}_{\beta}$ where $T_{\alpha, \beta}$ are signed Borel measures such that 
$ \frac{i}2\, \sum_{\alpha, \beta} T_{\alpha, \beta} \lambda_i \bar{\lambda}_j \ge 0$ for any choice of complex numbers $(\lambda_1,\lambda_2)$.

\smallskip

If $N$ is a closed analytic subvariety in $M$ of pure dimension $1$ (e.g. when $N$ is a closed Riemann surface in $M$), then we can define
the following current of integration
\[
\langle [N], \omega \rangle = \int_{\mathrm{Reg} (N)} \omega\]
for any smooth $(1,1)$-form $\omega$. Here $\mathrm{Reg} (N)$ denotes the set of smooth points of $N$, and it follows from a theorem of Lelong that 
$[N]$ is a well-defined closed positive $(1,1)$-current. When $N$ is defined by the vanishing of some holomorphic function
$f \colon M \to \C$ then the Poincar\'e-Lelong formula states that $[N] = dd^c \log |f|$.

Let now $u$ be any psh function such that $u|_N$ is not identically $-\infty$ on any of the irreducible component of $N$.
Then we may define a positive measure  supported on $N$ by setting
\[dd^c u \wedge [N] := \sum_i dd^c (u|_{\mathrm{Reg} (N_i)}) \]
where $N_i$ denotes the irreducible components of $N$.

It is a very delicate issue to prove the convergence $dd^c u_n \wedge [N] \to dd^c u \wedge [N]$
when $u_n$ is a sequence of psh functions converging in $L^1_{\mathrm{loc}}$ to $u$.
This is for instance true when $u_n$, $u$ are continuous and the convergence is uniform.

\subsection{Potential theory on Berkovich analytic curves}\label{sec:pot on curves}
Let $(K, |\cdot|)$ be any algebraically closed complete non-Archimedean metrized field, and let $C$ be any smooth connected projective curve defined over $K$.
For any open subset $U\subset C$, we define the notions of harmonic and subharmonic functions on $U$ and explain how to construct a natural Laplace operator
from the latter space of functions to the space of positive measures. Potential theory on arbitrary Berkovich curves was fully developed in Thuillier's PhD \cite{thuillier}, and we refer to this monograph for more details.

\paragraph*{Model functions}
Let $\mathfrak{C}$ be any snc model.\index{function!model} Recall that the skeleton $\Sigma(\mathfrak{C})$ is a finite graph included in $\H(C) \subset C^{\an}$, and can thus be endowed with a natural distance $d_\H$. Any segment of the skeleton comes with a unique (up to translation and change of direction) parameterization by
a segment of the real line. We may thus define the space $\PL(\mathfrak{C})$ of piecewise affine functions 
on $\Sigma(\mathfrak{C})$ as the space of continuous real-valued functions $h \colon \Sigma(\mathfrak{C}) \to \R$
whose restriction to any segment is affine. Any piecewise affine function is thus determined by its values on the set of vertices
of $\Sigma(\mathfrak{C})$, hence $\PL(\mathfrak{C})$  is isomorphic to $\R^N$ where $N$ is the number of irreducible components of the special fiber.
Define a model function $\varphi \colon C^{\an} \to \R$ as a function of the form  $\varphi = h \circ \tau_{\mathfrak{C}}$ where $\mathfrak{C}$ is any snc model,
and $h \in \PL(\mathfrak{C})$. 

If $U$ is an open subset of $C^{\an}$, we say that $\varphi \colon U \to \R$ is a model function when it has compact support in $U$ and its trivial extension to $C^{\an}$ is a model function. We denote by $\cD(U)$ the space of all model functions: it is an $\R$-algebra which is stable by $\max$. It follows from Stone-Weierstrass theorem that $\cD(U)$ is dense in the space of continuous function on $U$
for the topology of the uniform convergence on compact subsets.

\paragraph*{Subharmonic functions}
Pick any open subset $V$ of $\Sigma(\mathfrak{C})$. Note that $V$ is a countable union of finite metrized graphs
 having a finite number of branched and boundary points. We say that a function\index{function!subharmonic} $h \colon V \to \R$ is subharmonic when
 it is convex and continuous, and for any branched and end point $v \in V$, we have
 \begin{equation}\label{eq:conv graph}
 \sum_{\vec{v}  \in Tv} D_{\vec{v} } h \ge 0. 
 \end{equation}
Some explanations are in order here. If $v$ is any point in $V$, we let $Tv$ be the set of branches at $v$: when $v$ is an endpoint, then $Tv$ is reduced to a singleton whereas $v$ is a branched point precisely when $Tv$ has at least three points.

For any $\vec{v}  \in Tv$, we may fix an isometric embedding $\phi \colon [0, \epsilon) \to V$ such that
$\phi(0) = v$ and $\phi(t)$ belongs to the branch determined by $\vec{v} $ for $t$ small. The isometry condition 
ensures $d_\H (\phi(t), \phi(t')) = |t -t'|$ where $d_\H$ is the metric constructed in \S \ref{sec:nonArchcurves}.
In particular, any two parameterizations coincide on a small neighborhood of $0$.
Our assumption on $h$ to be convex is equivalent to say that  $h \circ \phi$ is convex so that we may define the directional derivative
\[D_{\vec{v} } h  := \left. \frac{d}{dt} \right|_{t=0^+} \!\!\!\!  h \circ \phi\in \R \cup \{ -\infty\}.\]
Observe that \eqref{eq:conv graph} actually implies $D_{\vec{v} } h$ to be finite for all $\vec{v} $.

\begin{definition}
Let $U$ be any open subset of $C^{\an}$. 
A function $u\colon U \to \R\cup\{-\infty\}$ is said to be subharmonic when
it is upper semi-continuous and for all snc models $\mathfrak{C}$, the function
$u|_{\Sigma(\mathfrak{C})\cap U}$ is subharmonic.
\end{definition}

The set $\sh(U)$ of all subharmonic functions on $U$ satisfies the same properties as its complex analog:
it is stable by scaling by positive constants, by taking sums and maxima, by decreasing limits, by composition by analytic functions and by restriction to smaller open subsets. Subharmonic functions satisfy the maximum principle. When $f$ is analytic on $U$, then $\log|f|$ is subharmonic. 

A function $h$ is harmonic when $+ h$ and $-h$ are both subharmonic. 

\paragraph*{The Laplace operator}
For any locally compact topological space $X$, let  $\cM(X)$ be the set of positive Radon measures on $X$, that is positive
linear functional on the space $\mathcal{C}_c^0(X)$ of continuous functions with compact support on $X$. 
We now define a linear operator $\Delta \colon \sh(U) \to \cM(U)$.

To that end we first define for any $\varphi \in \cD(U)$
\[
\Delta \varphi := 
\sum_{v} 
\left(\sum_{\vec{v}  \in Tv} D_{\vec{v} } \varphi \right) \delta_v
.\]
If $\varphi = h \circ \tau_{\mathfrak{C}}$ and $h \in \PL(\mathfrak{C})$ as above, then 
$\Delta g$\index{laplacian!non-archimedean} is a signed atomic measure supported on type $2$ points associated to the irreducible
components of the special fiber of $\mathfrak{C}$. 

\begin{propdef}
Let $U$ be any open subset of $C^{\an}$ and pick  any function $u\in \sh(U)$.

Then there exists a unique positive Radon measure $\Delta u$
such that for any $\varphi \in \cD(U)$ one has
\begin{equation}\label{eq:def laplace}
\int_U \varphi \, \Delta u = \int_U u\, \Delta \varphi
.\end{equation}
\end{propdef}

\begin{proof}[Sketch of proof]
We refer to \cite{thuillier} for a careful construction of $\Delta u$. 
Note that the uniqueness immediately follows from the density of model functions in $\mathcal{C}_c^0(U)$.
Here is one way to proceed for the construction of $\Delta u$.

We first suppose that $U$ is an open subset of $C^{\an}$ which has finitely many boundary points of type $2$ and is an $\R$-tree. Pick any $x_0\in \partial U$.
Define the Gromov product  $\langle x, y\rangle_{x_0} \in \R_+$ as the distance for the metric $d_\H$ between the segment $[x,y]$ and $x_0$.
For any positive measure $\rho$ on $U$, set $g_\rho(x) = \int \langle x, y\rangle_{x_0}\, d\rho(y)$.
One can then show using \cite[Theorem~7.50]{valtree} that the map $\rho \mapsto g_\rho$ is a bijection between the set of 
positive measures of finite mass on $\bar{U}$ and the set of subharmonic functions on $U$ which extends continuously to $\bar{U}$ and have value $0$ at $x_0$.
Denote by $\sh(\bar{U})_0$ this space and by $\sh(\bar{U})$ the set of sums of a function in  $\sh(\bar{U})_0$ and a constant. 
For any $u\in\sh(\bar{U})$, we define $\Delta u$ to be the unique measure such that 
$g_{\Delta u} - u $ is a constant. One can then check that this measure satisfies \eqref{eq:def laplace}.
In particular, the operator we have constructed is local in the sense if $U$ and $V$ are two open sets as above, and if we pick $u \in \sh(\bar{U})$ and
$v \in \sh(\bar{V})$ such that $u =v$ on $U\cap V$ then $\Delta u = \Delta v$ on $U\cap V$.

Now pick any open set $U$. We cover it by a countable family $U_i$ of open sets satisfying the above condition, and 
observe that $u|_{U_i} \in \sh(\bar{U}_i)$ for all $i$ so that one may define
\[ (\Delta u)|_{U_i} := \Delta \left(u|_{U_i} \right).\]
By the previous discussion, this measure is well-defined and satisfies \eqref{eq:def laplace}.
\end{proof}

Let us list a couple of properties of $\Delta u$ without proof. 

\begin{proposition}
Let $u$ be any subharmonic function on an open subset $U$ of $C^{\an}$. 
\begin{enumerate}
\item
For any connected open and relatively compact subset $V \subset U$ such that $\partial V$ is finite, we have
\[
\Delta u (V) =  \sum_{v\in \partial V} D_{\vec{v} } u
\]
where $\vec{v}$ denotes the unique direction at $v$ pointing towards $V$.
\item
Let $\mathfrak{C}$ be any snc model, and suppose that $u=  h \circ \tau_{\mathfrak{C}}$ for some convex and continuous function
$h \colon \Sigma(\mathfrak{C}) \cap U \to \R$.
Then $\Delta u$ is supported on the graph $\Sigma(\mathfrak{C}) \cap U$, and
\[ 
\Delta u = 
\sum_{v\in E} 
\left(\sum_{\vec{v}  \in Tv} D_{\vec{v} } h \right) \delta_v
+ 
\sum_j (\phi_j)_* \frac{d^2 (h \circ \phi_j(t))}{dt^2} 
\]
where $E$ denotes the set of end and branched points of  $\Sigma(\mathfrak{C}) \cap U$, and $\phi_j\colon I_j \to $ is a collection of isometries with $I_j\subset \R$
such that $\phi_j (I_j)\cap \phi_i (I_i)=\varnothing$ for all $i\neq j$ and $\bigcup_j \phi_j (I_j)=  \Sigma(\mathfrak{C})\setminus E$.
\end{enumerate}
\end{proposition}

Note that since $h$ is convex,  the function $h \circ \phi_j$ is also convex so that 
$\frac{d^2 (h \circ \phi_j(t))}{dt^2}$ is a well-defined positive measure. 

\smallskip

The Laplace operator defined above is natural in the sense that it satisfies 
the Poincar\'e-Lelong formula
\[ \Delta \log|f| = \sum_{f(p) =0} \ord_p(f)\, \delta_p\]
for any analytic function $f$ on $U$. 

It is also continuous in the following sense. 
If $u_n, u \in \sh(U)$ are subharmonic functions such that 
$u_n|_{\Sigma(\mathfrak{C})} \to u|_{\Sigma(\mathfrak{C})}$ for any snc model $\mathfrak{C}$, then we have
$\Delta u_n \to \Delta g$.
Finally if $h\in \sh(U)$, then $h$ is harmonic iff $\Delta h =0$.

\paragraph*{Pull-back of measures}\label{sec:pull-back meas}

Let $f \colon C \to C'$ be any regular surjective map between smooth projective irreducible curves $C$ and $C'$ defined over a complete metrized field $K$ (which may be Archimedean or non-Archimedean).
Then $f$ is a finite map, and for any $x\in C^{\an}$ the local ring $\cO_x$ with maximal ideal $\mathfrak{m}_{x}$
is a module of finite type over 
$\cO_{f(x)}$. We may thus define the local degree of $f$ at any point $x$  by setting
\[
\deg_f(x) = \dim_{\kappa(x)} \left(\mathcal{O}_x/\mathfrak{m}_{f(x)}\cdot\cO_x\right)~,\]
where $\kappa(x)$ denotes the residue field $\cO_x/\mathfrak{m}_x$.

When $x$ is a type $1$ point and $K$ is non-Archimedean, or for any $x$ when $K$ is Archimedean, 
then one may find local coordinates at $x$ and $f(x)$ so that $f$ is determined by a power series
$\sum_{i\ge 0} a_i z^i$. In this case, one has $\deg_f(x) = \min\{ i\ge 1, a_i \neq 0\}$.

For any open subset $U$ of $(C')^{\an}$, one can show that 
the integer-valued function
$y \mapsto \sum_{x \in f^{-1} (y) \cap U} \deg_f(x)$ is constant. 
When $U =(C')^{\an}$, one usually writes $\deg(f) = \deg_f((C')^{\an})$ and call it the degree of $f$\footnote{when $K$ is algebraically closed and $f$ is separable, it is the number of preimages of a generic closed point in $C'$}.
We refer to~\cite[\S 6.3]{MR1259429} for a more precise discussion of this notion in the case of Berkovich non-Archimedean analytic curves.

\medskip

For any function $\varphi \colon C \to \R$, set
\[ f_* \varphi (y) = \sum_{x\in f^{-1}(y)} \deg_f(x) \, \varphi(x)~.\]
It is a fact that if $\varphi$ is continuous, then $f_*\varphi$ is also continuous and 
$\sup |f_*\varphi| \le \deg(f) \times \sup |\varphi|$. The proof of this fact is purely local, and the arguments of~\cite[Proposition~2.4]{FRL-ergodique} 
apply verbatim over any complete metrized field.

One can thus define by duality the pull-back of any Radon measure $\mu$ on $C'$ as the unique Radon measure such that 
\[
\int_C \varphi\, d(f^*\mu) = \int_{C'}  (f_*\varphi)\, d(\mu)~.\]
The pull-back measure\index{pull-back of a measure} is positive when $\mu$ is, and
the total mass of $f^*\mu$ is equal to $\deg(f) \times \mass(\mu)$. 

Finally, if $U$ is any open subset of $(C')^{\an}$ over which $\mu|_U = \Delta u$ for some subharmonic function $u\colon U \to \R \cup \{-\infty\}$,
then we have $f^* \mu|_{f^{-1}(U)} = \Delta (u \circ f)$. This identity follows from the Poincar\'e-Lelong formula when $\mu$ is an atomic measure supported at type $1$ points, and we get the general case by continuity and by density of these measures in the space of  positive Radon measures.

\subsection{Subharmonic functions on singular curves}\label{sec:subharmonic on singular-C}

We shall also work on arbitrary singular curves. In this context, one can define the notion of subharmonic functions and its
Laplacian. We restrict ourself to the notion of continuous subharmonic functions for which the theory is better behaved, and
which will be sufficient for our purposes.

Let $C$ be any complex algebraic curve (possibly with some singularities) defined over a metrized field $(K,|\cdot|)$. Let $\reg(C)$ be its regular locus, and $\mathsf{n} \colon \hat{C} \to C$ be its normalization. 
A continuous function $g \colon C \to \R$ is said to be subharmonic whenever its restriction to $\reg(C)$ is subharmonic. 
Since any bounded subharmonic function\index{function!subharmonic} on the punctured disk extends through the origin (see e.g.~\cite[Lemma~3.7]{specialcubic} and the reference therein), it follows that a continuous function$g \colon C \to \R$ is subharmonic iff $g \circ \mathsf{n}$ is subharmonic.\index{laplacian!}

Let $g \colon C \to \R$ be any continuous subharmonic function. Then $\Delta g$ is defined as the trivial extension to $C$ of the 
positive measure $\Delta (g|_{\reg(C)})$. Since the Laplacian of a bounded subharmonic function does not charge closed points, see e.g.~\cite[Lemme~2.3 \& Lemme~4.2]{FRL}, 
$\Delta g$ is also equal to $\mathsf{n}_* \Delta (g \circ \mathsf{n})$.


\section{Line bundles on curves}

\subsection{Metrizations of line bundles}\label{sec:metric line}
We refer to \cite{ACL2} for more details. 

Let $C$ be any algebraic curve defined over a complete metrized field $(K, |\cdot|)$. 
A line bundle $L \to C$\index{line bundle} is an invertible sheaf on $C$. Since $C$ is a curve, one can always find a divisor $D$ such that
$L = \cO_C(D)$. When $C$ is complete, we define the degree of $L$ as the degree\index{line bundle!degree} of any of its defining divisor $\deg_C(L) = \deg_C(D)$.
To simplify notation, we still denote by $L$ the analytification of the line bundle over $C^{\an}$.

A continuous metrization $\|\cdot\|$ \index{metric} on  $L \to C$ is the data for each local analytic section $\sigma$ defined over an open subset $U\subset C^{\an}$
of a continuous function $\|\sigma\|_U \colon U \to \R_+$ such that: 
\begin{itemize}
\item
$\|\sigma\|_U$ vanishes only at the zeroes of $\sigma$;
\item
the restriction of $\|\sigma\|_U$ to any open subset $V\subset U$ is equal to $\|\sigma\|_V$; 
\item
for any analytic function $f$ on $U$, one has $\|f\,\sigma\|_U = |f| \times \|\sigma\|_U$.
\end{itemize}
A local frame on $U$ is a section $\sigma$ of the line bundle over $U$ which does not vanish. Any local frame induces 
a local trivialization, and the identity $\|f\,\sigma\|_U = |f| \times \|\sigma\|_U$ implies that one can write the metrization
over $U$ under the form $|\cdot| e^{- \varphi}$ for some continuous function $\varphi$.
In particular, two metrizations $\|\cdot\|_1, \|\cdot\|_2$ of the same line bundle $L$ differ by a multiplicative function $\|\cdot\|_1 = \|\cdot\|_2 e^{- \varphi}$ with 
$\varphi \colon C \to \R$.

Let $f \colon C ' \to C$ be any morphism between two algebraic curves.
If $L\to C$ is a line bundle, recall that one may define $f^* L\to C'$ as the line bundle whose local sections
over $f^{-1}(U)$ are given by sections of $L$ over $U$ so that for any $\sigma \in H^0(f^{-1}(U), f^*L)$
there exists $\sigma' \in H^0(U, L)$ such that $\sigma = \sigma' \circ f$.
We may thus transport any metric $|\cdot|_L$ on $L$, by imposing $|\sigma|_{f^*L} := |\sigma'|_L \circ f $.

\smallskip

When $K$ is Archimedean and $C$ is smooth, one can make sense of smooth (resp. $\mathcal{C}^k$, H\"older) metrics. In a local chart there are given by 
 $|\cdot| e^{- \varphi}$ with $\varphi$ smooth (resp. $\mathcal{C}^k$, H\"older).

\smallskip

In the non-Archimedean case, the notion of smooth metrics is not really relevant. Following Zhang \cite{MR1311351}, one defines instead the notion of model metrics.
\index{metric!model}
A model of the line bundle $L\to C$ is the choice of 
a model $\mathfrak{C}$ of $C$ together with a line bundle $\mathfrak{L} \to \mathfrak{C}$  whose restriction to the generic fiber is equal to $L$. When $L = \cO_C(D)$ is determined by a divisor $D$ on $C$, then $\mathfrak{L}$ is determined by a divisor
$\mathfrak{D}$ on $\mathfrak{C}$ whose restriction to the generic fiber is equal to $D$.

Any model $\mathfrak{L} \to \mathfrak{C}$ gives rise to a metrization of $L$ as follows.
Cover $\mathfrak{C}$ by affine charts $\mathfrak{U}_i = \spec(B_i)$ for some finitely generated $K^{\circ}$-algebras $B_i$. 
For each $i$, the space $\bar{U}_i$ of bounded multiplicative semi-norms on $B_i \otimes_{K^{\circ}} K$ that restrict to $|\cdot|$ on $K$ 
is compact, and the $\bar{U}_i$'s form a compact cover of $C^{\an}$.

Choose any invertible section $\sigma$ of $\mathfrak{L}$ on $\mathfrak{U}_i$.
Observe that any other local frame of $\mathfrak{L}$ over $\mathfrak{U}_i$ can be written as $\sigma' = \sigma \times h$ with $h\in B_i$ being invertible
so that $|h| =1$ on $\bar{U}_i$. One can thus define a continuous metric on $L$ by imposing
$\| \sigma \| = 1$ on $\bar{U}_i$. 

When $|\cdot|_1$ and $|\cdot|_2$ are two model metrics of the same line bundle, then $|\cdot|_1 = |\cdot |_2 e^{-\varphi}$ for some model function $\varphi$ in the sense of \S \ref{sec:pot on curves}.

\smallskip

Model metrics arise in practice by the following token. Let $\cF = \{f_1, \cdots, f_N\}$ by any non-empty finite set of non-constant
meromorphic functions on $ C$. Let $D_\cF$ be the effective divisor on $C$ such that 
\[ \ord_p(D_\cF) = \max \{0, -\ord_p(f_1), \cdots, -\ord_p(f_N)\}\]
for any $p\in C$.
Observe that any $f_i$ induces a section of $L_\cF:= \cO_C(D_\cF)$, 
so that $C$ can be covered by a family of charts $U_i$ such that $f_i$ is invertible on $U_i$, and $L_\cF$ is globally generated.
Define the function $g_\cF\colon C^{\an} \to \R \cup\{+\infty\}$ by
\[ g_\cF = \log^+ \max \{ |f_1|, \cdots, |f_N|\}.\]
 Any section   of $L_\cF$ is determined by a rational function $\sigma$ on $C$ 
 such that $\dv(\sigma) + D_\cF \ge0$ so that
 \[|\sigma|_\cF := |\sigma| \times e^{- g_\cF}\]
 defines a continuous metric on $L_\cF$.
 \begin{lemma}\label{lem:model metric}
 Any metric of the form  $|\cdot|_\cF$ is a model metric.
 \end{lemma}
 \begin{proof}
The sections induce a map $\Phi \colon C \setminus \supp (D_\cF)\to \p^N$ given in homogenous coordinates by 
$p\mapsto [1:f_1(p) :  \cdots :  f_N(p)]$ which extends to a regular map through the punctures, and such that $L_\cF = \Phi^* \cO_{\p^N}(1)$.
Recall that sections of $\cO_{\p^N}(1)$ are in bijection with linear forms in $(N+1)$ variables $Z_0, \ldots, Z_N$ so that 
each $f_i$ corresponds to $Z_i$.  
It follows that $|\cdot|_\cF$ is the metrization obtained by pulling-back the standard metrization on $\cO_{\p^N}(1)$
given by 
\[
|\sigma_u| = \frac{ u(Z_0, \cdots, Z_N)}{\max \{|Z_0|, \cdots, |Z_N|\}} 
\]
where $\sigma_u$ is the section associated to the linear form $u=u(Z_0, \ldots, Z_N)$.
The latter metric is a model metric arising from the standard model $\p^N_{K^{\circ}} = \proj (K^{\circ}[Z_0, \cdots, Z_N])$. 
 \end{proof}

\subsection{Positive line bundles}
Let $C$ be any complete algebraic curve defined over a complete metrized field $(K, |\cdot|)$, and 
let $L \to C$ be any line bundle. 

A metrization $(L,\|\cdot\|)$ is said to be semi-positive \index{metric!semi-positive} iff in any local chart
the  metric can be written under the form $|\cdot| e^{- \varphi}$ with $\varphi$ subharmonic. 
Since $\log | f|$ is harmonic
for any invertible analytic function $f$, this notion of positivity is independent on the choice of trivialization.

Let $\C_{K}$ be the completion of an algebraic closure of $K$. Observe that the notion of semi-positive (resp. model) metric
is stable by base change.  One may thus define  the curvature form $c_1(L, \|\cdot \|)\in \cM(C^{\an}_{\C_{K}})$ of a semi-positive metrization by setting
\[c_1(L, \|\cdot \|)|_U := \Delta \varphi\]
in any open set of trivialization $U$  where the metric writes  $\|\cdot\|= |\cdot| e^{- \varphi}$.
The curvature form is a positive measure of total mass $\deg_C(L)$ (the proof of this fact follows from the Poincar\'e-Lelong formula).

\smallskip

In the Archimedean case, $c_1(L, \|\cdot \|)$ is a smooth measure when the metric is smooth. 
In the non-Archimedean case, it is an atomic measure supported at type $2$ points when $\|\cdot\|$ is a model metric.
\begin{lemma}\label{lem:semi-positive model metric}
 Any metric of the form  $|\cdot|_\cF$ is semi-positive.
 \end{lemma}
 \begin{proof}
Indeed, any function of the form $\log \max \{ |f_0|, \cdots, |f_N|\}$ is subharmonic off its poles, and bounded subharmonic functions on a punctured disk extend
through the puncture, see e.g.~\cite[Lemma~3.7]{specialcubic}.
 \end{proof}

Let $f \colon C ' \to C$ be any finite morphism between two algebraic curves.
Since subharmonic functions are stable by composition by analytic maps, the metric $|\sigma|_{f^*L}$ is semi-positive
as soon as $|\sigma|_{L}$ is positive, and the curvature forms satisfy $f^* c_1(L, |\cdot|_L)=  c_1(f^*L, |\cdot|_{f^*L})$.

Finally in the non-Archimedean case, model metrics are preserved by pull-back since for any model $\mathfrak{C}$ there exists
a model $\mathfrak{C}'$ and a regular map $\mathfrak{f} \colon \mathfrak{C}' \to \mathfrak{C}$ which is equal to $f$ on the generic fiber
(to build $\mathfrak{C}'$ start with any model $\mathfrak{C}_0'$ of $C'$ and take the graph of the induced rational map  
 $\mathfrak{C}_0' \dashrightarrow \mathfrak{C}$).


\section{Adelic metric, Arakelov heights and equidistribution}\label{sec:adelic height}
A general reference for this section is \cite{ACL2}.

\subsection{Number fields}
Fix any number field  $\KK$, and denote by  $M_\KK$ its set of places, that is the set of multiplicative norms on $\KK$
whose restriction to $\Q$ is equal to either the standard euclidean norm $|\cdot|_\infty$ or to one of the $p$-adic norms $|\cdot|_p$
normalized by $|p|_p = \frac1p$. 

Given $v\in M_\KK$, we write $\KK_v$ for the completion of $\KK$ w.r.t. $|\cdot|_v$, and we let $\C_v$ be the completion of an algebraic closure of $\KK_v$. We also let $\Q_v$ be the completion
of the restriction of $|\cdot|_v$ to the prime field.
Then for any $x\in \KK$, the following product formula holds:
\[ 
\prod_{v\in M_\KK} |x|_v^{n_v} =1
\]
where $n_v = [\KK_v : \Q_v]$.

\subsection{Adelic metrics}
Any line bundle $L \to C$ over an algebraic curve $C$ defined over $\KK$ determines a line bundle
over the base change of $C$ by $\KK_v$. To simplify notation, we let $C_v$ be the Berkovich analytification of $C$ over $\KK_v$, and
denote by $L_v\to C_v$  the induced line bundle.

Recall from \S \ref{sec:metric line} the definition of $\KK^{\circ}$-model of the line bundle $L\to C$.
Observe that any $\KK^{\circ}$-model determines a $\KK^{\circ}_v$-model of the line bundle $L_v$ for any $v$, hence
a continuous metric $|\cdot|_{\mathfrak{L},v}$ over $L_v$. 

\smallskip

A semi-positive adelic metric \index{metric!adelic} on an ample line bundle $L \to C$ defined over $\KK$
is the data for each place $v \in M_\KK$ of a semipositive continuous metric $\|\cdot\|_v$ on  $L_v \to C_v$ 
such that there exists a model $\mathfrak{L} \to \mathfrak{C}$ of $L\to C$ over $\KK^{\circ}$ satisfying
$|\cdot|_v = |\cdot|_{\mathfrak{L},v}$ for all but finitely many places.

A simple adaptation of the proof of Lemma~\ref{lem:model metric} together with Lemma~\ref{lem:semi-positive model metric} yields 
\begin{lemma}\label{lem:adelic model metric}
 Any metric of the form  $|\cdot|_\cF$ is semi-positive and adelic.
 \end{lemma}
  
We simply write $\bar{L}$ to indicate that we have fixed a semi-positive adelic metric on an ample line bundle $L \to C$. 

If $\bar{L}$ is a semi-positive adelic metric on $L\to C$, and $f \colon C' \to C$ is a finite map
then the pull-back metrized line bundle $f^*\bar{L}$ is also adelic and semi-positive.

\subsection{Heights}
Let $\bar{\KK}$ be an algebraic closure of $\KK$, and suppose that $C$ is projective.
Then a semi-positive adelic metric $\bar{L}$ induces a height function \index{height} $h_{\bar{L}}:C(\bar{\KK})\to\R$ as follows. 

For any point $t \in C(\bar{\KK})$, we denote by $\mathsf{O}(t)\subset C(\bar{\KK})$
its orbit under the absolute Galois group of $\KK$, and write $\deg(t):=\Card(\mathsf{O}(t))$.
We then choose any rational section $\sigma$ of $L$ which has neither a zero nor a pole at $t$, and 
we set
\begin{align*}
h_{\bar{L}}(t)
&:=
\frac1{\deg(t)} \sum_{t'\in\mathsf{O}(t)} \sum_{v\in M_\KK}  - \log |\sigma|_{v}(t')~.
\end{align*}
Since the metrization is adelic, for all but finitely many terms $|\sigma|_{v}(t') =1$, hence the sum is well-defined.
It follows from the product formula that the definition does not depend on the choice of $\sigma$.

\smallskip

Look at the affine space $\A^1= \spec (K[x])$ and consider its completion $\p^1$. 
Endow $\cO(1) \to \p^1$ with its canonical metrics given by $|\cdot| \max \{ 1, |x|\}^{-1}$ at all places.
The induced metric is adelic and semi-positive, and 
the associated height is the standard height\index{height!standard} on $\p^1$ so that for any $x\in \bar{\Q}$ we have
\begin{equation}
h_{\st}(x):=\frac{1}{\deg(x)}\sum_{y\in \mathsf{O}(x)}\sum_{v\in M_\Q} \log^+|y|_v~.\label{eq413}
\end{equation}
One can alternatively define the height of $x$ by fixing a number field $\KK \ni x$
and set
\begin{equation}
h_{\st}(x):=\frac{1}{[\KK:\Q]} \sum_{v\in M_\KK}n_v\log^+|x|_v. \label{eq412}
\end{equation}

\smallskip

Let us return to our general context with $\bar{L}$ a semi-positive adelic metric on a line bundle $L \to C$.
The function $h_{\bar{L}}$ 
lies in the class of functions associated to $L$ and given by Weil's machinery (see~\cite[Theorem B.3.2]{Silvermandiophantine}). 
In the sequel, we shall call any such height an Arakelov height\index{height!Arakelov}.
In particular, $h_{\bar{L}}$ is always bounded from below, and satisfies the \emph{Northcott property}: for any integer $d\ge 1$, and for any real number $A$, the set of points $t\in C(\bar{\KK})$ such that $\deg(t) \le d$ and $h_{\bar{L}}(t) \le A$ is finite. 

\smallskip

The height\index{height!of a curve} of the total curve is defined as the following quantity
\[h_{\bar{L}}(\bar{C}) = 
\sum_{v\in M_\KK} \sum_{p\in C(\bar{\KK})} \ord_p(\sigma_1)\,  \log\|\sigma_0(p)\|_v^{-1}
 + 
  \int_{\hat{C}} \log\|\sigma_1\|_v^{-1} c_1(L, \|\cdot\|_v)
\]
where $\sigma_0$ and $\sigma_1$ are two sections of $L$ having disjoint sets of zeroes and poles. Again by the product formula this definition does not depend on 
the choice of sections.

It follows from the arithmetic Hilbert-Samuel theorem the following fundamental estimate~\cite[Th\'eor\`eme~4.3.6]{thuillier}, \cite[Proposition~3.3.3]{MR1810122}, or ~\cite[Theorem 5.2]{zhang}.
\begin{theorem}\label{thm:HSA}
Let $\bar{L}$ be any adelic semi-positive continuous metrization on $L\to C$. 
Then for any sequence of distinct points $x_n \in C(\bar{\KK})$ we have
\[\liminf_n h_{\bar{L}}(x_n) \ge \frac{h_{\bar{L}}(C)}{2\deg(L)}~.\]  
\end{theorem}

\subsection{Equidistribution}
In some situation, it is possible to understand the repartition of those points whose height tends to  the limit $h_{\bar{L}}(C)$.
 This is the content of the following result which plays a crucial role in any approach to the dynamical Andr\'e-Oort conjecture. 

If $\bar{L}$ is a semi-positive adelic metrization of $L$, and $v$ is any place on $\KK$, then  the line bundle $L_v$ is endowed with a continuous semi-positive metrization $|\cdot|_v$, and we may look at the curvature of $(L_v, |\cdot|_v)$ in the sense of the previous section. It is a positive measure on $C_v$ of total mass $\deg_C(L)$ which we denote by 
$c_1(\bar{L})_v$ to simplify notations. Observe that if a line bundle carries a semi-positive metric of non-zero curvature then it is automatically ample. 

\begin{theorem}[Equidistribution of points of small heights]\label{tm:yuan}\index{theorem!equidistribution of points of small heights}
~

Let $\bar{L}$ be any semi-positive adelic metrization of  a line bundle $L \to C$  over an irreducible projective curve defined over a number field $\KK$. Pick any sequence of distinct points $x_n\in C(\bar{\KK})$ such that  $h_{\bar{L}}(x_n)\to h_{\bar{L}}(C)$. 
Then, for any place $v\in M_\KK$, we have
\[\frac{1}{\deg(x_n)}\sum_{y\in\mathsf{O}(x_n)}\delta_y\longrightarrow \frac{1}{\deg_C(L)}c_1(\bar{L})_v\]
 on $C_v$ in the weak topology on the space of probability measures.
\end{theorem}

This theorem originated in the work of Szpiro-Ullmo-Zhang on the Manin-Mumford conjecture \cite{MR1427622} and was first proved in the case of abelian varieties, see also \cite{MR1470340} for tori. It was successively extended to the case of curves by Autissier \cite{MR1810122}, Baker-Rumely \cite{MR2244226}, Favre-Rivera-Letelier \cite{FRL}
and the statement above was finally obtained by Thuillier \cite{thuillier}. 

A far-reaching generalization of the previous theorem was later proved by Yuan \cite{yuan}
in any dimension.


\section{The parameter space of polynomials}\label{sec:spaceofpoly}

In this section we assume the defining field $K$ has characteristic zero.
Recall that a polynomial $P(z) = a_0 z^d + \cdots + a_d$ of degree $d$  is monic (resp. centered) if 
$a_0=1$ (resp. $a_1=0$).\index{polynomial!centered}
\index{polynomial!monic}

\paragraph*{Polynomials modulo affine conjugacy}
A polynomial of degree $d\ge2$ is determined by $(d+1)$ coefficients
$P(z) = a_0 z^d + \cdots + a_d$ with $a_0$ invertible so that 
the space $\poly^d$ of all polynomials of degree $d\ge2$  is canonically endowed with a structure
of affine variety which is isomorphic to $(\A^1)^* \times \A^d$.
The group $\Aff =\{ az +b, \, a\neq 0\}$ of affine transformations of the affine line acts by conjugacy on $\poly^d$
by $\phi \cdot P = \phi \circ P \circ \phi^{-1}$.  

In characteristic zero, any polynomial is conjugated by a unique translation to a centered polynomial so that over $\Q$
 the quotient of $\poly^d$ by $\Aff$
is isomorphic to the quotient $(\A^1)^* \times \A^{d-1}$ by the multiplicative group $\G_m$
under the action \[\lambda \cdot (a_0, a_2, \ldots, a_{d-1}, a_d) = (\lambda^{1-d} a_0 ,\lambda^{3-d} a_2, \ldots, a_{d-1}, \lambda a_d).\] 
Over an algebraically closed field a polynomial is conjugated by a suitable dilatation to a monic polynomial. It follows that  the quotient of $\poly^d$ by $\Aff$
is isomorphic to the space of monic and centered polynomials (which is isomorphic to $\A^{d-1}$) quotiented by the finite cyclic group 
$\U_{d-1}$ of $(d-1)$-th root of unity acting diagonally on $\A^{d-1}$ by $\zeta \cdot (a_2, \ldots, a_{d-1}, a_d) = (\zeta^{3-d} a_2, \ldots, a_{d-1}, \zeta a_d)$. 

The moduli space of polynomials $\mpoly^d$ thus exists a geometric group quotient and is an affine variety over $\Q$. 
It can in fact be identified with the product of $\A^1$ with an affine open subset of the weighted projective space
$\p(1, \ldots, d-1)$. In particular, it is an affine variety of dimension $(d-1)$ which is rational and 
has only cyclic quotient singularities\footnote{the corresponding statement for the moduli space of rational maps is due to Silverman and Levy, see~\cite{Silverman-spacerat,Levy-spaceendo}}.

\begin{example}
We have the isomorphisms $\mpoly^2  \simeq \A^1$;  and $\mpoly^3  \simeq \A^2$. 
However for any $d\ge 4$, the space $\mpoly^d$ admits  singularities. 

The  space $\mpoly^4$ is isomorphic to $\A^3$ modulo the action of $\U_3$ given by 
$\zeta \cdot (a_2, a_3, a_4) = (\zeta^{-1} a_2, a_3, \zeta a_4)$ which is the product of $\A^1$ by the cone
$ xy = t^3$.  Its singular locus is the image under the quotient map 
of the set of polynomials of the form $z^4 + a_3 z$. 
 \end{example}

\begin{remark}
When the characteristic of the field say $p>0$ divides the degree $d$, the discussion above does not apply since a polynomial is no longer conjugated to
a centered polynomial. In fact the action of the affine group becomes quite wild. When $p=2$ the stabilizer of \emph{any} separable quadratic polynomial 
$a_0 z^2 + a_1 z + a_2$ is equal to the group of translations $z+ \beta$ with $ a_0 \beta^2  - a_1 \beta = \beta$ which is always non-trivial.
\end{remark}

  \begin{lemma}\label{lem:at infty}
 Let $(K,|\cdot|)$ be any complete metrized field of characteristic $0$, and $\{P_t\}_{t\in \D^*_K(0,1)}$ be an analytic family of monic and centered polynomials defined over the punctured disk 
 that is meromorphic at $0$. 
Suppose that we can find a meromorphic family of affine transformations $A_t$ such that $A_t^{-1} \circ P_t \circ A_t$ 
extends analytically through $0$. 

Then the family $P_t$ is analytic at $0$.
  \end{lemma}
\begin{proof}
Write $P_t(z) = z^d + a_2(t)z^{d-2} + \cdots + a_d(t)$, and $A(t) = a(t) z + b(t)$, with $a,b, a_i$ analytic on $\D^*_K(0,1)$ and meromorphic at $0$.
Then $A_t^{-1} \circ P_t \circ A_t (z) = a(t)^{d-1} z^d + d a(t)^{d-2} b(t) z^{d-1} + \mathrm{l.o.t}$ so that the result follows.
\end{proof}

\paragraph*{Critically marked polynomials}
Even though $\mpoly^d$ is the most natural parameter space to consider, it is also important to work with polynomials with additional structures. 
A critically marked polynomials is a $d$-tuple $(P, c_0, \ldots, c_{d-2})$
with $P\in \poly^d$ and where $c_0, \ldots , c_{d-2}$ ranges over all critical points of $P$ (written with repetitions taking into account their multiplicities).  
The space of critically marked polynomials $\mpcrit^d$ is the quotient of this space
by the natural action of the group of affine transformations. A brief discussion of the geometry of this space is given in~\cite[\S 5]{favredujardin}. 

It is convenient to work in a finite ramified cover of $\mpcrit^d$ which is isomorphic to the affine space, i.e. with an "orbifold" parametrization of $\mpcrit^d$. 
For any field $K$, and any $(c,a)\in K^{d-1}$, we let
\begin{equation}\label{eq:defpoly}
P_{c,a}(z)\pe \frac{1}{d}z^d+\sum_{j=2}^{d-1}(-1)^{d-j}\sigma_{d-j}(c)\frac{z^j}{j}+a^d,
\end{equation}
where $\sigma_j(c)$ is the monic symmetric polynomial in $(c_1,\ldots,c_{d-2})$ of degree $j$.
Observe that the critical points of $P_{c,a}$ are exactly  $c_1, \ldots , c_{d-2}$, $c_{d-1}$ with the convention that $c_{d-1}=0$. One obtains a canonical projection $\pi\colon \A^{d-1}\to \mpcrit^d$ defined over $\Q$ which maps $(c_1,\ldots,c_{d-2},a) \in\A^{d-1}$ to the class of $P_{c,a}$ in $\mpcrit^d$ which is $d(d-1)$-to-one.
  
 \paragraph*{Polynomial dynamical  pairs}
One can also look at the space $\pair^d$ of polynomial dynamical pairs\index{dynamical pair!polynomial} $(P,a)$ with $P\in \poly^d$ and $a \in \A^1$ modulo the natural action of $\Aff$ given by $\phi \cdot (P,a) = (\phi \circ P \circ \phi^{-1}, \phi(a))$.  The structure of the quotient space $\mpair^d$ is similar to $\mpoly^d$: it is the product of $\A^1$ by an open affine subspace of $\p(1,2, \ldots, d-1, d-1)$, and therefore is a $d$-dimensional affine variety defined over $\Q$.
We have a natural submersion $\pi\colon \mpair^d \to \mpoly^d$. In the complex analytic category this map is an orbifold line bundle in the sense of~\cite[\S 2]{MR2819757}. 


\section{Adelic series and Xie's algebraization theorem}\label{sec:adelicseries}

Let $\KK$ be a number field, and $S$ a finite set of places of $\KK$ containing all archimedean ones. We denote by $\cO_{\KK,S}$ the ring of $S$-integers of $\KK$, 
i.e. the set of $x\in \KK$ such that $|x|_v \le 1$ for all $v \notin S$.

\paragraph*{Adelic series}
A power series $a(t) = \sum_{j\ge 1} a_j t^j$ is said to be adelic\index{adelic!series} if $a_j \in \cO_{\KK,S}$ for all $j$, and 
the radius of convergence $\rho_v$ of $a(t)$ is positive for each place $v \in M_\KK$. It is sufficient to impose $\rho_v >0$ for all $v \in S$
since $a(t)$ is analytic in the open unit disk for all $v \notin S$.

The set $\cO_{\KK,S}\{t\}$ of all adelic series is a $\cO_{\KK,S}$-module which is stable by products (hence is a ring), by quotients
by an adelic series $a(t)$ satisfying $a(0) \neq 0$, and by composition. 

For any adelic series $a(t) = \sum_{j\ge k} a_j t^j$ with $k\ge1$ and $a_k \neq 0$, there exist two adelic series $\rho$ and $\theta$ such that 
$a (t) = \rho(t)^k$, and $a(\theta(t)) = t^k$, see~\cite[Lemmas 3.2 $\&$ 3.3]{specialcubic}. 
In particular, any adelic series with $a_1 \neq 0$ is invertible.

\paragraph*{Adelic arcs and branches}
Let $X$ be any projective variety defined over $\KK$. Choose any projective model $\mathfrak{X}$ of $X$ over $\cO_{\KK,S}$.
An \emph{adelic arc} on $X$ is a $\cO_{\KK,S}\{t\}$-point in $\mathfrak{X}$.

When $X$ is embedded into $\p^N$ and $\mathfrak{X}$ is the closure of $X$ in the standard model of $\p^N$, 
so that $X= \cap_{i\in I} (P_i =0)$ for a collection of homogeneous polynomials in $(N+1)$ variables with coefficients in $ \cO_{\KK,S}$, then
an adelic arc\index{adelic!arc} $\gamma$ on $X$ is determined in homogeneous coordinates by $N+1$ adelic series $x_i \in \cO_{\KK,S}\{t\}$
such that $P_I(x_0(t), \cdots, x_N(t)) \equiv 0$ for all $i$, and
 $(x_0(0), \cdots , x_N(0)) \neq  (0)$. 
The point $\gamma(0)= [x_1(0):\cdots :x_N(0)]$ is the origin of the arc and belongs to $X(\KK)$.

Pick any place $v\in M_\KK$. An adelic arc $\gamma$ is defined by convergent power series, so it induces a natural analytic map
from $\gamma_v\colon \D_v(0,R_v(\gamma))\to X^{\an}_v$ where $R_v(\gamma)$ is the minimum of the radii of convergence of the 
series $x_i(t)$ determining $\gamma$. In particular, we have $R_v(\gamma)>0$ for all $v$, and $R_v(\gamma) =1$ for all but finitely many places.
Observe that the series converge only in open disks in general.

Let  $\gamma$ be any adelic arc, and $\theta$ be an invertible adelic series such that $\theta(0) =0$. 
Then $\gamma \circ \theta$ is an adelic arc, and we say that it is obtained by reparameterizing $\gamma$. 
An adelic branch\index{adelic!branch} $\fs$ is an equivalent class of adelic arcs modulo reparameterization. Note that the origin of a branch 
is well-defined.

\smallskip

\paragraph*{Adelic arcs on curves}

Suppose $C$ is an algebraic curve defined over $\KK$.
Any adelic arc $\gamma$ on $C$ centered at a point $p$ defines a formal arc $\gamma \in C(\KK[[t]])$, and 
induces a morphism of local rings $\widehat{\mathcal{O}_{C,p}} \to \KK[[t]]$. 

\begin{lemma}\label{lem:adelic exist}
For any smooth algebraic curve $C$ defined over $\KK$, and any $p \in C(\KK)$, there exists an adelic arc $\gamma$ originated at $p$
inducing an isomorphism of local rings $\widehat{\mathcal{O}_{C,p}} \simeq \KK[[t]]$. 
\end{lemma}
\begin{proof}
An argument is given in~\cite[Lemma~7]{ingram2010}. An alternative proof goes as follows.  We may pick an immersion of $C$ into $\p^2$ such that the image of $p$ is a smooth point. 
Locally in affine coordinates $(x,y)$ the curve can be defined by a polynomial of the form $f(x,y) = y + h(x,y)$ where $h = O(x,y)^2$. 
By the analytic implicit function theorem,  one can find a (unique) power series $\theta (t) = \sum_{k\ge 2} \theta_k t^k$ such that $f (t, \theta(t)) = 0$. 
If $h(x,y) = \sum a_I x^i y^j$, then the coefficients of $\theta$ are determined recursively and one sees that
$\theta_k$ is a polynomial with integral coefficients in the variables $(a_I, \theta_2, \ldots, \theta_{k-1})$.
It follows that $\theta \in \cO_{\KK,S}\{t\}$ for any set of places $S$ such that $a_I \in \cO_{\KK,S}$.
It is  clear that the formal arc $\gamma(t) =  (t, \theta(t))$ induces an isomorphism of local rings. 
\end{proof}

\begin{remark}
In fact $\gamma$ induces an analytic isomorphism from the Berkovich unit open disk $\D_v(0,1)$ onto its image for all but finitely many places $v$. 
\end{remark}

\begin{remark}
Suppose $C$ is singular, and consider $\mathsf{n} \colon \hat{C} \to C$ its normalization. Given any point $p\in C$, we may apply the preceding lemma
to each point in $\mathsf{n}^{-1}(p)$. In this way we obtain adelic arcs parameterizing each branch of $C$ at $p$.
\end{remark}

\paragraph*{Algebraization of adelic branches in the affine plane}

We say that an arc in $\p^2$ sits at infinity if its origin lies on the line at infinity $L_\infty = \{[x:y:0]\} \subset \p^2$. 
One can always find an affine chart $(z,w)$ centered at the origin of the arc such that the arc is actually given by 
two adelic series $(z(t), w(t))$ with $z(0) = w(0)=0$. 

Suppose $\fs$ is an adelic branch at infinity whose origin is not the point $[0:1:0]$. Then $\fs$ is determined by an adelic arc
of the form $\gamma(t) =[1: y(t):t^k]$ for some $k\ge1$ and $y \in \cO_{\KK,S}$. The integer $k$ is uniquely determined, since it is the order of vanishing at $0$
of $h\circ \gamma$ where $h$ is a local equation of $L_\infty$ at  the origin of $\fs$. The adelic series $y$ is not uniquely defined but
any other adelic series defining $\fs$ is of the form $y(\zeta t)$ with $\zeta^k=1$.

We set $R_v(\fs) = R_v(y)$, and define 
\[C_v(\fs) := \{ (\tau^{-k}, \tau^{-k} y(\tau)), \, |\tau|_v  > R_v(\fs)\} 
\]
which is a closed analytic irreducible curve in the open set $\max \{ |x|_v, |y|_v\} > R_v(\fs)$ of $\A^2_{\KK_v}$. 
When the origin of $\fs$ is  the point $[0:1:0]$, then $\fs$ is determined by an arc $\gamma(t) =[x(t):1:t^k]$ and we define analogously 
$R_v(\fs) = R_v(x)$, and  $C_v(\fs) =  \{ (\tau^{-k} x(\tau), \tau^{-k}) , \, |\tau|_v  > R_v(\fs)\}$.
\index{theorem!Xie's algebraization}

\begin{theorem}[Xie~\cite{Xie}]\label{thm:Junyi}
Suppose $\fs_1, \ldots, \fs_l$ is a finite set of adelic branches at infinity, and $\{B_v\}_{v \in M_{\KK,S}}$ is a collection
of positive real numbers $B_v \ge 1$ such that $B_v =1$ for all but finitely many places. 

Let $p_n = (x_n, y_n)$ be an infinite sequence of points in $\KK^2$ such that for each place $v\in M_\KK$ we have 
either $p \in \cup_{i=1}^l C_v(\fs_i)$, or $\max \{ |x_n|_v, |y_n|_v\} \le B_v$. 

Then there exists an algebraic curve $C\subset \A^2_\KK$ such that any branch of $C$ at infinity is contained in the set $\{\fs_1, \ldots \fs_l\}$ and $p_n \in C(\KK)$ for all $n$ sufficiently large. 
\end{theorem}
\begin{remark}
Under the assumption of the theorem  the genus of any resolution of singularities of the completion of $C$ in $\p^2$ is a most $1$ by Faltings' theorem.  
\end{remark}
The proof below is due to Junyi Xie.
\begin{proof}
For any branch at infinity $\fs$ determined by an arc $\gamma$ 
and for any polynomial $P\in \overline{\KK}[x,y]$, define the order of vanishing of $P$
along $\fs$ by
\[v_{\fs}(P) := \ord_t (P \circ \gamma(t))\in \Z \cup \{ + \infty\}\]
with the convention that $v_{\fs}(P) =+\infty$ if $P \circ \gamma$ is identically $0$.
\begin{lemma}
There exists a polynomial $P\in \bar{\KK}[x,y]$ such that $v_{\fs_i} (P) >0$ for all $i$. 
\end{lemma}
 For each branch $\fs_i$, we fix an adelic arc  $\gamma_i(t) =[1: y_i(t):t^{k_i}]$ defining $\fs_i$, and write 
 $\varphi_i(t) := P  (t^{-k_i}, t^{-k_i}y_i(t))$.
\begin{proof}
The space of all polynomials of degree $\le d$ such that $v_{\fs_i} (P) >0$
is an algebraic subvariety  of $\A^{d+1}$ given by the vanishing of all coefficients of non-positive powers in the expansion 
of $\varphi_i$.
 Note that the number of conditions on $P$ grows linearly with the degree since 
$\varphi_i (t) = t^{- dk_i}  O(1)$. On the other hand, the dimension
of the space of all polynomials of degree $d$ is quadratic in $d$. It follows that for $d \gg 0$, a generic polynomial satisfies 
 $v_{\fs_i} (P) >0$ for all $\fs_i$.
\end{proof}
Fix any polynomial $P$ as in the previous lemma. Replacing $\KK$ by a finite extension, we may suppose that 
$P \in \KK[x,y]$. Observe that 
$\varphi_i \in t\cdot \cO_{\KK,S}\{t\}$ for all $i$ since $v_{\fs_i} (P) >0$. 
\begin{lemma}\label{lem:estim height}
There exists a collection of positive real numbers $B'_v \ge 1$ such that $B'_v =1$ for all but finitely many places, and 
$|P(p_n)|_v \le B'_v$ for all $n$.
\end{lemma}
\begin{proof}
Set $C_v = \sup\{ |P(x,y)|_v,\, |(x,y)|_v \le B_v\}$, and \[D_v = \sup \left\{ |P  (t^{-k_i}, t^{-k_i}y_i(t))|_v, |t|_v < R_v(\fs)\right\}.\]
Note that $y_i$ is an adelic series so that for all but finitely places $v$ its coefficients lie in $\KK^{\circ}_v$.
It follows that for all but finitely many places, $R_v(\fs) =1$ and $D_v =1$.

The proof is complete by taking $B'_v = \max \{C_v, D_v\}$.
\end{proof}
Since all $p_n$ belongs to the same number field, the previous lemma gives the following
height estimate:
\[h_{\st}(P(p_n)) = \frac1{[\KK:\Q]} \sum_{v\in M_\KK} \log \max \{ 1, |P(p_n)|_v\} < \infty \]
so that by Northcott, $P(p_n)$ belongs to a finite set $T\subset \KK$, and 
$\{ p_n \} \subset D :=  \{ \prod_{\lambda \in T} (P - \lambda) = 0 \}$. 
Let $C$ be the Zariski closure of $\{ p_n \}$ in $D$.

To conclude the proof, we need to show that any 
branch at infinity of $C$ lies in $\{ \fs_1, \ldots, \fs_l\}$. Suppose by contradiction that $\fs$ is a branch at infinity of 
$C$ different from the $\fs_i$'s. Let $C_0$ be the irreducible component of $C$ containing $\fs$.

It may happen that several branches at infinity of $C_0$ meets at the same point in $\p^2$. 
To avoid that, we blow-up
finitely many points on the line at infinity, and build a smooth projective compactification $X$ of $\A^2$ such that 
the closure of $C_0$ intersects transversally the divisor at infinity $H$ of $X$. Let $\bar{C}_0$ be the closure of $C_0$ in $X$.

Choose an effective divisor $D$ supported on $\mathfrak{s}$
such that $L= \cO_{\bar{C}_0}(D)$ is globally generated. Pick any finite set of generating sections:
we get  rational functions in two variables $Q_1, \ldots, Q_N\in \KK(x,y)$ whose
restrictions to $\bar{C}_0$ have only a pole at $\mathfrak{s}$, and such that 
$\min_i \{ \ord_p(Q_i)\} = 0$ for any $p\neq \mathfrak{s}$ and $\min_i \{ \ord_{\mathfrak{s}}(Q_i)\} = \ord_ \mathfrak{s}(D)$.
It follows that the collection of functions $\max \{ |Q_0|_v, \ldots , |Q_N|_v\}$ is continuous on $C_0$ and defines an adelic semi-positive metrization $\bar{L}$ of $L$.

We now estimate the induced height $h_{\bar{L}}(p_n)$ for all $p_n \in C_0$.
The same proof as for Lemma \ref{lem:estim height} applies and we get uniform upper bounds on $\max \{ |Q_0|_v, \cdots , |Q_N|_v\}(p_n)$, so that
 $h_{\bar{L}}(p_n)$ is bounded from above. By Northcott property, the set of $p_n$'s lying in $C_0$ is
 finite which is absurd. 
 \end{proof}

%
%
%
%
%
%


\chapter{Polynomial dynamics}

This chapter is mainly expository: \S\ref{sec:fatou-julia} and \S\ref{sec:green} contain brief discussions of basic aspects of the iteration of complex and non-Archimedean polynomials in one variable
(the Fatou-Julia theory and the construction of the canonical invariant measure). We look at a few important  examples of polynomial dynamics in \S \ref{sec:exam-poly}.
The fourth section is devoted to the detailed study of the expansion of the B\"ottcher coordinates. 
Section \ref{sec:global} builds on the previous chapter and discuss the notion of canonical height.
In Section \ref{sec:holofamily}, we review
the  Ma\~{n}\'e-Sad-Sullivan theory of bifurcation of holomorphic 
dynamical systems in the context of polynomials. 
We conclude this chapter by a discussion of the locus of preperiodic points in an arbitrary family of polynomials. 
Our main result (Theorem \ref{thm:finite branches}) will play a key role in one of our specialization argument in \S\ref{sec:specialization}.


\section{Fatou-Julia theory}\label{sec:fatou-julia}
We fix any algebraically closed complete metrized field $(K, |\cdot|)$ of characteristic zero. When the norm $|\cdot|$ is Archimedean, then $K = \C$ and $|\cdot|$ is the standard Euclidean norm. 

\paragraph*{The filled-in Julia set}
Pick any polynomial $P(z)= a_0 z^d + \ldots + a_d$
of degree $d\ge 2$ with coefficients in $K$. It induces a continuous map $P: \A^{1,\an}_K \to \A^{1,\an}_K$ on the Berkovich analytification of the affine line, 
and we define the filled-in Julia set as
\[ K(P)= \{ z\in \A^{1,\an}_K, \, |P^n(z)| \text { is bounded as } n\to\infty\}~.\]
Observe that for some $\epsilon >0$ small enough and some $R > 1$ big enough we have 
\[|P(z)| \ge \epsilon |z|^d \ge 2 |z|\]
for all $|z| \ge R$. It follows that the basin of attraction of infinity
\[ \Omega(P)= \{ z\in \A^{1,\an}_K, \, |P^n(z)| \to \infty \text { as } n\to\infty\}~,\]
is an open set whose complement is equal to $K(P)$, and that the latter set is a non-empty compact set. 
Both sets $K(P)$ and $\Omega(P)$ are totally invariant by $P$. 

The Julia set\index{Julia set} $J(P)$ of $P$ is defined to be the boundary of $K(P)$: it is also a compact set which is totally invariant. 
The Fatou\index{Fatou!set} $F(P)$ is the complement of the Julia set: it is the disjoint union of $\Omega(P)$ and the interior of $K(P)$.
It can be characterized as the open set where the sequence of iterates $\{ P^n\}$ are normal (in the usual sense in the Archimedean case, and in the sense of~\cite{Non-arch-Montel} in
the non-Archimedean case).

\paragraph*{Periodic points}
Let $p \in \A^1(K)$ be a fixed point for $P$.\index{periodic point} Then $p$ is repelling (resp. neutral, attracting, or super-attracting)
when $|P'(p)|>1$ (resp. $|P'(p)| =1$, $|P'(p)| <1$, or $|P'(p)| =0$).

When $p$ is repelling or attracting, it is always possible to find an analytic  change of coordinates $\phi$ locally at $p$
such that $ \phi^{-1} \circ P \circ \phi (z) = \lambda z$ with $\lambda = P'(p)$.
When $p$ is super-attracting\footnote{In positive characteristic this result is no longer true, see~\cite{ruggiero-positivechar}.} then we can find $\phi$ such that  $ \phi^{-1} \circ P \circ \phi (z) = z^k$ where $k = \ord_p(P)$.

Observe that the polynomial $P (z) = a_0 z^d + \mathrm{l.o.t}$ also defines a continuous map on $\p^1_K$  for which the point at infinity is totally invariant, and super-attracting. It follows that one can find 
an analytic function $\varphi_P(z) = a z + \sum a_j z^{-j}$ with $a \neq 0$ converging in $|z| \ge R$ for $R$ sufficiently large and such that $ \phi_P \circ P = (\phi_P)^d$. 
This analytic function is uniquely determined by the previous equation and the choice of $a$ such that $a^{d-1} = a_0$, and is called the B\"ottcher coordinates. We shall discuss extensively the expansion 
of $\varphi_P$ in \S\ref{sec:bottcher}.

When $p$ is neutral, then the situation is quite delicate. 
\begin{itemize}
\item
When the multiplier $P'(p)$ is a root of unity\footnote{in which case $p$ is said to be a parabolic fixed point}, then one can never linearize $P$ near $p$ (otherwise we would have $P^r =\id$ for some $r>1$ which contradicts $d\ge2$). The dynamics is described by the Fatou-Leau theory in the complex case, see~\cite[\S10]{Milnor4}. In the non-Archimedean case,  the dynamics depends on the residue characteristic of $K$. When the residue characteristic is positive,  the dynamics can be investigated using the iterative logarithm, see~\cite[\S 10.2]{benedetto-book} or~\cite[\S 3.2]{rivera-corps-locaux}. 
\item
When the multiplier $P'(p)$ is not a root of unity, then $P$ is always linearizable when $(K,|\cdot|)$ is a non-Archimedean metrized field of characteristic zero, see~\cite[\S 3.3]{rivera-corps-locaux}. The
linearizability  of $P$ near $p$ in the Archimedean case depends in a subtle way on the continued fraction expansion of the argument of the multiplier.  We refer to~\cite[\S11]{Milnor4}
for a thorough discussion of this very intricate problem. 
\end{itemize}

We conclude this section by the following observation. 

\begin{lemma}
Let $p\in \A^1(K)$ be a fixed point for $P$.
\begin{itemize}
\item
When $K = \C$ is Archimedean, then $p$ belongs to the Fatou set iff
 it is attracting either $|P'(p)|<1$, or neutral $|P'(p)| =1$ and $P$ is linearizable at $p$. 
\item
When $K$ is non-Archimedean, then $p$ belongs to the Fatou set 
iff it is non-repelling, i.e. $|P'(p)| \le 1$. 
\end{itemize}
\end{lemma}

\begin{proof}
Assume $K = \C$. One direction is clear. For the converse suppose first that $p$ is repelling. Then the sequence of iterates $\{P^n\}$ cannot be normal at $p$
since its derivative explode. When $p$ is neutral and belongs to the Fatou set, then a simple argument for linearizability 
goes as follows, see~\cite[Corollary 5.3]{Milnor4}. 
Let $U\subset F(P)$ be the Fatou component of $P$ containing $p$. By the maximum principle, for all Jordan curve $\gamma\subset U$, the bounded component of $\mathbb{C}\setminus\gamma$ is contained in $U$, hence $U$ is simply connected. Let $\psi:U\to\D$ be a conformal isomorphism with $\psi(p)=0$, then $g:=\psi\circ f\circ \psi^{-1}$ satisfies $g:\D\to\D$, $g(0)=0$ and $g'(0)=P'(p)$. By Schwarz lemma, this gives $g(z)=P'(p)\cdot z$.

Suppose now that $K$ is non-Archimedean and write $\lambda = P'(p)$. Choose an affine coordinate $z$ such that $p=0$. When $|\lambda| \le 1$ we have
$|P(z) - \lambda z | \le C |z|^2$ for all $|z|$ small enough so that $|P(z)|  =  |\lambda z |$ is a neighborhood of $0$. It follows that $|P^n(z)| $ is bounded for all $n$
and $p$ belongs to the Fatou set. Conversely, when $p$ belongs to the Fatou set, then we can find a disk $D$ around $0$ such that $P^n(D) \subset B(0,1)$ for all $n$. 
It follows $|(P^n)'(0)| \le \diam(D)$ for all $n$, hence $p$ is not repelling. 
\end{proof}

\paragraph*{Non-rigid periodic points (non-Archimedean case)}
Suppose $K$ is a non-Archimedean metrized field. 
 Recall that we wrote $K^{\circ}:=\{z\in K\, : \ |z|\leq 1\}$, $K^{\circ\circ}:=\{z\in K\, : \ |z|<1\}$, and $\tilde{K}:=K^{\circ}/K^{\circ\circ}$.
 For any $z\in K^{\circ}$, we let $\tilde{z}\in \tilde{K}$ be the image of $z$ under the reduction map $K^{\circ} \to\tilde{K}$. 
 
 \smallskip
 
Given any polynomial $P\in K[T]$ of degree $d\ge2$, then it may appear that $P$ fixes some points in the Berkovich affine line that are not rigid. 

\begin{proposition}\label{prop:non-rigid fixed}
Suppose $x$ is a non-rigid point in  $\A^{1,\an}_K$ which is fixed by $P$. 
\begin{enumerate}
\item
If $x \in J(P)$, then there exists a finite extension $L/K$ and an affine map $\phi$ defined over $L$ such that 
$\phi (x)$ is the Gau{\ss} point, and the reduction of $\phi \circ P \circ \phi^{-1}$ is a polynomial with coefficients in $\tilde{K}$
of degree at least $2$.
\item 
If $x$ lies in the Fatou set, there exists a neighborhood of $x$ on which $P$ induces an analytic isomorphism. 
\end{enumerate}
\end{proposition} 

\begin{definition}
Any point satisfying Condition (1) of the previous proposition is called
a non-rigid repelling fixed point. 
\end{definition}

\begin{proof}
Suppose that $K$ is algebraically closed. If $x$ is of type 2, then there exists an affine map $\phi$ defined over $K$ such that 
$\phi (x)$ is the Gau{\ss} point, in which case the polynomial $P$ can be decomposed as the sum of two polynomials: $Q_1 \in \mathfrak{m}[T]$
and $Q_2$ having coefficients in $K^{\circ}\setminus \mathfrak{m}$.
When $\deg(Q_2) \ge2$, then we are in case (1) of the Proposition. When $\deg(Q_2) =1$, then we can find $r >1$ such that $P$ induces an analytic isomorphism
on the disk centered at $0$ of radius $r$, and we fall into case (2). 

When $x$ is of type 3 or of type 4, we are always in the second case, see \cite[Lemma~10.80 $\&$ Theorem~10.81]{baker-rumely-book} or \cite{benedetto-book}. \end{proof}
 
\paragraph*{Dynamics in the Fatou set}

A Fatou component\index{Fatou!component} is a connected component of $F(P)$. It is either bounded, or equal to $\Omega(P)$. The image by $P$ of a Fatou component remains a Fatou component.

When $K =\C$, then any Fatou component is pre-periodic by the famous non-wandering theorem of Sullivan, see, e.g.,~\cite[Theorem 16.4]{Milnor4}. One is thus reduced
to consider periodic (and even fixed) Fatou component to understand the dynamics of $P$ on $F(P)$.
Any fixed Fatou component $U$ is either the basin of attraction of a fixed attracting or super-atttracting point;
or a parabolic domain so that any point in $U$ converges under iteration toward a parabolic fixed point; or a Siegel domain, i.e. a disk on which $P$ is conjugated to an irrational rotation. 

When $K$ is non-Archimedean, the situation is more delicate and highly depends on the residual characteristic of $K$.
The classification of periodic components is due to Rivera-Letelier. We refer to~\cite[Theorem~9.14]{benedetto-book} for the following result. 
Any fixed Fatou component $U$ is either the basin of attraction of a fixed attracting or super-atttracting orbit; or it is an affinoid domain 
whose boundary consist of periodic type $2$ repelling points, and $P: U \to U$ is an analytic isomorphism. In the latter case, we say that $U$ is an indifferent component. 

An analog of Sullivan's non-wandering theorem has been proved by Benedetto~\cite{benedetto-non-wandering} when the residual characteristic of $K$ is $0$.
A Fatou component $U$ is either pre-periodic, or it is a ball and its (unique) boundary point $x$  is pre-periodic (and replacing $P$ by an iterate
we have $P^n(z) \to P(x)$ for all $z \in U$). When the residual characteristic of $K$ is positive, then this result is no longer true: there exists wandering Fatou components
which are disks and whose boundary point has an infinite orbit, see~\cite{benedetto-wandering}. 
There is no conjecture explaining the appearance of wandering domains in full generality.


\section{Green functions and equilibrium measure}\label{sec:green}
\subsection{Basic definitions}
Let $(K, |\cdot|)$ be any complete and algebraically closed metrized field of characteristic $0$. For any polynomial $P(z)  = Az^d + a_1 z^{d-1} + \ldots + a_d \in K[z]$ of degree $d\ge2$, there exists a constant $C\ge0$ such that 
 \[\left| \frac1d \log^+|P(z)| - \log^+|z| \right| \le C\] 
 Indeed 
 $|P(z)| \le \max \{|A|, |a_i|\} \max\{ 1, |z|\}^d$, and for any $\epsilon \ll 1$ small enough such that  $A > \epsilon \sum |a_i|$, we have
 $|P(z)|\ge (A - \epsilon \sum |a_i|) |z|^d$ when $|z| \ge \epsilon^{-1}$. 
 
It follows that the sequence of functions $\frac 1{d^n} \log^+|P^n|$ converges uniformly on $\A^{1,\an}_K$
to a continuous function $g_P$.

\begin{definition}
The function $g_P$ is called the Green function \index{function!Green} of $P$.
\end{definition}
The  proof of the next result is left to the reader, see~\cite[\S 10.7]{baker-rumely-book}.
\begin{proposition}\label{prop:basic-green}
The Green function of $P$ satisfies the following properties: 
\begin{enumerate}
\item
$g_P \circ P = d g_P$ on $\mathbb{A}^{1,\an}_K$; 
\item 
$g_P(z) = \log |z| + \frac1{d-1} \log |A| + o(1)$ as $|z|\to\infty$;
\item 
the set $\{ g_P = 0\}$ is the filled-in Julia set $K(P)$ of $P$; 
\item
the function $g_P$ is harmonic outside $J(P)$;
\item
the set of functions $\frac1{d^n} \log^+|P^n(z)|$ converges uniformly on 
$C \times \A^{1,\an}_K$ for any compact subset $C$ of $\A^{d+1,\an}_K$
so that 
the function $(P,z) \mapsto g_P(z)$ is continuous on $\A^{d+2,\an}_K$.
\end{enumerate}
\end{proposition}

As $g_P$ is a subharmonic function of on $\mathbb{A}^{1,\an}_K$ with $g_P(z)=\log^{+}|z|+O(1)$ as $|z|\to\infty$, its Laplacian is a probability measure.

\begin{definition}
The equilibrium measure\index{equilibrium measure} $\mu_P$ of $P$ on $\mathbb{A}^{1,\an}_K$ is $\mu_P:= \Delta g_P$.
\end{definition}

By the above properties of the Green function, the measure $\mu_P$ has the following properties:
\begin{itemize}
\item the measure $\mu_P$ is a probability measure supported on $J(P)$;
\item $P^*\mu_P=d\cdot \mu_P$ and $P_*\mu_P=\mu_P$;
\item $\mu_P$ is ergodic and mixing;
\item the entropy of $\mu_P$ is at most $\log d$ (with equality when $K = \C$).
\end{itemize}

One also defines the Lyapunov exponent\index{Lyapunov!exponent} of $P$ as the quantity
\[
\lyap(P) := \lim_{n\to \infty} \frac1n \int \log |(P^n)'| \, d\mu_P~.
\]
The Misiurewicz-Przytycki's formula states that
\begin{equation}\label{eq:MPform}
\lyap(P') = \log |d| + \sum_{P'(c)=0} g_P(c)~.
\end{equation}
Observe that $|d| = d$ if the norm is Archimedean, and $|d| \le 1$ when it is non-Archimedean so that $\lyap(P)$ may be negative in the latter case.
A proof is given over any metrized field of characteristic zero by Y. Okuyama in~\cite[\S 5]{MR2885787}.

\begin{remark}
When a polynomial $P$ is defined over a number field $\KK$, for each place $v\in M_\KK$, there is a Green function $g_{P,v}$ (and an equilibrium measure $\mu_{P,v}$) for the polynomial $P$. The function $g_{P,v}$ and the measure $\mu_{P,v}$ depend on the place $v$ (see \S\ref{sec:global}).
\end{remark}

\subsection{Estimates on the Green function}
Recall that if  $(c,a)\in K^{d-1}$, we set
\begin{equation}\label{eq:critical form}
P_{c,a}(z)\pe \frac{1}{d}z^d+\sum_{j=2}^{d-1}(-1)^{d-j}\sigma_{d-j}(c)\frac{z^j}{j}+a^d~.
\end{equation}
This parameterization is particulary adapted to estimate the variation of the Green function
in terms of the polynomial. We refer to \cite[\S2]{favregauthier} for a more detailed exposition.
\begin{proposition}\label{prop:classic-estim}
 There exist a constant $\theta \ge 0$ and $C\ge 1$ such that the following holds for all $c,a\in K^{d-1}$:
 \begin{enumerate}
  \item for all $z\in \A^{1,\an}_K$, we have
 \[g_{P_{c,a}}(z)\le \log^+\max\{|z|,|c|,|a|\}+\theta~,\]
\item for all $z\in \A^{1,\an}_K$ with $|z|> C\cdot \max\{1,|c|,|a|\}$, we have
\[\left\{\begin{array}{ll}
g_{P_{c,a}}(z)= \log^+|z|-\frac{1}{d-1}\log|d|, & \text{ when} \ K \ \text{ is non-Archimedean},\\
g_{P_{c,a}}(z)\geq \log^+|z|-\log8, & \text{ when} \ K \ \text{ is Archimedean}.
\end{array}\right.\].
 \end{enumerate}
 Furthermore, when $K$ is non-Archimedean then $C,\theta$ only depend on the residual characteristic, and they equal $C=1$ and $\theta=0$
 when the residual characteristic of $(K,|\cdot|)$ is $0$ or at least $d+1$.
 \end{proposition}

\begin{definition}
For any polynomial $P\in K[T]$, we set 
\begin{equation} \label{eq:def-GP}
G(P) = \max \{ g_P(c), P'(c) =0 \} \in \R_+.
\end{equation}
\end{definition}

The next result follows from Proposition~\ref{prop:classic-estim} and the Nullstellensatz.
\begin{proposition}\label{prop:growth Green}
The function $(c,a) \mapsto G(P_{c,a})$ extends continuously to $\A^{d-1,\an}_K$, and
there exists a constant $C\ge0$ such that 
\[\sup_{\A^{d-1,\an}_K} \left|G(P_{c,a}) - \log^+\max\left\{|c|, |a|\right\} \right| \le C~.\]
Furthermore, when $K$ is non-Archimedean then $C$ only depends on the residual characteristic, and it equals $0$
 when the residual characteristic is $0$ or $\ge d+1$.
 \end{proposition}
 
 \begin{proof}[Proof of Proposition~\ref{prop:classic-estim}]
Let us argue first in the non-Archimedean case. Let $p$ be the residue characteristic of $K$. By the strong triangle inequality, there exists $C_p\geq1$ depending only on $p$ with $C_p=1$ when $p=0$ or $p\geq d+1$, such that
\[
|P_{c,a}(z)| \le C_p \max \{1, |z|, |c|, |a| \}^d
\]
hence $g_{P_{c,a}}(z) \le \log^+ \max \{ |z|, |c|, |a| \} + \frac1{d-1} \log C_p$ by induction.
On the other hand, again by the strong triangle inequality, we have
\[
|P_{c,a}(z)| = \left|\frac{z^d}{d}\right|\geq |z|^{d},
\]
when $|z| > C'_p\max \{ 1, |c|, |a|\}$ where $C_p'=1$ when $p=0$ or $p\geq d+1$.
Again, an easy induction gives 
\[g_{P_{c,a}}(z)=-\frac{1}{d-1}\log|d|+\log^{+}|z|,\]
when $|z| > C'_p\max \{ 1, |c|, |a|\}$.

\medskip

For the Archimedean case, we refer to \cite[\S6.1]{favredujardin} for details. By the triangle inequality and the maximum principle, there exists $A\geq1$ which depends only on $d$ such that
\[|P_{c,a}(z)|\leq A\cdot\max\{1,|z|,|c|,|a|\}^d\]
and as above, we find $g_{P_{c,a}}(z) \le \log^+ \max \{ |z|, |c|, |a| \} + \frac1{d-1} \log A$ by induction. In particular, $G(P_{c,a})\leq \log^+ \max \{ |c|, |a| \} + \frac1{d-1} \log A$. For the second inequality, we use \cite[Lemma 6.5]{favredujardin}:
\[\max\{g_{P_{c,a}}(z),G(P_{c,a})\}\ge \log|z-\delta|-\log 4\]
where $\delta=\sum_jc_j/(d-1)$. Let $\tilde A:=\max \{A^{1/(d-1)},2\}$, so that
\[\log |z-\delta|\geq \log|z|-\log2\]
if $|z|\ge \tilde A\cdot\max\{1,|c|,|a|\}$. The conclusion follows taking $C:=8\tilde A$.
\end{proof}

\begin{lemma}\label{lm:nullstellesatzt}
There exists a constant $B\leq1$ such that for all $(c,a)\in K^{d-1}$,
\begin{equation*}
\max_{0\leq j\leq d-2}|P_{c,a}(c_j)|\ge B\cdot\max\{|c|,|a|\}^d.
\end{equation*}
Moreover, when $K$ is non-Archimedean, $B$ depends only on the residual characteristic of $K$ and $B=1$ if it is $0$ or $\ge d+1$.
\end{lemma}

\begin{proof}
Let $I$ be the ideal generated by $(a^d,P_{c,a}(c_1),\ldots,P_{c,a}(c_{d-2}))$ in $K[c,a]$. Observe that the generators of $I$ have no common zero other than $(0,\ldots,0)$. Let $R:=\Z[1/2,\ldots,1/d]$. As $R$ is an integral domain and as $P_{c,a}(c_j)\in R[c,a]$ are homogeneous of degree $d$, a standard fact from elimination theory asserts that there exists $m\geq d$ and homogeneous polynomials $Q_{j,k}\in R[c,a]$ of degree $m-d$ such that for $1\leq k\leq d-2$,
 \[c_k^m=\sum_j Q_{j,k}(c,a)P_{c,a}(c_j).\]
 Let $A$ be the maximum of the norms of coefficients appearing in one of the $Q_{j,k}$'s. and let $B>0$ be such that $B^{-1}:=\max\{1,A\}$. Then  $|c_k|^m\leq B^{-1}\cdot \max\{|c|,|a|\}^{m-d}\cdot \max_{0\leq j\leq d-2}|P_{c,a}(c_j)|$. Since $a^d=P_{c,a}(c_0)$, this gives the lemma.
\end{proof}

 \begin{proof}[Proof of Proposition~\ref{prop:growth Green}]
 The inequality
 \[G(P_{c,a})\leq \log^+\max\{|c|,|a|\} + C\]
  is an immediate consequence of Proposition~\ref{prop:classic-estim} (1). We also see that $C=0$ when $K$ is non-Archimedean and its residue characteristic is $0$ or $\ge d+1$.

\medskip

 \par\noindent Assume now $K$ is non-Archimedean and its residual characteristic is $0$ or at least $d+1$. Assume first that $\max\{|c|,|a|\}\leq 1$. Then $G(P_{c,a})=0$, again by Proposition~\ref{prop:classic-estim} (1).
If $\max\{|c|,|a|\}>1$, then either there is an index $i$ such that $|c_i|=\max\{|c|,|a|\}$ and $g_{P_{c,a}}(c_i)\geq \log^+|c_i|=\log^+\max\{|c|,|a|\}$ by the second point of Proposition~\ref{prop:classic-estim}; or $|a|=\max\{|c|,|a|\}$ and we easily see that $|P_{c,a}^{n}(0)|=|a|^{d^{n}}$ for $n\geq1$. The conclusion follows.

Assume now the residual characteristic $p$ of $K$ satisfies $0<p\leq d$. Let $C\geq1$ be given by Proposition~\ref{prop:classic-estim}. When $|P_{c,a}(c_j)|\leq C\max\{1,|c|,|a|\}$ for all $j$, then Lemma \ref{lm:nullstellesatzt} yields $B\cdot \max\{|c|,|a|\}^{d-1}\leq C$, so that $G(P_{c,a})\geq0\geq \frac{1}{d-1}\log(B/C)+\log^+\max\{|c|,|a|\}$.
When there exists $j$ such that $|P_{c,a}(c_j)|>C\max\{1,|c|,|a|\}$, the second point of Proposition \ref{prop:classic-estim} and Lemma \ref{lm:nullstellesatzt} give
\begin{align*}
dG(P_{c,a}) & \geq dg_{P_{c,a}}(c_j)=g_{P_{c,a}}(P_{c,a}(c_j))\\
&\geq \log^{+}|P_{c,a}(c_j)|\geq d\log^{+}\max\{|c|,|a|\}+\log B
\end{align*}
and the conclusion follows.

To conclude, we assume $K$ is Archimedean. If $\max\{|c|,|a|\}\leq 2$, we have $G(P_{c,a})\leq \log 2+\theta$. We may thus assume $A:=\max\{|c|,|a|\}\geq 2$. As in the proof of Proposition \ref{prop:classic-estim}, let $\delta:=\frac{1}{d-1}\sum_jc_j$. If $|c_j|\leq A/2$ for all $j$, we also have $|\delta|\leq A/2$ and $A=|a|$. By \cite[Lemma 6.5]{favredujardin},
\[dG(P_{c,a})\geq \max \{g_{P_{c,a}}(a^d),G(P_{c,a})\}\geq \log|a^d-\delta|-\log 4\geq d\log A-\log8.\]
In the opposite case, either $|\delta|\geq A/2$ or $|c_j-\delta|\geq A/2$ and \cite[Lemma 6.5]{favredujardin} directly gives
$G(P_{c,a})\geq \log A-\log8$, as required.
 \end{proof}

\section{Examples}  \label{sec:exam-poly}

\subsection{Integrable polynomials}
Fix an integer $d\geq2$. The degree $d$ Chebychev polynomial is the unique polynomial $T_d\in\Z[z]$ satisfying
\[T_d(z+z^{-1})=z^d+z^{-d}\]
in the field $\Q(z)$. It is a monic and centered polynomial of degree $d$, and 
for all $d,k\geq2$, we have
\[T_d\circ T_k=T_{kd} \ \text{and} \ T_d(-z)=(-1)^dT_d(z).\]
In particular, $T_d^n=T_{d^n}$ for all $n\geq1$.

Observe that $\pi(z) = z+ z^{-1}$ defines a  Galois cover of degree $2$ from $\A^1\setminus\{0\}$ onto $\A^1$, 
with two ramified points at $\pm 1$, and that we have $ \pi (M_d(z)) = T_d(\pi(z))$ where $M_d(z) = z^d$. In particular
the critical points of $T_d$ are $0, \pm1$, and are pre-periodic. 

When $K=\C$, the Julia set of $T_d$ is the image under $\pi$ of the unit circle so that 
$K(T_d)=J(T_d)=[-2,2]$. The equilibrium measure $\mu_P$ is absolutely continuous with respect to the normalized Lebesgue measure on $[-2,2]$. 

When $K$ is non-Archimedean, then the filled-in Julia set of $T_d$ coincides with the closed unit ball so that its Julia set is reduced to a singleton, namely to the 
Gau{\ss} point. 

\begin{definition}
Let $K$ be a field of characterisic $0$ and let $P\in K[z]$ be a degree $d\geq2$ polynomial. We say $P$ is integrable\index{polynomial!integrable} if $P$ is affine conjugated (in a finite extension of $K$) to either  $M_d$ or  $\pm T_d$.
\end{definition}

This terminology is taken from~\cite{MR1098340} and inspired from hamiltonian dynamics. 

Observe that when $d$ is odd, then $- T_d$ is conjugated to $T_d$. There are many characterizations of integrable polynomials in algebraic or analytic terms. 
Over the complex numbers, we have the following famous theorem of Zdunik~\cite{zdunik}.\index{theorem!Zdunik} 
\begin{theorem}\label{th:zdunik}
Let $P$ be a complex polynomial of degre at least $2$. 
If its equilibrium measure is absolutely continuous with respect to the Hausdorff measure of its Julia set, then $P$ is integrable.
In particular a complex polynomial $P$ is integrable iff its Julia set is smooth near one of its point. 
\end{theorem}
One can in fact push a little further the analysis and get
\begin{corollary}\label{cor:zdunik}
Let $P$ be a complex polynomial of degre at least $2$. 
Then $P$ is integrable iff $J(P)$ is either a circle or a closed segment.
\end{corollary}

We refer to Theorem~\ref{thm:SS} below for a characterization of integrable complex polynomials in terms of the symmetries of their Julia set. 

\subsection{Potential good reduction}

We assume here that the metrized field $(K,|\cdot|)$ is non-Archimedean. Recall that given $z\in K^{\circ}$ we let  $\tilde{z}\in \tilde{K}$ be the image of $z$ under the reduction map $K^{\circ} \to \tilde{K}$.

Fix $d\geq2$ and pick any polynomial $P$ of degree $d$ with coefficients in $K^{\circ}$. Then $P$ induces a polynomial
\[\tilde{P}\colon\tilde{K}\to \tilde{K}\]
by taking the residue class of its coefficients. We say that $P$ has \emph{good reduction} if $\deg(\tilde{P})=\deg(P)$.
More generally, we say that $P$ has potential good reduction\index{polynomial!with potential good reduction} if there exists a finite extension $L/K$ and an affine transformation $\phi(z)=az+b$ with $a,b\in L$ such that $Q:=\phi^{-1}\circ P\circ \phi\in L^\circ[z]$ has good reduction.

Observe that when $P$ has good reduction, then we have $g_P(z)=\log^+|z|$ so that the filled-in Julia set of $P$ is the closed unit ball of $(K,|\cdot|)$ and 
$J(P)$ is reduced to the Gau{\ss} point. 

\begin{proposition}
Let $P\in K[T]$ be any polynomial of degree $d\ge2$. The following are equivalent.
\begin{enumerate}
\item
$P$ has potential good reduction; 
\item
the Julia set of $P$ is reduced to a point;
\item
the filled-in Julia set is a closed ball. 
\end{enumerate}
\end{proposition}

Several other characterizations of potential good reduction polynomials are given in~\cite[Th\'eor\`eme~E]{FRL-ergodique}.

\begin{proof}
The implications $(1) \Rightarrow (2) \Rightarrow (3)$ are easy. When the filled-in Julia set is a closed ball, then the Julia set is reduced to a singleton $\{x\}$. 
The point $x$ is periodic of type 2 or 3 and belongs to the Julia set. It follows that $x$ is a type 2 repelling point, and up to conjugacy we may suppose that
it is the Gau{\ss} point. Any polynomial for which the Gau{\ss} point is totally invariant has coefficients in $K^\circ$ and has good reduction, which concludes the proof. 
\end{proof}

\subsection{PCF maps}
Suppose $K$ is a metrized field of characteristic $0$. 
A  polynomial $P$ of degree $d\ge2$ is said to be post-critically finite (PCF) when all its critical points have a finite orbit. \index{polynomial!post-critically finite (PCF)}

Integrable maps are PCF. We shall see that any PCF map is conjugated to a polynomial with coefficients in $\bar{\Q}$ (see Corollary~\ref{cor:PCF} below), and 
that the set of PCF maps forms a set of bounded height in the parameter space. 

When $K = \C$ and all critical points are periodic then the dynamics of $P$ on its Julia set is hyperbolic in the sense that it expands some Riemannian metric, see~\cite[\S V.2]{carleson-gamelin}. When $P$ is PCF but none of its critical point is periodic, then we say that $P$ is strictly PCF (these polynomials are sometimes called Misiurewicz). 

When $K$ is non-Archimedean, then any PCF polynomial has good reduction since all its critical points have a bounded orbit.


\section{B\"ottcher coordinates \& Green functions}\label{sec:bottcher}

\subsection{Expansion of the B\"ottcher coordinate}
In this section, $R$ is any integral domain whose fraction field $K$ has characteristic zero, and $P\in R[z]$ is a polynomial  of degree $d \ge2$ given by
\[P(z)=A z^d+a_{1} z^{d-1} + a_{2}z^{d-2}+\cdots+a_d,\]
with $a_i, A\in R^{d+1}$. We fix any element  $\alpha$ of an algebraic closure of $K$ such that $\alpha^{d-1} = A$. 

The following result holds, see~\cite{specialcubic} in degree $3$, \cite[\S 5.5]{BD} or~\cite[\S 6]{boundedheight}.

\begin{propdef}\label{prop:bottcher}
There exists a unique formal Laurent series $\varphi_P$  in the variable $z^{-1}$ of the form 
\begin{equation}\label{eq-bott}
\varphi_P(z)=\alpha \left(z+ \frac{a_{1}}{dA}\right) +\sum_{j\geq1}\alpha_j z^{-j} \in \Z\left[\frac{1}{dA},\alpha, a_1, \ldots, a_d\right](\!(z^{-1})\!)
\end{equation}
such that
\begin{equation}
\varphi_P\circ P=(\varphi_P)^d~.\label{eq:bottcher}
\end{equation}
It is called the B\"ottcher coordinate\index{B\"ottcher coordinate} of $P$ at infinity.
\end{propdef}

\begin{proof}
We look for coefficients $\alpha_j \in R[\frac{1}{dA},\alpha]$ such that the power series
\[ \varphi (z)= \alpha \left(z+ \frac{a_{1}}{dA}\right) +\sum_{j\geq1}\alpha_j z^{-j}\]
solves the equation~\eqref{eq:bottcher}. Observe that 
\begin{align*}
\varphi_P\circ P(z)
&
= \alpha \left( P(z) + \frac{a_{1}}{dA}\right) + \sum_{j \ge1} \frac{\alpha_j}{(Az^d)^j} \left(1 + \frac{a_{1}}{Az} + \ldots + \frac{a_d}{Az^d}\right)^{-j}\\
&
= \alpha \left( P(z) + \frac{a_{1}}{dA}\right) + \sum_{l \ge1} \frac1{z^l}\, \left(\sum_{l/d\ge j \ge 1} \alpha_j Q_{j,l}(a_1, \ldots, a_d)\right)
 \end{align*}
with $Q_{j,l}\in \Z[\frac1A,a_1, \ldots, a_d]$, and
\begin{align*}
\varphi_P(z)^d 
=&  \left( \alpha \left(z+ \frac{a_{1}}{dA}\right)+ \sum_{j\geq1}\alpha_j z^{-j} \right)^d  \\
=&
\alpha^d z^ d + \alpha a_1 z^{d-1} +  \left(dA  \alpha_{1}+   \frac{(d-1)\alpha a_1^2}{2dA}\right) z^{d-2} +  \\
& + \sum_{l \le d-3} z^l \left( dA \alpha_{d-1-l} + Q_l(\alpha_1, \ldots, \alpha_{d-2-l})\right) 
 \end{align*}
 where $Q_l$ is a polynomial having coefficients in $\Z[\frac1{dA},a_1,\alpha]$. The terms in $z^d$ and $z^{d-1}$ of both series agree by our choice of $\alpha$. 
Identifying the terms in $z^l$ successively for $l = d-2, d-3, \ldots, 1, 0, -1,\ldots$, 
 we see that the coefficients $\alpha_s$ are uniquely determined by the relations:
 \begin{align}
 dA  \alpha_{1}&= -   \frac{(d-1)\alpha a_1^2}{2dA}+ \alpha a_2  ; \label{eq010}
 \\
  dA \alpha_{s} &=  - Q_{d-1-s}(\alpha_1, \ldots, \alpha_{s-1}) + \alpha a_{s+1}; \label{eq020}
  \end{align}
 when $ 2\le s \le d-2$, 
\begin{equation}
dA \alpha_{d-1} =  - Q_0(\alpha_1, \ldots, \alpha_{d-2}) +  \alpha \left(a_d + \frac{a_1}{dA}\right); \label{eq021}
\end{equation}  and
\begin{equation}
 dA \alpha_{s} =  - Q_{d-1-s}(\alpha_1, \ldots, \alpha_{s-1}) +  \sum_{j=1}^{(s-d+1)/d} \alpha_j Q_{j,s-d+1}(a_1, \ldots, a_d) \label{eq030}
\end{equation}
for all $s \ge d$. 
This shows that for all $s\ge 1$ the coefficient $\alpha_s$ can be expressed as a polynomial in the variables $a_1, \ldots, a_d$ with coefficients in $\Z[\frac1{dA},\alpha]$. 
\end{proof}

We shall need some precise informations on the dependence of the coefficients of the B\"ottcher coordinates in terms of the coefficients of the polynomials. Compare with~\cite[Lemma 2.2]{specialcubic} and \cite[Theorem 6.5]{boundedheight}.
We use the same notation as in the previous proposition. 

Define the weighted degree  of a polynomial $P= \sum  c_Ia^I$ in the variables $a_1, \ldots, a_d$ as 
$\wdeg(P):= \min \{ \sum j i_j, \, c_I \neq 0 \}$
so that $\wdeg(a_i) =i$.

\begin{proposition}\label{prop:expan bottcher}
For any $k\geq1$, one has the following expansion
\begin{equation}\label{eq008}
\left(\varphi_P(z)\right)^k=\hat{P}_{k}(z)+\sum_{j=1}^{+\infty}\frac{\alpha_{k,j}}{z^j},
\end{equation}
where  $\hat{P}_k$ is a polynomial of degree $k$ with leading coefficient $\alpha^k$ such that  \[(2dA)^{2k}\hat{P}_k\in \Z[\alpha,a_1,\ldots,a_{d}][z]~;\] 
and for any $k,j\geq1$,  \[(2dA)^{2(k+j-1)} \alpha_{k,j} \in \Z[\alpha, a_1,\ldots,a_{d}]\] is a polynomial of degree $\le  d (k-1) + (d-1) j$ in $\alpha$, and  of weighted degree  $\le k+j$ in the variables $a_i$. 
\end{proposition}
 
 \begin{proof}
 Observe that $\alpha_j = \alpha_{1,j}$ for all $j$, and recall that $\alpha_j$ is a polynomial in the coefficients $\alpha, a_1, \ldots, a_d$. Denote by  $\deg_\alpha(\alpha_j)$
the degree of this polynomial in $\alpha$.  We claim that 
\begin{align}
& \deg_\alpha(\alpha_j) \le (d-1) j,\,  \label{eq090}
\\
&  \wdeg(\alpha_j) \le j+1, \text{ and } \label{eq09}
\\
&(2dA)^{2j} \alpha_j  \in \Z [\alpha, a_1, \ldots, a_d]~.\label{eq10}
\end{align} 
Grant this claim. The proof then goes as follows. 

Fix $k \in \N^*$. Then, we first have
\begin{align*}
\varphi^k(z)  = \left( \alpha \left(z+ \frac{a_{1}}{dA}\right) +\sum_{j\geq1}\alpha_j z^{-j}\right)^k
=& 
\left( \alpha \left(z+ \frac{a_{1}}{dA}\right) +\sum_{j=1}^k\alpha_j z^{-j}\right)^k + O\left(\frac1z\right)
\\=& \alpha^k z^k + \sum_{j=0}^{k-1} \beta_j z^j + O\left(\frac1z\right)
\end{align*}
where $\beta_j$ is a sum of terms of the form \[\alpha^\kappa \left(\frac{\alpha a_1}{dA}\right)^\tau \alpha_{j_1} \cdots \alpha_{j_\mu}\]
with $\kappa, \tau, j_1, \ldots, j_\mu \in\N$, $\kappa + \tau +\mu  = k$, and $\kappa - (j_1 + \cdots + j_\mu) = j$. 
Observe that $\tau + 2 (j_1 + \cdots + j_\mu) = \tau + 2\kappa -2j \le 2k$. 
By~\eqref{eq10}, we infer that $(2dA)^{2k} \beta_j \in \Z [\alpha, a_1, \ldots, a_d]$ as required.

Now we focus on the coefficient $\alpha_{k,j}$ in the power series expansion~\eqref{eq008}. 
As above $\alpha_{k,j}$ is a sum of terms of the form
 \[\alpha^\kappa \left(\frac{\alpha a_1}{dA}\right)^\tau \alpha_{j_1} \cdots \alpha_{j_\mu}\]
where $\mu \ge1$, $\kappa + \tau +\mu = k$, and $(j_1 + \cdots + j_\mu) - \kappa = j$.
Since $\tau + 2(j_1 + \cdots + j_\mu) = \tau + 2 ( \kappa +j) \le 2 (k+j-1)$,~\eqref{eq10}
implies $(dA)^{2 (k+j-1)} \alpha_{j,k} \in  \Z [\alpha, a_1, \ldots, a_d]$. 
We have
\[\kappa + \tau + (d-1) (j_1 + \cdots + j_\mu) \le k-1 + (d-1) (\kappa +j)\le d (k-1) + (d-1) j\]
hence $\deg_\alpha(\alpha_{j,k}) \le d (k-1) + (d-1) j$.
On the other hand 
$\tau + (j_1 + 1 + \cdots  + j_\mu + 1) = \tau + \kappa + j +\mu = k+j$, hence
the weighted degree of $\alpha_{j,k}$ as a polynomial in the variables $a_i$'s is at most $k+j$
by~\eqref{eq09}.   
\medskip

Let us now prove the claim. We proceed by induction building on the formulas defining $\alpha_j$ obtained in the proof of the previous proposition. 
The case $j=1$ follows from the equality $ 2(dA)^ 2 \alpha_1 =  - (d-1) \alpha a_1^2 + 2\alpha a_2$.
Assume the claim has been proved for all $j =0, \ldots , s-1$. 

We first observe  that the polynomial $Q_{d-1-s}$ is a sum of terms of the form
\[ 
\alpha^\kappa\left(\frac{\alpha a_1}{dA}\right)^\tau \alpha_{j_1} \cdots \alpha_{j_\mu}\]
with $\tau \in \N$, $d-2 \ge \kappa \in \N$, $\mu\ge1$, $j_1, \ldots, j_\mu \in \{ 1, \ldots, s-1\}$, 
 $\kappa + \tau +\mu =d$, and $\kappa - (j_1 + \cdots + j_\mu) = d-1-s$. 
 Since \[\tau + (j_1 + 1 + \cdots  + j_\mu + 1) = \tau + \kappa - d +1 +s + \mu = s+1 ,\] we get from the induction hypothesis
 that  $\wdeg(Q_{d-1-s}) \le s+1$. Similarly \[\tau + 2 (j_1 + \cdots  + j_\mu) = \tau + 2 (\kappa -d +1 +s) \le 2(s+1),\] hence
 $(2dA)^{2s}Q_{d-1-s} \in \Z[\alpha, a_1, \ldots, a_d]$.
 Finally, one has
 \[\kappa + \tau + (d-1) (j_1 + \cdots  + j_\mu) \le d-1 + (d-1) (\kappa +s -d +1) \le (d-1) ( 1 + s -1) = (d-1) s\]
 so that $\deg_\alpha(Q_{d-1-s}) \le (d-1)s$. 
  By~\eqref{eq020} and~\eqref{eq021}, this proves the induction step when $ s \le d-1$.

 For $s\ge d$, we also need to control the polynomials $Q_{j,s-d+1}$ with $j \le (s-d+1)/d$. 
The polynomial $Q_{j,s-d+1}$ is a sum of terms of the form
\[ 
\frac{1}{A^j} \frac{a_{j_1}}{A} \cdots \frac{a_{j_\mu}}{A}
\]
with $j_1, \ldots, j_\mu \ge1$, 
$dj + (j_1 + \ldots + j_\mu) = s - d +1$, and $\mu \le j$.
We have 
$s-j \ge (d-1)j + (j_1 + \ldots + j_\mu) +d -1 \ge (j_1 + \ldots + j_\mu)$, hence
$\wdeg(Q_{j,s-d+1}) \le s-j$. Similarly, 
one has
\begin{equation}\label{eq676}
j + \mu +1
\le
2j +1 
\le
2 (s-j) 
\end{equation} 
since $j \le (s-d+1)/d$.
It follows that $(2dA)^{2(s-j)} Q_{j,s-d+1} \in \Z[a_1, \ldots, a_d]$.
Observe also that $\deg_\alpha(Q_{j,s-d+1}) =0$. 

Using~\eqref{eq030}, we can now conclude the induction step for $s \ge d$ which proves our claim.
This concludes the proof of the proposition.
\end{proof}

\begin{proposition}\label{prop:expand power}
For any $k\ge1$ we have
\[
\hat{P}_k \circ P = \hat{P}_{kd}~.
\]
\end{proposition}

\begin{proof}
Observe that $\varphi_P(P(z))^k = \varphi_P(z)^{dk}$. The result follows from~\eqref{eq008} and identifying the polar parts of both members.
\end{proof}

\subsection{B\"ottcher coordinate and Green function}
We now fix an algebraically closed complete metrized field $(K,|\cdot|)$ of characteristic $0$ and pick an integer $d\geq2$. Recall that, if $X$ is an algebraic variety over $K$, we denote by $X^\an$ its Berkovich analytification. 

Recall that if  $(c,a)\in K^{d-1}$, we let
\[P_{c,a}(z)\pe \frac{1}{d}z^d+\sum_{j=2}^{d-1}(-1)^{d-j}\sigma_{d-j}(c)\frac{z^j}{j}+a^d,
\]
see~\eqref{eq:defpoly}; and that the critical set of $P_{c,a}$ is given by $c_1, \ldots, c_{d-2}, c_{d-1}=0$. 
Recall that $G(P_{c,a})= \max_{1\leq i \leq  d-1} \{ g_{P_{c,a}}(c_i)\}$, and that
the map $(z,c,a) \mapsto g_{P_{c,a}}(z)$ is continuous, see \S\ref{sec:green}. In particular, for any $\tau>0$ the set 
$\{(c,a,z)\in\mathbb{A}^{d,\an}_K\, ; \ g_{P_{c,a}}(z)>G(P_{c,a}) + \tau \}$ is open. 

\smallskip

We shall rely on the following estimates, see~\cite[Proposition 2.3]{specialcubic} in the cubic case. 
\begin{proposition}~\label{prop:GreenBottcher}
\begin{enumerate}
\item
There exists a constant $\sigma=\sigma(K)\geq0$ such that the B\"ottcher coordinate $\varphi_{P_{c,a}}$ is a convergent power series in 
the neighborhood of infinity $\{z\in\mathbb{A}^{1,\an}_K\, ; \ \log|z|>G(P_{c,a})+C\}$ for any $(c,a) \in K^{d-1}$.
\item
There exists a constant $\rho=\rho(K)\geq0$ such that 
the map $(c,a,z)\longmapsto \varphi_{P_{c,a}}(z)$ extends analytically to the open set
\[\{(c,a,z)\in\mathbb{A}^{d,\an}_K\, ; \ g_{P_{c,a}}(z)>G(P_{c,a})+\rho\}~,\]
and we have  the relations
\[g_{P_{c,a}}=\log|\varphi_{P_{c,a}}| \text{ and } \varphi_{P_{c,a}} \circ P_{c,a} = \varphi^d_{P_{c,a}} \text{ on }U_{c,a}=\{ g_{P_{c,a}}> G(P_{c,a})+\rho\}.\]
\item
There exists a constant $\tau=\tau(K)\geq0$ such that 
\[ \varphi_{P_{c,a}}\colon \{g_{P_{c,a}}>G(P_{c,a})+\tau\}\longrightarrow \mathbb{A}^{1,\an}_K\setminus\overline{\D(0,\exp( G(P_{c,a})+\tau))}\] 
induces an analytic isomorphism.
\item
Finally, when $K$ is Archimedean, or when the residual characteristic of $K$ is zero, or larger than $d+1$ we can take $C= \tau=\rho=0$.
\end{enumerate}
\end{proposition}
\begin{proof}
Let us treat first the case $K$ is Archimedean. In that case most of the statements are proved in~\cite{orsay1} (see also~\cite[\S 1]{BH}). 
In particular,  $\varphi_{c,a}(z)$ is analytic in a neighborhood of $\infty$ and extends to $U_{c,a}:=\{g_{P_{c,a}} > G(P_{c,a})\}$ by invariance (so that $\rho=0$), and it 
defines an isomorphism between the claimed domains with $\tau =0$. It is moreover analytic in $c,a,z$,
the relation $\varphi_{P_{c,a}} \circ P_{c,a} = \varphi^d_{P_{c,a}}$ holds on $U_{c,a}$ since the latter set is connected and the relation is satisfied at a formal level.
Since $\varphi_{P_{c,a}}(z) - \alpha z$ is bounded near infinity, 
we have
\begin{equation}\label{eq:blop}
g_{P_{c,a}}(z) = \lim_{n\to\infty}\frac1{d^n} \log |P_{c,a}^n(z)| = \lim_{n\to\infty}\frac1{d^n}\log |\varphi_{P_{c,a}}(P_{c,a}^n(z))| = \log |\varphi_{P_{c,a}}(z)|
\end{equation}
on $U_{c,a}$.

To estimate more precisely the radius of convergence of the power series~\eqref{eq-bott}, we rely on \cite[\S 4]{BH} as formulated in
\cite[\S 6]{favredujardin}. Write $\delta = \sum_{i=1}^{d-1} c_i /(d-1)$. First choose $C= C_K>0$ such that $G(P_{c,a})  >   \log^+ \max \{|a| , |c| \} - C$, and set 
$\sigma := C + \log 5$. 
Suppose  $\log |z| > G(P_{c,a}) + \sigma$. We first infer 
\[|z- \delta| > 5 \max \{1, |a| , |c| \} - |\delta|\ge (4 +\frac1{d-1})  \max \{1, |a| , |c| \}\]
 hence
$\log |z- \delta| > G(P_{c,a})+ \log (4+(d-1)^{-1})$, so that \[g_{P_{c,a}}(z)\ge \log | z - \delta| - \log 4 > G(P_{c,a})\] by~\cite[Lemma~6.5]{favredujardin}.
We have shown that  $\varphi$ is analytic in $\{z, \, \log |z| > G(P_{c,a}) + \sigma\}$ hence converges in this domain.

\smallskip

From now on, we assume that the norm on $K$ is non-Archimedean. 
Choose any $\alpha$ such that $\alpha^{d-1} = \frac1d$. 
Recall that 
by Proposition~\ref{prop:expan bottcher}, we have
\[\varphi_{P_{c,a}}(z)=\alpha \left(z - \frac{\sigma_1(c)}{d-1}\right) +\sum_{j\geq1}\alpha_j z^{-j} \in \Z\left[\alpha, \frac{\sigma_{1}(c)}{d-1}, \ldots, \frac{\sigma_{d-2}(c)}{2}, a^d\right](\!(z)\!)\]  
where 
\begin{equation}\label{eq481}
2^{2j} \alpha_{j} \in \Z\left[\alpha,  \frac{\sigma_{1}(c)}{d-1}, \ldots, \frac{\sigma_{d-2}(c)}{2}, a^d\right]
\end{equation} is a polynomial in the variables $c$ and $a$ of degree  $\le j+1$, and in the variable $\alpha$ of degree $\le (d-1) j$.

\smallskip

Suppose first that  the  residual characteristic of $K$ is either zero, or larger than $d+1$. Then~\eqref{eq481} implies
$|\alpha_j| \le \max\{1, |c|, |a|\}^{j+1}$ so that $\varphi_{P_{c,a}}$ converges in  $U_{c,a} :=\{|z| >  \max\{1, |c|, |a|\}\}$, and $\log|\varphi_{P_{c,a}}(z)| = \log|z|$
in $U_{c,a}$. It is easy to check that $U_{c,a}$ is invariant by the dynamics, hence~\eqref{eq:blop} applies and $g_{P_{c,a}} = \log|\varphi_{P_{c,a}}(z)|$.
Recall that we have $G(P_{c,a}) = \log^+ \max\{|c|, |a|\}$ by Proposition~\ref{prop:growth Green}.  
It follows that $\varphi_{P_{c,a}}$ is a well-defined analytic function on 
$\{ g_{P_{c,a}} > G(P_{c,a})\}= \{ |z| > \max\{1, |c|, |a|\}\}$.
It induces an isomorphism between  $U_{c,a}$ and $\A^{1,\textup{an}}_{K}\setminus\overline{\mathbb{D}_K(0,e^{G(P_{c,a})}})$ since $\log|\varphi_{P_{c,a}}(z)| = \log|z|$. 
The proposition is thus proved in this case with $C = \rho = \tau =0$.

\smallskip

In residual characteristic $0<p \le d$, the estimates are more delicate. 
Let us set \[\tilde{B}_p := \max \left\{|j|_p^{-1/j}, j =1, \ldots, d-2 \right\}~.\] 
Then~\eqref{eq481}
shows that
\[|\alpha_j| \le (\tilde{B}_p \max\{1, |c|, |a|\})^{j+1} |4|_p^{-j} |d|^{-j}_p\]
hence $\varphi_{P_{c,a}}$ converges for $|z| >   B_p \max\{1, |c|, |a|\}$, with $B_p = \tilde{B}_p\, \max\{|4|_p^{-1}, |d|_p^{-1}\}$, 
and we have $g_{P_{c,a}}=\log|\varphi_{P_{c,a}}| = \log |\alpha z| $ in that range.

Recall the definition of the constants $\theta$ and $C$ from Propositions~\ref{prop:classic-estim} and~\ref{prop:growth Green}. 
Define 
\[ U_{c,a} = \{ g_{P_{c,a}}>G(P_{c,a})+\tau\}\]
with $\tau = C+\theta+ \log B_p$, and pick any $z \in U_{c,a}$. 
Since $G(P_{c,a}) \ge \log^+ \max \{ |c|, |a|\}  - C$, we have
\[\log^+ \max \{ |z|, |c|, |a|\} \ge g_{P_{c,a}}(z) - \theta>  G(P_{c,a})+\tau -\theta 
\ge \log^+ \max \{ |c|, |a|\} + \log B_p
\]
so that $|z| > B_p \max \{ 1, |c|, |a|\}$. It follows that the series $\varphi_{P_{c,a}}$ converges on $U_{c,a}$. 
Since  $g_{P_{c,a}}=\log|\varphi_{P_{c,a}}|= \log |\alpha z| $ on $U_{c,a}$, the map $\varphi_{c,a}$ induces an isomorphism
onto $\A^{1,\textup{an}}_{K}\setminus\overline{\mathbb{D}_K(0,e^{G(P_{c,a})+\tau})}$ with $\rho = \tau$, 
and the proof is complete. 
\end{proof}


\section{Polynomial dynamics over a global field}\label{sec:global}

For any number field $\KK$, recall from \S\ref{sec:adelic height} that  $M_\KK$ is the set of places of $\KK$.
For any $v\in M_\KK$, we write $\KK_v$ for the completion $(\KK, |\cdot|_v )$, and we shall also let $\C_v$ be the completion of the algebraic closure $\bar{\KK}_v$ of  $\KK_v$.
Recall that $n_v = [\KK_v:\Q_v]$, and that the standard height of $x\in\KK$ is defined by 
\[
h_{\st}(x)
=
\frac{1}{[\KK:\Q]} \sum_{v\in M_\KK}n_v\log^+|x|_v
=
\frac{1}{\deg(x)}\sum_{y\in \mathsf{O}(x)}\sum_{v\in M_\Q} \log^+|y|_v~.
\]

\smallskip
Pick $P\in \KK[z]$ of degree $d\geq2$. Following Call and Silverman~\cite{call-silverman}, we define the canonical height\index{height!canonical} $h_P$ of $P$ as the limit
\[h_P(x)=\lim_{n\rightarrow\infty}\frac{1}{d^n}h_{\st}(P^n(x)).\]
This limit always exists and the height function $h_P$ satisfies
$h_P\circ P=dh_P$ and 
$\sup_\KK |h_P-h_{\st}| < +\infty$.
Furthermore, the Northcott property holds:  for all $x\in\bar{\KK}$, we have $h_P(x)\geq0$ and  $h_P(x)=0$ if and only if $x$ is preperiodic under iteration of $P$, i.e. if there exists $n>m\geq0$ such that $P^n(x)=P^m(x)$.

\smallskip

Just like the standard height function, the height function $h_P$ can also be decomposed as a sum of local contributions either under the form~\eqref{eq413} or~\eqref{eq412}.
For any $v\in M_\KK$, denote by $g_{P,v}$ the Green function of $P$ at the place $v$:
\[g_{P,v}(z):=\lim_{n\rightarrow\infty}\frac{1}{d^n}\log^+|P^n(z)|_v, \ z\in\C_v.\]
Fix $x\in \bar{\Q}$. Then
\[h_P(x)
=
\frac{1}{[\KK:\Q]} \sum_{v\in M_\KK} n_v g_{P,v}(x) =
\frac{1}{\deg(x)}\sum_{y\in \mathsf{O}(x)}\sum_{v\in M_{\Q}}g_{P,v}(y).\]
These relations reflect the fact that $h_P$ is induced from the (unique $P^*$-invariant) semi-positive adelic metric on $\cO(1)$. 

\begin{definition}\label{def:bif-height}
Let $P$ be any polynomial defined over a number field $\KK$. 
Its bifurcation height \index{height!bifurcation} is by definition $h_\bif(P):= \sum_{P'(c) =0} h_P(c)$.
\end{definition}

It was proved by Ingram \cite{MR2885981} that the function $P \mapsto h_\bif(P)$ defines a Weil height on the parameter space $\mathrm{MPoly}^{d}$ of polynomials of a fixed degree $d$. 
It was further noticed in~\cite{favregauthier} that if one views $\A^{d-1}$ as an open subset of the projective space,
then using the orbifold parameterization $(P_{c,a})_{c,a}$ by critically marked polynomials 
the height function 
\[
\widetilde{h_\bif}(c,a) := \frac1{\deg(c,a)} \sum_{c'a'\in \mathsf{O}(c,a)}\sum_{v\in M_\KK} G(P_{c',a'}) 
\]
is in fact
determined by a semi-positive continuous adelic metrization on the line bundle $\mathcal{O}(1) \to \p^{d-1}$, 
and satisfies $\frac1{d-1} h_\bif(c,a) \le \widetilde{h_\bif}(c,a) \le h_\bif(c,a)$.
In fact, one has the following.
\begin{proposition}
There exists a constant $C>0$ such that for all $(c,a) \in \bar{\Q}^{d-1}$, 
\[ |\widetilde{h_\bif}(P_{c,a}) - h_{\st}(c,a)| \le C~.\]
\end{proposition} 
\begin{proof}
Pick any $(c,a) \in \bar{\Q}^{d-1}$, and choose any place $v\in M_\Q$. 
By Proposition~\ref{prop:growth Green}, for any $(c',a') \in \bar{\Q}^{d-1}$, we have
\[
|G(P_{c',a'}) - \log^+ \max \{ |c'|, |a'|\}| \le C_v~,
\]
and the set $S$ of places for which $C_v \neq 0$ is finite. 
 Summing up all contributions over $\mathsf{O}(c,a)$ we get
\begin{align*}
|h_\bif(P_{c,a}) - h_{\st}(c,a)| \le &\frac1{\deg(c,a)} \sum_{c'a'\in \mathsf{O}(c,a)}\sum_{v\in S} \left|G(P_{c',a'}) - \log^+\max \{ |c'|, |a'|\}\right|
\\
\le & \sum_{v \in S} C_v<\infty~,
\end{align*}
as required.
\end{proof}
\begin{corollary}\label{cor:PCF}
The set of complex polynomials $P_{c,a}$ of degree $d\ge2$ which are PCF is a countable subset of $\bar{\Q}^{d-1}$
and forms a set of bounded (standard) height. 
\end{corollary}
\begin{proof}
The set $\mathcal{P}$ of PCF polynomials of the form $P_{c,a}$ is defined as the union of  countably many algebraic subvarieties defined 
by equations with coefficients in $\Q$ of the form $\cap_{i=0}^{d-1} \{ P^{n_i}_{c,a}(c_i) = P^{m_i}_{c,a}(c_i)\}$ with $n_i > m_i$. 
By the preceding result, $\mathcal{P} \cap \bar{\Q}^{d-1}$ forms a set of bounded height, hence $\mathcal{P}$ cannot
contain a positive dimensional subvariety. This proves the result. 
\end{proof}


\section{Bifurcations in holomorphic dynamics}\label{sec:holofamily}

We review briefly Ma\~{n}\'e-Sad-Sullivan's theory of bifurcation of holomorphic dynamical systems in the context of polynomials. 
We refer to either the original papers \cite{MSS,lyubich-bif} or to the survey~\cite{bsurvey} by Berteloot for a more detailed account.  

Let $\Lambda$ be any connected complex manifold. A holomorphic family $(P_\lambda)_{\lambda\in\Lambda}$ of polynomials of degree $d\ge 2$ parameterized by $\Lambda$ is by definition
a holomorphic map $(\lambda,z) \mapsto P_\lambda (z)$ from $\Lambda \times \C \to \C$ such that 
for all $\lambda \in \Lambda$, the map $z \mapsto P_\lambda(z)$ is a polynomial of degree $d$.\index{family of polynomials!holomorphic}

A critically marked holomorphic family of polynomials is a holomorphic family $(P_\lambda)_{\lambda\in\Lambda}$ together with $d-1$ holomorphic functions $c_1, \ldots, c_{d-1}:\Lambda\to\C$
such that $\crit(P_\lambda) = \{ c_1(\lambda), \ldots, c_{d-1}(\lambda)\}$ for all $\lambda$. 

\begin{definition}
Let $(P_\lambda)_{\lambda \in \Lambda}$ be any critically marked holomorphic family of polynomials of degree $d$. 
The stability locus\index{stability locus} $\stab(P)$ of a holomorphic family is the union of all open subsets  $U \subset \Lambda$ over which 
the families $\{\lambda\mapsto P_\lambda^n(c_i(\lambda))\}_n$ are normal on $U$ for all $i=1, \ldots, d-1$. 
\end{definition}

When $(P_\lambda)_{\lambda \in \Lambda}$ is an arbitrary holomorphic family of polynomials, then one can define the 
stability locus as follows. 

\smallskip

The behaviour of the Julia set on the stability locus is governed by holomorphic motions. 
We recall briefly this crucial notion here. Fix any $\lambda_0 \in \Lambda$. 
\begin{definition}
A holomorphic motion\index{holomorphic motion} of a subset $X \subset \C$ parameterized by $(\Lambda,\lambda_0)$ is a family of holomorphic maps
$\{\lambda \mapsto x(\lambda)\}_{x\in X}$ such that $x(\lambda_0) = x$, and
for all $\lambda$, $x \mapsto x(\lambda)$ is injective from $X$ to $\C$.
\end{definition}
A holomorphic motion can also be viewed as a map $h\colon \Lambda \times X \to \C$ such that 
$h(\lambda_0, \cdot) = \id$, $h(\lambda, \cdot)$ is injective and $h(\cdot, x)$ is holomorphic for all $x$.

\smallskip

Using Montel's and Hurwitz's theorems, one can show that any holomorphic motion of a set $X$ extends canonically to its closure $\bar{X}$. 
Moreover for any fixed parameter $\lambda$, the injective map $X \to \C$ sending $x$ to $x(\lambda)$ is quasi-symmetric, see~\cite[\S5.2]{hubbard-teichmuller1}.

\begin{theorem}
The stability locus $\stab(P)$ of a holomorphic family of polynomials is the union of all connected pointed open subsets $(U,\lambda_0)\subset \Lambda$ 
on which there exists a holomorphic motion
$h\colon U \times J(P) \to \C$ such that $ P_\lambda ( h_\lambda(z)) = h_\lambda(P_{\lambda_0}(z))$.

 In particular the dynamics of any two polynomials
$P_{\lambda_0}$ and $P_{\lambda_1}$ for which $\lambda_0$ and $\lambda_1$ lie in the same connected component of $\stab(P)$ are topologically 
conjugated on their Julia sets. 
\end{theorem}
Because of this theorem it is of common use to refer to the stability locus as the $J$-stability locus. 
Beware that in general however, the stability locus of $P$ is not connected (this phenomenon already appears in degree $2$). 

The previous result relies on the characterization of the stability locus in terms of the stability of periodic points. 
\begin{definition}
A parameter $\lambda$ is said to have an unstable periodic point $z$, iff there exist two sequences of parameters $\lambda_n^{\pm}$ 
such that  $\lambda_n^{\pm} \to \lambda$, and a sequence $z_n^+$ (resp. $z_n^-$) of repelling (resp. attracting) periodic points for 
$P_{\lambda_n^+}$ (resp. $P_{\lambda_n^-}$) such that $z_n^\pm \to z$ as $n\to \infty$.
\end{definition}
Observe that any unstable periodic point is necessarily neutral. 

When $\lambda_0$ has an unstable periodic point at $z_0$, then one can find a finite map $r\colon (B,0) \to (\Lambda, \lambda_0)$
defined on a small open ball $B$ centered at the origin in $\C^{\dim(\Lambda)}$, and a holomorphic map $p\colon (B,0) \to \C$
such that  $p(r(t))$ is a periodic point of some fixed period $n$ for all $t\in \B$, $p(r(0)) = z_0$ and the multiplier of 
$P_{r(t)}$ at $p(r(t))$ is a non-constant holomorphic function
whose value at $0$ equals $|(P_{\lambda_0}^n)' (z_0)| = 1$.

\begin{theorem}
The complement of stability locus of a holomorphic family of polynomials coincides with the closure of the set of parameters having an unstable periodic point. 
\end{theorem}

Again using Montel's theorem, one can infer the following crucial density statement.

\begin{theorem}
The stability locus  of any holomorphic family of polynomials is open and dense. 
\end{theorem}

\begin{remark}
We shall return to the more general notion of stability of a pair in Chapter~\ref{chapter-pairs} where a characterization of the stability locus will be given in terms
of the variation of the Green function, see Proposition~\ref{prop:bifurcation} and Theorem~\ref{thm:active-characterization1}. 
\end{remark}

The previous theorem was made more precise by McMullen and Sullivan~\cite[Theorem~7.1]{McMS}.
\begin{theorem}\label{thm:full hol motion} 
Let $(P_\lambda)_{\lambda \in \Lambda}$ be any analytic family of polynomials. Then there exists an open and dense subset $C(P) \subset \stab(P)$
and, for any simply connected pointed domain $(U,\lambda_0)\subset C(P)$, a holomorphic motion $h \colon U \times \C \to \C$ that conjugates $P_\lambda$ to $P_{\lambda'}$ on $\C$ for any pair of parameters $\lambda, \lambda'\in U$.
\end{theorem}

\begin{remark}\label{rem:even for rat}
The previous theorem is valid for any family of rational maps of the projective line.
\end{remark}


\section{Components of preperiodic points} \label{sec:preper comp}

In order to develop a specialization argument in \S\ref{sec:specialization} we shall need a detailed analysis of the locus of preperiodic points in arbitrary families
of polynomials.  Let $K$ be any algebraically closed field of characteristic $0$, and $\Lambda$ be any algebraic variety defined over $K$.
Let $P:(\lambda,z)\in \Lambda \times \A^1 \mapsto (\lambda,P_\lambda(z))\in \Lambda\times \A^1$ be any algebraic family of  degree $d$ polynomials, and define
\[\preper:=\left\{(\lambda,z)\in\Lambda\times \C\, : \ \{P_\lambda^n(z)\}_{n\geq1} \ \text{is finite}\right\}.\]
Observe that the set $\preper$ is a countable union of irreducible analytic hypersurfaces of $\Lambda\times\A^1$.

\begin{theorem}\label{thm:finite branches}
Let $(\lambda_0,z_0)$ be any point in $\Lambda\times\A^1$ which belongs to an infinite sequence $\{Z_i\}_{i\ge0}$ of distinct irreducible hypersurfaces included in $\preper$.
Then after a base change, conjugating the family by suitably affine transformations, and possibly replacing $P$ by some iterate
we are in the following situation.
\begin{itemize}
\item
The point $z_0$ is fixed for any parameter $\lambda$. Further it is super-attracting for $P_{\lambda_0}$ and the local degree function 
$\lambda \mapsto \ord_{z_0}(P_\lambda)$ is not locally constant near $\lambda_0$.
\item
One can find a sequence of integers $n_i \to \infty$ such that $Z_i$ is a component of $\{(\lambda,z) , P^{n_i}_\lambda(z) = z_0 \}$.
\end{itemize}
\end{theorem}
\begin{remark}
The statement above is actually true in the analytic category, for any holomorphic family of endomorphisms of degree at least $2$ of the projective space of any dimension. 
\end{remark}

\begin{example}
Consider the family $P_\lambda(z) = \lambda z + z^2$ parameterized by the affine line $\Lambda = \A^1_K$. 
Define the  hypersurfaces  $Z_1 = \{ z + \lambda =0\}$, and $Z_{n+1} = \{ P^n_\lambda(z) + \lambda = 0 \}$ for all $n\ge 1$. 
Since $P^n_\lambda(z)$ is a polynomial of degree $2^n$ of dominant term $z^{2^n}$ with no linear terms in $\lambda, z$, these varieties are all smooth at $(0,0)$ and tangent to $\{ \lambda =0\}$ up to order $2^n$. It follows that for each $n$ the variety $Z_n$
is irreducible and included in $\preper$. 
When $\lambda =0$, then $\{Z_n(0)\}_{n\in \N^*} = \{ (0,0)\}$, and since $-\lambda$ is strictly prefixed, $\{ Z_n(\lambda)\}_{n\in \N^*}$ is infinite when $\lambda \neq 0$.
\end{example}

\begin{proof}[Proof of Theorem~\ref{thm:finite branches}]
Since $K$ has characteristic $0$, we can resolve $\Lambda$, and suppose it is a smooth variety.
Replacing $P$ by some iterate, we may suppose that $z_0$ is preperiodic to a fixed point $z_\star$ for $P_{\lambda_0}$. 
This fixed point also belongs to an infinite set of irreducible hypersurfaces $W_i$ included in $\preper$ since $P$ is finite. 
Up to an affine change of coordinates and base change, $z_\star=0$ is a persistent fixed point i.e. $P(0) =0$.
Moreover, since the problem is local, we shall also 
argue in a formal neighborhood of $\lambda_0=0$ in $\Lambda$. Denote by $\mu:=(P_{0})'(0)$ the multiplier of the fixed point $0$.

We shall prove that $\mu =0$. 

\medskip

We first claim that $\preper$ is locally an irreducible subvariety near $(0, 0)$ when $\mu$ is non-zero and not a root of unity. 
Observe that $P_\lambda(z) = \mu z + O(\lambda z, z^2)$ so that 
\[P_{\lambda}^{m}(z)=\mu^{m} z + O(\lambda z, z^2) \text{ for all } m.\]
It follows that for all $l,m$ we have 
\[P_\lambda^{m}(z)-P_{\lambda}^{m+l}(z)=\left(\mu^m-\mu^{l+m}\right) z +O(\lambda z, z^{2}),\]
Since $\mu^{m}-\mu^{l+m}\neq0$, we conclude that 
\[ 
W_{m,l} := \{P_\lambda^{m}(z) = P_\lambda^{m+l}(z)\}\]
is locally at the origin given by the equation $\{z=0\}$, ending the proof in this case.

\medskip

Next suppose that $\mu$ is a $q$-th root of unity. 
Then we claim that  at most $(d-1) q+2$ irreducible hypersurfaces included in $\preper$ may contain $(0,0)$. 

By~\cite[\S 1]{MR542732}, 
we may formally conjugate $P_0$  to the following normal form:
\[P_0(z)=\mu \left(z+z^{\nu q+1}\right)+O(z^{\nu q+2}), \text{ for some } \nu\ge 1~.\]
By \cite[Proposition 6 p. 72]{orsay1}, we have\footnote{the arguments of Douady and Hubbard only works when $K=\C$, but we may invoke Lefschetz principle over an arbitrary field of characteristic zero.}  $\nu \le d-1$.
For any $m\geq1$, we find
\[P_\lambda^m(z)=\mu^m \left(z+mz^{\nu q+1}\right)+O(\lambda z, z^{\nu q+2}),\]
so that 
\[ P_\lambda^{m}(z)-P_\lambda^{m+l}(z)=(\mu^{m}-\mu^{m+l})z+(m\mu^m-(m+l)\mu^{m+l})z^{\nu q+1}+O(\lambda z, z^{\nu q+2}).\]
Observe that when $\mu^l\neq 1$, then the unique component of $W_{m,l}$ passing through the origin is given by $z=0$.
Otherwise $q$ divides $l$, and $P_\lambda^{m}(z)-P_\lambda^{m+l}(z)$ is equivalent to a Weierstra{\ss} polynomial in $z$ of degree $\nu q +1$. Since this degree is independent on $m$ and $l$ and $W_{0,q} \subset W_{m,l}$ we conclude that $W_{0,q} = W_{m,l}$. The maximal number of irreducible components of $W_{0,q}$ is
$\nu q +1$ which concludes the proof of the claim.

\medskip

At this point we have proved that $0$ is a super-attracting fixed point so that we can write
\[P_0(z)= z^k+O(z^{k+1}), \text{ for some } 2 \le k \le d~,\]
and 
\[P_\lambda(z)= \sum_{i=i_0}^{k} a_i(\lambda) z^i +O_\lambda(z^{k+1}),\] 
with $a_{i_0}\not\equiv0$, $a_i(0) = 0$ for $i \le k-1$,  and $a_k(0) =1$.
If the local degree of $0$ is locally constant, we have $a_i \equiv 0$ for all $i \le k-1$ so that $i_0 =k$, and we can write
\[ P_\lambda^{m}(z)-P_\lambda^{m+l}(z)=a_{m,l}(\lambda) z^{k^m}+O_\lambda(z^{k^m+1}),\]
with $a_{m,l}(0)= 1$. As above, we infer that there exists a unique irreducible hypersurface passing through $(0,0)$ and included in $\preper$, namely $\{z=0\}$.

It remains to show that any component of $W_{m,l}$ passing through $0$ is included in $\{ P_\lambda^{m}(z) = 0\}$. 
We observe that the image under $(\lambda, z) \mapsto (\lambda, P_\lambda^m(z))$ of $W_{m,l}$ is included
in $W_{0,l}$ which is defined by the equation
\[0 = z - P_\lambda^{l}(z)= z+O_\lambda(z^{2}),\]
hence locally equal to $z=0$. 
This concludes the proof of the theorem.
\end{proof}

The proof of the previous theorem works verbatim for any family of rational maps of the projective line. 
It implies in particular the following statement. 
\begin{theorem}
Let $\{R_t\}_{t\in \D}$ be any holomorphic family of rational maps of degree $d\ge 2$ parameterized by the unit disk.
Let $Z_n \subset \p^1_\C \times \D$ be any sequence of distinct irreducible closed curves
such that for all $t$, the point $Z_n(t) := Z_n \cap (\p^1_\C \times \{t\})$ is preperiodic for $R_t$.
 
Then there exists an open subset $U\subset \Lambda$ such that 
$\{ Z_n(t)\}_{n\ge0}$ is infinite for all $t\in U$.
\end{theorem}
\begin{proof}
By Theorem \ref{thm:full hol motion} we may find a connected open set $U\subset \D$ over which the family is stable
and the dynamics of $R_t$ and $R_{t'}$ are topologically conjugated for any $t, t'\in U$, see Remark \ref{rem:even for rat}. In particular, the number of super-attracting orbits remains the same with the same multiplicity
since two super-attracting periodic points are topologically conjugated iff their local degree is the same. 
The result then follows from Theorem \ref{thm:finite branches}.
\end{proof}

\begin{remark}
The previous result may be used to clarify the specialization argument of \cite[p.16]{GNY2}.
\end{remark}

%
%
%
%
%
%


\chapter{Dynamical symmetries}\label{chapter:symmetries}

Symmetries of Julia sets play an important role in problems of unlikely intersections. 
In this chapter we study various notions of dynamical symmetries for polynomials. 

In \S\ref{sec:symmetry} we define the group of dynamical symmetries  $\Sigma(P)$ of a single polynomial $P$ and
present various characterizations of it especially in the Archimedean case. We investigate the variation 
of this symmetry group when $P$ belongs to an algebraic family in \S\ref{sec:symmetriesinfamily}.
We then prove (Theorem~\ref{tm:algebraiccharacterization}) that $\Sigma(P)$ consists of those affine transformations
mapping the set of preperiodic points of $P$ onto itself.

In \S\ref{sec:primitivity} we introduce the notion of primitive polynomials, which are polynomials that cannot be written as iterates
of polynomials of lower degree up to symmetries. Families of non primitive polynomials are sources of undesirable examples 
of unlikely intersections. We show that any family of non primitive polynomials is induced by a family of lower degree polynomials
(Proposition~\ref{prop:prim}).

It was progressively realized that a polynomial might have symmetries induced by polynomials of degree $\ge2$, see~\cite{BD,Ghioca-Hsia-Tucker2}.
We investigate this phenomenon in \S\ref{sec:Ritt-theory} building on Ritt's theory which aims at describing all compositional factors of a given polynomial.
We introduce the notion of intertwined polynomials. Two polynomials $P$ and $Q$ of the same degree are intertwined if the map $(z,w) \mapsto (P(z),Q(w))$ fixes a non trivial curve. Building on works by M\"uller-Zieve~\cite{zieve-muller}, 
Medvedev-Scalon~\cite{medvedev-scanlon}, Ghioca-Nguyen-Ye~\cite{MR3632102,GNY}, and Pakovich~\cite{pako:preimages,pako:rational,pakovich}, we explore this equivalence relation in the moduli space of polynomials.

We conclude this chapter by a description in \S\ref{sec:stratif} of the basic stratification of the moduli space of polynomials in degree $d\le 6$
induced by the presence of symmetries.


\section{The group of dynamical symmetries of a polynomial}\label{sec:symmetry}

Let $K$ be any algebraically closed field of characteristic zero.

A reduced presentation\index{polynomial!reduced presentation of} of a monic and centered polynomial $P\in K[z]$ is 
the choice of two integers $\mu$ and $m$ and a polynomial $P_0$
such that $P_0(0) \neq 0$, $P(z) = z^{\mu} P_0(z^m)$, and $P_0$ cannot be further written as 
$Q(z^l)$ for some polynomial $Q$ and some integer $l\geq 2$. Such a presentation is unique, 
$P_0$ is again monic and it is centered when $m =1$.

\begin{definition}
The group $\Sigma(P)$ of dynamical symmetries\index{polynomial!dynamical symmetries of} of a monic and centered polynomial having a reduced presentation 
$P(z) = z^{\mu} P_0(z^m)$ is the cyclic group  $\U_m$ when $P_0 \neq 1$ and equals the group of all roots of unity when $P$ is a monomial map. 
\end{definition}
When $P$ is not a monomial map, the order of $\Sigma(P)$ is always less than $d$, and it equals $d$ iff $P(z) = z^d + c$ for some $c \in K^*$, i.e. $P$ is unicritical and non monomial (whence non-integrable if $d\geq3$). 

Since two monic and centered polynomials can be conjugated only by mutiplication by a root of unity, it follows
that $\Sigma(P)$ does \emph{not} depend on the conjugacy class of $P$ in the set of monic and centered polynomials. 
In fact since the group of roots of unity is abelian, $\Sigma(P)$ is \emph{canonically} isomorphic to $\U_m$.

We can thus define the group of dynamical symmetries of any polynomial $P$ 
by setting 
\[ \Sigma(P) = \{ A^{-1} \circ \sigma \circ A, \, \sigma \in \Sigma (A^{-1} P A)\}\]
for any affine map $A$ such that $A^{-1}\circ P\circ A$ is  monic and centered.

\begin{remark}
When $K$ is not algebraically closed, we embed it into some algebraically closed field $L$. If $P$ is a polynomial with coefficients in $K$, we may view it as a polynomial in $L$ and consider $\Sigma_L(P) \subset \Aff(L)$ its group of dynamical symmetries. We shall set 
$\Sigma(P): = \Sigma_L(P) \cap \Aff(K)$. In general $\Sigma(P)$ is much smaller than  $\Sigma_L(P)$.
\end{remark}
Recall the definition of  the automorphism\index{polynomial!automorphism of} group of the polynomial $P$:
\[
\aut(P) = \{ g \in \Aff(K), \, P(g\cdot z) = g \cdot P(z) \}~. 
\]
\begin{example}
The group of symmetries and the group of automorphisms of the Chebyshev polynomials $T_d$ are equal to $\Sigma(T_d)=\mathbb{U}_2$ and 
\[\mathrm{Aut}(T_d)=\left\{\begin{array}{ll}
\{\mathrm{id}\} & \text{if} \ k \ \text{is even},\\
\mathbb{U}_2 & \text{if} \ k \ \text{is odd}.
\end{array}\right.\]
\end{example}

\begin{proposition}\label{prop:Sigma}
Let $P$ be any polynomial of degree $d\ge2$. The group $\Sigma(P)$ is the union of all finite subgroups $G$ of $\Aff(K)$ 
such that there exists a morphism $\rho : G \to G$ satisfying $P(g\cdot z) = \rho(g) \cdot P(z)$.
\end{proposition}

\begin{remark}\normalfont
In particular, we get $\mathrm{Aut}(P)\subset\Sigma(P)$, so that $\mathrm{Aut}(P^n)\subset \Sigma(P)$ for all $n\geq1$.
In general the inclusion $\cup_{n\in \N} \mathrm{Aut}(P^n)\subset \Sigma(P)$  is strict, for instance for any quadratic polynomial not conjugated to the square map. 
\end{remark}

\begin{proof}
One may suppose that $P$ is monic and centered and has a reduced presentation  $P(z) = z^{\mu} P_0(z^m)$.
The case $P$ is monomial (i.e. $P_0 = 1$) is easy to deal with so that we assume $\deg(P_0) \ge 1$ during the whole proof. 

\smallskip
Take any finite subgroup $G$ of $\Aff(K)$ with a morphism  $\rho : G \to G$ satisfying $P(g\cdot z) = \rho(g) \cdot P(z)$.
Any finite group is of the form $G = \{ \zeta z + a (\zeta -1), \, \zeta \in \U_r\}$ for some $r$ and some $a\in \C$. 
For any $\zeta$ one can thus find another root of unity $\xi$ such that 
$$
( \zeta z + a (\zeta -1))^{\mu} P_0( (\zeta z + a (\zeta -1))^m) = \xi z^{\mu} P_0(z^m) + a (\xi-1)
$$
Since $P$ is centered the term of degree $\deg(P) -1$ is zero on the right hand side. 
As $P_0$ is centered when $m=1$, the left hand side can be written as 
$$
(\zeta^\mu z^\mu + \mu a (\zeta-1) \zeta^{\mu-1}z^{\mu-1} + \text{l.o.t.}) 
(\zeta^{d-\mu} z^{d-\mu} + (d-\mu) \zeta^{{d-\mu}-1} a (\zeta-1) z^{{d-\mu}-1} +  \text{l.o.t.})
$$
so that comparing both sides we get 
$d a (\zeta -1) \zeta^{d-1} =0$
hence $a =0$. We thus have 
$$
( \zeta z )^{\mu} P_0( (\zeta z )^m) = \xi z^{\mu} P_0(z^m)
$$
Since $P_0$ is monic, we have $\zeta^\mu = \xi$ and therefore 
$P_0( (\zeta z )^m) = P_0(z^m)$ which implies $\zeta^m =1$. This proves $G \subset \Sigma(P)$.
Since $\Sigma(P)$ is a finite group, the conclusion follows. 
 \end{proof}

\begin{definition}
We define the group of \emph{reduced symmetries} by letting
\[\Sigma_0(P):=\bigcup_{n\geq1}\ker(\rho^n)\]
where $\rho$ is the morphism given by Proposition~\ref{prop:Sigma}.\index{polynomial!dynamical symmetries of! reduced}
\end{definition}
When $P$ is a monic and centered polynomial having reduced presentation $z^\mu P_0(z^m)$ then the group $\Sigma_0(P)$ satisfies
\[\Sigma_0(P) = \bigcup_{n\geq1} \{ \zeta \in \U_m, \, \zeta^{\mu^n} =1\}.\]
It is thus clear that
\begin{enumerate}
\item $\Sigma_0(P)$ is trivial if and only if $\mu\neq0$ and $\mu$ and $m$ are coprime,
\item $\Sigma_0(P)=\Sigma(P)=\U_m$ if and only if $\mu=0$, or all prime divisors of $m$ divide $\mu$.
\end{enumerate}

Over the field $K=\C$ of complex numbers, the dynamical symmetries have been studied in details. We have the following result by~\cite{Baker-Eremenko,Schmidt-Steinmetz}.
For any compact set $J$ of the complex plane, denote by $\aut(J)$ the subgroup of affine transformations $g \in \Aff(\C)$ fixing $J$.

\begin{theorem}\label{thm:SS}
Suppose $J$ is the Julia set of a complex polynomial of degree at least $2$ which is  not integrable. 

Then  there exists a polynomial $Q$ such that any polynomial $P$ with  $J(P) = J$ is of the form 
$P= g \circ Q^n$
for some integer $n$, and some $g \in \aut(J)$.
\end{theorem}

From this theorem, we obtain
\begin{proposition}\label{prop:sameSigma}
Pick any complex polynomial $P\in \C[z]$ of degree at least $2$. 
\begin{enumerate}
\item 
The  group $\aut(J(P))$  of affine transformations fixing $J(P)$ coincides with the group
of all $g\in \Aff(\C)$ such that  the Green function satisfies $g_P ( g\cdot z) = g_P(z) $.
\item 
When $P$ is not conjugated to a monomial, then $\Sigma(P) = \aut(J(P))$; otherwise $\aut(J(M_d)) = S^1$
 and $\Sigma(M_d) = \U_\infty$ is
the set of torsion elements of $\aut(J(M_d))$. 
\end{enumerate}
  \end{proposition}

\begin{remark}\normalfont
It would be interesting to generalize the second item to any metrized non-archimedean field $(K,|\cdot|)$. Note however that when $P$ has good reduction, then $g_P = \log^+|z|$ and the whole group $\Aff(K^{\circ})$ preserves the Green function. It is likely that for a polynomial $P$ not having potential good reduction  then  $\Sigma(P)$  coincides with the set of $g\in \Aff(K)$ such that  $g_P ( g\cdot z) = g_P(z)$. 
\end{remark}

 \begin{proof}
 We assume $P$ is monic and centered. 
 Denote by  $G$ be the  group of all $g \in \Aff(\C)$ such that $g_P(g\cdot z) = g_P(z)$. 
 Any $g\in G$ leaves  the filled-in Julia set of $P$ invariant, hence belongs to $\aut(J(P))$. Conversely, any $g\in \aut(J(P))$
 preserves the Green function with a logarithmic pole at infinity of $J(P)$ hence belongs to $G$. 
 We thus have $G=\aut(J(P))$.

 In the case $P$ is a monomial, $\aut(J(P)) = S^{1}$ and $\Sigma(P)$ is the set of roots of unity, hence the result follows. 
 We suppose from now on that $P$ is not conjugated to a monomial map.  Observe that for any element $g$ in the group $\Sigma(P)$ then $P^n (g\cdot z) = \rho^n(g) \cdot P^n(z)$ for all $n$
hence $$g_P(g\cdot z) = \lim_n \frac1{d^n} \log^+|\rho^n(g) \cdot P^n(z)| = g_P(z)~,$$
 and $\Sigma(P) \subset \aut(J(P))$.
 
Since $J(P)$ is compact, the group $\aut(J(P))$ is also compact hence finite, since otherwise $J(P)$ would be a circle and $P$ is monomial by Corollary~\ref{cor:zdunik}.
Assume $J(P)$ is not a segment.
For any $g\in \aut(J(P))$ the polynomial $P(g\cdot z)$ fixes $J(P)$ hence by Theorem~\ref{thm:SS} there exists $\rho(g) \in \aut(J(P))$
such that  $P(g\cdot z) = \rho(g) \cdot P(z)$, and Proposition~\ref{prop:Sigma} (1) implies $\aut(J(P)) \subset \Sigma(P)$.

When $J(P)$ is a segment, then $P$ is conjugated to $\pm T_d$ by Corollary~\ref{cor:zdunik} so that we may assume that $J(P) = [-2,2]$ in which case we have $\aut(J(P)) = \{ \pm \id\}$ which is included in $\Sigma(P)$. 
\end{proof}

\begin{corollary}
Let $K$ be any field of characteristic zero and pick any $P\in K[z]$ of degree $d\geq2$. Then $\Sigma(P^n) = \Sigma(P)$ for all $n\in \N^*$.
\end{corollary}

\begin{proof}
Replacing $K$ by the field generated by the coefficients of $P$ over $\bar{\Q}$ we may suppose that 
$K$ is finitely generated over $\bar{\Q}$, and further fix an embedding $K \subset \C$. 
It is sufficient to treat the case $P$ is not conjugated to a monomial map. 
By the previous proposition, we obtain that
\[\Sigma(P^n) = \aut(J(P^n)) = \aut(J(P)) =\Sigma(P)~.\]
This concludes the proof.
\end{proof}
\begin{remark}
It is not clear how to get a purely algebraic proof of the previous fact which could also be applied in positive characteristic. 
\end{remark}


\section{Symmetry groups in family} \label{sec:symmetriesinfamily}
As in the previous section,  $K$ is any algebraically closed field of characteristic zero.
Let $V$ be a $K$-affine variety. 
Recall that an algebraic family\index{family of polynomials!algebraic} $P$ of polynomials of degree $d\ge2$ parameterized by $V$ is 
a map $P : V \times \A^1_K \to  V\times \A^1_K$ such that 
$P(t,z) = (t, P_t(z))$ and $P_t$ is a polynomial of degree $d$ for all $t\in K$.

Given any regular morphism $\pi: W \to V$, one can lift an algebraic family parameterized by $V$ to an algebraic family $Q$ parameterized by $W$
by setting $Q(t,z) = (\pi(t), P_{\pi(t)}(z))$.

Observe that in the case $\Sigma(P)$ is finite, it is a cyclic group hence its isomorphism class is determined by its cardinality.
The following explains the relation between the dynamical symmetries of a family $P$ and those of $P_t$ for a generic parameter $t$.

\begin{proposition}\label{prop:sameZariskiopen}
Let $V$ be an irreducible affine variety defined over $K$, and let $P$ be an algebraic family  parametrized by $V$. 
Then the function $t\mapsto \mathrm{Card} (\Sigma(P_t)) \in \N^* \cup \{ +\infty \}$ is upper semi-continuous with respect to the Zariski topology.

In particular there exists a Zariski dense open subset of $V$ on which all group of dynamical symmetries are isomorphic.
\end{proposition}

\begin{proof}
Write $P_t(z)=\alpha_0(t)z^d+\alpha_1(t)z^{d-1}+\mathrm{l.o.t.}$ and
consider the affine subvariety $$W = \{ (a, b, t) \in \A^2_K \times V , \, \alpha_0(t)a^{d-1}=1 \text{ and } d\alpha_0(t)a^{d-1}b+\alpha_1(t)=0\}~.$$
The second projection $\pi: W \to V$ is a finite ramified cover, and the lift of the family $P$ by $\pi$
is a family of polynomials of degree $d$ which is conjugated to a family of monic and centered polynomials 
 $\tilde P_t(z) = z^d + \sum_{i=0}^{d-2} \tilde{\alpha}_{d-i}(t) z^i$ parametrized by $W$.
 The result then follows from the observation 
 that the set of $t \in W$ such that $\mathrm{Card}  (\Sigma(\tilde P_t))$ is divisible by $ m$  is equal to  
 the union over all integers $\mu\le d$ such that $m | d - \mu$ of the sets 
of  $ t \in W$ such that $\tilde{\alpha}_{d-i}(t) = 0$  when $i- \mu$ is not divisible by $m$.
 \end{proof}

\begin{remark}\normalfont
We may also consider a family $P$ as a polynomial $\mathsf{P}\in K(V)[T]$ over the field $K(V)$. The above implies that $\Sigma(\mathsf{P})\simeq \Sigma(P_t)$ for all $t$ in a dense Zariski open set of $V$.
\end{remark}

\begin{proposition}\label{prop:generic group}
For any degree $d\ge3$, there exists a non-empty Zariski open subset of $U \subset \mpoly^d$ such that 
$\Sigma(P) = \{ \id\}$ for all polynomials $P$ of degree $d$ whose conjugacy class lies in $U$.
\end{proposition}

\begin{remark}\normalfont
Any quadratic polynomial $P$ is conjugated to $z^2 + c$ for some $c \in K$, hence $\Sigma(P)$ has cardinality equal to $2$ except if $P$ is conjugated to a quadratic monomial map.
\end{remark}

\begin{proof}
The set of monic and centered polynomials $P(z) = z^d + a_2 z^{d-2} + \ldots + a_d$  having a trivial group of dynamical symmetries
is Zariski open, since it contains all polynomials for which $\prod_{i=2}^d a_i \neq 0$.
The proposition follows from Proposition~\ref{prop:sameZariskiopen}, and the fact that the projection map from the space of monic and centered polynomials to $\mpoly^d$ is finite.
 \end{proof}


\section[Algebraic characterization of dynamical symmetries]{Algebraic characterization of the group of dynamical symmetries}  
\label{sec:alg-char-sym}

Let us first treat the case of a polynomial $P$ defined over a number field. Recall the definition of the canonical height $h_P$ from \S \ref{sec:global}.
\begin{proposition}\label{prop:algebraiccharacterization}
Let $\KK$ be a number field and pick $P\in \KK[z]$ of degree $d\geq2$ and $g\in\Aff(\KK)$. The following assertions are equivalent:
\begin{enumerate}
\item $g\in \Sigma(P)$,
\item $h_P(g\cdot z)=h_P(z)$ for all $z\in \bar{\KK}$,
\item $g (\preper(P,\bar{\KK})) = \preper(P,\bar{\KK})$,
\item $g (\preper(P,\bar{\KK})) \cap \preper(P,\bar{\KK})$ is infinite.
\end{enumerate}
\end{proposition}

\begin{proof}
Assume first $g\in\Sigma(P)$. Recall that $P^n(g\cdot z) = \rho^n(g) \cdot P^n(z)$ for all $n$. Pick any place $v\in M_\KK$ and $g\in\Sigma(P)$. For any $z\in\bar{\KK}$, we have
 $$g_{P,v}(g\cdot z) = \lim_n \frac1{d^n} \log^+|\rho^n(g) \cdot P^n(z)|_v = g_{P,v}(z)~,$$
 so that the canonical heights are equal $h_P(g\cdot z)=h_P(z)$. This implies
 $g(\preper(P,\bar{\KK})) = \preper(P,\bar{\KK})$ and shows the sequence of implications $(1)  \Rightarrow (2) \Rightarrow (3) \Rightarrow (4)$. 
 
 \smallskip
 
Assume now that $g (\preper(P,\bar{\KK})) \cap \preper(P,\bar{\KK})$ is infinite and pick an infinite sequence of distinct preperiodic points $z_n$
 and $g\in \Aff(\KK)$ such that $g\cdot z_n$ is still preperiodic. Then the canonical heights of both $z_n$ and $g\cdot z_n$ are zero and by~Theorem~\ref{tm:yuan}
 we get $g_* \mu_{v,P} = \mu_{v,P}$ for any place $v$. Choose any archimedean place $v$.  The previous invariance implies  that $g_{P,v}(g\cdot z) - g_{P,v}(z)$
 is a harmonic function on $\C$ with at most logarithmic growth, hence is a constant. Since it is zero at all points $z_n$ we get $g_{P,v}(g\cdot z) = g_{P,v}(z)$. We conclude that  $g$ belongs to $\Sigma(P)$ using the first point of Proposition~\ref{prop:sameSigma}.
\end{proof}
Using a specialization argument, we now extend the previous result to any field of characteristic zero.

\begin{theorem}\label{tm:algebraiccharacterization}
For any field $K$ of characteristic zero and any $P\in K[z]$ of degree $d\geq2$, the group $\Sigma(P)$ coincides with the set of $g\in \Aff(K)$ such that 
$g (\preper(P,\bar{K})) \cap \preper(P,\bar{K})$ is infinite.
\end{theorem}

\begin{remark}
Observe that \cite[Theorem 5.7]{BDM} is very similar to our theorem. It states that if $P$ is defined over a number field, then 
$g \in \Sigma (P)$ iff $P$ and $g \circ P$ have an infinite number of common preperiodic points. 
\end{remark}

\begin{proof}[Proof of Theorem \ref{tm:algebraiccharacterization}]
One direction is clear since $P^n(g\cdot z) = \rho^n(g) \cdot P^n(z)$
 so that $g(\preper(P,\bar{K})) = \preper(P,\bar{K})$.

 \medskip
 
Let us now focus on the converse implication. If $K$ is a number field, this is part of the statement of Proposition \ref{prop:algebraiccharacterization}, so we may  assume that $K$ 
has transcendence degree $\ge 1$ over $\Q$. Let $R$ be the $\bar{\Q}$-algebra of finite type generated over $\bar{\Q}$ by all coefficients of $P$.
 The spectrum of $R$ is an affine variety $X$ of positive dimension defined over $\bar{\Q}$ and the coefficients of $P$ can be viewed as regular functions on $X$, so that 
$P$ defines an algebraic family of degree $d$ polynomials parameterized by $X$. 
Observe that any maximal ideal $t$ of $R$ corresponds to a point in $X(\bar{\Q})$, and that you can associate to $t$ a polynomial $P_t$ with coefficients in $\bar{\Q}$.

Pick any $g\in \mathbb{G}_a(K)$ such that $g (\preper(P,\bar{K})) \cap \preper(P,\bar{K})$ is infinite. 
We claim that $g$ has finite order, and that there exists a unique automorphism $\rho(g)\in \Aff(K)$
such that \[P(g\cdot z)=\rho(g)\cdot P(z)~.\]
This claim implies that the group  $G:=\langle g\rangle\subset\Aff(K)$ generated by $g$ is finite, 
and that the map  $\rho(g^j) := \rho(g)^j$, $j\ge 0$ is a well-defined group homomorphism $\rho:G\to \Aff(K)$.
One concludes  $g \in \Sigma(P)$ by Proposition~\ref{prop:Sigma}.

\smallskip

Let us prove this claim, and pick $g\in \Aff(K)$ and a sequence of distinct points $z_n \in \preper(P,\bar{K})$ such that $g \cdot z_n$ is preperiodic.
We may suppose that $z_n$ and $z_m$ are not $\Gal(\bar{K}/K)$-conjugated for any $n\neq m$.

The polynomial $P$ induces an algebraic family of polynomials $(t,z) \in X\times \mathbb{A}^1 \mapsto (t,P_t(z))\in X\times\mathbb{A}^1$. 
According to Proposition \ref{prop:sameZariskiopen},  and possibly replacing $X$ by a Zariski open dense subset  we have
$\Sigma(P_t)\simeq \Sigma(P)$ for all $t \in X(\bar{\Q})$.
Restricting again $X$ if necessary, the affine automorphism $g$ can also be identified with a family of affine maps
\[(t,z)\in X\times \mathbb{A}^1 \mapsto (t,g_t\cdot z)\in X\times\mathbb{A}^1\]
with $g_t\cdot z=\alpha(t)z+\beta(t)$, $\alpha,\beta\in\bar{\Q}[X]$ and $\alpha$ is not identically zero.

We fix any embedding of $\bar{\Q}$ into $\C$ so that $P$ now induces a holomorphic family of complex polynomials parameterized by $X(\C)$.
By Theorem~\ref{thm:full hol motion}, we may find an open polydisk $\Delta \subset X(\C)$
and a holomorphic motion $h \colon \Delta \times \C \to \Delta \times \C$ conjugating the dynamics, i.e. 
such that $ P_t ( h_t(z)) = h_t(P_{t_0}(z))$, and $h_{t_0} = \id$ (where $t_0 \in \Delta$ is any marked point).

Fix some integer $n$. The point $z_n$ is defined over some finite extension of $K = \bar{\Q} (X)$, hence the field $\bar{\Q}(X)[z_n]$ is isomorphic to $\bar{\Q} (X) [Z] / (R_n)$ where $R_n \in \bar{\Q} (X)[Z]$ is the minimal polynomial of $z_n$. In particular $z_n$ can be viewed as a rational function on the graph $Y_n$ of $R_n$ in $X \times \p^1$. 
Observe that the natural projection map $\pi_n \colon Y_n \dashrightarrow X$ defines a generically finite rational map so that 
one can find Zariski dense open subsets $V_n \subset X$, and $W_n \subset Y_n$ 
over which $\pi_n \colon W_n \to V_n$ defines a finite (proper) unramified cover.  With these notations, $\pi_n(z_n(\tau))$ is a preperiodic point for $P_{\pi_n(\tau)}$ for all $\tau \in W_n$.
Since $V_n$ is Zariski dense, one can find a simply connected and connected open dense set  $\Delta_n \subset \Delta \cap V_n$ (see \cite{berlanga}) and finitely many holomorphic  functions $z^{(i)}_n \colon \Delta_n \to \C$ such that \[F_n(t) := \{ z_n (\tau), \, \tau \in \pi_n^{-1}(t)\} = \{z^{(i)}_n(t)\}\subset \preper(P_t,\C)\] for all $t \in \Delta_n$. 

Changing the base point if necessary, one can assume that $t_0 \in \Delta_n$.  Since the holomorphic motion on $\Delta$ conjugates the dynamics of $P_{t_0}$ with $P_t$
and $z^{(i)}_n(t)$ is preperiodic, it follows that the holomorphic function $ t\in \Delta \mapsto h_t (z^{(i)}_n(t_0))$ equals $z^{(i)}_n(t)$ on $\Delta_n$. In other words, 
$z^{(i)}_n(t)$ can be continued analytically over $\Delta$ for all $i$. 

\begin{lemma}
For any $n \neq m$, for any $t \in \Delta$ and for any $i,j$ we have 
 $z^{(i)}_n(t) \neq z^{(j)}_m(t)$.
\end{lemma}
\begin{proof}
Pick any parameter $t_* \in \Delta$ such that $z^{(i)}_n (t_*)= z^{(j)}_m(t_*)$ for some $i,j$.
Since there is a holomorphic motion of the complex plane conjugating the dynamics on $\Delta$, and the set of points whose orbit has a fixed cardinality is finite, 
it follows that $z^{(i)}_n = z^{(j)}_m$ in a neighborhood of $t_*$.

The Zariski closure of the set of points $\{ (t, z^{(i)}_n(t)), t\in \Delta \}$ in the complex algebraic variety $X \times \A^1$
is equal to  $\{R_n =0\}$ since the latter is irreducible. It follows that 
$\{R_n =0\} = \{R_m =0\}$,
hence 
the two points $z_n$ and $z_m$ are Galois conjugated in $\bar{K}$. This contradicts our standing assumption. 
\end{proof}

Fix any $t \in \Delta$.  The preceding lemma shows that the set $\{ z_n^{(i)}(t)\}$ is infinite. 
Since $g_t \cdot z_n^{(i)}(t)$ is preperiodic by assumption, we conclude from Proposition~\ref{prop:algebraiccharacterization} that $g_t$ belongs to $\Sigma(P_t)$. 
By Proposition~\ref{prop:Sigma}, $g_t$ has finite order, and there exists an affine map $\rho_t(g_t) \cdot z=a(t)z+b(t)$
such that 
\[P_t(g_t\cdot z)=\rho_t(g_t)\cdot P_t(z).\]
If we write  $P_t(z)=\sum_{i=0}^da_i(t)z^i$, then the preceding equation yields
\begin{equation}\label{eq:unique!}
a(t)=\alpha(t)^d \ \text{ and } \ b(t)=P_t(\beta(t))-\alpha(t)^da_0(t)~.
\end{equation}
Now observe that $\alpha, \beta, a_i$ all define regular functions on $X$. We may thus set
$a:=\alpha^d$, and $b:=P(\beta)-\alpha^da_0$ which both belong to $\bar{\Q}[X]\subset K$, and
we set $\rho(g) = a z + b \in \Aff(K)$.

Since for all $t\in \Delta$, $g_t$ has finite order, and $\Delta$ is a complex open subset of $X$,  it follows that $g$ has finite order in $\Aff(K)$.
Similarly $P_t(g_t\cdot z)- \rho_t(g_t)\cdot P_t(z)$ vanishes on the complex open subset of points $(t,z) \in \Delta \times \C\subset X \times \A^1$ hence
$P(g\cdot z)= \rho(g)\cdot P(z)$. The uniqueness of $\rho(g)$ follows from~\eqref{eq:unique!}, and the proof of the Claim is complete.
\end{proof}


\section{Primitive families of polynomials}\label{sec:primitivity}

In this section, we fix any field $K$ of characteristic $0$ (not necessarily algebraically closed).

\begin{definition}
A polynomial $P$ of degree $d\ge 2$ defined over a field $K$ is said to be \emph{primitive}\index{polynomial!primitive} when the following holds. 
 For any polynomial $Q$ defined over an algebraic extension of $K$, and for any  $\sigma \in \Sigma(P)$ such that 
 $P = \sigma Q^n$ for some integer $n\ge1$, we have $n=1$.
\end{definition}
A polynomial is imprimitive when it is not primitive.

\begin{remark}
A polynomial $P$ whose degree is not a power of an integer is primitive. 
\end{remark}

We also introduce the notion of weak primitivity\index{polynomial!weakly primitive}. A polynomial $P$ is weakly primitive
iff $P = Q^n$ implies $n=1$. Any primitive polynomial is weakly primitive. However the two notions do not coincide in general as the next example shows.

\begin{example}
Pick any $a \neq0$, and set $Q(z) = z (z^2+a)$. Observe that $\Sigma(Q) = \{ \pm \id \}$.
The polynomial $P_1 = Q^2$ is imprimitive, but $P_2 = - Q^2$ is weakly primitive. 
Indeed if $P_2 = (Q_1)^2$ then $J(Q_1) = J(Q)$ hence $Q_1 = \pm Q$ by~\cite{Baker-Eremenko} which is absurd.
\end{example}

\begin{example}
Observe that any polynomial of the form $P(z) = (z^2+c)^2+ c  = z^4 + 2cz^2+ c^2+c$ for some $c\in K$ is not weakly primitive. Any other centered quartic polynomial
is weakly primitive.
Here is the complete list of  monic and centered quartic polynomials which are not primitive: 
\begin{itemize}
\item 
$P(z) = z^2 (z^2+c)$ with $c=0$ or $c = -2 \zeta$ and $\zeta^3 = -1$;
\item
$P(z) = z^4+ az^2 + c$ with $4 c = a^2 - 2a \zeta$ and $\zeta^3 = -1$.
\end{itemize}
\end{example}

\begin{definition}
An algebraic family $P$ parameterized by an irreducible algebraic variety $V$ defined over  a field $K$
is primitive when the induced polynomial $P \in K(V) [z]$ is primitive over the field $K(V)$.
A family parameterized by an arbitrary algebraic variety is primitive when its restriction to any irreducible components is primitive.
\end{definition}

\begin{proposition}\label{prop:prim}
Let $P$ be any algebraic family of polynomials parameterized by an irreducible algebraic variety $Z$ defined over a field $K$.
\begin{enumerate}
\item
Either the family is primitive, and the set of parameters $t\in Z$ such that $P_t$ is primitive forms a Zariski open subset of $V$. 
\item 
Or there exist an integer $n\ge 2$,  a finite proper and surjective map $\pi\colon W \to Z$,
a primitive family of polynomials $Q$ parameterized by $W$
and $\sigma\in\Sigma(P)$ such that $P_{\pi(t)} = \sigma_{\pi(t)}Q^n_t$ for all $t\in W$.
\end{enumerate}
\end{proposition}

This result relies on the properness of the composition map. 

\begin{proposition}\label{prop:properness composition}
The composition map $(P,Q) \mapsto P \circ Q$ induces a regular map $\Phi_{d,l} \colon \poly^d_{\rm mc} \times \poly^l_{\rm mc} \to \poly^{dl}_{\rm mc}$
which is proper when $d, l \ge2$.
\end{proposition}

\begin{proof}
We first need to check that the composition of two monic and centered polynomials remains monic and centered. This easy fact follows from the computation
$P(z) = z^d + O(z^{d-2})$, $Q(z) = z^l + O(z^{l-2})$, $P \circ Q(z) = (z^d + O(z^{d-2}))^l + O(z^{l(d-2)}) = z^{dl} + O(z^{dl-2})$. 
To check the properness of the composition map, we first  suppose $K = \C$. 
Write
\[ P(z) = z^d + a_2 z^{d-2} + \cdots + a_d,\,  Q(z) = z^l + b_2 z^{l-2} + \cdots + b_d\]
so that
\begin{align*}
P \circ Q (z) &:= z^{dl} + c_2 z^{dl-2} + \cdots + c_{dl} \\
&= (z^l + b_2 z^{l-2} + \cdots + b_d)^d + a_2 (z^l + b_2 z^{l-2} + \cdots + b_d) ^{d-2} + \cdots + a_d~. 
\end{align*}
Choose any constant $C>0$. 
Suppose by contradiction that we have a sequence of polynomials with coefficients $a_i^{(n)}, b_j^{(n)}$
such that $|c^{(n)}_l| \le C$ for all $n$, but $\max \{ |a_i^{(n)}|, |b_j^{(n)}|\} \to \infty$. 
If $\max \{ |b_j^{(n)}|\} \to \infty$, choose $i_0$ minimal such that $|b_{j_0}^{(n)}| \to \infty$. 
Since $l(d-1) + j_0 \ge l (d-2)$,
we have
\[c^{(n)}_{l(d-1) + j_0} = d b_{j_0}^{(n)} + \text{ Polynomial in } b_{2}^{(n)}, \ldots, b_{j_0-1}^{(n)} \]
which gives a contradiction. 
When $\sup_n \max \{ |b_j^{(n)}|\} < \infty$, and $\max \{ |a_i^{(n)}|\} \to \infty$, choose $i_0$ minimal with 
$|a_{i_0}^{(n)}| \to \infty$. 
This time 
\[ 
c^{(n)}_{l (d-i_0)} = a_{i_0}^{(n)} + \text{ Polynomial in } b_{2}^{(n)}, \ldots, b_{d}^{(n)}, a_2^{(n)},  \ldots, a_{i_0-1}^{(n)}\]
which again gives a contradiction. 

Since the properness of a map is preserved by faithfully flat descent, see~\cite[Lemma 35.20.14]{stacks-project},
it follows that the composition map is 
proper when $K = \Q$. It follows that it remains proper over any field of characteristic zero by base change, see~\cite[Lemma 29.39.5]{stacks-project}.
\end{proof}
\begin{corollary}\label{cor:properness compo}
For any finite collection of integers $d_1, d_2, \ldots, d_n \ge 2$, the composition map
\[ 
\Phi_{d_1, \ldots, d_n} \colon \prod_i \poly^{d_i}_{\rm mc} \to \poly^{d_1 \cdots d_n}_{\rm mc}, \, 
\]
sending $(P_1, \ldots, P_n)$ to 
$\Phi_{d_1, \ldots, d_n} (P_1, \ldots, P_n) := P_1 \circ \cdots \circ P_n$
is finite (hence proper).
\end{corollary}
\begin{proof}
Since $\poly^d_{\rm mc}$ is affine, it is sufficient to prove that the map $\Phi_{d_1, \ldots, d_n}$ is  finite. 
We proceed by induction on $n$. The case $n=2$ follows from the previous proposition. 
The induction step can be proved using the observation 
\[ \Phi_{d_1, \ldots, d_n}(P_1, \ldots, P_n) = \Phi_{d_1 \cdots d_{n-1}, d_n} \left( \Phi_{d_1, \ldots, d_{n-1}}(P_1, \ldots, P_{n-1}) , P_n\right)~.\]
 and the fact that a composition of two proper maps remains proper. 
\end{proof}

\begin{remark}
It is \emph{not} the case that the composition map 
$(P,Q) \mapsto P \circ Q$ is proper on $\poly^d \times \poly^l$. Indeed if $\alpha_t$  is a one-parameter subgroup of $\G_m$, e.g. $\alpha_t(z) = tz$, 
then $(P \circ \alpha_t^{-1}, \alpha_t \circ Q)$ is sent to the polynomial $P \circ Q$. However if $P$ is monic and centered and $P'(0) \neq 0$, then 
over the complex  $P \circ \alpha_t^{-1}$ diverges in the moduli space when $t\to 0$.
\end{remark}

\begin{remark}
Corollary~\ref{cor:properness compo} implies the iteration map $\Phi^n\colon \mpoly^d \to \mpoly^{d^n}$ defined by $\Phi^n(P):= P^n$ to be finite.  
This iteration map is in fact proper (hence finite) on the moduli space of complex rational maps of a fixed degree by~\cite[Corollary 0.3]{Demarco-Duke}. 
A short argument in the moduli space of complex polynomials bypassing the computations above goes as follows. Recall that 
the critical Green function   $G(P) = \max\{ g_P(c), \, P'(c) =0\}$ defines a proper function on $\poly^d$, see Proposition~\ref{prop:growth Green}. 
Suppose $P^n$ lies in a compact set $A$ of $\poly^{d^n}$. Then we have
$G(P) = G(P^n) \le \sup_A G =: C' <\infty$ so that $P$ belongs to $G^{-1} ([0, C'])$
which is compact. This proves $\Phi^n \colon \poly^d \to \poly^{d^n}$ is proper. 
\end{remark}

\begin{proof}[Proof of Proposition~\ref{prop:prim}]
For any $l$, we introduce the space $\widetilde{\poly}^l_{\rm mc}$ of all centered polynomials of degree $l$ 
whose dominant term is a root of unity of order $\le d$. It is a disjoint union of copies of $\poly^l_{\rm mc}$, and the $n$-th composition map
$\Phi_n \colon \widetilde{\poly}^l_{\rm mc} \to \widetilde{\poly}^{l^n}_{\rm mc}$ is proper by Corollary~\ref{cor:properness compo}. 

Pick any irreducible subvariety $Z$ of $\widetilde{\poly}^d_{\rm mc}$ for some $d\ge1$. We get a polynomial $P_Z$ defined over the field $K(Z)$. 
Given any root of unity $\sigma \in \Sigma(Z)$, the composition $\sigma P_Z$
determines a family of monic and centered polynomials hence an irreducible subvariety $\sigma(Z)$ in $\widetilde{\poly}^d_{\rm mc}$.

For any (possibly reducible) subvariety $V$ of  $\widetilde{\poly}^d_{\rm mc}$, 
we define 
\[S(V) := \bigcup_{\sigma \in \Sigma(W)}\sigma (W)\] where $W$ ranges over all irreducible subvarieties of $V$.
\begin{lemma}
The set $S(V)$ is Zariski-closed.
\end{lemma}
\begin{proof}
Observe that the symmetry group of any polynomial $P\in \widetilde{\poly}^d_{\rm mc}$ is a subgroup of $\U_{d!}$ except if $P= \zeta M_d$ in which 
case it is equal to $\U_\infty$. For each subgroup $G$ of $\U_{d!}$, the space $W_G \subset \widetilde{\poly}^d_{\rm mc}$
of polynomials such that $G \subset \sigma(P)$ is a closed subvariety, so that 
\[
S(V)
=
\bigcup_{G \subset  \U_{d!}} \bigcup_{\sigma \in G} \sigma (W_G \cap Z).
\]
This implies the claim.
\end{proof}

Observe that if a monic and centered polynomial $P$ of degree $d\ge2$ equals $\sigma Q^n$ for some $\sigma \in \Sigma(P)$
then $Q (T)= \zeta T^l + O(T^{l-2})$ where $l^n  = d$, and $\zeta$ is a root of unity of order $\le d$.
It follows that the set of polynomials in $\widetilde{\poly}^d_{\rm mc}$ which are not primitive is equal to
\[\bigcup_{l^n =d} S\left(\Phi^n(\widetilde{\poly}_{\rm mc}^l)\right)~.\]
It is thus Zariski closed since the image of a closed set by a proper morphism remains closed. 
In particular the subset $\mathrm{Prim}\subset \poly^d_{\rm mc}$ of primitive polynomials is Zariski open and dense in $\poly_{\rm mc}^d$.

\medskip

Pick now any family of polynomials of degree $d$ parameterized by an irreducible algebraic variety $V$. 
By base change, we may suppose that $P_t$ is monic and centered for all $t\in V$, so that we have 
an induced map $\pi\colon V \to \poly^d_{\rm mc}$.

When $\pi(V)$ intersects $\mathrm{Prim}$, then the set of parameters $t \in V$ such that $P_t$ is primitive is equal to $\pi^{-1}(\mathrm{Prim})$ which is Zariski dense since $V$ is irreducible. We claim that the family is primitive which implies (1). We argue by contradiction,  and pick an algebraic extension $L/K(V)$ and a family of centered polynomials $Q \in L[z]$ whose dominant term is a root of unity such that $P = \sigma Q^n$ for some root of unity $\sigma \in \Sigma(P)$ and some $n\ge2$. Since $Q$ has finitely many coefficients, we can assume that $L/K(V)$ is finite, and find an irreducible algebraic variety $W$ with a generically finite rational map $\pi \colon W \dashrightarrow V$ such that $L = K(W)$ and the field extension is induced by $\pi$.
Reducing $V$ and $W$ to suitable Zariski open subset, we may assume that $\pi$ is regular and surjective. We get a contradiction since for any $t \in V$ there exists $\tau\in W$
such that  $P_t = \sigma Q^n_\tau$. 

\smallskip

When $\pi(V)$ is disjoint from $\mathrm{Prim}$, one can find a maximal integer $n \ge2$ such that there exists $l\ge 2$
with $\pi(V) \subset \sigma \Phi^n(\widetilde{\poly}_{\rm mc}^l)$ for some root of unity $\sigma\in \Sigma(P_{\pi(V)})$, and we can form the 
fiber product:
\begin{displaymath}
    \xymatrix{ W \ar@{.>}[d]_{\Psi^n} \ar@{.>}[r]^(.4){\varpi} & \widetilde{\poly}^l_{\rm mc} \ar[d]^{\sigma \Phi^n} \\
               V \ar[r]^(.4){\pi} & \poly^d_{\rm mc}}
\end{displaymath}
Concretely when $K$ is algebraically closed the set of $K$-points of $W$ is obtained as the set of pairs $\tau = (t, Q) \in V(K) \times \widetilde{\poly}^l_{\rm mc}(K)$ such that $P_t = \sigma Q^n$.

In any case, since $\Phi^n$ is finite, the map $\Psi \colon W \to V$ is also finite and we get a family $Q_\tau$ of monic and centered polynomials of degree $l$
parameterized by $W$ such that $Q^n_\tau = \sigma P_{\Psi(\tau)}$. Observe that $\varpi(W)$ cannot be included in 
\[\bigcup_{\substack{j^m= l\\ m\ge 2}} S\left( \Phi^m(\poly_{\rm mc}^j) \right) \subset \poly^l_{\rm mc}\]
since we chose $n$  to be maximal. It follows from our previous arguments that there exists at least one irreducible component $W'$ of $W$ for which $Q_\tau$ is primitive for a Zariski dense open subset of $\tau \in W'$, and this family is primitive which proves (2).
\end{proof}


\section{Ritt's theory of composite polynomials}\label{sec:Ritt-theory}

In this section we review some aspects of Ritt's theory of decomposition of polynomials extended by Medvedev and Scanlon in~\cite{medvedev-scanlon} and further developed by Ghioca, Nguyen and his co-authors~\cite{MR3632102,GNY}, and Pakovich in a series of papers~\cite{pako:preimages,pako:rational,pakovich}. A modern account on the original approach of Ritt is described by M\"uller and Zieve in~\cite{zieve-muller}. 
Ritt's theorems are proved over a field of arbitrary characteristic by Zannier in~\cite{MR1244972} (see also~\cite{MR1770638}).

As in the previous section the base field $K$ is any field of characteristic zero. 
\subsection{Decomposability}

We start with the following basic definition
\begin{definition}
A polynomial $P$ of degree $d\ge 2$ is said to be \emph{decomposable} if it may be written $P = Q \circ R$ with $\deg(Q), \deg(R) \ge 2$, and
\emph{indecomposable}\index{polynomial!indecomposable} otherwise.
\end{definition}
\begin{remark}
If the degree of $P$ is prime, then $P$ is indecomposable. If $P$ is indecomposable, then it is primitive in the sense of the previous section. Observe that 
an integrable map (i.e. $P=M_d$ or $\pm T_d$) is indecomposable iff $d$ is prime.
\end{remark}

It is easy to see that any polynomial admits a complete decomposition, i.e. can be written $P = P_1 \circ \cdots \circ P_s$ with $P_1, \ldots, P_s$ indecomposable. 
Complete decompositions are not unique, but Ritt  
described how to pass from one decomposition to another, see~\cite{Ritt}.

Let $P$ and $Q$ be two indecomposable polynomials. A Ritt move\index{Ritt move} for $(P,Q)$ is a pair of two indecomposable polynomials $(\bar{P}, \bar{Q})$ such that
$P \circ Q = \bar{P} \circ \bar{Q}$. There is a short list of possible Ritt moves: 
\begin{itemize}
\item[(M1)]
$P$ and $Q$ arbitrary and $\bar{P} = P \circ \sigma^{-1}$, $\bar{Q} =  \sigma \circ Q$ for some affine map $\sigma$;
\item[(M2)]
$P=\nu \circ z^s R^n(z) \circ \sigma_1^{-1}$, $Q=\sigma_1 \circ z^n \circ \mu $, and
$\bar{P} = \nu \circ z^n\circ \sigma_2^{-1}$, $\bar{Q} = \sigma_2 \circ z^sR(z^n) \circ \mu$, 
where $\nu, \sigma_1, \sigma_2$ and $\mu$ are affine, $R$ is a polynomial,  $n\ge 1$ and $s\ge 0$ are coprime;
\item[(M3)] 
$
P= \nu \circ \pm T_m \circ \sigma_1^{-1}$, $Q=\sigma_1 \circ \pm T_n \circ \mu$, and
$\bar{P} = \nu \circ \pm T_n \circ \sigma_2^{-1}$, $\bar{Q} =\sigma_2 \circ \pm T_m \circ \mu$
where $n$ and $m$ are coprime, and  $\nu, \sigma_1, \sigma_2$ and $\mu$ are affine. 
\end{itemize}

\begin{theorem}\label{thm:Ritt}
Any polynomial $P \in K[z]$ of degree $\ge2$  admits a  complete decomposition 
$P = P_1 \circ \cdots \circ P_s$ with $P_1, \ldots, P_s$ indecomposable. 

Any other complete decomposition $P = Q_1 \circ \cdots \circ Q_s$ has the same cardinality
and there exists a sequence of complete decompositions
$P = P^{(i)}_1 \circ \cdots \circ P^{(i)}_s$ such that $P^{(0)}_j = P_j$, $P^{(n)}_j = Q_j$
and the decomposition at step $(i+1)$ is obtained by applying a Ritt move to a pair
of consecutive polynomials  $P^{(i)}_{j_i}, P^{(i)}_{j_i+1}$ at step $i$. 
\end{theorem}

Put it broadly, any two complete decompositions are connected by a sequence of Ritt moves.
We shall call the number of factors in any complete decomposition of $P$ its \emph{complexity}.

\begin{theorem}\label{thm:indecomp-family}
The complexity function is lower-semicontinuous for the Zariski topology in any algebraic family of polynomials. 
Moreover the set of conjugacy classes of indecomposable polynomials is an open and dense Zariski subset in $\mpoly^d$. 
\end{theorem}

\begin{proof}
It is only necessary to prove that the complexity function is lower semicontinuous on the space of monic and centered polynomials $\poly^d_{\rm mc}$.
\begin{lemma}\label{lem:decompo mc}
Any monic and centered polynomial $P$ admits a complete decomposition $P = P_1 \circ \ldots \circ P_s$ where $P_1, \ldots, P_s$ are again monic and centered.
\end{lemma}
It follows that the set of polynomials $P \in \poly^d_{\rm mc}$ whose complexity is larger than a fixed integer $k$ is the union of the images under the composition map
\[\Phi_{d_1, \ldots, d_k} \left(\poly^{d_1}_{\rm mc} \times \ldots \times \poly^{d_k}_{\rm mc}\right) \subset \poly^{d}_{\rm mc}\]
over all integers $d_1, \ldots, d_k \ge 2$, such that $d_1 \cdots d_k = d$. By Corollary~\ref{cor:properness compo} these images are Zariski closed in $\poly^{d}_{\rm mc}$. 
This shows the lower semicontinuity of the complexity function. 

Observe that $\poly^d_{\rm mc}$ is an affine variety of dimension $d-1$ whereas for all $k$ as above the dimension of 
$\Phi_{d_1, \ldots, d_k} \left(\poly^{d_1}_{\rm mc} \times \ldots \times \poly^{d_k}_{\rm mc}\right)$ is at most
$\sum (d_i -1) \le k d/2^{k-1} - k \le d -k < d-1$. It follows that the set of monic and centered decomposable polynomials
forms a strict algebraic subvariety of $\poly^d_{\rm mc}$. 
This ends the proof.
\end{proof}

\begin{proof}[Proof of Lemma~\ref{lem:decompo mc}]
Let $P$ be any monic an centered polynomial and choose an arbitrary complete decomposition $P = P_1 \circ \ldots \circ P_s$. 
Write \[P_i(z) = a_i z^{d_i} + b_i z^{d_i-1} + \lot\] 
with $a_i \neq 0$. 
We first choose inductively dilatations $\nu_s(z) = a_s^{-1} z$, $\nu_{s-1}(z) =(a_{s-1} a_s^{d_{s-1}})^{-1}z$, etc,
such that  $\bar{P}_s = \nu_s \circ P_s$, and $\bar{P}_i := \nu_i \circ P_i \circ \nu_{i+1}^{-1}$ becomes monic for all $2 \le i$. 
In this way, we obtain a complete decomposition $P = \bar{P}_1 \circ \cdots \circ \bar{P}_s$ for which all  polynomials $\bar{P}_2, \ldots, \bar{P}_s$ are monic. 
A direct computation shows that the leading term of $\bar{P}_1$ should also be equal to $1$ since
$P$ is monic.

Replacing $P_i$ by $\bar{P}_i$ we may thus assume that $a_i=1$ for all $i$ in the expansion above. 
Next we choose inductively translations $\tau_1(z) = z + \gamma_1, \tau_2 = z+ \gamma_2$, etc, so that
$\bar{P}_1 = P_1\circ \tau_1$, $\bar{P}_i = \tau_{i-1}^{-1}\circ P_i \circ \tau_i$ are centered  for all $i \le s-1$.
If we write $\bar{P}_s = \tau_{s-1}^{-1} \circ P_s = z^{d_s} + \alpha z^{d_s-1} + \lot$, then we obtain
\[ 
P(z) = \bar{P}_1 \circ \cdots \circ \bar{P}_s (z) = z^{d_1 \cdots d_s} + (d_1 \cdots d_{s-1}) \alpha z^{d_1 \cdots d_s -1} + \lot
\]
which implies $\alpha =0$ since $P$ is centered. 

This concludes the proof of the lemma.
\end{proof}

\subsection{Intertwined polynomials}\label{sec:intertwining}

We introduce the following terminology. 
\begin{definition}\label{def:intertwining}
Let $P$ and $Q$ be two polynomials of the same degree.
\begin{enumerate}
\item
We say that $P$ and $Q$ are semi-conjugate\index{polynomial!semi-conjugate}
if there exists a polynomial $\pi$ (possibly of degree $1$) such that 
$\pi \circ P = Q \circ \pi$. We write $P \ge Q$, or $P \ge_\pi Q$ if we want to emphasize the semi-conjugacy.
\item 
We say that $P$ and $Q$ are strictly intertwined\index{polynomial!strictly intertwined} iff there exists a polynomial $R$ such that
$R \ge P$ and $R \ge Q$.
\item
We say that $P$ and $Q$ are  intertwined\index{polynomial!intertwined} iff there exists a polynomial $R$ and $n\ge1$ such that
$R \ge P^n$ and $R \ge Q^n$.
\end{enumerate}
\end{definition}

We have the following basic observations (see~\cite[Theorem~4.4]{pakovich} for 2. and 3.).
\begin{theorem}\label{thm:basic-intertwin}
\begin{enumerate}
\item
Semi-conjugacy implies strict intertwining which implies intertwining. 
\item
The polynomial $P$ is intertwined with $M_d$ iff it is conjugated to $M_d$.
\item
The polynomial $P$ is intertwined with $\pm T_d$ iff it is conjugated to $\pm T_d$.
\item
Two polynomials $P$ and $Q$ of the same degree are intertwined (resp. strictly intertwined) iff 
there exists an algebraic subvariety $Z \subset \A^2$ (resp. an irreducible subvariety $Z \subset \A^2$)
whose projections to both axis are onto, and which is fixed by the map $(z,w) \mapsto (P(z),Q(w))$.
\item
Intertwining defines an equivalence relation in the moduli space of polynomials
of a fixed degree. 
\end{enumerate}
\end{theorem}

As in~\cite{MR3632102,medvedev-scanlon}, we shall write $P \approx Q$ when $P$ and $Q$ are intertwined. 

\begin{remark}
There exists a semi-conjugacy $\pi(z) = z + \frac1z$ between $M_d$ and $T_d$ but this semi-conjugacy is
not given by a polynomial.
\end{remark}

\begin{proof}
The first item is obvious. 
\smallskip

Let $P_*$ be an integrable polynomial, and let $P$ be any polynomial satisfying  $P_* \ge P$ so that $\pi \circ P_* = P \circ \pi$ for some polynomial $\pi$. 
We embed the defining field of the coefficients of $P$ and $\pi$ in the field of complex numbers.
By the previous lemma, $\preper(P, \C) = \pi (\preper(P_*, \C))$ hence $J(P)\subset \C$ is smooth near
any point  outside finitely many exceptions, and $P$ is integrable thanks to Theorem~\ref{th:zdunik}.

If $P_* = \pm T_d$, then $\pi^{-1}(K(P)) = [-2,+2]$ hence $P$ is also equal to $\pm T_d$. 
If $P_* = M_d$, then $P$ cannot be a Chebyshev polynomial since $K(M_d)$ has non-empty interior
whereas $K(P) = J(P)$, hence $P = M_d$.

Suppose now that $P \ge P_*$ with $P_*$ integrable. The same argument applies and show that $P$ is monomial (resp. Chebyshev) when $P_*$ is. 

\smallskip
For the fourth item, one direction is clear. Indeed if $P \le_\pi R$ and $Q \le_\varpi R$ then the curve $Z = \{ (\pi(\tau) , \varpi(\tau)), \tau \in \A^1\}$ 
is fixed by $(P,Q)$.

Suppose that $Z$ is irreducible. Let $\bar{Z}$ be the closure in $\p^1 \times \p^1$ of $Z$, and
 $\mathsf{n} \colon \hat{Z} \to \bar{Z}$ its normalization. The restriction of map $f$ to $Z$ lifts to $\hat{Z}$ and defines a non-invertible map 
 $g \colon \hat{Z} \to \hat{Z}$. Observe that $\mathcal{E} = \mathsf{n}^{-1} ( \bar{Z}\setminus Z) \subset \hat{Z}$ is a finite totally invariant set of cardinality $1$ or $2$
 hence $\hat{Z}$ is isomorphic to $\p^1$.   When $\mathcal{E}$ is reduced to one point, the restriction of $g$ to $\hat{Z}\setminus \mathcal{E} = \A^1$ is a polynomial say $R$, and  the composition of $\mathsf{n}$ with the first (resp. the second) projection semi-conjugates $R$ to $P$ (resp. to $Q$) so that $P$ and $Q$ are strictly intertwined.   When $\mathcal{E}$ has $2$ points, $g$ is a monomial map and (2) and (3) show that both $P$ and $Q$ are integrable and strictly intertwined. 

\smallskip

The fifth item can be treated as follows. Suppose $S \ge_{\pi_1} P$, $S \ge_{\pi_2}  Q$, and
$T \ge_{\varpi_1} Q$, $T \ge_{\varpi_2}  R$. The curve $C:= \{ \pi_2(y) = \varpi_1(z)\}$ in $\A^2$
is fixed by $\phi= (S,T)$. The projections $\pi_1,\varpi_2 \colon C \to \A^1$ 
satisfy $\phi\ge_{\pi_1} P$ and $\phi\ge_{\varpi_2} R$, hence the result.
\end{proof}

\begin{proposition}\label{prop:intertwined example}
\begin{enumerate}
\item
The polynomials $P \circ g $ and $P$ are intertwined for any $g \in \Sigma(P)$.
\item
Two commuting polynomials are intertwined.
\item
Let $A$ and $B$ be two arbitrary polynomials. Then $P=A\circ B$ and $Q = B\circ A$ are strictly intertwined.
\end{enumerate}
\end{proposition}

\begin{lemma}
Suppose $\pi \circ P = Q \circ \pi$ for some non-constant $\pi$. Then $\preper(P,\bar{K}) = \pi^{-1} (\preper(Q,\bar{K}))$. 
\end{lemma}
 \begin{proof}
Indeed one has $\pi \circ P^n (z) = Q^n \circ \pi(z)$ for all $n$ so that $z$ has a finite $P$-orbit iff $\pi(z)$ has a finite $Q$-orbit. 
The result follows from the fact that $\pi$ is necessarily finite.\footnote{two such polynomials are called congruent, see page \pageref{congruent} below} 
\end{proof}

\begin{proof}[Proof of Proposition~\ref{prop:intertwined example}]
Let us prove 1. We may suppose that $P$ is not integrable. Pick $g \in \Sigma(P)$. By Proposition~\ref{prop:Sigma} for any $n\ge1$ there exists 
$g_n \in \Sigma(P)$ such that $(P \circ g)^n = g_n \circ P^n$. Since $\Sigma(P)$ is finite, it follows that we can find $n_1>n_0$ such that
$g_{n_1} = g_{n_0}$. 
Let $\Delta$ be the diagonal in $\A^2$, and define the map $f(z,w) = (P(z), P \circ \sigma (w))$. We have
\[
f^n(\Delta) = \{ (P^n(z), g_n \circ P^n(z)), z \in \A^1 \}
\]
so that $f^{n_1}(\Delta) = f^{n_0}(\Delta)$. In particular $\cup_{n\ge 0} f^n(\Delta)$ is an algebraic subvariety of $\A^2$ which is $f$-invariant, and Theorem~\ref{thm:basic-intertwin} proves (1). 

To prove (2), suppose that $Q$ is a polynomial commuting with $P$ (and of the same degree). As above, we may work over the field of complex numbers. 
The set of preperiodic points of $P$ and $Q$ are then equal which implies the Julia set of $P$ and $Q$ to coincide. 
By Theorem~\ref{thm:SS}, and using the fact that $P$ and $Q$ are supposed to have the same degree, 
we conclude to the existence of $\sigma \in \aut(J(P)) = \Sigma(P)$ such that $Q = \sigma \circ P$. 
It follows that $Q$ is conjugated to $P \circ \sigma$ which is intertwined with $P$ by (1). 

Finally $B\circ P = Q\circ B$ so that $P\ge Q$  which proves (3).
\end{proof}

\subsection{Uniform bounds and invariant subvarieties}

It is a striking fact that one may obtain uniform bounds in the context of Ritt's theory when degrees are fixed. 
A first example of such bounds was given in~\cite[Theorem~1.4]{zieve-muller}. In this book, we shall use the next two results. 

\begin{theorem}[\cite{MR3632102}]\label{thm:bdd-nguyen}
For any two polynomials $P,Q$ of the same degree $d\ge 2$
such that $P \approx Q$, there exists an integer $n \le 2 d^4$ such that 
$P^n$ and $Q^n$ are strictly intertwined.
\end{theorem}

\begin{theorem}[\cite{pakovich}]\label{thm:min-pako}
For any integer $d$, there exists a constant $c(d)$ such that the following holds.

For any polynomial $P$ of degree $d\ge 2$, there exist a polynomial $P_{\min}$ of degree $d$
and  $\pi_{\min}$ of degree $\le c(d)$, 
such that for any $Q \le P$ there exist polynomials $\pi,\varpi$ 
such that $P\ge_\varpi Q\ge_\pi P_{\min} $ and $\pi_{\min} = \pi \circ \varpi$.

Moreover the set of monic and centered polynomials $Q\le P$ is finite of cardinality $\le c_2(d)$ for some constant depending only on $d$.
\end{theorem}

This result implies the following characterization of invariant curves by product maps
which is due to~\cite{medvedev-scanlon}. We sketch the proof given in~\cite{pakovich} thereafter.

\begin{theorem}\label{thm:inv-curve}
Let $P$ and $Q$ be two non-integrable polynomials of the same degree $d\ge2$.

Let $C$ be any algebraic irreducible curve in $\A^2$ which is invariant by the map $(x,y)\mapsto (P(x),Q(y))$.
Then we can find two polynomials $u,v$ whose degrees are coprime such that 
$C=\{u(x)=v(y)\}$ and which satisfy
\begin{align*}
P \circ u = u \circ R
\\
Q \circ v = v \circ R
\end{align*}
for some polynomial $R$.
\end{theorem}

\begin{proof}[Sketch of proof]
The normalization $\tilde{C}$ of $C$ is a smooth affine curve over which $\phi(x,y) :=  (P(x),Q(y))$
induces a non-invertible finite surjective map. It follows that $\tilde{C}$ is either the affine line
and $\phi$ is a polynomial, or $\tilde{C}$ is the punctured affine line and $\phi$ is a monomial map. 
The composition of the normalization map and the first projection semi-conjugate $\phi$ to $P$ hence the latter
case cannot appear, see Theorem~\ref{thm:basic-intertwin}.

We now apply Theorem~\ref{thm:min-pako} and set $R := \phi_{\min}$. 
We get the existence of two polynomials $u$ and $v$ such that 
$P \ge_u R$ and $Q\ge_v R$ such that $C=\{u(x)=v(y)\}$. 
Using Ritt's theorem (Theorem~\ref{thm:reduction} below), one can argue that the degrees of $u$ and $v$ are coprime.
\end{proof}

\subsection{Intertwining classes}
Given any polynomial $P$ of degree $d\ge2$, we are interested in the description of the set of polynomials that are intertwined with $P$. 
Note that since any two conjugated polynomials are intertwined, it makes sense to consider the set of conjugacy classes
in $\poly_d$ that are intertwined with $P$.

More precisely, for any integer $D$ we define $\inter_D(P)$ to be the set of monic and centered polynomials of degree $D$
such that $Q^m \approx P^n$ for some $n,m\ge1$. Observe that $\inter_D(P)=\varnothing$ whenever $D^m \neq d^n$ for all $n,m\ge1$.
We also set $\inter(P) = \bigcup_{D \ge1} \inter_D(P)$.

\begin{theorem}
There exists a constant $C = C(d)$ such that 
for any polynomial $P$ of degree $d$, 
we have $\# \inter_D(P) \le C$ for all $D$.
\end{theorem}

\begin{remark}
The proof actually shows that if $P$ has coefficients in $\bar{\Q}$ then any $Q\approx P$ does. 
Using the critical height $h_\bif$ defined on p.\pageref{def:bif-height}, and Misiurewicz-Prytycky's formula~\eqref{eq:MPform}, it also implies that if $Q\approx P$ then $Q$ is PCF iff $P$ is. 
\end{remark}

\begin{proof}
We begin with two lemmas. Recall the definition of the Lyapunov exponent from~\S\ref{sec:green}.

\begin{lemma}\label{lem:final2}
If $P\approx Q$, and both polynomials are defined over some metrized field $K$, then $\lyap(P) = \lyap(Q)$.
\end{lemma}
\begin{proof}
Indeed suppose $ \pi \circ Q = P \circ \pi$ for some polynomial $\pi$. Since $\pi^*\mu_P=\Delta (g_P\circ \pi)=\deg(\pi)\cdot \Delta(g_Q)=\deg(\pi)\cdot \mu_Q$, we have $\pi_*\mu_Q=\mu_P$.
It follows that 
\[
|(P^n)' \circ \pi| = |\pi'\circ Q^n| \cdot | (Q^n)'| \cdot |\pi'|^{-1} 
\]
for all $n$,
hence
\begin{multline*}
\lyap(P) = 
\lim_n \frac1n \int \log |(P^n)' \circ \pi|\, d\mu_Q
= \\
\lim_n 
\frac1n \int \log |\pi'\circ Q^n| \, d\mu_Q
+ 
\lim_n 
\frac1n \int \log | (Q^n)'| \, d\mu_Q
+ 
\lim_n 
\frac1n \int \log  |\pi'|^{-1}  \, d\mu_Q
 \end{multline*}
 and the result follows. 
\end{proof}

\begin{lemma}\label{lem:final1}
For any integer $N\ge 2$, there exists a constant $C_1(d,N)$ (independent on $P$) such that 
the set of  monic and centered polynomials $Q$ such that $ \pi \circ Q = P \circ \pi$ for some polynomial $\pi$ 
with $\deg(\pi) \le N$ is 
finite of cardinality $\le C_1(d,N)$.
\end{lemma}
\begin{proof}
Write $P(T) = T^d + \sum_{j=0}^{d-2} p_{d-j} T^j$. We look for polynomials $Q(T) = T^d + \sum_{j=0}^{d-2} q_{d-j} T^j$ such that there exists
$\pi = \zeta T^N + \pi_2 T^{N-2} + \cdots + \pi_N$ with $\zeta^{d-1} =1$, and $\pi \circ Q (T) = P \circ \pi(T)$.
Identifying the terms in $T^{Nd-j}$ with $j= 2, 3, \ldots, N$ yields equations of the form
$\zeta^{d-1} d \pi_j = L_j(\pi_2, \cdots, \pi_{j-1},q_2, \cdots , q_d)$ with $L_j$ a polynomial with integral coefficients
so that $\pi_2, \ldots, \pi_N$ can be expressed as polynomials in the variables $q_2, \ldots , q_N$. 
It follows that the equation $\pi\circ Q (T) = P \circ \pi(T)$ is equivalent to the vanishing of the coefficients in $T^j$ from $j=0$ to $j = Nd-N-1$, hence of
$Nd-N$ polynomials in the variables $q_2, \ldots, q_N$ of degree depending only on $d$ and $N$.
We get that $\{Q, \, \pi \circ Q = P \circ \pi, \, \deg(\pi)\le N \}$ is  an algebraic
subvariety defined by the intersection of $C_1$ hypersurfaces of degree $\le C_2$ where $C_1, C_2$ are
constants depending only on $d$ and $N$. But Lemma~\ref{lem:final2} shows that over the field of complex numbers, this variety is bounded, since the Lyapunov function is proper on the space of monic and centered polynomials by~\eqref{eq:MPform} and Proposition~\ref{prop:growth Green}. It is hence finite, and the result follows from~\cite[Theorem~12.3]{MR1644323}.
\end{proof}

Now fix any monic and centered polynomial $P$ of degree $d$ and choose an embedding of the defining field of $P$ in $\C$. 
Pick any monic and centered polynomial $Q \approx P$. By Theorem~\ref{thm:bdd-nguyen}, there exists an integer $n \le 2d^4$ and a polynomial $R$
such that $R \ge P^n$ and  $R \ge Q^n$.

By Theorem~\ref{thm:min-pako}, there exists $S$ of degree $\le c(d^{2d^4})$ such that 
$S \le R$ and $S$ is universal for $R$, so that we may find polynomials  $\pi, \varpi$ with $\deg(\pi), \deg(\varpi) \le c(d^{2d^4})$
such that  $R \ge_\pi P^n$ and $R \ge_\varpi Q^n$. 

By Lemma~\ref{lem:final1}, there are at most $C_1(d^{2d^4},c(d^{2d^4}))$ possibilities for $R$, and for each $R$ at most $c_2(d^{2d^4})$ for $Q$ by Theorem~\ref{thm:min-pako}.
This ends the proof.
\end{proof}

\subsection{Intertwining classes of a generic polynomial}

For convenience, we say that $P$ is pseudo-integrable if it is of the form 
$P= T_d, M_d, T^sR(T^n)$ or $T^sR^n(T)$ for some $n\ge2$. 
Also write $P \sim Q$ if there exist two affine maps such that $P= \sigma \circ Q\circ \tau$.

We say that $P$ has an integrable (resp. pseudo-integrable) factor if it admits
a complete decomposition $P = P_1 \circ \cdots \circ P_s$ for which one of the factor satisfies
$P_i \sim Q$ where $Q$ is integrable (resp. pseudo-integrable).
\begin{proposition}
\begin{itemize}
\item
 If $P$ has an (pseudo)-integrable factor, then any complete decomposition
contains an (pseudo)-integrable polynomial (up to composition by affine maps). 
\item
A polynomial $P$ has a (pseudo)-integrable factor iff one of its iterate has a (pseudo)-integrable factor.
\item
A polynomial $P$ without any pseudo-integrable factor has a trivial group of dynamical symmetries.
\end{itemize}
\end{proposition}
\begin{proof}
Recall that any two complete decompositions are related by Ritt's moves. The first item follows from the observation
that any non trivial move (not of the form (M1)) involves an integrable polynomial. 

Observe that if $P = P_1 \circ \cdots \circ P_s$ is a complete decomposition, then 
$P^n = (P_1 \circ \cdots \circ P_s)^{\circ n}$ is also a complete decomposition. Indeed
if it were not, one of the polynomial $P_i$ would be decomposable. 
It follows that $P$ has no (pseudo)-integrable factor iff $P^n$ does. 

Take any complete decomposition $P = P_1 \circ \cdots \circ P_s$ and suppose $P$ has no pseudo-integrable factors. 
If $\Sigma(P)$ is non-trivial, then there exist affine maps $g, g'$ such that $(g' \circ P_1) \circ \cdots \circ P_s = P_1 \circ \cdots \circ (P_s\circ g)$.
By Ritt's theorem, these two decompositions are connected by a sequence of Ritt moves which are necessarily of the type (M1)
since $P$ has no pseudo-integrable factors. It follows that $g'' \circ P_s = P_s\circ g$ for some affine $g''$ hence $P_s$ has a non-trivial group
of symmetries, and $P_s$ is pseudo-integrable. This yields a contradiction.
\end{proof}

In degree $2$, any intertwining class is trivial. This was shown by Ghioca, Nguyen, and Ye~\cite[Theorem~1.4]{GNY}.

\begin{theorem}[\cite{GNY}]
Two quadratic polynomials $T^2+c$ and $T^2 +c'$ are intertwined iff $c=c'$.
\end{theorem}

In higher degree, we have the following result.

\begin{theorem}
Suppose that $P$ is an indecomposable polynomial with  no pseudo-integrable factors. 
Then any $Q \in \inter(P)$ is conjugate to an iterate of $P$.

In particular for any $d\ge3$, the set of polynomials $P$ such that 
$ \inter(P)$ is reduced to a single conjugacy class is a Zariski open dense subset of $\poly_d$.
\end{theorem}

\begin{proof}
We rely on the following lemma.
\begin{lemma}\label{lem:congruent}
Suppose that $P$ has no pseudo-integrable factors. If $Q \ge P$ or $Q \le P$, then there exists some integer $n$ such that 
$Q^n = U \circ V$ and $P^n = V \circ U$ for some polynomials $U$ and $V$.
\end{lemma}
Following the terminology of~\cite{pakovich}, we shall say that $P$ and $Q$ are congruent when the conclusion
of the lemma is satisfied. \label{congruent}

Suppose that $Q\in \inter(P)$. By definition, one can find a polynomial $R$ such that 
$R \ge Q^m$ and $R \ge P^n$ for some $n,m\ge1$. By the previous lemma, $R$ and $P^n$ are congruent (maybe after increasing $n$) so that 
$R = U \circ V$ and $P^n = V \circ U$.
Since $P$ is indecomposable and has no pseudo-integrable factor, any complete decomposition of $P^n$ is trivial. 
It follows that $V = P^k\circ \sigma$ and $U =\sigma^{-1} \circ P^{n-k}$ with $k\ge0$ and $\sigma$ affine, hence
$R$ is conjugated to $P^n$, and $P^n \ge Q^m$. The same argument implies that $Q^m$ and $P^n$ are conjugated. 

Since $P$ has no non-trivial symmetries and is indecomposable, it is also primitive, and 
it follows from Theorem~\ref{thm:SS} that $P$ and $Q$ are conjugated.

The last statement follows from  Theorem~\ref{thm:indecomp-family}.
\end{proof}

The proof of Lemma~\ref{lem:congruent} relies on two fundamental theorems in Ritt's theory that we now recall.
Fix four polynomials $A,B,C,D$ such that $D \circ B= A \circ C$, i.e. the following diagram is commutative 
\begin{displaymath}
    \xymatrix{ 
    \A^1 \ar[r]^{B} \ar[d]_{C} 
    & \A^1 \ar[d]^{D} \\
        \A^1 \ar[r]_{A} 
    & \A^1}
\end{displaymath}

\begin{theorem}[Reduction, \cite{MR3599}]\label{thm:reduction}
There exist two polynomials $U$ and $V$ such that 
$ A = U \circ \tilde{A}$, $ D = U \circ \tilde{D}$,
$ B = \tilde{B}\circ V$, and $ C = \tilde{C}\circ V$, 
\[ \deg(U) = \gcd \{\deg(A), \deg(D)\} \text{ and } \deg(V) = \gcd \{\deg(B), \deg(C)\}\]
and  $\tilde{D} \circ \tilde{B}= \tilde{A} \circ \tilde{C}$.
\end{theorem}

\begin{theorem}[Solution in the primitive case, \cite{Ritt}]\label{thm:solution}
Suppose that $\gcd \{\deg(A), \deg(D)\} = \gcd \{\deg(B), \deg(C)\} =1$. Then we are in one of the following cases:
\begin{enumerate}
\item
$B\sim  z^sR(z^n)$, $A \sim z^s R^n(z)$, $C \sim  D\sim  z^n$ with $\gcd\{ n ,s \}=1$;
\item 
$C\sim  z^sR(z^n)$, $D \sim  z^s R^n(z)$,  $A \sim  B\sim  z^n$  with $\gcd\{ n ,s \}=1$;
\item
$A \sim  B \sim T_n$, $C \sim  D \sim T_m$ with $\gcd\{ n ,m \}=1$.
\end{enumerate}
\end{theorem}

\begin{proof}[Proof of Lemma~\ref{lem:congruent}]
Suppose that $Q\ge_\pi P$.
Observe that if $\pi= \pi' \circ Q$ for some polynomial $\pi'$, then 
we may replace $\pi$ by $\pi'$. We may thus assume that $\pi$ is not a polynomial in $Q$.

Write
$\deg(\pi) = l \times b$ where $l = \prod_{p\wedge d =1} p^{v_p(\deg(\pi))}$, and
pick $n$ minimal such that $b$ divides $d^n$.
By the reduction Theorem~\ref{thm:reduction}, we can write
\begin{align*}
P^n = U \circ P_0, \pi = U \circ \pi_0 
\\
Q^n = Q_1 \circ V, \pi = \pi_1 \circ V 
\end{align*}
with 
$\deg(U) = \deg(V)= \gcd\{\deg(P)^n, \deg(\pi)\} = \gcd\{d^n, l\times b\} = b$, and 
$\pi_0 \circ Q_1 = P_0 \circ \pi_1$.
Note that $\deg(P_0) \ge2$ since otherwise we would have 
$\deg(Q_1) = \deg(P_0) = 1$, and $\pi$ would be a polynomial in $Q^n$.

Observe that 
$\gcd\{\deg(P_0), \deg(\pi_0)\} = \gcd\{\deg(Q_1), \deg(\pi_1)\}=1$.
Apply Theorem~\ref{thm:solution} to the quadruple $P_0,Q_1, \pi_0, \pi_1$. 
Since $P_0$ is a factor of $P^n$, and $P^n$ has no pseudo-integrable factor,
we fall into case 1: $\pi_0$ and $\pi_1$ have thus degree $1$, and
$P^n$ and $Q^n$ are congruent.

The same argument applies when $Q\le_\pi P$.
\end{proof}


\section{Stratification of the parameter space in low degree}\label{sec:stratif}
The five tables below summarize the stratification of the space of monic and centered polynomials
of degree $\le 6$ in terms of the size of its group of dynamical symmetries.
All computations are done over $\C$ (but the results are valid over any algebraically closed field of characteristic $0$).

\medskip

For any triple $(d,k,\mu)$ with $k\ge2$ and $\mu \le d-2$, we let $\Sigma(d,k,\mu)$ 
be the set of monic and centered polynomial
of degree $d$ which can be written under the form $z^\mu Q(z^k)$ with $Q(0)\neq0$ and $k$ maximal.
It is an open Zariski-dense subset of a linear subspace of $\C^{d-1}$ of dimension $\frac{d-\mu}k$.

Recall that 
\begin{align*}
\aut(P) & = \{ g \in \Sigma(P), gP = Pg\}; \\
\Sigma_0(P) &= \{ g \in \Sigma(P), P^n g = P^n \text{ for some } n\}.
\end{align*}
The numbers is the colums "complexity" corresponds to the number of polynomials appearing in some/any decomposition 
$P = P_1 \circ \cdots \circ P_s$ with $P_s$ indecomposable, see Theorem~\ref{thm:Ritt}. This number is always $1$ when the degree of $P$ is prime.

\medskip

\begin{center}

\begin{tabular}{ | l |c || c | c | c |}
 \hline
 \multicolumn{5}{|c|}{$P(z) = z^2 + c$} \\
 \hline
Range & Domain & $\aut(P)$ & $\Sigma(P)$& $\Sigma_0(P)$\\
 \hline
$c \neq 0$ & $\C^*$ & $\U_1$ & $\U_2$& $\U_2$ \\
 $z^2$& $\{0\}$ & $\U_1$  & $\U_\infty$ & $\U_{2^{\infty}}$\\
 \hline
\end{tabular}

\end{center}

\medskip

\begin{center}

\begin{tabular}{ | l |c || c | c | c |}
 \hline
 \multicolumn{5}{|c|}{$P(z) =  z^3 + a z + b$} \\
 \hline
Range & Domain & $\aut(P)$ & $\Sigma(P)$ & $\Sigma_0(P)$\\
 \hline
$ab\neq 0$ & $(\C^*)^2$ & $\U_1$ & $\U_1$& $\U_1$ \\
$\Sigma(3,3,0)$: $z^3+b$  & $\C^*$ & $\U_1$ & $\U_3$ &  $\U_3$ \\
$\Sigma(3,2,1)$: $z(z^2+a)$ & $\C^*$ & $\U_2$ & $\U_2$ & $\U_1$ \\
 $z^3$& $\{0\}$ & $\U_2$ & $\U_\infty$ & $\U_{3^{\infty}}$ \\
 \hline
\end{tabular}

\end{center}

\medskip

\begin{center}
\tiny{
\begin{tabular}{ | l |c || c | c | c |  c |c|}
 \hline
 \multicolumn{7}{|c|}{$P(z) =  z^4 + a z^2 + bz + c$} \\
 \hline
Range & Domain & $\aut(P)$ & $\Sigma(P)$& $\Sigma_0(P)$& Complexity& Primitivity\\
 \hline
 No symmetry & $U_4$  & $\U_1$ & $\U_1$& $\U_1$ & $1$ & Yes \\
$\Sigma(4,2,2)$: $z^2 (z^2 + a)$ & $\C^*$ & $\U_1$ & $\U_2$ & $\U_2$ & 2 & iff $a \neq -2\zeta$, $\zeta^3 =-1$\\
$\Sigma(4,3,1)$: $z(z^3+b)$ & $\C^*$ & $\U_3$ & $\U_3$ & $\U_1$ &1& Yes \\
$\Sigma(4,2,0)$: $z^4 + az^2 + c$ & $(\C^*)^2$ & $\U_1$ & $\U_2$ & $\U_2$ & 2 & iff $4c \neq a^2 - 2\zeta a$, $\zeta^3 =-1$\\
$\Sigma(4,4,0)$: $z^4 +c$  & $\C^*$ &  $\U_1$ &  $\U_4$ & $\U_4$ & 2 & Yes\\
 $z^4$& $\{0\}$ & $\U_3$ & $\U_\infty$ & $\U_{4^{\infty}}$ & 2 & No \\
 \hline
\end{tabular}
}
\end{center}
\[ U_4 = \C^3 \setminus \{b=0\} \cup\{a=c=0\}\]

\medskip

\begin{center}

\begin{tabular}{ | l |c || c | c | c |}
 \hline
 \multicolumn{5}{|c|}{$P(z) =  z^5 + a z^3 + bz^2 + cz + d$} \\
 \hline
Range & Domain & $\aut(P)$ & $\Sigma(P)$ & $\Sigma_0(P)$ \\
 \hline
 No symmetry & $U_5$  & $\U_1$ & $\U_1$ & $\U_1$  \\
$\Sigma(5,2,3)$: $z^3(z^2+a)$ & $\C^*$ &  $\U_2$ & $\U_2$ & $\U_1$ \\
$\Sigma(5,3,2)$: $z^2(z^3+b)$ & $\C^*$ & $\U_1$ & $\U_3$ & $\U_1$ \\
$\Sigma(5,2,1)$: $z (z^4+az^2+c)$ & $(\C^*)^2$ &  $\U_2$ & $\U_2$ & $\U_1$ \\
$\Sigma(5,4,1)$: $z(z^4+c)$ & $\C^*$ & $\U_4$ & $\U_4$ & $\U_1$ \\
$\Sigma(5,5,0)$: $z^5+d$ & $\C^*$ &$\U_1$ & $\U_5$ & $\U_5$ \\ 
 $z^5$& $\{0\}$ & $\U_4$ & $\U_\infty$ & $ \U_{5^{\infty}}$ \\
 \hline
\end{tabular}

\end{center}
\[U_5 = \C^4 \setminus \{ b=d=0\} \cup \{a=c=d=0 \} \cup \{a=b=c=0 \}\]
\medskip

\begin{center}
\small{
\begin{tabular}{ | l |c || c | c | c | c | }
 \hline
 \multicolumn{6}{|c|}{$P(z) =  z^6 + a z^4 + bz^3 + cz^2 + dz + e $} \\
 \hline
Range & Domain & $\aut(P)$ & $\Sigma(P)$& $\Sigma_0(P)$& Complexity\\
 \hline
 No symmetry & $U_6$ & $\U_1$ & $\U_1$& $\U_1$ &2 iff $4c=a^2$ and $ab=2d$ \\
$\Sigma(6,2,4)$: $z^4(z^2+a)$ & $\C^*$ & $\U_1$ & $\U_2$ & $\U_2$ &2\\
$\Sigma(6,3,3)$: $z^3(z^3+ b)$ & $(\C^*)^2$ & $\U_1$ & $\U_3$ & $\U_3$ & 2 \\
$\Sigma(6,4,2)$: $z^2 (z^4 + c)$ & $\C^*$ & $\U_1$ & $\U_4$ & $\U_4$ & 2 \\
$\Sigma(6,2,2)$: $z^2(z^4 +a z^2+c)$  & $(\C^*)^2$ &  $\U_1$ &  $\U_2$ & $\U_2$ & 2 \\
$\Sigma(6,5,1)$: $z(z^5+d)$  & $\C^*$ &  $\U_5$ &  $\U_5$ & $\U_1$ & 1 \\
$\Sigma(6,3,0)$: $z^6 +bz^3+e$  & $(\C^*)^2$ &  $\U_1$ &  $\U_3$ & $\U_3$ & 2 \\
$\Sigma(6,2,0)$: $z^6 +a z^4 +cz^2 +e$  & $U_{6,2,0}$ &  $\U_1$ &  $\U_2$ & $\U_2$ & 2 \\
$\Sigma(6,6,0)$: $z^6 +e$  & $\C^*$ &  $\U_1$ &  $\U_6$ & $\U_6$ & 2 \\
$z^6$& $\{0\}$ & $\U$ & $\U_\infty$ & $\U_{6^\infty}$ & 2  \\
 \hline
\end{tabular}
}
\end{center}
\begin{align*}
U_6  &= \C^5\setminus \{b=d=0\}\cup\{a=c=d=e=0\} \cup \{a=b=c=e=0\}\\
U_{6,2,0} &= \C^3\setminus \{a=0\}\cup\{a=c=d=e=0\} \cup \{a=b=c=e=0\}
\end{align*}



\section{Open problems}\label{sec:Ritt-open}

Many questions about the intertwining relation remain unclear. 
We have selected a few below. 
\begin{itemize}
\item[(I1)]
Prove that strict intertwining does not define an equivalence relation. In other words, 
there exist polynomials of the same degree $P$, $Q$, $R$ such that 
$\phi \ge P, \phi \ge R$ and $\varphi \ge R, \varphi \ge Q$ for some $\phi,\varphi$ but no polynomial
$\Phi$ satisfies 
$\Phi \ge P, \Phi \ge Q$.
\item[(I2)] 
Is it true that the smallest equivalence relation generated by strict intertwining
is strictly weaker than the intertwining relation?
Equivalently, does there exist two polynomials $P \approx Q$ such that 
there exists  no sequence of polynomials $P_i, \phi_i$ with $P = P_0$, $Q = P_N$
such that $\phi_i \ge P_i$, and $\phi_i \ge P_{i-1}$ for all $i$.
\item[(I3)] 
Suppose that $P\circ g$ and $P$ are intertwined, and $P$ is not integrable. Is it true that $g\in \Sigma(P)$?
\item[(I4)]
Fix any polynomial $P$. Design an algorithm to determine all monic and centered polynomial of a fixed degree
such that $Q \approx P$.
\item[(I5)]
Describe all intertwining classes for any polynomial of low degree, say $d = 3,4,5$. 
\item[(I6)]
Is the intertwining class of any unicritical polynomial trivial? 
Observe that~\cite[Theorem~1.4]{GNY} proves that two unicritical polynomials are intertwined iff they are conjugated. 
Is the intertwining class of any PCF polynomial trivial?
\end{itemize}

We would like to ask for a uniform boundedness result.
\begin{conj}
For any $d\ge2$, there exists a constant $c(d)$ such that
for any polynomial $P$ of degree $d$, one can find $N\le c(d)$  polynomials $Q_1,\cdots, Q_N$
such that $Q\in \inter(P)$ implies $Q$ to be conjugated to  $g\cdot Q_i^n$ for some integer $n\ge 1$ and some $g\in \Sigma(Q_i)$.
\end{conj}
The problem is open even for $d=2$.

\smallskip

Finally observe that for any two complex polynomials such that $P \approx Q$, there exists a local biholomorphism $\sigma$ defined on an open disk $U$ such that
$\sigma(J(P)\cap U) = J(Q) \cap \sigma(U)$. Since Julia sets determine polynomials up to symmetry (by Theorem~\ref{thm:SS}) and since by invariance under the dynamics the shape of Julia set is locally the same near any of its points,
it is natural to expect the following to be true.
\begin{conj}
Suppose that there exists a univalent map $\sigma \colon U \to \C$ such that
$\sigma(J(P) \cap U) = J(Q) \cap \sigma(U)\neq \varnothing$. Then $P \approx Q$.
\end{conj}

Several results in that direction have been already obtained, but the general case remains elusive, see~\cite{MR3875252,MR2342966}.

%
%
%
%
%
%


\chapter{Polynomial dynamical pairs}\label{chapter-pairs}

A dynamical pair $(P,a)$ is a family of polynomials together with a marked point. 
We first review basic notions of bifurcation and activity for holomorphic dynamical pairs,
and prove the following important rigidity property  when the bifurcation locus is included in a smooth real curve. 

\rigidity*

In \S\ref{sec:algebraic dynamical pairs}, we turn to algebraic dynamical pairs. We first explain how to attach a canonical line bundle to such a pair, and discuss
 the continuity of the Green function associated  to a non-isotrivial pair (Theorem~\ref{thm:main FG}) which turns out to be 
a key technical point for applications. We recall DeMarco's theorem (Theorem~\ref{thm:active-characterization1}) stating that a non isotrivial 
complex algebraic dynamical pair admits bifurcations.

We conclude this chaper by discussing in \S\ref{sec:arithmetic dynamical pair} dynamical pairs defined over a number field and prove that they induce a natural height arising from an adelic semi-positive
metrization of a suitable divisor. This allows one to characterize isotrivial adelic pairs in terms of  their height function. 


\section{Holomorphic dynamical pairs and proof of Theorem~\ref{tm:rigidaffine}}

In this section, we prove a rigidity property for holomorphic dynamical pairs  parametrized by the unit disk 
whose bifurcation locus is included in 
a smooth curve (Theorem~\ref{tm:rigid}). This result  plays an important role in the proof of our main results.


\subsection{Basics on holomorphic dynamical pairs}

Let $V$ be any complex manifold. 

\begin{definition}
 A \emph{holomorphic dynamical pair} $(P,a)$ of degree $d\ge 2$ parametrized by $V$ is a holomorphic family $P:V\times \C \to V\times\C$ of degree $d$ polynomials together with a holomorphic map $a\colon V \to \C$. \index{dynamical pair!holomorphic}
\end{definition}
We say that a family $P$ parametrized by $V$ is \emph{isotrivial} if it has dimension $0$ in moduli, i.e. if there exists a finite branched cover $\pi\colon\tilde V\to V$ and a holomorphic family of affine transformations $\phi_t\in\mathrm{Aut}(\C)$ parametrized by $\tilde{V}$ such that $\phi_t^{-1}\circ P_{\pi(t)}\circ \phi_t$ is independent of $t$.\index{family of polynomials!isotrivial}
 
We also say that a dynamical pair $(P,a)$ parametrized by $V$ is \emph{isotrivial}\index{dynamical pair!isotrivial} if there exists a finite branched cover $\pi\colon \tilde V\to V$ and a holomorphic family of affine transformations $\phi_t\in\mathrm{Aut}(\C)$ parametrized by $\tilde{V}$ such that $\phi_t\circ P_{t}\circ \phi_t^{-1}$ is independent of $t$ and $\phi_t(a(t))$ is constant.

Observe that when the pair $(P,a)$ is isotrivial, then $P$ is too, but the converse is not true. 

We let $\preper(P,a)$ be the set of parameters $t \in V$ such that $a(t)$ is preperiodic for $P_t$.

\medskip

We shall mostly be interested in the case the parameter space is a Riemann surface. We thus pick any (connected) Riemann surface $S$ and 
let $(P,a)$ be a dynamical pair of degree $d$ parametrized by $S$.

Observe that the set $\preper(P,a) = \bigcup_{n>m\geq0} \{ t \in S, \, P^n_t(a(t)) = P^m_t(a(t)) \}$ is either equal to $S$, or
is countable. 

\smallskip

We say the pair $(P,a)$ is \emph{stable} at a parameter $t_0\in S$ if there exists an open set $U\subset S$ with $t_0\in S$ and such that the sequence of holomorphic maps $t\in U\longmapsto P_t^n(a(t))$ forms a normal family on $U$. The set of stable parameters is an open set called
the \emph{stability locus} whose complement is the \emph{bifurcation locus} denoted by $\mathrm{Bif}(P,a)$.\index{stability locus!of a pair}\index{bifurcation!locus (of a pair)}

\begin{proposition}\label{prop:bifurcation}
For any any holomorphic dynamical pair $(P,a)$ parametrized by a Riemann surface $S$, the function 
\[g_{P,a}(t):=g_{P_t}(a(t)), \ t\in S~\]
is continuous and subharmonic. 

Moreover the boundary of $\{g_{P,a}=0\}$ coincides with the support of the positive measure $\Delta g_{P,a}$, 
which in turn equals $\mathrm{Bif}(P,a)$. In particular, $\mathrm{Bif}(P,a)$ is closed and perfect, i.e. has no isolated point, and has empty interior.
\end{proposition}
\begin{remark}\label{rem:stab-dense}
In particular, the stability locus of any holomorphic dynamical pair is open and dense.
\end{remark}

\begin{definition}
The \emph{bifurcation measure} of the dynamical pair $(P,a)$ is the positive measure
$\mu_{P,a} :=\Delta g_{P,a}$ on $S$. 
\end{definition}\index{bifurcation!measure (of a pair)}

 \begin{proof}
 Recall that $g_t(z)$ is the uniform limit on the product of any compact subset of $S$ with the complex plane of the sequence of continuous psh functions
 $\frac1{d^n} \log^+|P^n_t(z)|$, so that $g_{P,a}(t)$ is continuous and psh. 
 
 We have $g_{P,a}(t) \ge 0$ and $g_{P,a}(t) > 0$ iff $a(t) \notin K(P_t)$ so that 
 $g_{P,a}(t) = \lim_n \frac1{d^n} \log|P^n_t(z)|$ when $g_{P,a}(t) > 0$ which implies
 the harmonicity of $g_{P,a}$ on  $\{g_{P,a} > 0\}$.  It follows from the maximum principle
 that the support of $\mu_{P,a}$ is equal to the boundary of  
 $\{g_{P,a}=0\}$. 
 
Pick a point $t_0$ in the support of $\mu_{P,a}$. Then the sequence of functions $t \mapsto P^n_t(a(t))$ cannot be normal
near $t_0$ since $g_{P,a}(t) >0$ for some point close to $t_0$ hence $P^n_t(a(t)) \to \infty$ whereas $P^n_t(a(t_0))$ remains bounded.
When $t_0$ is not in the support of $\mu_{P,a}$, then either $g_{P,a}(t) >0$ and $P^n(a(t)) \to \infty$ uniformly in a neighborhood of $t_0$; 
or $g_{P,a}(t) =0$ on a small disk containing $t_0$ and $P^n(a(t))$ takes its value in a fixed compact set thereby being normal by Montel's theorem.

Finally, since $\mathrm{Bif}(P,a)=\partial \{g_{P,a}=0\}$ it is closed and has empty interior. If $t_0\in \mathrm{Bif}(P,a)$ is isolated, as $g_{P,a}\geq0$, there exists a neighborhood $U$ of $t_0$ with $\mathrm{supp}(\Delta g_{P,a})\cap U=\{t_0\}$, whence $\Delta g_{P,a}$ gives mass to $t_0$. In particular, $\Delta g_{P,a}(\{t_0\})>0$. This is impossible since $g_{P,a}$ is locally bounded near $t_0$.
 \end{proof}
 
 Let us include here for completeness the following
 
 \begin{theorem}\label{thm:discrete preper}
Let $(P,a)$ be any holomorphic dynamical pair parametrized by the unit disk. If $(P,a)$ is stable and $a$ is not stably preperiodic, then 
the accumulation points of the set $\preper(P,a)$
is included in the analytic subset $Z$ of $\D$ where the multiplicity of periodic points of $P$ strictly increases. 
\end{theorem}
This result is a direct consequence of~\cite[Theorem 1.1]{favredujardin}, which is a refinement of~\cite{LYUBICH}, see \cite[\S 2.3]{Chio-roeder-chromatic} for this precise formulation.


\subsection{Density of transversely prerepelling parameters}

Given any holomorphic dynamical pair $(P,a)$, we say that the marked point  is \emph{prerepelling} at $t_0\in S$ if $a(t_0)$ eventually lands on a repelling periodic point $z_0$. 
In such a situation there exists $m\in \N$ such that $z_0=P_{t_0}^m(a(t_0))$ is a repelling periodic point of $P_{t_0}$ of exact period $k$. 
By the Implicit Function Theorem, there exists $\epsilon>0$ and an analytic function $z:\D(t_0,\epsilon)\to\C$ 
such that $P^k_t(z(t)) = z(t)$ for all $t$ and $z(t_0)=z_0$. 
\begin{definition}
The marked point $a$ is said to be properly (resp. transversally) prerepelling at $t_0$, if the two graphs $\{(t,P_t^m(a(t)))\in\D(t_0,\epsilon)\times\C\}$ and $\{(t,z(t))\in\D(t_0,\epsilon)\times\C\}$ intersect properly (resp. transversally)  at the point $(t_0,z_0)$ in $\D(t_0,\epsilon)\times\C$. \index{preperiodic parameter!transversally}\index{preperiodic parameter!properly}
\end{definition}
Our aim is to prove the following characterization of the bifurcation locus. This result is essentially due to Dujardin, see~\cite[Theorem 0.1]{dujardin-higher}
except that the latter reference only deals with marked critical point. 

\begin{theorem}\label{tm:suppT}
Let $(P,a)$ be any dynamical pair of degree $d$ parametrized by a Riemann surface $S$. 
Then the bifurcation locus $\mathrm{Bif}(P,a)$ of the pair $(P,a)$ is the closure of the set of parameters $t\in S$ at which the marked
point $a$ is transversely prerepelling.
\end{theorem}

\begin{remark}
The second author has obtained a version of this theorem for holomorphic dynamical pairs in arbitrary dimensions, see~\cite{gauthier-smooth-bif} for details. We present here an argument that can be adapted to treat \emph{properly} prerepelling parameters.
\end{remark}

We begin with the following 
\begin{lemma}\label{lm:suppT}
 Assume the marked point $a$ is properly prerepelling at $t_0\in S$. Then $t_0\in\mathrm{Bif}(P,a)=\supp(\mu_{P,a})$.
\end{lemma}
\begin{proof}
Pick $t_0\in S$ such that $a$ is properly prerepelling at $t_0$. We use the notation above: $k$ is the period of $z_0$ and $\epsilon$ is a sufficiently small positive number. 

We proceed by contradiction assuming that the family $(P_t^n(a(t)))_n$ is normal at $t_0$. Let $K>1$ and $\delta>0$ be small enough so that $|(P_t^k)'(z)|\geq K>1$ for all $(z,t)\in \D(z_0,\delta)\times\D(t_0,\epsilon)$. Reducing $\epsilon$ if necessary, we may assume that  $z(t)$ and $P_t^{m+kn}(a(t))$ belong to $\D(z_0,\delta/2)$ for all $t\in\D(t_0,\epsilon)$ and for all $k\ge0$.

For any integer $n\geq0$, and for every $t\in\D(t_0,\epsilon)$  set $\varepsilon_{n}(t):= P_t^{m+kn}(a(t))-z(t)$. Differentiating the quantity $\varepsilon_{n+1}(t)=P_t^k(P_t^{m+kn}(a(t)))-P_t^k(z(t))$, we get
\begin{multline*}
\varepsilon_{n+1}'(t) =  (P_t^k)'(P^{m+kn}_t(a(t)))\cdot \varepsilon_{n}'(t)-
\\
\big((P_t^k)'(z(t))-(P_t^k)'(P^{m+kn}_t(a(t)))\big)\cdot z'(t) + 
\\
\frac{\partial P_t^k}{\partial t}\left(P^{m+kn}_t(a(t)))\right)-\frac{\partial P_t^k}{\partial t}(z(t)).
\end{multline*}
Pick now $\tau>0$ as small as you wish. Since $\{P^{m+kn}_t(a(t))\}_n$ forms a normal family whose value at $t=t_0$ equals $z(t_0)$, 
we may again reduce $\epsilon$
to obtain
\[|\varepsilon_{n+1}'(t)|\geq K|\varepsilon_{n}'(t)|-\tau\]
for all $t\in \D(0,\epsilon)$. By induction, we infer 
\[|\varepsilon'_{n}(t)|\geq K^n\left(|\varepsilon_0'(t)|-\frac{\tau}{(K-1)}\right)~.\]
By assumption $a$ is properly repelling hence $\varepsilon_0$ cannot be identically zero and we get a contradiction. 
\end{proof}

\begin{proof}[Proof of Theorem~\ref{tm:suppT}]
According to Lemma~\ref{lm:suppT}, it is sufficient to prove the density of transversely prerepelling parameters in $\supp(\mu_{P,a})$. We follow closely the proof of~\cite[Theorem 0.1]{dujardin-higher}. 

Pick $t_0\in\supp(\mu_{P,a})$. 
According to~\cite[Lemma~4.1]{dujardin-higher}, there exists an integer $m\geq1$ and a $P_{t_0}^m$-compact set $K\subset\C$ such that 
\begin{itemize}
\item $P_{t_0}^m|_K$ is uniformly hyperbolic and conjugated to the one-sided shift on two symbols,
\item the unique invariant measure $\nu$ on $K$ satisfying $(P_{t_0}^m)^*\nu=2\nu$ has continuous potential.
\end{itemize}
Moreover, by \cite[\S 2]{Shishikura2}, there exists $\epsilon>0$ and a holomorphic motion $h:\D(t_0,\epsilon)\times K\to\C$ which conjugates the dynamics, i.e. satisfying 
\[h_t\circ P^m_{t_0}(z)=P^m_t\circ h_t(z), \text{ for all } (t,z)\in \D(t_0,\epsilon)\times K~.\]
The function $\hat{h}$ defined by
\[\hat{h}(t,z):=\int_{K}\log|z-h_t(w)|\mathrm{d}\nu(w)\]
is psh and it is continuous on $\D(t_0,\epsilon)\times\C$ by~\cite[Lemma 6.4]{dujardin-structure-laminar}.
Observe also that 
\[\hat{h}(t,z)=\log^+|z|+O(1) \text{ as } |z|\to\infty\]
and $O(1)$ is locally uniform in $t$ since $K_t$ is included in a fixed closed ball for all $t$.  Approximating $\nu$ by Dirac masses, we see that 
\[dd^c\hat{h}=\int_{K}[\Delta_z]\mathrm{d}\nu(z), \text{ with } \Delta_z:=\{(t,h_t(z))\in\D(t_0,\epsilon)\times\C\}~.\] 
Write  $u_n(t):=\hat{h}(t,P_t^n(a(t)))$ for all $t\in\D(t_0,\epsilon)$. We claim that $d^{-n}\Delta u_n\to\mu_{P,a}$ as $n\to\infty$. Indeed for any compact set $E\subset \D(t_0,\epsilon)$, there exists $C>0$ such that
\[\left|\log^+|z|-\hat{h}(t,z)\right|\leq C\]
for all $t\in E$ and all $z\in\C$. In particular, for all $n\geq0$ and all $t\in E$, we have
\[\left|\frac{1}{d^n}\log^+|P_t^n(a(t))|-\frac{1}{d^n}u_n(t)\right|\leq \frac{C}{d^n},\]
and the claim follows by taking the Laplacian of the left hand side and letting $n\to\infty$.

\smallskip

To conclude the proof we interpret the bifurcation measure as the image of the intersection of the graph  $\Gamma_n = \{ (t,P_t^n(a(t)))\in\D(t_0,\epsilon)\times\C \}$
with the positive closed $(1,1)$ current $dd^c \hat{h}$. 

For each $z \in K$ such that $\Gamma_n \cap\Delta_z$ is  discrete, the intersection of the two positive closed currents
\[[\Gamma_n] \wedge [\Delta_{z}] = dd^c_{t,w} ( \log|w-P_t^n(a(t))|  [\Delta_z]) \] is well-defined by~\cite[Proposition 4.12, page 156]{demailly}, and equals
to the atomic measure supported on the set of intersection points of $\Gamma_n$ and $\Delta_z$
with weight given by the multiplicity of intersection.  
If $\Gamma_n \cap\Delta_z$ is not discrete then $\Gamma_n = \Delta_z$ (so that  $z$ is uniquely determined) 
and we set by convention  $[\Gamma_n] \wedge [\Delta_z] := 0$.
We get
\begin{equation}\label{eq:laminar}
 \mu_n := [\Gamma_n] \wedge dd^c \hat{h} = \int_{K }[\Gamma_n]  \wedge [\Delta_z]\mathrm{d}\nu(z)~. 
 \end{equation}
Indeed, for any $\varphi\in\mathcal{C}^{\infty}_c(\D(t_0,\epsilon)\times \C)$, Fubini gives
\begin{align*}
\left\langle \mu_n,\varphi\right\rangle & =\int \log|w-P_t^{n}(a(t))|dd^{c}\hat{h}\wedge dd^{c}\varphi\\
& =\int_K\left(\int_{\Delta_z}\log|w-P_t^{n}(a(t))| dd^{c}_{w,t}\varphi\right)\mathrm{d}\nu(z)\\
& = \int_K\left\langle[\Gamma_n]\wedge [\Delta_z],\varphi\right\rangle\mathrm{d}\nu(z).
\end{align*}
Since $(\pi_1)_* ([\Gamma_n] \wedge dd^c \hat{h}) = \Delta u_n$ and $t_0 \in \supp(\mu_{P,a})$,
by the claim above, we get a point $t_*$ arbitrarily close to $t_0$ 
such that $q_* := \pi_1^{-1}(t_*) \cap \Gamma_n$ lies in the support of $\mu_n$ for some large $n$.

By~\eqref{eq:laminar}, $\Gamma_n$ intersects  $\Delta_{z_*}$ for some $z_* \in K$
near the point $q_*$. Since the set of repelling periodic points of $P^m_{t_0}$ is dense in $K$, we may find a sequence 
$z_p \in K$ such that $z_p \to z$, and $h_{t}(z_p)$ is $P_t$-periodic for all $t$. 
We conclude by~\cite[Lemma 6.4]{BLS} which implies that $\Gamma_n$ and $\Delta_{z_p}$ should intersect transversally near $q_*$
for all $p$ sufficiently large. 
\end{proof}

\subsection{Rigidity of the bifurcation locus}\label{sec:realfamily}

Our next (somehow technical) result gives a precise description of all situations in which the bifurcation locus of a holomorphic dynamical pair is included in a
smooth real curve. 

We say that a holomorphic family of polynomials $P$  parametrized by the unit disk is \emph{real}\index{family of polynomials!real} when the coefficients of $P$
are defined by power series with real coefficients converging in the segment $\mathopen]-1, +1 \mathclose[$.
We call a \emph{real dynamical pair} is a dynamical pair $(P,a)$ for which $P$ is a real family and the marked point is also
defined by a real power series. 

\begin{theorem}\label{tm:rigid}
Let $(P,a)$ be any dynamical pair of degree $d$ parametrized by the unit disk $\D$, and suppose that 
$\mathrm{Bif}(P,a)$ is non-empty and included in a smooth real curve. 

Then $\mathrm{Bif}(P,a)$ is included in some real-analytic curve.  Moreover if the family  is not conjugated to a constant integrable polynomial, then 
$\mathrm{Bif}(P,a)$ is a closed totally disconnected perfect set, and
the following holds. 

Any point $t_0 \in \mathrm{Bif}(P,a)$ outside possibly a discrete subset of $\D$
admits a small neighborhood $U$ such that :
\begin{itemize}
\item[(1)]
for any $t\in U$, a critical point $c$ of $P_t$ either escapes to $\infty$ or satisfies $P^4_{t}(c)= P^2_{t}(c)$,
\item[(2)] 
$J(P_t) = K(P_t)$ is totally disconnected for all $t\in U$,
\item[(3)]
the family $(P_t)_{t \in U}$ is $J$-stable.
\end{itemize}
Moreover there exists a reparametrization of $(P_t)_{t\in U}$ 
for which the family is conjugated to a real family on $U$, and $\mathrm{Bif}(P,a)$
is included in the real line.
\end{theorem}

\begin{remark}
Suppose that $P_t \equiv P_*$ is a constant family, that $P_*$ is real and satisfies (1) and (2). Since
we have $\mathrm{Bif}(P,a) = \{ t, \, a(t) \in J(P_*)\}$ by Proposition~\ref{prop:bifJstab} below, 
the bifurcation locus is then included in the real-analytic curve $a^{-1}(\R)$.

It is not clear to the authors whether the conditions (1) -- (3) conversely imply the bifurcation locus to be included in the real line.
\end{remark}

\begin{proof}[Proof of Theorem~\ref{tm:rigidaffine}]
Pick a dynamical pair $(P,a)$ parametrized by some connected Riemann surface $C$ and a holomorphic disk $U\subset C$, such that $\mathrm{Bif}(P,a) \cap U \neq \varnothing$ is 
included in a smooth curve. 
Theorem~\ref{tm:rigid} implies that we fall into two cases. 
Either the family $(P_t)_{t\in U}$ is isotrivial and $P_t$ is integrable for all $t$ which
 implies the family is also isotrivial over $C$. 
 Or we may find another holomorphic disk $\imath \colon \D \to U$ intersecting the bifurcation locus over which the family is conjugated to a real family
 and the bifurcation locus is included in the real line. 
\end{proof}

The proof of Theorem~\ref{tm:rigid} ranges over the next three subsections.


\subsection{A renormalization procedure}\label{sec:renorm}
We expose here how we reinterpret the similarity argument of Tan Lei \cite{similarity} in our setting. Note that we need a more general and more precise statement than that in \cite{gauthier-smooth-bif}.

Let $(P_t)_{t\in \D}$ be a holomorphic family of degree $d$ polynomials parametrized by the unit disk and $a:\D\to\C$ be a marked point. We assume there exists a holomorphically moving repelling periodic point $z:\D\to\C$ of period $k$ with $P^m_0(a(0))=z(0)$. We say $a$ is properly prerepelling of \emph{order} $q\geq1$ if
\[P_t^m(a(t))-z(t)=\alpha\cdot t^q+O(t^{q+1})\]
for some $\alpha\in\C^*$. Let $\rho(t):=(P_{t}^k)'(z(t))$ and denote by $\phi_t$ the linearizing coordinate of $P_t^k$ at $z(t)$ which is tangent to the identity, i.e. such that
$P^k_{t}\circ \phi_t(z)=\phi_t(\rho(t)\cdot z)$, and $\phi_t'(0)=1$ for all $t\in\D$ and all $z\in \D(0,r)$ for some $r>0$.

The next result is standard, and reflects the similarity between the Julia set of $P_t$ and the bifurcation locus at $t$.
We give here a complete proof for sake of completeness, adapting the proof by Buff and Epstein~\cite{buffepstein}.

 \begin{lemma}\label{lm:cvrenorm}
Assume that the marked point $a$ is properly prerepelling at $0$ with order $q\geq1$ and pick any $q$-th root $\lambda$ of $\rho(0)$. Then
\[g_{P_0}\circ\phi_0(t^q)=\lim_{n\to\infty}d^{m+kn}g_{P,a}(\lambda^{-n}t)~,\]
and the convergence is uniform on some small disk containing $0$.
 \end{lemma}

 \begin{proof}
For $n\geq1$, we set $r_n(t):=\lambda^{-n}t$ and
\[a_n(t):=P_{r_n(t)}^{m+kn}\left(a\circ r_n(t)\right).\]
First prove that,  there exist a constant $C>0$ and $\epsilon>0$ small enough 
such that for all $t\in\D(0,\epsilon)$ and all $n\geq1$, we have
 \begin{eqnarray}
\left|a_n(t)-\phi_0(t^q)\right| &\leq & C\frac{n|t|}{|\rho(0)|^{n/q}} .\label{speedrenorm}
\end{eqnarray}
We fix $\epsilon>0$ small enough such that $P_{t}^m(a(t))$ lies in the range of $\phi_t$ for all $t\in\D(0,\epsilon)$. We may thus define 
$h(t):=\phi_t^{-1}(P_t^m(a(t)))$ for all $t\in \D(0,\epsilon)$. As $\phi_t(z)$ depends analytically on $(t,z)$, the map $h$ is holomorphic and
\begin{align*}
h(t) & =\phi_t^{-1}\left(z(t)+P_t^m(a(t))-z(t)\right)\\
& =\phi_t^{-1}\left(z(t)+\alpha t^q+O(t^{q+1})\right)= \alpha t^q+ O(t^{q+1}),
\end{align*}
where we used that $\phi_t(0)=z(t)$, $\phi_t'(0)=1$ so that $\phi_t^{-1}(z(t) + w) = w + O(w^2)$. 

To simplify notation we reparametrize the unit disk and assume $\alpha=1$. 
In particular, there exists a constant $M>0$ such that $|h(t)-t^q|\leq M|t|^{q+1}$.
Again as $\phi_t(z)$ is analytic, there exist $C_1,C_2>0$ such that 
$\left|\phi_t(z)-\phi_s(w)\right|\leq C_1|z-w|+C_2|t-s|$,
for all $z,w\in\D(0,r)$ and all $s,t\in\D(0,\epsilon)$. In particular, for all $t\in\D(0,\epsilon)$, all $n\geq1$ and all $z,w\in\D(0,r)$, we find
\[\left|\phi_{r_n(t)}(z)-\phi_{0}(w)\right|\leq C_1|z-w|+C_2\frac{|t|}{|\rho(0)|^{n/q}}.\]
Similarly there exists a constant $C_3\geq1$ such that $\left|\rho(t)/\rho(0)-1\right|\leq C_3|t|$ for all $t\in\D(0,\epsilon)$, and for $|t|$ small enough we get
\[\left|\left(\frac{\rho(r_n(t))}{\rho(0)}\right)^n-1\right|\le 
\left|\frac{\rho(r_n(t))}{\rho(0)}-1\right| \times (2n)
\le 2nC_3 \frac{|t|}{|\rho(0)|^{n/q}}~.
\]
Observe that
\begin{eqnarray*}
a_n(t) & = & P_{r_n(t)}^{kn}\left(P_{r_n(t)}^m(a \circ r_n(t))\right)=P_{r_n(t)}^{kn}\left(\phi_{r_n(t)}\left(h\circ r_n(t)\right)\right)\\
 & = & \phi_{r_n(t)}\Big(\rho(r_n(t))^n\left(h\circ r_n(t)\right)\Big).
\end{eqnarray*}
Putting all the above together, for all $t\in\D(0,\epsilon)$, we find
\begin{align*}
\left|a_n(t)-\phi_0(t^q)\right| \leq \, &  C_2\frac{|t|}{|\rho(0)|^{n/q}}+C_1\left|\rho(r_n(t))^n\left(h\circ r_n(t)\right)-t^q\right|\\
 \leq \, & C_2\frac{|t|}{|\rho(0)|^{n/q}}+C_1\left|\left(\frac{\rho(r_n(t))}{\rho(0)}\right)^n-1\right||t|^q  \\
 & \hspace*{1.93cm} +C_1|\rho(r_n(t))|^n|h\circ r_n(t)-r_n(t)^q|\\
 \leq \, & C_2\frac{|t|}{|\rho(0)|^{n/q}}+2nC_3 \frac{ |t|^{q+1}}{|\rho(0)|^{n/q}}
+C_1M\left|\rho(r_n(t))\right|^n\left|r_n(t)\right|^{q+1}, 
\end{align*}
and
\[\left|\rho(r_n(t))\right|^n\left|r_n(t)\right|^{q+1} =\left|\frac{\rho(r_n(t))}{\rho(0)}\right|^n\frac{|t|^{q+1}}{|\rho(0)|^{n/q}}
\leq2\frac{|t|^{q+1}}{|\rho(0)|^{n/q}},
\]
for $|t|$ small enough which implies~\eqref{speedrenorm}. 

In particular, the sequence $a_n(t)$ converges uniformly on $\D(0,\epsilon)$ to $\phi_{0}(t^q)$ and, if
\[g_n(t):=g_{P_{r_n(t)}}\!\left(a_n(t)\right), \ t\in\D(0,\epsilon),\]
since $r_n\to0$ uniformly on $\D(0,\epsilon)$, the above implies $g_n(t)\to g_0\circ\phi_0(t^q)$ uniformly on $\D(0,\epsilon)$.
Using that $g_{P_t}\circ P_t=dg_{P_t}$, we get the wanted convergence.
\end{proof}

\subsection{Bifurcation locus of a dynamical pair and $J$-stability}

\begin{proposition}\label{prop:bifJstab}
Let $(P,a)$ be any dynamical pair parametrized by the unit disk $\D$. If $P$ is $J$-stable and $\mathrm{Bif}(P,a)\neq\varnothing$, we have
\[\mathrm{Bif}(P,a)=\{t\in \D\, ; \ a(t)\in J(P_t)\}.\]
\end{proposition}

\begin{proof}
Let $h:\D\times J(P_0)\to\p^1$ be the unique holomorphic motion of $J(P_0)$ such that
\[h_t\circ P_0=P_t\circ h_t \ \text{on} \ J(P_0).\]
Observe that the set $\{ (t,z), t \in \D, z \in J(P_t)\}$ is equal to 
$ \{ (t, h_t(z)), t \in \D, z \in J(P_0)\}$ hence is closed. 
Since $\{t\in \D\, ; \ a(t)\in J(P_t)\}$ is the image under the first projection of the intersection of this set with the graph of $a$, it is also closed in $\D$.

By Theorem~\ref{tm:suppT},  the set of parameters $t_0\in\D$ such that $a$ is transversely prerepelling at $t_0$ is a dense subset of $\mathrm{Bif}(P,a)$ . As repelling points of $P_{t_0}$ are contained in $J(P_{t_0})$ and $\{t\in \D\, ; \ a(t)\in J(P_t)\}$ is closed, we get $\mathrm{Bif}(P,a)\subset \{t\in \D\, ; \ a(t)\in J(P_t)\}$.

Suppose conversely that $t_*\in \{t\in \D\, ; \ a(t)\in J(P_t)\}$ so that $a(t) = h_t(z_*)$ for some $z_* \in J(P_0)$. 
Since the bifurcation locus is assumed to be non-empty, the two curves $\Gamma:= \{ (t, a(t)), t\in \D\}$ and $\{ (t, h_t(z_*)), t \in \D\}$
cannot coincide. Choose any sequence of repelling periodic point $z_n$ accumulating $z_*$. Then by~\cite[Lemma 6.4]{BLS}
the curves $\{ (t, h_t(z_n)), t\in \D\}$ intersect transversally $\Gamma$ near $t_*$ for all $n$ sufficiently large. We conclude by Lemma~\ref{lm:suppT} 
that $t_0$ is accumulated by points in the bifurcation locus, hence  $ \{t\in \D\, ; \ a(t)\in J(P_t)\}\subset \mathrm{Bif}(P,a)$.
\end{proof}


\subsection{Proof of Theorem~\ref{tm:rigid}}

Suppose $\mathrm{Bif}(P,a)$ is included in a smooth curve $\gamma$. Our first objective is to prove the following fact, whose proof is essentially contained in~\cite{Eremenko-vanStrien}.

\begin{proposition}\label{prop:real}
 Suppose $t_0$ is a transversally prerepelling parameter such that  $P_{t_0}$ is not integrable. 
Then $\{ t_0 \}$ is a connected component of $\mathrm{Bif}(P,a)$, and 
$P_{t_0}$ is conjugated to a real polynomial whose critical points are either escaping to $\infty$, or satisfy
$P^4_{t_0}(c) = P^2_{t_0}(c)$, 
and which has a totally disconnected Julia set included 
in $\R$.  
\end{proposition}
\begin{proof}
Let $m$ and $k$ be  two integers such that $P^m_{t_0} (a(t_0))$ is a repelling periodic point $p_*$ of period $k$ 
of multiplier $\lambda$. We let $\phi \colon \C \to \C$ be the unique linearization map at this periodic point which is tangent to the identity
and so that  $\phi (\lambda z) = P^k_{t_0} (\phi(z))$. 
 By Lemma~\ref{lm:cvrenorm}, we have 
\[\Delta g_{P_{t_0}}(\phi(t)) = \lim_{n\to\infty} \Delta (d^{m+kn} g_{P,a}(\lambda^{-n} t))~.\] 
Let $L$ be the real line tangent to $\gamma$ at the point $t_0$. In a fixed disk $D$ centered at $t_0$, observe that 
$\lambda^{n} \gamma \cap D$ becomes $C^\infty$-close to $L$. Since the support of the measure $\Delta (g_{P,a}(\lambda^{-n} t))$
is included in $\lambda^{n} \gamma$ we conclude that 
\[
J_* := \phi^{-1} (J(P_{t_0})) = \supp (\phi^* \Delta g_{P_{t_0}}) \subset L~.\]
Since $J_*$ is a closed subset of a real line, it is either totally disconnected or it contains a segment. In the latter case, 
$J(P_{t_0})$ is locally real-analytic at some point, and $P_{t_0}$ is integrable by Theorem~\ref{th:zdunik} which contradicts our assumption. 
In particular, $\{t_0\}$ is a connected component of  $\mathrm{Bif}(P,a)$. 

Observe that $J_*$ is invariant by the dilatation by the dynamics hence $\lambda \in \R$, and conjugating $P_{t_0}$ by a suitable homothety
we may assume that $J_* \subset \R$.

Pick any periodic point $p$ for $P_{t_0}$ of period $N$ which does not lie in its post-critical set so that $\phi$ is a local diffeomorphism at any preimage of $\phi^{-1}(p)$. 
Then we obtain that locally at $p$, $J(P_{t_0})$ is included in $\phi(J_*)$ which is a smooth (in fact real-analytic) curve. We may now repeat the argument
with the linearization map $\phi_p$ at $p$, and we conclude that the multiplier $\lambda_p$ of $p$ is real, and that $\phi_p^{-1}(J(P_{t_0}))$ is real and totally disconnected.

Now we follow the arguments of Eremenko and Van Strien. We only sketch the main ideas referring to the original paper for details. 
Recall that the order of the entire function $\phi_p$ is defined by 
\[ \rho(p) := \limsup_{r\to\infty} \frac{ \log \left(\log \sup_{|z|\le r} |\phi_p(z)|\right)}{\log r}~.\]
Let $C = \sup_{\D(0,|\lambda_p|)} \max\{|\phi_p|,1\}$, and pick any constant $C'\geq1$ such that $|P_{t_0}(z)| \le C' \max\{|z|,1\}^d$. Since $\phi_p(\lambda_p z) = P^N_{t_0}(\phi_p(z))$ we get 
\[|\phi_p(z)| = |P^{nN}_{t_0}(\phi_p(\lambda_p^{-n}z))| \le \left(C'\max\{|\phi_p(\lambda_p^{-n}z)|,1\}\right)^{d^{nN}}.\]
If $n$ is chosen so that $|\lambda_p| \ge |\lambda^{-n}_p z| \ge 1$, then we have
\[
\log |\phi_p(z)| \le (\log C'C) |z|^{\frac{N \log d}{\log |\lambda_p|}}
\]
so that $\rho(p) \le \frac{N \log d}{\log |\lambda_p|}$. 

Since $J(P_{t_0})$ is totally disconnected, at least one critical point escapes, and by the Misiurewicz-Przytycki's formula~\eqref{eq:MPform}, 
the Lyapunov exponent $\chi$ of $P$ satisfies $\chi >\log d$. 
By e.g.~\cite[proof of Theorem 30]{bsurvey} we can find a sequence of periodic points accumulating $p_*$ whose multiplier is arbitrarily
close to $\chi$. In particular, we may find $p$ of period $N$ having a multiplier $|\lambda_p| > d^N$ so that   
$\rho(p) < 1$. 
Conjugate $P_{t_0}$ to have $p=0$.
Denote by $z_j\neq 0$ the zeroes of $\phi_p$.
Since $p$ belongs to $J(P_{t_0})$, these zeroes are all real.
By Hadamard's theorem~\cite[Theorem 8 p.209]{Ahlfors}, we may write
\[ 
\phi_p (z) = z^m \prod_j \left( 1 - \frac{z}{z_j} \right) e^{\frac{z}{z_j} + \frac{1}{2}\left(\frac{z}{z_j}\right)^{2} + \cdots + \frac{1}{m_j}\left(\frac{z}{z_j}\right)^{m_j}}
\]
for some $m\in \N$, $ m_j \in \N^*$. 

We infer that  $\phi_p(\R) \subset \R$ which shows that $J(P_{t_0})$ is real.
Let $I$ be the convex hull of $J(P_{t_0})$: it is a compact segment whose extremeties $b_- < b_+$
are either pre-fixed or forms a $2$-cycles.  
Since all preimages of $b_\pm$ belongs to $I$, the intermediate value theorem implies that 
all critical points of $P_{t_0}$ belong to $I$, and all critical values lie outside the interior of $I$. 
Since $J(P_{t_0})$ is totally disconnected at least one critical escapes and any other critical point $c$ is mapped by $P_{t_0}$
to one of boundary point of $I$, hence $P^4_{t_0}(c) = P^2_{t_0}(c)$. 

This concludes the proof of the proposition. 
\end{proof}

We now come back to the proof of Theorem \ref{tm:rigid}. When the family is isotrivial, we may assume that $P_t =P_*$ is a fixed polynomial, and
the previous proposition implies that $P_*$ is real, that all its critical points are either escaping or satisfy  $P^4_*(c) = P^2_*(c)$, 
and $J(P_*) = K(P_*)$ is a totally disconnected subset of $\R$.
The statements (1), (2) and (3) are thus clear in this case so that the theorem is proved in this case. 

\medskip

From now on, we assume that the family is not isotrivial. We reparametrize the family and assume that 
$P_t$ is monic and centered for all $t$, and all critical points are marked. We denote by $\mathcal{E}$ the set of parameters for which $P_t$ is integrable. 
As the family is not isotrivial, $\mathcal{E}$ is discrete. 

Since the bifurcation locus is non-empty, it follows from Theorem~\ref{tm:suppT} and the previous Proposition \ref{prop:real} that for a set $\mathcal{D}$ of parameters $t$
that is dense in $\mathrm{Bif}(P,a)$, there exists $\alpha_t\in \C^*$ and $\beta_t \in \C$ s.t. $\alpha_t^{-1} P_t(\alpha_t z + \beta_t) -\beta_t$ has real coefficients. 
Pre-composing by the dilatation of factor $1/|\alpha_t|$,  we may suppose that $|\alpha_t| =1$, and since $P_t$ is monic we get  $\alpha_t^{d-1} = \pm1$. 
Pre-composing by a real translation of vector $B$ replaces $\beta_t$ by $B \alpha_t + \beta_t$. 
As $P_t$ is centered we have $\beta_t \in \alpha_t \R$ so that we may actually choose $\beta_t =0$. It follows that  $\alpha^{-1} P_t(\alpha z) \in \R[z]$
for at least one $(2d-2)$-th root of unity $\alpha$.

For each  $\alpha \in \U_{2(d-1)}$, and for any $t$ write 
\[P_{t,\alpha}(z) = \alpha^{-1}P_t(\alpha z) = z^d + \sum a_{i,\alpha}(t) z^i,\]
and set 
\[ \Gamma_\alpha:= \bigcap_{i=0}^{d-1} a_{i,\alpha}^{-1}(\R), \text{ and } \Gamma = \bigcup_{\U_{2(d-1)}} \Gamma_\alpha . \]
Then $\Gamma$ is a real-analytic subvariety of $\D$ containing $\mathcal{D}$ which is dense in $\mathrm{Bif}(P,a)$,  so that 
$\mathrm{Bif}(P,a)\subset \Gamma$. Observe that $\mathrm{Bif}(P,a)$ is totally disconnected otherwise we could 
find a segment in the bifurcation locus containing a transversally repelling parameter and this would contradict our standing assumption by 
Proposition~\ref{prop:real}.

We consider the set of points $t_0 \in \Gamma$ which admits a neighborhood $U$ such that 
$\Gamma \cap U = \Gamma_\alpha$ for some $\alpha \in \U_{2(d-1)}$. The complement of this set is a discrete
subset of $\Gamma$ that we adjoin to $\mathcal{E}$ together with all singular and all isolated points of $\Gamma$. 

For any $t_0 \in \mathrm{Bif}(P,a)\setminus \mathcal{E}$, we may thus replace the family $(P_t)$ by $(P_{t,\alpha})$ on $U$.
Reducing $U$ if necessary we also suppose that $\Gamma \cap U$ is a segment, and reparametrizing
the family in $U$, we have $\Gamma \cap U = \R \cap U$.
For a dense subset of $\mathrm{Bif}(P,a)$ we get $P_t \in \R[z]$, thus
$P_t \in \R[z]$ for all $t \in \Gamma \cap U$ which implies the family to be real on $U$.

Recall that all critical points are marked. Let $c_1, \ldots, c_r$ be the critical points satisfying $P^4_t(c_i(t)) = P^2_t(c_i(t))$
persistently in $U$,  and let $\tilde{c}_1, \ldots, \tilde{c}_s$ be the other critical points.  

For any $t \in \mathcal{D}$,  we have $J(P_t) = K(P_t)\subset \R$ and $\tilde{c}_j(t)$ escape for all $j$. 
Since the measure $\mu_{P_t}$ varies continuously, it follows that $J(P_t)$ remains included in the real line
for all $t \in \mathrm{Bif}(P,a)$ hence $J(P_{t_0}) = K(P_{t_0})$.

\begin{lemma}\label{lem:ouf}
If $P^4_{t_0}(\tilde{c}_j(t_0))\neq P^2_{t_0} (\tilde{c}_j(t_0))$ for all $j$, then $\tilde{c}_1(t_0), \ldots, \tilde{c}_s(t_0)$ escape and
$J(P_{t_0})$ is totally disconnected.  
\end{lemma}
Since for $t \in \mathcal{D}$ all critical points $\tilde{c}_j(t)$ escape under $P_t$, the set of parameters $t$ having a critical point $c =\tilde{c}_j$
satisfying  $P^4_{t}(c)= P^2_{t} (c)$ is discrete and we may add it to $\mathcal{E}$. Our assumption $t_0 \notin \mathcal{E}$ implies that 
$J(P_{t_0})= K(P_{t_0})$ is real, totally disconnected and that all critical points $\tilde{c}_j(t_0)$ escape. 
The latter property implies the $J$-stability of the family $(P_t)$ in a neighborhood of $t_0$, and
this concludes the proof of the theorem.

\begin{proof}[Proof of Lemma~\ref{lem:ouf}]
Recall that $t_0$ is a limit of (real) parameters $t_n \in \mathcal{D}$ for which all critical points are real and all critical points $\tilde{c}_j(t_n)$ escape.
More precisely, we know that  the convex hull $I_t$ of $J(P_t)$ is a segment  containing all critical points, and $P_{t_n} (\tilde{c}_j(t_n)) \notin I_{t_n}$ for all $j$. 

By continuity, either this property remains true for $P_{t_0}$ and $J(P_{t_0})$ is totally disconnected; 
or an image of one critical point $\tilde{c}_j(t_0)$ is equal to a boundary point of $I_{t_0}$.
\end{proof}


\section{Algebraic dynamical pairs}\label{sec:algebraic dynamical pairs}


\subsection{Algebraic dynamical pairs} 

Fix any algebraically closed field $K$ of characteristic $0$.
Assume $V$ is an irreducible affine $K$-variety and fix an integer $d\geq2$.

\begin{definition}\index{dynamical pair!algebraic}
 An (algebraic) \emph{dynamical pair} $(P,a)$ of degree $d$ parametrized by $V$ is an algebraic family $P:V\times\A_K^1\to V\times\A_K^1$ of degree $d$ polynomials together with a marked point $a\in K[V]$. 
\end{definition}

Given a dynamical pair, we let $\preper(P,a)$ be the subset of parameters $t\in V(K)$ for which $a(t)$ is preperiodic for $P_t$. It is a countable union of algebraic subvarieties.

\smallskip

One defines the notion of isotriviality of an algebraic family of map and of a dynamical pair exactly as in the holomorphic case. 

\begin{definition}\index{dynamical pair!active}
We say that a dynamical pair $(P,a)$ parametrized by an irreducible algebraic variety $V$ is \emph{active} if the set $\preper(P,a)$ is a proper Zariski dense
subset of $V$.
Otherwise, we say the pair $(P,a)$ is \emph{passive}.
\end{definition}

Suppose that $(P,a)$ is an algebraic dynamical pair parametrized by a complete algebraic curve $C$. 

Since any regular function on $C$ is a constant, it follows that the pair $(P,a)$ is constant. 
If $(P,a)$ is active, the curve $C$ is necessarily affine. 
Observe also that  $(P,a)$ is active when $\preper(P,a)$ is infinite countable, Zariski dense, 
and its complement is non-empty. We shall see below that 
in this case $C(K) \setminus \preper(P,a)$ is actually Zariski dense too, see Remark~\ref{rem:zar dense}.
When $C$ is not irreducible, $(P,a)$ is active when its restriction to any irreducible component is. 

We shall see below a characterization of passive dynamical pairs that was obtained by DeMarco (see Theorem~\ref{thm:demarco-stab} from the Introduction).


\subsection{The divisor of a dynamical pair} \label{sec:div dyn pair}
Let $C$ be any irreducible affine curve defined over an algebraically closed field $K$ of characteristic $0$.
We let $\bar{C}$ be any projective compactification of $C$ such that $\bar{C}\setminus C$ is a finite set of smooth points
on $\bar{C}$. Let $\mathsf{n}\colon\hat C\to\bar{C}$ be the normalization map.
A branch ${\mathfrak c}$ of $C$ is by definition a (closed) point in $\hat{C}$.\footnote{note that  this definition is consistent with our definition of adelic branch in \S\ref{sec:adelicseries}} 
We say the branch lies at infinity if 
 ${\mathfrak c}$ belong to $\mathsf{n}^{-1}(\hat{C}\setminus C)$. We may (and shall) identify the set of branches at infinity with $\bar{C}\setminus C$.

\begin{lemma}\label{lem:divisor}
Let $(P,a)$ be any dynamical pair parametrized by $C$. For any branch $\mathfrak{c}$ of $C$, the sequence
$- \frac{1}{d^n}\min\{\mathrm{ord}_{{\mathfrak c}}(P^n(a)),0\}$ converges to a non-negative rational number $q_{\mathfrak{c}} (P,a)\in \Q_+$.

Moreover, when $q_{\mathfrak{c}} (P,a)>0$, then we have $q_{\mathfrak{c}} (P,a) = - \frac{1}{d^n} \mathrm{ord}_{{\mathfrak c}}(P^n(a))$ for all $n$ large enough.
And when $q_{\mathfrak{c}} (P,a)=0$, then the sequence $- \min\{\mathrm{ord}_{{\mathfrak c}}(P^n(a)),0\}$ is in fact bounded. 
\end{lemma}

\begin{proof}
Consider the local ring $R=\mathcal{O}_{\hat{C},\mathfrak{c}}$ endowed with the unique $\mathfrak{m}_{\hat{C},\mathfrak{c}}$-norm. 
In other words, one choose a (formal) parametrization $t \mapsto \theta(t)$ of $\hat{C}$ at $\mathfrak{c}$ and set 
$|\phi| = \exp(- \ord_t (\phi \circ  \theta))$ for any $\phi \in R$. 
The completion $\hat{R}$ of $(R, |\cdot|)$ is isomorphic to $K[[t]]$ with its usual $t$-adic norm, and we denote by $L$ the fraction field of $\hat{R}$.

The pair $(P,a)$ induces a polynomial $P_\mathfrak{c} \in L[z]$ and a point $a \in L$. 
Unwinding definitions, we get
\[ 
q_n:=
\frac1{d^n}
\log^+ |P^n_\mathfrak{c}(a)|_L =  \max \{ 0, -\ord_t (P^n \circ a \circ \mathsf{n})\} 
= 
- \frac{1}{d^n}\min\{\mathrm{ord}_{{\mathfrak c}}(P^n(a)),0\}
~.\]
It follows from \S \ref{sec:green} that $q_n$ converges which justifies the existence of $q_{\mathfrak{c}} (P,a)$. 
More precisely, when $q_{\mathfrak{c}} (P,a)>0$ then for $n$ sufficiently large
 we can apply Proposition~\ref{prop:classic-estim} (2) to $z = P^n_\mathfrak{c}(a)$
and we infer $q_{\mathfrak{c}} (P,a) = \frac1{d^n}
\log^+ |P^n_\mathfrak{c}(a)|_L \in \Q_+$.

When $q_{\mathfrak{c}} (P,a)=0$, then the point $a$ lies in the filled-in Julia set of $P_\mathfrak{c}$
and $|P^n_\mathfrak{c}(a)|_L$ is bounded. This implies the lemma. 
\end{proof}

Since $P$ and $a$ are determined by regular functions on $C$, it follows that $q_{\mathfrak{c}} (P,a) =0$ for all branch $\mathfrak{c}$ mapped 
by  $\mathsf{n}$ to a point in $C$.
By the previous lemma, we can set the following
\begin{definition}[The divisor of a dynamical pair]\label{def:divisor dynamical}
The divisor of the dynamical pair $(P,a)$ is by definition 
\[\mathsf{D}_{P,a}:= \sum_{\mathfrak{c} \in \bar{C}} q_{\mathfrak{c}} (P,a) [\mathfrak{c}]~. \]
It is an effective rational divisor on $\bar{C}$ whose support is included in the set of branches at infinity of $C$.
\end{definition}\index{dynamical pair!divisor of}


\subsection{Meromorphic dynamical pairs parametrized by the punctured disk}

In order to relate Green functions to the divisor defined above we need to do a short d\'etour through
meromorphic family and review the main result of~\cite{continuity}.

Let us fix any complete metrized field $(K,|\cdot|)$ of characteristic zero. Let us introduce the ring $\mathbb{M}_K$
of analytic functions on the punctured unit disk $\D_K^*(0,1)$ that are meromorphic at $0$. 
When  $K$ is non-Archimedean, $\mathbb{M}_K$ is equal to the set of Laurent series
$\sum a_n T^n$ such that $a_n =0$ for all $n$ sufficiently negative, and $\sup |a_n| < \infty$.

By convention, a meromorphic family of degree $d$ is a polynomial $P \in \mathbb{M}_K[z]$
such that $P_t$ is a polynomial with coefficients in $K$ of degree $d$ for all $t \neq 0$. 
In other words, we assume the leading coefficient of $P$ to be invertible in $ \mathbb{M}_K$
(i.e. to have no zero on the punctured disk).

\begin{theorem}\label{thm:main FG}
For any meromorphic family $P \in \mathbb{M}_K[z]$ of polynomials of degree $d\ge2$ and for any function $a (t) \in \mathbb{M}_K$,
there exists a nonnegative rational number $q(P,a) \in \Q_+$ such that the function
\[\hat{g}(t):= g_{P_t}(a(t)) - q(P,a)\, \log|t|^{-1}\]
on $\D_K^*(0,1)$ extends continuously across the origin. Moreover, one of the three possibilities following occurs.
\begin{enumerate}
\item
There exists an affine change of coordinates depending analytically on $t$ conjugating $P_t$ to $Q_t$ such that the function $a$ and the family $Q$ are analytic, $\deg(Q_0) = d$, 
and the constant $q(P,a)$ vanishes.
\item
The constant $q(P,a)$ is strictly positive and $\hat{g}$ is harmonic in a neighborhood of $0$. 
\item
The constant $q(P,a)$ vanishes and $\hat{g}(0) = 0$.
\end{enumerate}
\end{theorem}

The proof of the theorem relies on delicate estimates inspired from a work by Ghioca and Ye~\cite{Ghioca-Ye-Cubic}.  We shall refrain from giving a proof of
this result. We shall prove though (see Proposition~\ref{prop:behave green}) the weaker statement that there exists  a non-negative rational number such that the function $\hat{g}$ as in the theorem is subharmonic 
and $\Delta \hat{g}$ has no mass at the origin. 
To that end, we shall need the next lemma.
\begin{lemma}\label{lem:base null}
For any meromorphic family $P \in \mathbb{M}_K[z]$ of polynomials of degree $d\ge2$ there exists a positive constant $C>0$ such that
\begin{equation}\label{eq:basic}
\left|
\frac1d \, \log \max \{ 1, |P_t(z)|\} - \log \max \{ 1, |z|\} 
\right|
\le C \, \log|t|^{-1}
\end{equation}
for all $t \in \D^*_K(0,\frac12)$ and for all $z$ in the affine plane. 
\end{lemma}

This basic lemma holds for any endomorphism of the projective space in any dimension, see~\cite[Lemma~3.3]{demarco} and~\cite[Proposition~4.4]{favre}.

\begin{proof}
Observe that for any $a \in \mathbb{M}_K$, there exist constant $A>1$, $\alpha \in \N^*$ such that
$ |a(t)| \le A |t|^{-\alpha}$ for all $t \in \D^*_K(0,\frac12)$, and that we may further assume
$|a(t)| \ge A^{-1} |t|^{\alpha}$ when $a \in \mathbb{M}^*_K$.

Since $P_t(z) = a_0(t) z^d + \cdots + a_d(t)$ for some $a_i \in \mathbb{M}_K$, we get
$|P_t(z)| \le B |t|^{-\beta} \max \{ 1, |z|^d\} $ for some $B,\beta>0$ which implies the upper bound
\[
\frac1d \, \log \max \{ 1, |P_t(z)|\} \le \log \max \{ 1, |z|\} 
+ C \, \log|t|^{-1}~.
\]
For the lower bound 
we write
\[
P_t(z) = a_0(t) z^d \left( 1 + \frac{b_1(t)}{z} + \cdots + \frac{b_d(t)}{z^d}
\right)
\]
where $a_0 \in \mathbb{M}^*_K$, and $b_i \in \mathbb{M}_K$. 
When $|z|\ge 3 \max \{ |b_i(t)|, 1\}$, we have
\[
|P_t(z)| \ge
\frac12 A|t|^\alpha |z|^d~, 
\]
hence 
\begin{equation}\label{eq:basic001}
\frac1d \, \log \max \{ 1, |P_t(z)|\} \ge \log \max \{ 1, |z|\} 
- C \, \log|t|^{-1}~,
\end{equation}
for some $C>0$, 
whereas when $|z|\le 3 \max \{ |b_i(t)|, 1\}$ we may increase $C>0$ so that 
\[ \log \max \{ 1, |z|\} \le C \, \log|t|^{-1}\]
and~\eqref{eq:basic001} holds trivially.
\end{proof}


\subsection{Metrizations and dynamical pairs}

Let $(K,|\cdot|)$ be  any algebraically closed complete metrized field of characteristic zero. We fix
any dynamical pair $(P,a)$ parametrized by an affine curve $C$ defined over $K$. 
Set $g_{P,a}(t):=g_{P_t}(a(t))$ for all $t\in C^{\an}$. In this section, we relate the behaviour of $g_{P,a}$ near the branches
at infinity of $C$ to the divisor $\mathsf{D}_{P,a}$.

By a local parametrization of $\mathfrak{c}$, we mean an analytic map from the open unit disk $\D_K(0,1)$ to $\hat{C}^{\an}$
which is an isomorphism onto its image and sends the origin to $\mathfrak{c}$.

\begin{proposition}\label{prop:behave green}
Let $(P,a)$ be any dynamical pair parametrized by an affine curve $C$. 
Then $g_{P,a}$ defines a non-negative continuous subharmonic function on $C^{\an}$. 

Moreover, for any branch $\mathfrak{c}$ of $C$ at infinity, and for any local parametrization of a punctured neighborhood of $\mathfrak{c}$
one can write
\[
g_{P,a}(t) = q_\mathfrak{c}(P,a) \, \log |t|^{-1} + \tilde{g}(t)
\]
where $\tilde{g}$ is a continuous and subharmonic. 
\end{proposition}
\begin{proof}
By Proposition~\ref{prop:basic-green} (5), $g_{P,a}$ is locally a uniform limit of continuous and subharmonic functions, hence it is continuous and subharmonic. 
From Theorem~\ref{thm:main FG}, we may write $g_{P,a}(t) = q(P,a) \, \log |t|^{-1} + \tilde{g}(t)$ with $\tilde{g}$ continuous, and Lemma~\ref{lem:base null}
implies
\[ 
\left| g_{P,a}(t) - \frac1{d^n} \log^+|P_t^n(a(t))| \right|
\le \frac{C}{d^n} \log|t|^{-1}
\]
for some positive constant $C$.
Since 
\[ 
\frac1{d^n} \log^+|P_t^n(a(t)| = - \frac{1}{d^n}\min\{\mathrm{ord}_{{\mathfrak c}}(P^n(a)),0\} \log|t|^{-1} + \tilde{g}_n(t)
\]
with $\tilde{g}_n$ continuous, we deduce that 
\[ 
q_\mathfrak{c}(P,a) = \lim_{n\to\infty} - \frac{1}{d^n}\min\{\mathrm{ord}_{{\mathfrak c}}(P^n(a)),0\}  = q(P,a)\]
as required. 
\end{proof}

Recall how a function on $C$ can define a metrization on a suitable line bundle on $\bar{C}$. 
Let $\mathsf{D}$ be any integral divisor on $\bar{C}$ supported on the set of branches at infinity of $C$, and let $g \colon C \to \R$ be any continuous function. 

Any local section $\sigma \in \mathcal{O}_{\bar{C}}(\mathsf{D})(U)$ on an open subset $U$ of $\bar{C}$ is determined by a rational function on $\bar{C}$ satisfying 
$\ord_\mathfrak{c}(\sigma) + \ord_\mathfrak{c}(\mathsf{D}) \ge 0$ for all $\mathfrak{c} \in U$. 
We say that $g$ defines a continuous metrization $|\cdot|_g$ on $ \mathcal{O}_{\bar{C}}(\mathsf{D})$ whenever the continuous function $t \mapsto |\sigma(t)| e^{-g(t)}$
defined on $U \setminus \supp (\mathsf{D})$
extends continuously to $U$ for any local section $\sigma$.

This condition can be rephrased as follows. 
Pick any branch at infinity $\mathfrak{c}$, and any local parametrization $\theta \colon \D_K(0,1) \to \bar{C}$ mapping $0$ to $\mathfrak{c}$. 
This parametrization induces a meromorphic map $N_\mathfrak{c} \colon \D_K(0,1) \to C$ which is analytic on $\D^*_K(0,1)$. Through the map $\sigma \mapsto \sigma \circ\theta$, sections defined in a local (analytic) neighborhood of $\mathfrak{c}$
are identified with meromorphic functions with at most one pole at $0$ of order $\le \ord_{\mathfrak{c}}(\mathsf{D})$. 
It follows that $g$ defines a continuous metrization on $\mathcal{O}_{\bar{C}}(\mathsf{D})(U)$ iff one can write 
\[g \circ N_\mathfrak{c}(t) = \ord_{\mathfrak{c}}(\mathsf{D}) \log|t|^{-1} + g_\mathfrak{c}(t)~,\]
for some \emph{continuous} function $ g_\mathfrak{c}$.

One can translate Proposition~\ref{prop:behave green}  into the language of metrizations as follows.

\begin{theorem}\label{thm:continuity}
The function $g_{P,a}$ is continuous and subharmonic on $C^{\mathrm{an}}$. 
It  induces a continuous semi-positive metrization on the ($\Q$-)line bundle $\mathcal{O}_{\bar{C}}(\mathsf{D}_{P,a})$.
In particular, we have
\[\int_{C^{\an}}\Delta g_{P,a}=\deg(\mathsf{D}_{P,a})~.\]
\end{theorem}\index{metric!induced by a dynamical pair}

We call $\mu_{P,a} = \Delta g_{P,a}$ the bifurcation measure of the pair $(P,a)$. It is a positive measure on $\bar{C}^{\an}$
of total mass $\deg(\mathsf{D}_{P,a})$ which has no atoms.


\subsection{Characterization of passivity}

The next result is well-known and was proved by DeMarco~\cite{demarco}
in the more general case of families of rational maps. We include a proof for sake of completeness.

For any non-constant rational function $f \in K(C)$, let $\deg(f) \in \N^*$ be the number of poles (or zeroes) counted
with multiplicities. 

Also given any family $P$ parametrized by $C$, we say that a marked point $a$ is \emph{stably preperiodic} if 
there exists $n>m\geq0$ such that $P_t^n(a(t))=P_t^m(a(t))$ for all $t\in C$.

\begin{theorem}\label{thm:active-characterization1}
Let $(P,a)$ be a dynamical pair of degree $d\ge 2$ parametrized by an affine irreducible curve $C$ 
defined over an algebraically closed field $K$ of characteristic $0$.
Then the following assertions are equivalent:
\begin{itemize}
\item[(1)] the pair $(P,a)$ is passive on $C$,
\item[(2)]  the pair $(P,a)$ is either isotrivial or the marked point is stably preperiodic,
\item[(3)] the divisor $\mathsf{D}_{P,a}$ of the pair $(P,a)$ vanishes,
\item[(4)] there exists a constant $M>0$ such that $\deg(P_t^n(a(t)))\leq M$ for all $n\geq1$.
\end{itemize}
When $K$ is a complete metrized field, these assertions are equivalent to: 
\begin{itemize}
\item[(5)] the bifurcation measure $\mu_{P,a}$ vanishes.
\end{itemize}
When $K = \C$, this is further equivalent to: 
\begin{itemize}
\item[(6)] $\mathrm{Bif}(P,a)$ is empty.
\end{itemize}
\end{theorem}

Observe that this implies all characterizations stated in Theorem~\ref{thm:demarco-stab} except for (5) which will be dealt with in \S\ref{sec:arithmetic dynamical pair}.

\begin{proof}

The implication (2) $\Rightarrow$ (3) is easy. When $(P,a)$ is isotrivial, then $t \mapsto P^n(a(t))$ is constant for all $n$, 
and when $a$ is stably preperiodic, then we have $P^n(a) = P^m(a)$ for some $n>m$. In both cases, 
$\ord_\mathfrak{c} (P^n(a))$ is bounded for all branches of $C$ at infinity.

Assume (4). Since $\deg(P^n(a)) = \sum_{\mathfrak{c}} \ord_\mathfrak{c} (P^n(a))$, we the sum ranges over all branches at infinity of $C$, the sequence $ \ord_\mathfrak{c} (P^n(a))$ is bounded for all
branches of $C$ at infinity therefore (3) holds. Conversely, if (3) holds, then Lemma~\ref{lem:divisor} implies that $ \ord_\mathfrak{c} (P^n(a))$ is bounded for all
branches of $C$ at infinity so that $\deg(P^n(a))$ is bounded. This shows (3) $\Leftrightarrow$ (4).

Suppose (4) holds and $K$ is a metrized field. Then by Theorem~\ref{thm:continuity} we get that $g_{P,a}$ is a continuous semi-positive
metrization on the trivial line bundle on $\hat{C}$ which implies $g_{P,a}$ to be constant, and $\mu_{P,a}$ to be equal to $0$. 
This shows (4) $\Rightarrow$ (5). When $K= \C$, the previous argument and Proposition~\ref{prop:bifurcation} shows (4) $\Rightarrow$ (6). 
Conversely when either (5) or (6) holds, the divisor $\mathsf{D}_{P,a} =0$ by Theorem~\ref{thm:continuity} which implies (3).

Suppose that $\deg(P^n(a)) \to \infty$. We shall prove that (1) cannot hold.  
Since the family is algebraic we may replace $K$ by an algebraically closed field which is finitely generated over $\bar{\Q}$, 
and fix an embedding $K \subset \C$. Let $C^\an$ be the Riemann surface defined by this embedding. 
By what precedes, the bifurcation measure $\mu_{P,a}$ is non-zero, hence 
$\preper(P,a)$ is dense (for the complex topology) in $\mathrm{Bif}(P,a)$. Since  $\mathrm{Bif}(P,a)$ is the support of the Laplacian
of a continuous subharmonic function, it is Zariski-dense and we conclude that $\preper(P,a)$ is Zariski-dense in $C^{\an}$. 
But $\preper(P,a)$ is a subset of $C(K)$, hence it is also Zariski-dense in $C$ (viewed as a curve over $K$).
We claim that $C(K) \setminus \preper(P,a)$ is non-empty showing that $(P,a)$ is active. 
In $C^{\an}$ the complement of $\mathrm{Bif}(P,a)$ is the stability locus which is open and dense, see Remark~\ref{rem:stab-dense}.
Take any connected component $U$ of the stability locus. Then $a$ cannot be stably periodic on $U$, otherwise it would be stably periodic
on $C$. The set of $t\in U$ for which $a$ is not preperiodic is thus discrete by Theorem~\ref{thm:discrete preper}. Since $K$ is algebraically closed, $C(K)$ is dense in $C^{\an}$, and
we get parameters $t \in C(K) \cap U$ which do not pertain to $\preper(P,a)$.
We have proved (1) $\Rightarrow$ (4). 

Suppose finally that $\deg(P^n(a))$ is bounded. As in the previous argument, we may suppose that $K\subset \C$ so that the bifurcation locus of the pair is empty. By Lemma~\ref{lm:suppT}, there exists no properly prerepelling parameter. When $a$ is stably preperiodic, then the dynamical pair is passive. 
When it is not stably preperiodic, we fix any $t_0 \in C$. Conjugating $P_{t_0}$ if necessary, and replacing $P$ by a suitable iterate, we may suppose that $0$ and $1$ are 
fixed and repelling. By base change, we may also suppose that the fixed points $0$ and $1$ can be followed over $C$: we get two regular functions 
$p_0, p_1 \colon C \to \A^1$ such that $P_t(p_i(t)) = p_i(t)$ for all $t$ and $p_0(t_0)= 0$, and $p_1(t_0) =1$.
Let $C^*$ be the open Zariski dense subset of $C$
where $p_0\neq p_1$. Replacing $\phi_t \circ P_t \circ \phi_t^{-1}$ with  $\phi_t(z) = \frac{z - p_0}{p_1 - p_0}$, we 
get a family $(P_t)_{t \in C^*}$ such that $0$ and $1$ are fixed and repelling for $P_{t_0}$. It follows that 
$P^n(a)$ forms a sequence of regular functions from $C^*$ to $\A^1 \setminus \{0,1\}$. By De Franchis theorem, see e.g.~\cite{Tsushima}, 
there exists only finitely many such non-constant maps  $C^* \to \A^1 \setminus \{0,1\}$, hence $P^n(a)$ is constant
for all $n$ sufficiently large. This implies $t \mapsto P_t(b)$ to be constant for infinitely many $b \in \A^1$ hence $P_t(b)$ is constant for all $b$. 
And we conclude that $(P,P^n(a))$ is isotrivial for some $n\gg1$ hence $(P,a)$ is isotrivial. This concludes the proof (4) $\Rightarrow$ (1). 
\end{proof}

\begin{remark}\label{rem:zar dense}
Observe that our proof of (1) $\Rightarrow$ (4) implies that for any active dynamical pair $(P,a)$ both sets $\preper(P,a)$ and $C(K) \setminus \preper(P,a)$ are Zariski-dense. 
\end{remark}


\section{Family of polynomials and Green functions}

We explain how the results of the previous section applies to study the bifurcation locus of a family of polynomials.
Recall that $G(P)= \max \{ g_P(c), \, P'(c) =0 \}$.
As usual we write $\bar{C}$ is a projective compactification of $C$ such that $\bar{C}\setminus C$ is a finite set of smooth points.
\begin{theorem} \label{thm:the case of a family}
Let $P$ be any family of polynomials parametrized by a curve $C$ defined over a field $K$ of characteristic $0$.
For each branch $\mathfrak{c}$ of $C$ at infinity, there exists a constant $q_\mathfrak{c}(P)\in \Q_+$ such that for any norm $|\cdot|$ on $K$, 
and for any analytic parametrization $t \mapsto \theta(t)$ of a neighborhood of $\mathfrak{c}$ in $\hat{C}^\an$ we have
\[
G(P_{\theta(t)}) = q_\mathfrak{c}(P) \log |t|^{-1} + \tilde{G}(t)
\]
where $\tilde{G}$ extends continuously through the origin.
\end{theorem}
\begin{remark}
This result is completely analogous to~\cite[Corollary~1]{continuity} where $G(P)= \max \{ g_P(c), \, P'(c) =0 \}$ is replaced by 
the Lyapunov exponent of $P$ with respect to the equilibrium measure which satisfies the Misiurewicz-Prytycky's formula
$\lambda (P) = \log |d| + \sum_{P'(c) =0} g_P(c)$, see~\eqref{eq:MPform}. 
\end{remark}

\begin{proof}

Fix a branch and do a finite cover $\bar{D} \to \bar{C}$ to follow all critical points $c_1, \ldots, c_{d-1}$.
Apply Proposition~\ref{prop:behave green} to each $c_i$. Then 
\[g_{P,c_i}(t)= q_\mathfrak{c}(P,c_i) \log |t|^{-1} + \tilde{g}_i(t)\]
and take the max
Then $q_\mathfrak{c}(P) =\max_i \{ q_\mathfrak{c}(P,c_i)\}$ and $\tilde{G} = \max \tilde{g}_i$ over all $i$ such that $q_\mathfrak{c}(P) = q_\mathfrak{c}(P,c_i)$.
\end{proof}

Assume $P$ is parametrized by a curve $C$ which is defined over a complete and algebraically closed field $(K,|\cdot|)$ of characteristic $0$ and that $P$ is critically marked, i.e. there exist $c_1,\ldots,c_{d-1}\in K[C]$ that follow the critical points of $P$.

\begin{definition}
The bifurcation measure\index{bifurcation!measure (of a polynomial family)} of the family $P$ is the probability measure $\mu_{\mathrm{bif}}$ on $C^\mathrm{an}$ which is proportional to the positive measure $\sum_{i=1}^{d-1}\mu_{P,c_i}$.
\end{definition}


\section{Arithmetic polynomial dynamical pairs} \label{sec:arithmetic dynamical pair}

Assume $C$ is an algebraic curve defined over a 
number field $\KK$, and $(P,a)$ is an algebraic dynamical pair parametrized by $C$ and defined over $\KK$. 
Recall that $\bar{C}$ is a projective compactification of $C$ such that $\bar{C}\setminus C$ is a finite set of smooth points. 

If $t \in C(\bar{\KK})$, then $P_t$ is defined over a number field, and we attach to it a canonical height $h_t$, 
see~\S \ref{sec:global}.

\paragraph*{The canonical height of a dynamical pair}
Ingram~\cite{ingram2010} proved that the function $t \mapsto   \hat{h}_{P_t}(a(t))$ defines a Weil height on $\bar{C}$. 
We improve his result and show that it is an Arakelov height.\index{height!induced by a dynamical pair}

\begin{proposition}\label{prop:adelic}
Assume $(P,a)$ is  active. There exists a positive integer $n\geq1$ such that $n\cdot \mathsf{D}_{P,a}$ is a positive integral divisor on $\bar{C}$ and the collection of subharmonic functions $\{n\cdot g_{P,a,v}\}_{v\in M_\KK}$
induces a semi-positive adelic metrization on the ample line bundle $\mathcal{L}:=\mathcal{O}_{\bar{C}}(n\cdot \mathsf{D}_{P,a})$.
The induced height function $h_{\bar{\mathcal{L}}}$ on $\bar{C}$ satisfies
\begin{enumerate}
\item $h_{\bar{\mathcal{L}}}(t)=n\cdot \hat{h}_{P_t}(a(t))$ for all $t\in C(\bar{\KK})$, i.e.
\[h_{\bar{\mathcal{L}}}(t)=\frac{n}{\deg(t)}\sum_{v\in M_\KK}\sum_{t'\in\mathsf{O}(t)}g_{P_{t'},v}(a(t')),\]
where $\mathsf{O}(t)$ is the $\mathrm{Gal}(\bar{\KK}/\KK)$-orbit of $t$ and $\deg(t)=\mathrm{Card}(\mathsf{O}(t))$;
\item $h_{\bar{\mathcal{L}}}(t)=0$ if and only if $t\in\mathrm{Preper}(P,a)$;
\item the global height of $\bar{C}$ is equal to $h_{\bar{\mathcal{L}}}(\bar{C})=0$;
\item at any place $v\in M_\KK$, the associated measure $c_1(\bar{\mathcal{L}})_v$ on $\bar{C}^{\an}_v$ is $n\cdot\Delta g_{P,a,v}$.
\end{enumerate}
\end{proposition}

\begin{proof}
Let $n$ be any integer such that $n\cdot \mathsf{D}_{P,a}$ has integral coefficients. 
For any place $v \in M_\KK$, Theorem~\ref{thm:continuity} implies that $n\cdot g_{P,a,v}$ defines 
a continuous and semi-positive metrization $|\cdot|_{a,v}$ on $\mathcal{O}_{\bar{C}}(n\cdot \mathsf{D}_{P,a})$. Let us justify that the collection $\{|\cdot|_{a,v}\}_{v\in M_\KK}$ is adelic.

\smallskip

We make some preliminary  comments. 
By a suitable base change over $C$ we may (and shall) assume that the family $\{P_t\}$ is of the form~\eqref{eq:critical form}. To avoid conflict of notation we write
$P_t = P_{c(t), \alpha(t)}$ with $c, \alpha \in \KK[C]$.

We fix an embedding $C \subset \A^N$, so that the completion  $\bar{C}$  of $C$ in $\p^N$ is smooth at infinity. 
For any non-Archimedean place $v$, let $\mathrm{red}_v \colon \p^N(\C_v) \to \p^N(\widetilde{\C_v})$
be the canonical reduction map, defined by sending a point $[z_0: \cdots : z_N]$ to
$[\widetilde{\lambda z_0}: \cdots  : \widetilde{\lambda  z_N}]$ where $\lambda = \max \{ |z_0|, \cdots , |z_N| \}^{-1}$.
We let $\tilde{C} \subset \p^N(\widetilde{\C_v})$ 
be the image of $C$ under $\mathrm{red}_v$.
We shall make a systematic use of the following  results.
\begin{lemma}\label{lem:reduc}
Let $f \in \KK(C)$. Then for all but finitely many places $v\in M_\KK$, we have
\[\{t \in C^{\an}_v,\,  |f(t)| < 1 \} = \mathrm{red}_v^{-1} (\mathrm{red}_v( f^{-1}(0)))~.
\]
\end{lemma}
Observe that replacing $f$ by its inverse implies $\{t \in C^{\an}_v,\,  |f(t)| > 1 \} = \mathrm{red}_v^{-1} (\mathrm{red}_v( f^{-1}(\infty)))$.
\begin{lemma}\label{lem:v-adic loja}
Let $f,g   \in \KK[C]$ such that $-\ord_\mathfrak{c} (f) \ge -\ord_\mathfrak{c} (g)$ for all branches at infinity of $C$.
Then for all but finitely many places we have
$\max \{ 1, |f|_v\} \ge |g|_v $.
\end{lemma}
During the proof $S$ will be a set of places on $\KK$ containing all archimedean 
ones which may vary from line to line but which shall remain finite.

\smallskip

Let us now prove that $\{|\cdot|_{a,v}\}_{v\in M_\KK}$ is adelic.
Denote by $\mathcal{B}$ the set of branches at infinity that do not lie in the support of $\mathsf{D}_{P,a}$. 

Suppose first that $\mathcal{B} = \varnothing$. By Lemma~\ref{lem:divisor}, we can find a sufficiently large integer 
 $q \in \N^*$ such that
\[-d^q \cdot \ord_\mathfrak{c} (\mathsf{D}_{P,a}) > - \max \{ \ord_\mathfrak{c} (c), \ord_\mathfrak{c} (\alpha) \}
\text{ and } \ord_\mathfrak{c} (\mathsf{D}_{P,a}) = -\frac{1}{d^q} \mathrm{ord}_{{\mathfrak c}}(P^q(a))\]
for all branches at infinity $\mathfrak{c}$. 
It follows from Lemma~\ref{lem:v-adic loja} that 
\[
\max \{ 1, |P^q_t(a(t))|_v\} \ge \max \{ |c(t)|_v, |\alpha(t)|_v\}~,  
\]
for all $v \notin S$, so that Proposition~\ref{prop:classic-estim} yields
$g_{P,a,v}(t) = \frac1{d^q} g_{P,v}(P^q(a)) =  \frac1{d^q}  \log^+ |P^q(a)|_v$.
The metric is thus adelic by Lemma~\ref{lem:adelic model metric}.

\smallskip

When $\mathcal{B}$ is non-empty, one proceeds as follows.  
Replacing $\mathsf{D}_{P,a}$ by a sufficiently high multiple, we may suppose that it is very ample so that we can find a rational function 
$h \in \KK(C)$ whose divisor of poles is equal to $\mathsf{D}_{P,a}$. 
We need to prove that $g_{P,a,v}= \log^+|h|_v$ for all but finitely many places $v\notin S$. 

As above, we pick a sufficiently large integer 
 $q \in \N^*$ such that
\[-d^q \cdot \ord_\mathfrak{c} (\mathsf{D}_{P,a}) > - \max \{ \ord_\mathfrak{c} (c), \ord_\mathfrak{c} (\alpha) \}
\text{ and } \ord_\mathfrak{c} (\mathsf{D}_{P,a}) = -\frac{1}{d^q} \mathrm{ord}_{{\mathfrak c}}(P^q(a))\]
for all $\mathfrak{c} \in \supp(\mathsf{D}_{P,a})$ (observe that the latter set is now strictly included in the set of branches at infinity of $C$).

Since any pole of $h$ is a pole of $P^q(a)$, Lemma~\ref{lem:reduc} shows that 
\[U:= \{ |h|_v >1 \} \subset \{ |P^q(a)|_v >1\}~.\] Similarly
the function $h^{d^q}/P^q(a)$ does not vanish at any point in $\supp(\mathsf{D}_{P,a})$, hence we have
$|h|_v^{d^q} = |P^q(a)|_v$ on $U$.

We fix an auxiliary rational function $\eta \in \KK(C)$ whose zero locus is equal to the support of 
$\mathsf{D}_{P,a}$, whose set of poles is equal to $\mathcal{B}$ and which satisfies
\[ - \ord_\mathfrak{c} (\eta \times P^q(a)) \ge \max \{ - \ord_\mathfrak{c}(c), -\ord_\mathfrak{c}(\alpha) \}\]
for all branches at infinity (one may need to increase $q$). 

Applying again Lemma~\ref{lem:reduc}, we get $\{ |\eta|_v<1 \} = U$. On the other hand
Lemma~\ref{lem:v-adic loja} implies
\[
\max \{ 1, |P^q_t(a(t))|_v\times |\eta(t)|\} \ge \max \{ |c(t)|_v, |\alpha(t)|_v\}~,  
\]
so that $|P_t^q(a(t))| >  \max \{ |c(t)|_v, |\alpha(t)|_v\}$ for all $t \in U$. 
We infer by Proposition~\ref{prop:classic-estim} that 
\[g_{P,a,v} = \frac1{d^q} \log|P^q(a)|_v=\log|h|_v\] on $U$. 

Observe now that $g_{P,a,v}$ extends to $C_0:= \bar{C} \setminus \supp(\mathsf{D}_{P,a})$ as a non-negative continuous subharmonic function by Proposition~\ref{prop:behave green}. The set $\{ |h| \le 1\}$ is compact in $C_0^{\an}$ and its boundary
is equal to  $\partial U$, so that $g_{P,a,v} \equiv 0$ on $\partial \{ |h| \le 1\}$. 
Note that the latter set consists of finitely many type 2 points corresponding to the irreducible components of the reduction of $C$ containing a pole of $\tilde{h}$ (see the discussion on models in \S\ref{sec:nonArchcurves}). 
We may now apply the maximum principle to any connected component of the interior of $\{ |h| \le 1\}$, and we obtain
that $g_{P,a,v} \equiv 0$ on $\{ |h| \le 1\}$. 

This concludes the proof that  $g_{P,a,v} =\log|h|_v$ everywhere. 

\medskip

Properties (1), (2) and (4) follow from the definitions. 
To prove (3), we use~\cite[(1.2.6) \& (1.3.10)]{ACL2}.
Choose any two meromorphic functions $\phi_0, \phi_1$ on $\hat{C}$ such that 
$\mathrm{div}(\phi_0)+ n\mathsf{D}_{P,a}$ and $\mathrm{div}(\phi_1)+ n\mathsf{D}_{P,a}$ are both effective with disjoint support included in $C$.  Let $\sigma_0$ and $\sigma_1$ be the associated sections of $\mathcal{O}_{\hat{C}}(n\mathsf{D}_{P,a})$. Let $\sum n_i [t_i]$ be the divisor of zeroes of $\sigma_0$, and
$\sum n'_j [t'_j]$ be the divisor of zeroes of $\sigma_1$.
Then 
\begin{eqnarray*}
h_{\bar{\mathcal{L}}}(\hat{C}) 
& = &
\sum_{v\in M_\KK} (\widehat{\mathrm{div}}(\sigma_0)\cdot \widehat{\mathrm{div}}(\sigma_1) | \hat{C})_v
\\
& = &
\sum_i n_i \cdot n\cdot \hat{h}_{P_{t_i}}(a(t_i)) - 
\sum_{v\in M_\KK} \int_{\bar{C}} \log| \sigma_0|_{a,v}\,  \Delta (n\cdot g_{P,a,v})
\\
& = &
\sum_{v\in M_\KK} \int_{\bar{C}} n\cdot g_{P,a,v}\,  \Delta (n\cdot g_{P,a,v}) \ge0~,
\end{eqnarray*}
where the third equality follows from Poincar\'e-Lelong formula and writing 
$\log| \sigma_0|_{a,v} = \log |\phi_0|_v - n\cdot g_{P,a,v}$.

\smallskip

Pick any archimedean place $v_0$. The total mass on $\hat{C}$ of the positive measure $\Delta g_{P,a,v_0}$ is the degree of $\mathcal{L}$, hence is non-zero. It follows from e.g.~\cite[Lemma 2.3]{favredujardin} that any point $t_0$ in the support of $\Delta g_{P,a,v_0}$ is accumulated by parameters $t_* \in C(\bar{\KK})$ such that $P_{t_*}^n(a(t_*))=P_{t_*}^m(a(t_*))$
 for some $n>m\geq0$. For any such point, the point (1) implies
$h_{\bar{\mathcal{L}}}(t_*)=0$.  The result follows from Theorem~\ref{thm:HSA}.
\end{proof} 

\begin{proof}[Proof of Lemma~\ref{lem:reduc}]
In some affine chart,  $C$ is given by the vanishing of polynomials with coefficients in $\cO_{\KK,S}$
and $f$ is the restriction of a polynomial with coefficients in the same ring of integers.
Enlarging $S$ if necessary, for any place $v \notin S$, the reduction $\tilde{C}_v$ of $C$ over $\C_v$ is an affine curve over $\widetilde{\C}_v$, and
$f$ induces a regular function $\tilde{f}$ on $\tilde{C}_v$. We have the following commutative diagram:
\begin{displaymath}
    \xymatrix{ 
    C(\C_v) \ar[r]^-{\mathrm{red}} \ar[d]^{f}
    &    \tilde{C}_v \ar[d]^{\tilde{f}}
    \\
    \p^1(\C_v)  \ar[r]^-{\mathrm{red}} & \p^1({\widetilde{\C_v}})
               }
\end{displaymath}
from which the result follows.
\end{proof}

\begin{proof}[Proof of Lemma~\ref{lem:v-adic loja}]
We first embed $\KK$ into $\C$ and argue analytically.
The function field $K(C)$ is finitely generated over $K(f)$, hence we may find rational functions $R_0, \cdots, R_n$ such that 
$ g^n + g^{n-1} R_1(f) + \cdots + g R_{n-1}(f) + R_n(f) = 0$. In fact $R_j (z)$ is the symmetric polynomial of degree $j$ 
of the collection of points $g(w_j)$ where $w_j$ are the roots of $f(w) = z$. Since the set of poles of $g$ is included in the set of poles of $f$, 
all $R_j$ are regular on $\A^1$ and are hence polynomials. By Galois invariance, they 
are defined over a finite extension of $\KK$.
Fix a local coordinate $w$ at a branch at infinity $\mathfrak{c}$. 
We may then write $|f(w)| \asymp |w|^{d_\mathfrak{c}}$ and $|g(w)| \asymp |w|^{c_\mathfrak{c}}$ for some integers
$c_\mathfrak{c} \le d_\mathfrak{c}$ so that 
\[
\log |R_j (z)| \lesssim \left(\sum_{|I| =j} \prod_{i\in I} \frac{c_{\mathfrak{c}_i}}{d_{\mathfrak{c}_i}} \right)\,  \log|z| , \] 
near infinity, and $R_j$ has degree $\le j$. 

Take a place $v\in M_\KK$ for which all coefficients of $R_j$ have norm $\le 1$. 
For each $w \in C$, we get 
$|g(w)|^n \le \max \{ |g^j(w)| \cdot |R_{n-j}(f(w))|
\} \le \max \{ |g^j(w)| \cdot \max \{ 1, |f(w)|\}^{n-j} \}$
hence $|g(w)| \le \max \{ 1, |f(w)|\}$.
\end{proof}

\paragraph*{Canonical dynamical height and Weil heights}
Since $h_{P,a}$ is a height induced by a semi-positive adelic metrization,  it differs from any Weil height attached to the ample line bundle $\mathcal{L}$
by a bounded function, a result that was first proved by Ingram~ \cite{ingram2010}.
We then obtain from Northcott's theorem the following
\begin{corollary}
Let $(P,a)$ be an active dynamical pair parametrized by an affine curve defined over a number field $\KK$. 
For any integer $N$, and any $B >0$, then there exists a constant $C = C(N, B)$
such that the set of parameters $t \in C(\bar{\KK})$ defined over a field of degree at most $\le N$ over $\KK$
for which $h_{P,a}(t) \le B$ is finite. 

In particular $\preper(P,a) \cap C(\KK)$ is finite. 
\end{corollary}

The following characterization of activity of an arithmetic pair then
follows quickly from Theorem~\ref{thm:active-characterization1} (note that the next statement implies Theorem~\ref{thm:demarco-stab}).

\begin{theorem}\label{tm:active-characterization2}
Let $(P,a)$ be a dynamical pair parametrized by an affine irreducible curve $C$ defined over $\KK$.
If the pair $(P,a)$ is not isotrivial, then the following assertions are equivalent:
\begin{enumerate}
\item the pair $(P,a)$ is passive on $C$,
\item the height function $h_{P,a}$ satisfies $h_{P,a}(t)=0$ for all $t\in C(\bar{\KK})$,
\item for any place $v\in M_\KK$, we have $g_{P,a,v}\equiv0$ on $C^{\an}_v$,
\item there exists a place $v\in M_\KK$ such that $g_{P,a,v}\equiv0$ on $C^{\an}_v$,
\item for any place $v\in M_\KK$, we have $\Delta g_{P,a,v}=0$ on $C^{\an}_v$,
\item there exists a place $v\in M_\KK$ such that $\Delta g_{P,a,v}=0$ on $C^{\an}_v$.
\end{enumerate}
\end{theorem}

\begin{proof}
If $(P,a)$ is passive, then Theorem~\ref{thm:active-characterization1} implies that $(P,a)$ is stably preperiodic
so that (2) -- (6) are clearly satisfied. 
The next diagram of implications is also clear: 
\begin{displaymath}
\xymatrix
{
(2) \ar@{=>}[r]
&
(3) \ar@{=>}[r] \ar@/_/@{=>}[dr]
&
(4) \ar@{=>}[r]
&
(6)
\\
&
&
(5)\ar@/_/@{=>}[ur] 
&}
\end{displaymath}
So assume $\Delta g_{P,a,v}=0$ on $C^{\an}_v$ for some $v\in M_\KK$.  Then Statement (5) of Theorem~\ref{thm:active-characterization1} is satisfied, and
we conclude that $(P,a)$ is passive. 
\end{proof}


%
%
%
%
%
%


\chapter{Entanglement of dynamical pairs}

We introduce the notion of entanglement of dynamical pairs and prove that for any two entangled pairs $(P,a)$ and $(Q,b)$
there exist iterates of $P$ and $Q$, that have the same degree, and are intertwined in the sense of \S \ref{sec:intertwining}. 
This result leads to the proof of Theorem~\ref{tm:unlikely} in the number field case. We then deduce Theorem~\ref{tm:unlikely-car0} 
over an arbitrary field of characteristic zero using a specialization argument.


\section{Dynamical entanglement}
 
  
\subsection{Definition} 
We fix a field $K$ of characteristic $0$, and let $\bar{K}$ be an algebraic closure of $K$.
We introduce the following terminology. 
\begin{definition}
Let $(P,a)$ and $(Q,b)$ be two active dynamical pairs parametrized by an affine curve $C$ defined over $K$.  
We say that $(P,a)$ and $(Q,b)$ are dynamically entangled\index{entanglement!of pairs}\index{dynamical pair!entangled} when 
the set of parameters $t \in C(\bar{K})$ for which $a(t)$ and $b(t)$ are preperiodic 
for $P_t$ and $Q_t$ respectively is Zariski dense, i.e. $\preper(P,a,\bar{K}) \cap \preper(Q,b,\bar{K})$ is infinite.
\end{definition}

Being entangled defines an equivalence relation on the set of active dynamical pairs parametrized by a fixed affine curve.

\begin{remark}
When $K$ is algebraically closed, the entanglement is equivalent to say that $\preper(P,a) \cap \preper(Q,b)$ is Zariski dense. 
When $K$ is not algebraically closed, e.g. when $K$ is a number field, $\preper(P,a) \subset C(K)$ can be finite, and
one needs to look at all parameters in $C(\bar{K})$ to test the entanglement. 
\end{remark}

\begin{remark}
The dynamical entanglement is a rigidity property that  links two dynamical pairs.  
Since passive pairs are already rigid, the entanglement property is only relevant for active pairs.
\end{remark}

\begin{remark}
One can extend the notion of entanglement to a dynamical pair defined over a base of arbitrary dimension.
\end{remark}

Recall the definition of intertwining from \S \ref{sec:intertwining} and its geometric characterization  (Theorem~\ref{thm:basic-intertwin}). 
\begin{theorem}\label{thm:basic ex of entanglement}
Let $P, Q$  be any two active families of polynomials of the same degree $d\ge2$ parametrized by an irreducible affine curve $C$. 
Assume that they are intertwined as polynomials with coefficients in $K(C)$, 
and pick any algebraic subvariety $Z \subset \A^1_C \times \A^1_C$ which projects onto each factor and which is invariant by the map $(z,w) \mapsto (P(z),Q(w))$. 

For any $a,b \in K(C)$ such that $(a, b)\in Z$, there exists a Zariski open dense subset $U \subset C$ 
over which $(P,a)$ and $(Q, b)$  define dynamical pairs parametrized by $U$ that are dynamically entangled. 
\end{theorem}

\begin{proof}
Pick any Zariski open set $U \subset C$ such that $a,b$ define regular functions on $U$, and the restriction of the two natural projections 
$Z' =  Z \cap  \A^2_U \to \A^1_U$ are finite. 

Then  $(P,a)$ and $(Q, b)$  define active dynamical pairs
and by definition $\preper(P,a)$ is Zariski dense in $U$. Pick any point $t \in \preper(P,a) \subset U$. 
Then for all $n\in \N$, we have $(P^n(a(t)), Q^n(b(t))) \in Z$. Since the set $\{P^n(a(t))\}_{n\in \N}$ is finite, 
and the second projection $Z_U \to U$ is finite, it follows that the set $\{Q^n(b(t))\}_{n\in \N}$ is also finite
which proves that $\preper(Q,b) \cap U = \preper(P,a) \cap U$ is Zariski-dense.
The two pairs are thus entangled.
\end{proof}

As a corollary of the previous result, one obtains:
\begin{corollary}
Let $(P,a)$ be an active dynamical pair parametrized by an affine curve $C$.
\begin{itemize}
\item
For any $n\in \N$, and for any $g \in \Sigma(P)$, the pairs
$(P,a)$, $(P,g\cdot a)$,  $(g\cdot P, a)$, and $(P, P^n(a))$ are entangled. 
\item
For any polynomial family $Q$ parametrized by $C$ and commuting with $P$, 
the pairs $(P,a)$, $(P,Q(a))$ and $(Q,a)$ are entangled.
\end{itemize}
\end{corollary}

\begin{proof}
When $\Sigma(P)$ is infinite, then we may suppose $P=M_d$ and the result is easy. 
Otherwise the curve $Z= \cup_{g'\in \Sigma(P)} \{ (z, g'\cdot z)\} \subset \A^1\times \A^1$ is invariant
by $(P,P)$ and $(P,g\cdot P)$, hence  $(P,a)$, $(P,g\cdot a)$, and $(g\cdot P, a)$ are entangled by Theorem~\ref{thm:basic ex of entanglement}. 
Also the graph $\{(z,P^n(z))$ is invariant by $(P,P)$ so that $(P,a)$,  and $(P, P^n(a))$ are entangled. 

If $Q$ commutes with $P$, then the graph $\{(z,Q(z))\}$ is $(P,P)$-invariant hence $(P,a)$, and $(P,Q(a))$ are entangled.
Finally if $a(t)$ has a finite orbit for $P$ of length $\le N$, then $Q^n(a(t))$ has also a finite orbit of length $\le N$ for any $n\ge0$, 
hence $a$ is also preperiodic for $Q$. It follows that $(P,a)$ and $(Q,a)$ are entangled.
\end{proof}

  
\subsection{Characterization of entanglement} 

Let us recall the following theorem from the introduction.

\unlikely*

The most interesting implication is  (1) $\Rightarrow$ (3) which is an arithmetic rigidity result. 
In plain words it implies that two entangled pairs define families of intertwined polynomials.
The proof of Theorem~\ref{tm:unlikely} will occupy most of this chapter including all sections from \S \ref{sec:identical measures} to \S \ref{sec:biggest proof}.
For the convenience of the reader, we give an overview of this technically involved proof in the next paragraph. 
The proof of Theorem~\ref{tm:unlikely-car0} is given in \S \ref{sec:specialization}.


\subsection{Overview of the proof of Theorem~\ref{tm:unlikely}} \label{sec:overviewB}

Assume (2), and pick any parameter $t$ for which $\{P_t^n(a(t))\}_{n\in \N}$ is a finite set. 
Since 
\[\pi \circ Q^{nM}(Q^s(b)) = R^n \tau(P^r(a)) = \tau \circ P^{nN}(a)\]
and the two maps $\tau, \pi$ are finite, it follows that  $\{Q_t^n(b(t))\}_{n\in \N}$
is also a finite set, hence $t \in \preper(Q,b)$. This implies (1). 

\smallskip

The proof of the implication (1) $\Rightarrow$ (2) is given in \S \ref{sec:identical measures}. By Proposition~\ref{prop:adelic}, we may apply 
Thuillier-Yuan's theorem (Theorem~\ref{tm:yuan}) to both height functions, which gives equality of bifurcation measures
$n \cdot \mu_{P,a,v} = m\cdot \mu_{Q,b,v}$ at all places (for suitable integers $n$ and $m$).
The equality $n\cdot h_{P,a} = m\cdot h_{Q,b}$ is equivalent to 
the equality of Green functions $n \cdot g_{P,a,v} = m\cdot g_{Q,b,v}$ at all place $v$, which we obtain in Theorem~\ref{thm:identical measures}
using our rigidity result in the parameter space (Theorem~\ref{tm:rigidaffine}). 
 
\smallskip

The implication (2) $\Rightarrow$ (3) is substantially harder. 
We first prove that (2) implies that $\deg(P)$ and $\deg(Q)$ are multiplicatively dependent, see Theorem~\ref{tm:multiplicative}.
We argue by contradiction. We compute the H\"older constant of continuity of the Green function (at a fixed Archimedean place) of the two 
dynamical pairs at a parameter $t \in C$ which is prerepelling (for the two marked points). Since the Green functions are proportional, 
the H\"older constants are equal so that the multiplier of $P_t$ at $a(t)$ and $Q_t$ at $b(t)$ are multiplicatively independent. 
Using ideas from Levin~\cite{Levin-symmetries}, we build a real-analytic flow commuting with the dynamics which implies $P$ and $Q$ to be integrable. 

The core of the argument is the content of \S \ref{sec:biggest proof}. The main ideas are already present in Baker and DeMarco's original paper~\cite{BDM} and relies
on the expansion of the B\"ottcher coordinate and Ritt's theory. Note however that $C$ may have several places at infinity, and this makes things considerably more technical. 

We have divided the proof into four steps indicated (I) to (IV).  Step (IV) is not essential. In this last step, we explain how to reduce the problem to families of monic and centered polynomial. We thus focus on the description of the first three steps. 

Step (I) consists in the analysis of the B\"ottcher coordinate near a suitable branch at infinity of $C$. More precisely, we fix such a branch $\mathfrak{c}$
lying in the support of the divisor $\mathsf{D}_{P,a}$ (i.e. of $\mathsf{D}_{Q,b}$ since these divisors are proportional). 
Let us take an adelic parametrization $t \mapsto \theta(t)$ of $\mathfrak{c}$. To simplify notation, we write $P_t$, $a(t)$, instead of  $P_{\theta(t)}$, $a(\theta(t))$, etc.

Using the expansion of the B\"ottcher coordinates explained in \S \ref{sec:bottcher}, we next show that $\varphi_{P_{t}}(P^{lN}_{t}(a(t)))$ and $\varphi_{Q_{t}}(Q^{lM}_{t}(b(t)))$
are well-defined as adelic series, and satisfies
\[\left(\varphi_{P_{t}}\left(P^{lN}_{t}(a(t))\right)\right)^n =  \zeta \, \left(\varphi_{Q_{t}}\left(Q^{lM}_{t}(b(t))\right)\right)^m \text{ for all } l~,\] 
for some root of unity $\zeta$ that may depend on $l$ and on $\mathfrak{c}$. We get some control on $\zeta$ by observing that 
it always belongs to a fixed finite extension of our field of definition $\KK$.

\smallskip

During the whole duration of Step (II) we work with a suitable fixed parameter $t$ sufficiently close (at some place fixed Archimedean place $v_0$)
to $\mathfrak{c}$. 
We consider the analytic curve $\mathfrak{s}_\zeta = \{ (x,y),\, \varphi_{P_t}(x)^n =  \zeta \varphi_{Q_t}(y)^m\}$ defined in an open set of the form $\{ |x| \ge R_v, |y| \ge R_v\}$ 
inside $\A^{2,\an}_v$. We observe that this analytic curve contains many algebraic points of the form $(P^{lN}_{t}(a(t)), Q^{lM}_{t}(b(t)))$. 
This allows us to apply the algebricity result of Xie (Theorem~\ref{thm:Junyi} from Chapter~\ref{chapter:geom}), which proves that the Zariski closure of $\mathfrak{s}_\zeta$ is in fact an algebraic curve $Z_t$. 
Using a result from Pakovich relying on Ritt's theory, we may also assume that $Z_t$ is of the form $\{ u(x) = v(y)\}$ for some polynomials $u$ and $v$ (depending on $t$). 
Unwinding definitions, this implies that $P_t$ and $Q_t$ are intertwined.

\smallskip

Finally we exhibit in Step (III) a series of conditions $(\mathcal{P}_1)$ -- $(\mathcal{P}_3)$ on the polynomials $P_t, Q_t$ and the points $a(t), b(t)$
which imply $P_t$ and $Q_t$ to be intertwined and that depend algebraically on $t$. The second step furnishes an infinite set 
of parameters $t$ satisfying these conditions,  hence they are satisfied for all $t \in C$. 
This implies the families $P$ and $Q$ to be intertwined and concludes the proof.


\section{Dynamical pairs with identical measures} \label{sec:identical measures}

In this section, we focus on dynamical pairs having identical bifurcation measures and prove the implication  (1) $\Rightarrow$ (2) of Theorem~\ref{tm:unlikely}. 

\subsection{Equality at an Archimedean place}

In this section we assume $K = \C$.

\begin{theorem}\label{thm:identical measures}
Let $(P,a)$ and $(Q,b)$ be two active and non-integrable dynamical pairs parametrized by an irreducible affine complex curve $C$. 
Assume there exist two integers $n,m\geq1$ such that $n\cdot\Delta g_{P,a}=m\cdot \Delta g_{Q,b}$.

Then we have equality of Green functions $n\cdot g_{P,a}=m\cdot g_{Q,b}$ on $C$.
\end{theorem}\label{atoneplace}

\begin{proof}
Recall from \S\ref{sec:subharmonic on singular-C} that a continuous function is subharmonic iff its pull-back on the normalization  $\tilde{C}$ of $C$ is subharmonic, and  that its Laplacian is given by the push-forward by the normalization map of the Laplacian on $\tilde{C}$.
We may therefore replace $C$ by its normalization, and assume the curve to be smooth.

By assumption, the function
\[h(t):=n\cdot g_{P,a}(t)-m\cdot g_{Q,b}(t)\]
is harmonic on $C$.
By Proposition~\ref{prop:bifurcation},  we have 
\begin{align*}
\mathrm{Bif}(P,a)&=\supp(\Delta g_{P,a})=\partial\{g_{P,a}=0\} \text{ and } \\
\mathrm{Bif}(Q,b) &=\supp(\Delta g_{Q,b})=\partial\{g_{Q,b}=0\}~,
\end{align*}
so that 
\[\mathrm{Bif}(P,a)=\mathrm{Bif}(Q,b)\subset \{h=0\}.\]
Since the pairs are active the bifurcation loci are non-empty. 
Observe that there is nothing to prove when $h$ is identically zero, so 
that we may assume that $\{h=0\}$ is a non-empty real-analytic curve. 
If $\mathrm{Bif}(P,a)$ is not totally disconnected, then we can find a holomorphic disk $U$ 
for which $\mathrm{Bif}(P,a) \cap U$ is a smooth curve, and Theorem~\ref{tm:rigidaffine} implies $P$ and $Q$ to be integrable which contradicts
our assumption.

When $\mathrm{Bif}(P,a)$ is totally disconnected, 
the open set $U:= C\setminus\mathrm{Bif}(P,a)$ is connected, the functions $g_{P,a}$ and $g_{Q,b}$ are harmonic on $U$, and we have
\[U=\{g_{P,a}>0\}=\{g_{Q,b}>0\}.\]
For any $\varepsilon>0$, we set $U_\varepsilon:=\{g_{P,a}>\varepsilon\}\cap \{g_{Q,b}>\varepsilon\} \subset C$. 
Observe that $\cup_{\varepsilon >0} U_\varepsilon = U$. 

\begin{proposition}\label{prop:Bottcherrational}
For any $\varepsilon>0$ small enough, there exist two integers $N,M\geq1$ depending on $\varepsilon$ such that 
both functions $\varphi_{P_t}(P_t^{N}(a(t)))$ and $\varphi_{Q_t}(Q^{M}_t(b(t)))$ are well-defined on $U_{\varepsilon}$, 
and 
\begin{equation}\label{eq:639}
U_\varepsilon \ni t \mapsto  \frac{\varphi_{P_t}\left(P^{N}_t(a(t))\right)^{n\delta^M}}{\varphi_{Q_t}\left(Q^{M}_t(b(t))\right)^{md^N}}
\end{equation}
extends as a holomorphic nowhere vanishing function $f$ on 
$C$ satisfying  
\begin{equation}\label{eq:log}
d^N\delta^Mh=\log|f|~.
\end{equation} 
\end{proposition}
Let us take this proposition for granted.
By Theorem~\ref{tm:rigidaffine}, we may find a univalent holomorphic map $\imath \colon \D \to C$ whose image intersects $\mathrm{Bif}(P,a)$
so that the dynamical pairs $(\bar{P}, \bar{a}) = (P \circ \imath, a \circ \imath)$ and $(\bar{Q}, \bar{b})=(Q \circ \imath, b \circ \imath)$ 
induced on $\D$ are real in the sense of \S \ref{sec:realfamily}, and $\mathcal{B} := \mathrm{Bif}(\bar{P}, \bar{a})= \mathrm{Bif}(\bar{Q}, \bar{b})$ is a closed perfect totally disconnected subset of $\mathopen] -1, 1\mathclose[$.

We may thus choose $\varepsilon >0$ so small that the open set $I := \mathopen] -1, 1\mathclose[ \cap \imath^{-1}(U_\varepsilon)$ is non-empty. 
Proposition~\ref{prop:Bottcherrational} furnishes a holomorphic map $f \colon \D \to \C$ satisfying~\eqref{eq:639} and~\eqref{eq:log}.
By Proposition-Definition~\ref{prop:bottcher}, $\varphi_{\bar{P}_t}$ and $\varphi_{\bar{Q}_t}$ 
are also formal Laurent series with real coefficients, hence the real-analytic map $f|_I$ is real, which implies $f(\mathopen] -1, 1\mathclose[ ) \subset \R$.
Recall that the harmonic function $h$ is vanishing on $\mathcal{B}$, hence on $\mathopen] -1, 1\mathclose[$, and we conclude from~\eqref{eq:log} that
$f(t)\in\{-1, 1\}$ for all $t$ real. This implies $f$ to be a constant, hence $h$ too. This contradicts our standing assumption. 
\end{proof}

\begin{proof}[Proof of Proposition~\ref{prop:Bottcherrational}]
Recall that we assumed $C$ to be smooth. Let $\bar{C}$ be any smooth projective compactification  $C$. 
As above, write  $\mathcal{B} := \mathrm{Bif}(P, a)= \mathrm{Bif}(Q, b)$. 

We let $C_+$ be the union of $C$ together with all  branches at infinity of $C$
lying in the closure of $\mathcal{B}$. Observe that $\bar{\mathcal{B}}$ is compact and totally disconnected in $C_+$. 

Let $\mathfrak{c}$ be any branch at infinity of $\tilde{C}$ (i.e. a branch at infinity of $C$ which does not lie in $\bar{\mathcal{B}}$). 
Apply Theorem~\ref{thm:main FG} to the dynamical pair $(P,a)$ restricted to a punctured disk centered at $\mathfrak{c}$. 
Three possibilities may appear. In case (1), one can conjugate the family so that $(P,a)$ extends analytically through $\mathfrak{c}$. 
By base change over $C$, we may suppose that the family of polynomials is monic and centered. By Lemma~\ref{lem:at infty}
the family $P$ extends analytically through $\mathfrak{c}$. 
Replacing $C$ by a finite ramified cover, we may thus suppose that this case does not appear. 

Let us argue that case (3) of Theorem~\ref{thm:main FG} can not appear. Indeed, otherwise $g_{P,a}$ would extend continuously through $\mathfrak{c}$ and vanish at this point. 
Since $\mathfrak{c}\in\bar{\mathcal{B}}$
and $\bar{\mathcal{B}}$ is totally disconnected,  $g_{P,a}$ is subharmonic and positive in a punctured neighborhood of $\mathfrak{c}$ 
which contradicts the maximum principle.  We are thus automatically in case (2), and $q_\mathfrak{c}(P,a)>0$ for any branch at infinity of $C_+$. 

Recall that for any $t$ close enough to $\mathfrak{c}$ one has $g_{P,a}(t) \ge q_\mathfrak{c}(P,a) \log|t|^{-1} + O(1)$, whereas
$G(P_t) \le \kappa \log |t|^{-1} + O(1)$ for some constant $\kappa \ge0$ by Proposition~\ref{prop:growth Green}. 
It follows that  $g_{P,a}(t) \ge \eta \cdot G(P_t) $ for some $\eta>0$ and for all $t$ outside a compact subset of $C_+$. 
We conclude that for $n$ large enough, $g_{P_t}(P^n(a(t))) = d^n g_{P,a}(t) > G(P_t)$  for all $t$ outside a compact subset of $C_+$. 
It follows that given $\varepsilon>0$ there exists a sufficiently large integer $N \ge1$ such that 
 $g_{P_t}(P^N(a(t))) > G(P_t)$ for all $t \in U_\varepsilon \subset \{ g_{P,a} >\varepsilon\}$.
Similarly we get the existence of an integer $M\ge1$ such that  $g_{Q_t}(Q^M(b(t))) > G(Q_t)$ for all $t \in U_\varepsilon$. 

Proposition~\ref{prop:GreenBottcher} shows that for any  $t \in U_\varepsilon$ the points $P_t^N(a(t))$ and $Q_t^M(b(t))$ fall into the range of 
the B\"ottcher coordinates of $P_t$ and $Q_t$ respectively, and we have 
\begin{align*}
d^N\delta^Mh(t) & = d^N\delta^M\Big(ng_{P_t}(a(t))-mg_{Q_t}(b(t))\Big)\\
& = n\delta^Mg_{P_t}(P^N_t(a(t)))-md^Ng_{Q_t}(Q^M_t(b(t)))\\
& = n\delta^M\log\left|\varphi_{P_t}\left(P^N_t(a(t))\right)\right|-md^N\log\left|\varphi_{Q_t}\left(Q^M_t(b(t))\right)\right|\\
& = \log\left|\frac{\varphi_{P_t}\left(P^N_t(a(t))\right)^{n\delta^M}}{\varphi_{Q_t}\left(Q^M_t(b(t))\right)^{md^N}}\right|
\end{align*}
for all  $t \in U_\varepsilon$. 

Since $\bar{\mathcal{B}}$ is compact and disconnected in $C_+$, for any $\varepsilon>0$ small enough, the open set $U_\varepsilon$ is connected and 
we can cover its complement by finitely many  topological disks $D_1,\ldots,D_p\subset C_+$ such that $C_+\setminus U_\varepsilon\Subset\bigcup_jD_j$. As $D_j$ is simply connected and $d^Nh$ is harmonic, there exists a nowhere vanishing holomorphic function $f_j\colon D_j\to \C$ such that $d^N\delta^Mh=\log|f_j|$ on $D_j$. 
We infer that 
\[|f_j(t)|=\exp(d^N\delta^Mh(t))=\left|\frac{\varphi_{P_t}\left(P^N_t(a(t))\right)^{n\delta^M}}{\varphi_{Q_t}\left(Q^M_t(b(t))\right)^{md^N}}\right|\]
for any $t\in U_\varepsilon\cap U_j$.
Multiplying $f_j$ by a suitable complex number of modulus $1$, we may suppose that
the two holomorphic functions $f_j$ and  $\frac{\varphi_{P_t}\left(P^N_t(a(t))\right)^{n\delta^M}}{\varphi_{Q_t}\left(Q^M_t(b(t))\right)^{md^N}}$ coincide on $U_\varepsilon\cap U_j$
so that the function 
\[
f 
=
\begin{cases}
f_j \text{ on } D_j\\
\frac{\varphi_{P_t}\left(P^N_t(a(t))\right)^{n\delta^M}}{\varphi_{Q_t}\left(Q^M_t(b(t))\right)^{md^N}} \text{ on } U_\varepsilon
\end{cases}
\]
is a well-defined holomorphic function on $C_+$. By construction it satisfies $d^N\delta^Mh=\log|f|$ on $U_\varepsilon$ hence everywhere. 
\end{proof}


\subsection{The implication  (1) $\Rightarrow$ (2) of Theorem~\ref{tm:unlikely}}\label{sec:1on3}
Let $(P,a)$ and $(Q,b)$ be active non-integrable  dynamical pairs of respective degree $d,\delta\geq2$ parametrized by an affine curve $C$ defined over a number field $\KK$.
By assumption one can find a sequence of distinct points $p_l \in C(\bar{\KK})$ such that  
$p_l \in \preper(P,a) \cap \preper(Q,b)$ for all $l$.

Since both dynamical pairs are assumed to be active, Proposition~\ref{prop:adelic} implies that the two height functions $h_{P,a}$ and $h_{Q,b}$
are induced by semi-positive adelic metrizations on two (a priori different) ample line bundles $L_{P,a} = \mathcal{O}_{\bar{C}}(n\cdot\mathsf{D}_{P,a})$
 and $L_{Q,b} = \mathcal{O}_{\bar{C}}(m\cdot\mathsf{D}_{Q,b})$ where $n$ and $m$ are integers.
Recall also that the curvatures of  $L_{P,a}$ and $L_{Q,b}$ in $C^{v,\an}$ are the positive measures $n\cdot \Delta g_{P,a,v}$ and $m\cdot \Delta g_{Q,b,v}$ respectively
 by  Proposition~\ref{prop:adelic} (4).

Since $h_{P,a}(p_l) = h_{Q,b}(p_l) =0$ for all $l$, we may apply Thuillier-Yuan's theorem, see Theorem~\ref{tm:yuan}, and we infer that for any 
place $v \in M_\KK$, we have 
\[
\frac1{\deg(p_l)} \sum_{q \in \mathsf{O}(p_l)}\delta_q \longrightarrow n\cdot \Delta g_{P,a,v}= m\cdot \Delta g_{Q,b,v} \text{ as } l \to \infty~.
\]
Theorem~\ref{atoneplace} gives the equality  $n  g_{P,a,v} = m  g_{Q,b,v}$ at all Archimedean places, hence $n\cdot\mathsf{D}_{P,a} = m\cdot\mathsf{D}_{Q,b}$
by Proposition~\ref{prop:behave green}.
Let $v$ be any non-Archimedean place. The function $n  g_{P,a,v} - m  g_{Q,b,v}$ is harmonic on $C^{\an}_v$ and can be extended continuously through all branches of $C$ at infinity. 
By the maximum principle, it is thus a constant, and this constant must be zero since  $g_{P,a,v} (p_l) =  g_{Q,b,v} (p_l) =0$ for all $l$. 

From Proposition~\ref{prop:adelic} (1) we finally infer $n h_{P,a} = m h_{Q,b}$ which concludes the proof. 

\smallskip

For further reference we note that we have actually proved the following
\begin{theorem}\label{thm:partial1}
Let $(P,a)$ and $(Q,b)$ be active non-integrable  dynamical pairs parametrized by an affine curve $C$ defined over a number field $\KK$.
The following are equivalent:
\begin{enumerate}
\item
$\preper(P,a,\bar{\KK}) \cap \preper(Q,b,\bar{\KK})$ is infinite. 
\item
there exist $n,m\in \N^*$ such that 
$n\cdot g_{P,a,v} = m \cdot g_{Q,b,v}$ for all places $v\in M_\KK$;
\item
there exist $n,m\in \N^*$ such that 
$n \cdot h_{P,a} = m \cdot h_{Q,b}$;
\item
$\preper(P,a,\bar{\KK}) = \preper(Q,b,\bar{\KK})$.
\end{enumerate}
\end{theorem}
\begin{proof}
The arguments above show (1)$\Rightarrow$(2) and (3). 
Then (4) follows from (3)  since $\preper(P,a,\bar{\KK}) = \{ h_{P,a} = 0\} = \preper(Q,b,\bar{\KK}) = \{ h_{Q,b} = 0\}$. 
And (4) implies trivially (1).
\end{proof}


\section{Multiplicative dependence of the degrees}

In this section, we prove an important step of the proof of Theorem~\ref{tm:unlikely}: 
the degrees of any two entangled dynamical pairs $(P,a)$, $(Q,b)$ 
are multiplicatively dependent. 

Throughout this section, we assume $K= \C$.

\begin{theorem}\label{tm:multiplicative}
Let $(P,a)$ and $(Q,b)$ be two active non-integrable complex dynamical pairs of respective degree $d,\delta\geq2$ parametrized by an  irreducible affine curve $C$. 
Assume that $\preper(P,a)=\preper(Q,b)$, and that $g_{P,a}$ and $g_{Q,b}$ are proportional on $C$.

Then one may find two positive integers $N,M\geq1$ such that $d^N=\delta^M$. 
\end{theorem}

We refer to~\cite[Theorem~5.1]{GNY} for a statement implying equality of degrees in a related context.

The idea of the proof is to look at the expansion of the Green functions $g_{P,a}$ and $g_{Q,b}$ near a common prerepelling parameter. 
We proceed in two steps, first proving that such a parameter exists (Proposition~\ref{lm:bothprerepelling}), and then analyzing the Green functions near that parameter (Proposition~\ref{multiplicative}). In the second step, we compute
the H\"older exponents of the Green functions and relate them to the degrees of the families: this argument is inspired from~\cite[Proposition~3.1]{Favre-Dujardin}.

\begin{proposition}\label{lm:bothprerepelling}
Assume $(P,a)$ and $(Q,b)$ are two dynamical pairs satisfying the assumptions of Theorem~\ref{tm:multiplicative}. 
Then for any parameter $t_0\in\mathrm{Bif}(P,a)$ at which $a$ is transversally prerepelling, the point $b$ is properly prerepelling for $Q$ at $t_0$. 
\end{proposition}
\begin{proof}
Pick any transversely prerepelling parameter $t_0$ for the pair $(P,a)$. 
This parameter automatically belongs to $\mathrm{Bif}(P,a)$ by Lemma~\ref{lm:suppT}, 
hence $t_0 \in \mathrm{Bif}(Q,b)\cap\preper(Q,b)$ by our assumptions and Proposition~\ref{prop:bifurcation}. 
In particular $b(t_0)$ cannot be preperiodic to an attractive cycle.
We need to argue that $b(t_0)$ is preperiodic to a \emph{repelling} cycle.

To simplify notation write $t_0=0$. We let $m$ be any integer such that $P_{0}^m(a(0))$ is periodic, and let $k$ be the period of that point. 
Denote by $\lambda = (P^k_{0})'(P_0^m(a(0)))$ the multiplier of this cycle. 
Let also $\phi$ be the linearizing coordinate at $P_{0}^{m}(a(0))$, i.e. the univalent map $\phi:\D(0,r) \to\C$ such that $\phi'(0) =1$, $\phi(0)=P_{0}^{k}(a(0))$ and $\phi(\lambda z)=P_{0}^k(\phi(z))$ for all $z\in\D(0,r)$ (for some $r \ll 1$). 
By Lemma~\ref{lm:cvrenorm} we have
\[d^{m+kn}g_{P_{t/\lambda^n}}(a(t/\lambda^n))\rightarrow g_{P_{0}}(\phi(t)),\]
uniformly on $\D(0,r)$ as $n\to\infty$. 

We may assume that $Q_0^m(b(0))$ is also periodic of period $l$, and let $\mu$ be its multiplier. 
We suppose by contradiction that $|\mu| = 1$.
Up to translation, we can assume $b(0)=0$. Since we may locally follow the periodic orbit 
we shall also suppose that $Q_t^l(0)=0$ for all $t\in\D(0,r)$. Recall that $\deg(Q) = \delta$.
\begin{claim}
For any $|t| , |z| >0$ small enough, one has
\[
g_{Q_t}(z)\leq C\delta^{- l\cdot \min \left\{ |t|^{-1/2} , |z|^{-1} \right\}
}
\]
for some constant $C>0$.
\end{claim}
Since $|Q_t^m(b(t))|\leq C'|t|$ for some constant $C'>0$ and for all $|t|\le r$, and for all $n$, 
we find
\begin{align*}
g_{Q_{t/\lambda^n}}(b(t/\lambda^n))=\frac{1}{\delta^m}g_{Q_{t/\lambda^n}}\left(Q_{t/\lambda^n}^m(b(t/\lambda^n))\right)\leq 
C \delta^{-m} \delta ^{ - C'' \sqrt{|\lambda^n/t|}}
\end{align*}
By assumption, we have $ g_{P_t}(a(t))= c g_{Q_t}(b(t))$ for some positive $c>0$, hence the above gives
\[d^{kn+m}\,g_{P_{t/\lambda^n}}(a(t/\lambda^n))=c \cdot d^{kn+m}g_{Q_{t/\lambda^n}}(b(t/\lambda^n))
\lesssim
 d^{kn}\cdot 
\delta ^{ - C'' \sqrt{|\lambda^n/t|}},\]\\
so that $\varlimsup_n  d^{kn+m}\,g_{P_{t/\lambda^n}}(a(t/\lambda^n)) \le 0$.
This implies
\[g_{P_0}(\phi(t))=\lim_{n\to\infty}\frac1{d^{kn+m}}g_{P_{t/\lambda^n}}(a(t/\lambda^n)) =0,\]
which contradicts the fact that $a(0)$ lies in the Julia set of $P_0$.
\end{proof}

\begin{proof}[Proof of the Claim]
 Set $f_t:=Q_t^l$. Since $f_t(0) = 0$ for all $t$, and $f'_0(0) =1$, we may write 
\[|f_t(z)|\leq |z|+ C(|t|+|z|^2) \]
for some $C>0$ and for all $t$ and  $z$ small enough, say $\le r$ with $1/C \le r$. 
Let us prove by induction on $n$ that for all $|t|\leq \frac1{(4Cn)^2}$ and all $|z|\leq \frac1{16Cn}$, then
\begin{equation}
|f_t^n(z)|\leq |z|+4Cn(|t|+|z|^2) \le \frac1{Cn} \le r
~.\label{stark}
\end{equation}
Only the first inequality needs an argument. For $n=1$, this is obvious. Assume $|t|\le  \frac1{(4C(n+1))^2}$  and $|z|\leq \frac1{16C(n+1)}$. Then we have
\begin{align*}
|f_t^n(z)|^2 & \leq (|z|+4Cn(|t|+|z|^2))^2\leq (2|z|+ 4Cn|t|)^2\\
& \leq 4|z|^2+ 16Cn |z|\cdot|t|+ 16 C^2 n^2|t|^2\\
& \leq 4|z|^2+2|t|
\end{align*}
so that 
\begin{align*}
|f_t^{n+1}(z)| & \leq |f_t^n(z)|+C( |t|+|f_t^n(z)|^2)\\
& \leq |z|+4Cn(|t|+|z|^2)+C (3|t|+ 4|z|^2)\\
& \leq |z|+4C(n+1)(|t|+|z|^2).
\end{align*}
ending the proof of \eqref{stark}.

Now pick any $t, z$ small enough, and set $n:= \min \{ (4C\sqrt{|t|})^{-1}, (16C|z|)^{-1}\}$
so that $|f^n_t(z)|\le r$ by~\eqref{stark}.
We conclude that
\[\delta^{ln} g_{Q_t}(z)=g_{Q_t}(Q_t^{ln}(z))=g_{Q_t}(f_t^n(z))\leq C := \sup_{|t|, |z|\le r} g_Q ~,\]
which proves the claim, observing that $C$ is a large constant which may be assumed to be $\ge1$. 
\end{proof}

\begin{proposition}\label{multiplicative}
Assume $(P,a)$ and $(Q,b)$ are two dynamical pairs satisfying the assumptions of Theorem~\ref{tm:multiplicative}. Suppose further that 
there exist infinitely many parameters $t\in C$ such that $a$ is transversely prerepelling and $b$ is simultaneously properly prerepelling at $t$.

Then one may find two positive integers $N,M\geq1$ such that $d^N=\delta^M$. 
\end{proposition}

\begin{proof}
Assume by contradiction that 
\[\theta:=\frac{\log d}{\log \delta}\notin\Q ~.\]
We shall prove that for any parameter $t$  at which $a$ is transversely prerepelling
and $b$ is simultaneously properly prerepelling, then $P_t$ is integrable. 
If this occurs at infinitely many parameters, then $P$ is isotrivial and integrable which contradicts 
the assumptions of Theorem~\ref{tm:multiplicative}.

So let us pick a parameter $t_0$ such that $a$ is transversely prerepelling
and $b$ is properly prerepelling. Write $t_0 =0$, suppose $P_0^m(a(0))$ and $Q_0^m(b(0))$ are both preperiodic
of period $k$ and $l$ respectively. Denote by $\lambda:=(P_{0}^k)'(P_{0}^m(a(0)))$ and $\mu:=(Q_{0}^l)'(Q_{0}^m(b(0)))$
the multipliers of these periodic points so that $|\lambda|>1$ and $|\mu|>1$.

\begin{lemma}\label{lem:holder exp}
\[\lim_{r \to 0} \frac1{\log r} \log \left( \sup_{|t|\le r} g_{Q,b}(t)
\right)
 = \frac{ l \log \delta}{\log |\mu_1|}\]
where $ q$ is the order of vanishing of $Q^{l+m}(b)-Q^m(b)$ at $0$ and $\mu_1$ is 
is a $q$-th root of $\mu$. 
\end{lemma}

Since $g_{P,a}$ and $g_{Q,b}$ are proportional, the previous lemma implies
that 
\[\frac{ l \log \delta}{\log |\mu_1|} = \frac{ k \log d}{\log |\lambda |}.\] 
\begin{lemma}
For any $\alpha \in \R$, there exists sequences $n_j , m_j \to \infty$
such that
\[\lambda^{n_j}\mu_1^{-m_j} \to e^{\alpha}~.\]
\end{lemma}
\begin{proof}
Our assumption implies $\beta:= \log |\lambda|/ \log |\mu_1|$ to be irrational so that the abelian
group generated by $\beta$ and $1$ is dense in $\R$.  
We conclude that there exist sequences of integers 
$n_j, m_j \to \infty$ such that 
$n_j \log |\lambda| -  m_j \log |\mu_1|\to \alpha$, and
up to extracting a subsequence we have
$\lambda^{n_j}\mu_1^{-m_j} \to e^\alpha e^{i\theta}$ for some $\theta\in \R$.
Pick $r_j \to\infty$ with $e^{ir_j\theta} \to1$, so that we get
$\lambda^{r_jn_j}\mu_1^{-r_jm_j} \to e^\alpha$, as required.
\end{proof}
Fix any real number $\alpha$, and choose sequences of integers $n_j, m_j$ as in the previous lemma. 
Observe that $d^{kn_j} \delta^{-lm_j} \to e^{C\alpha}$ for some positive constant $C>0$.
Write $g_{P,a}/g_{Q,b} = \kappa>0$,
so that 
\begin{align*}
g_{P_0} \circ \phi (t) &= 
\lim_j
d^{m+kn_j}  g_{P,a} ( t/\lambda^{n_j}) 
=
\left(\kappa \frac{d^m}{\delta^m}e^{C\alpha}\right)\, 
\lim_j
\delta^{m+ln_j}  g_{Q,b} ( t/\lambda^{n_j}) 
\\
&=
(\kappa'e^{C\alpha})\, 
\lim_j
\delta^{m+lm_j} 
g_{Q,b} ( e^{-\alpha}t/\mu^{m_j})
=
(\kappa'e^{C\alpha})\,  g_{Q_0} \circ \psi(e^{-q\alpha} t^q)
\end{align*}
for some $\kappa'>0$.
For each $\alpha$, let us introduce the real-analytic flow of local analytic isomorphisms:
$\sigma_\alpha(t) = \psi ( e^{q\alpha} \psi^{-1}(t)) = e^{\alpha} t + O(t^{2})$. 
Our previous computations used twice imply 
\[e^{C\alpha} g_{Q_0} \circ \psi(e^{-q\alpha} t^q) =   g_{Q_0} \circ \psi(t^q)\]
so that
\[g_{Q_0} = e^{C \alpha} g_{Q_0} \circ \sigma_\alpha\]
for all $\alpha$. 

It follows that the flow $\sigma_\alpha$ preserves locally the Julia set of $Q_0$ near $0$, 
hence $J(Q_0)$ is smooth near at least one of its point. We conclude 
by Theorem~\ref{th:zdunik} that  $Q_0$ is integrable. 
\end{proof}

\begin{proof}[Proof of Lemma~\ref{lem:holder exp}]
By Lemma~\ref{lm:cvrenorm} we have uniform convergence for all $t$ small enough
\[\delta^{m+ln}g_{Q_{t/\mu_1^n}}(b(t/\mu_1^n)) \longrightarrow  g_{Q_{0}}\circ\psi(t^q)\]
where  $\psi$ is the linearizing coordinate of $Q_0^l$ at $Q_{0}^m(b(0))$.
In particular
$\delta^{ln} g_{Q,b}(t/\mu_1^n)$ converges to a continuous 
subharmonic function $h$ which vanishes at the origin
and is not identically zero. In particular $\rho \mapsto H(\rho) = \sup_{|t| = \rho} h(t)$ is a continuous function
which is positive for any $\rho>0$, see, e.g., \cite[Theorem 2.6.8]{ransford}.

For any $r \le \rho$, there exists a unique integer $n$ such that 
$\rho/|\mu_1|\le  r |\mu_1|^n \le \rho$ and we get
\begin{align*}
\sup_{|t|\le r} g_{Q,b}(t) & = \sup_{|\tau|\le r |\mu_1|^n} g_{Q,b} (\tau/\mu_1^n)
\\
& 
= \frac1{\delta^{ln}} \sup_{|\tau|\le r |\mu_1|^n} \delta^{ln} g_{Q,b} (\tau/\mu_1^n)
\begin{cases}
\le& \delta^{-ln} 2H(\rho)
\\
\ge & \delta^{-ln} \frac{H(\rho/|\mu_1|)}2
\end{cases}
\end{align*}
which implies 
\[
\frac1{\log r} \log \left( \sup_{|t|\le r} g_{Q,b}(t)
\right)
= 
\frac{ - ln \log \delta + O(1)}{ - n\log |\mu_1| + O(1)} 
\longrightarrow 
\frac{ l \log \delta}{\log |\mu_1|}~.\]
when $r \to 0$. 
\end{proof}


\section{Proof of the implication (2) $\Rightarrow$ (3) of Theorem~\ref{tm:unlikely}}\label{sec:biggest proof}

The setting is as follows. We let $(P,a)$ and $(Q,b)$ be two dynamical pairs of degree $d$ and $\delta$ respectively, parametrized by an affine curve $C$ defined over a number field $\KK$. 
We assume that they are both active and non-integrable and that $h_{P,a}$ and $h_{Q,b}$ are proportional.

We shall first work under the following  assumption $(\diamond)$:
there exist two regular invertible functions $\alpha, \beta \in \KK[C]$ such that 
\[
P_t(z) = \alpha(t)^{d-1} z^d + o_t(z^d)
\text{ and }
Q_t(z) = \beta(t)^{\delta-1} z^\delta + o_t(z^\delta)
~.\]
Under this assumption, one can define the B\"ottcher coordinate of $P_t$
as the unique formal Laurent series in $z^{-1}$ of the form
\[
\varphi_{P_t} = \alpha(t) z + O(z^{-1})
\]
such that 
$\varphi_{P_t} \circ P_t = (\varphi_{P_t})^d$, see Proposition-Definition~\ref{prop:bottcher}.
One defines analoguously the B\"ottcher coordinate $\varphi_{Q_t}$.

\paragraph*{(I) Behaviour of the Green functions near the branches at infinity}

Let $\bar{C}$ be a projective compactification of $C$ such that $\bar{C}\setminus C$ is a finite set of smooth points. 
Up to taking a finite extension of $\KK$, we 
can assume all branches at infinity of $C$ are defined over $\KK$. We let $U_\KK$ be the number of roots of unity lying in $\KK$.

\smallskip

Recall from Theorem~\ref{thm:active-characterization1} that the divisor $\mathsf{D}_{P,a}$ is effective and non-zero since $(P,a)$ is active. 
We also consider the effective divisor $\mathsf{D}_{P} = \sum q_\mathfrak{c}(P) [\mathfrak{c}]$ where $q_\mathfrak{c}(P)$ is 
the constant defined by Theorem~\ref{thm:the case of a family}. 

The set of branches at infinity of $C$ can be decomposed into two disjoints sets: the set $\mathcal{B}$ of branches 
$\mathfrak{c}$ such that $\ord_\mathfrak{c}(\mathsf{D}_{P,a})>0$; and the set $\mathcal{G}$ of branches 
$\mathfrak{c}$ such that $\ord_\mathfrak{c}(\mathsf{D}_{P,a})=0$. 

Pick any branch $\mathfrak{c}$ of $C$ at infinity, and choose a local adelic parametrization $t \mapsto \theta(t)$ of that branch, see~\S\ref{sec:adelicseries} for a precise definition. 
To simplify notation we shall abuse notations and write  $P_t = P_{\theta(t)}$ and $a(t) = a(\theta(t))$. 
If $\mathfrak{c} \in \mathcal{B}$, then we have
\[g_{P,a,v}(t) = \ord_\mathfrak{c}(\mathsf{D}_{P,a}) \log |t|^{-1}_v + O(1) \to \infty\]
for all places $v \in M_\KK$, by~Proposition \ref{prop:behave green}. Otherwise $\mathfrak{c} \in \mathcal{G}$ and
$g_{P,a,v}(t)$ extends continuously at $t=0$ for all places. 

The same discussion applies to the dynamical pair $(Q,b)$.

\medskip

By Theorem~\ref{thm:partial1} we have  $\preper(P,a,\bar{\KK}) =  \preper(Q,b,\bar{\KK})$, and there exist two integers $n,m>0$ such that 
$n \cdot\mathsf{D}_{P,a} = m \cdot \mathsf{D}_{Q,b}$, and
$n\cdot g_{P,a,v} = m \cdot g_{Q,b,v}$ for all places $v \in M_\KK$. We may and shall assume that $n$ and $m$ are \emph{coprime} (in particular they
are uniquely determined).

Since these equalities hold in particular at an Archimedean place $v$, 
Theorem~\ref{tm:multiplicative} applies and we may thus find two positive integers $N,M \in \N^*$
such that $d^N = \delta^M$.  We take $N$ and $M$ to be minimal over all integers satisfying this equality: again these integers
are uniquely determined. In the sequel we write
\[
\Delta = d^N = \delta^M~.\]
Observe that the set of branches $\mathcal{B}$ coincides 
with $\{ \mathfrak{c}, \ord_\mathfrak{c}(\mathsf{D}_{Q,b})>0\}$.
For each branch $\mathfrak{c} \in \mathcal{B}$, we let $l(\mathfrak{c})$ be the least integer
such that
\[
\Delta^{l(\mathfrak{c})} \times \ord_\mathfrak{c}(\mathsf{D}_{P,a}) >
\ord_\mathfrak{c}(\mathsf{D}_{P})
\text{ and }
\Delta^{l(\mathfrak{c})} \times \ord_\mathfrak{c}(\mathsf{D}_{Q,b}) >
\ord_\mathfrak{c}(\mathsf{D}_{Q})
~.\]

\begin{lemma}\label{lem:344}
For any branch at infinity $\mathfrak{c}\in\mathcal{B}$, and for any integer $l\ge l(\mathfrak{c})$ there exists
a root of unity $\zeta = \zeta(\mathfrak{c},l) \in \KK$ such that $\varphi_{P_t}\left(P^{lN}_t(a(t))\right)$ and $\varphi_{Q_t}\left(Q^{lM}_t(b(t))\right)$
are well-defined adelic series and 
\[\varphi_{P_t}\left(P^{lN}_t(a(t))\right)^n = \zeta \times  \varphi_{Q_t}\left(Q^{lM}_t(b(t))\right)^m
~.\]
\end{lemma}

\begin{proof}
Fix any integer $l\ge1$ such that $\Delta^l \cdot \ord_\mathfrak{c}(\mathsf{D}_{P,a}) > \ord_\mathfrak{c}(\mathsf{D}_{P})$, and fix any place $v\in M_\KK$. 

\smallskip

\noindent{\bf 1.} Our first objective is to show that $P^{lN}_t(a(t))$ belongs to the domain of definition of the B\"ottcher coordinate $\varphi_{P_t,v}$. 
Recall from Definition~\ref{def:divisor dynamical} that $\ord_\mathfrak{c}(\mathsf{D}_{P,a})= q_\mathfrak{c}(P,a)$, and similarly $\ord_\mathfrak{c}(\mathsf{D}_{P})=q_\mathfrak{c}(P)$.
By Proposition~\ref{prop:behave green} and Theorem~\ref{thm:the case of a family} respectively, we have
\begin{align*}
g_{P,a,v}(t) & =  \ord_\mathfrak{c}(\mathsf{D}_{P,a})\cdot \log|t|^{-1} +O(1)~,
\\
G_{v}(P_t) &= \ord_\mathfrak{c}(\mathsf{D}_{P})\cdot \log|t|^{-1}+ O(1)
~,
\end{align*}
so that for any $t$ small enough, $g_{P_t,v}(P_t^{lN}(a(t))) - G_v(P_t)$ is very large. It follows from Proposition~\ref{prop:GreenBottcher} (2) that $P_t^{lN}(a(t))$ belongs to the 
domain of definition of the B\"ottcher coordinate.

\smallskip

\noindent{\bf 2.}
By the previous step, the function $\Phi_{v}(t) := \varphi_{P_t}\left(P^{lN}_t(a(t))\right)^n$ is well-defined and analytic in some punctured
disk $\D_\KK^*(0,r_v)$. Since we have \[\log |\Phi_{v}(t)|_v = n\Delta^l \cdot g_{P,a}(t) = c \log|t|^{-1} + O(1)\] for some $c>0$,  it follows
that $\Phi$ is meromorphic at $0$.  

In the next two steps, we argue that the Laurent series associated to $\Phi_{v}$ is in fact adelic (in particular we shall see that this series is independent of $v$). 

Observe that for any $L>l$, we have
\[\Phi_v(t)^{\Delta^{L-l}}=\varphi_{P_t}\left(P^{lN}_t(a(t))\right)^{n\Delta^{L-l}} =  \varphi_{P_t}\left(P^{NL}_t(a(t))\right)^{n}
~.\] 
Since any root of an adelic series remains adelic by~\cite[Lemma~3.2]{specialcubic} (but possibly over a finite extension of $\KK$), 
we may suppose that  $l$ is arbitrarily large. 

\smallskip

\noindent{\bf 3.} \label{step3}
We now identify the Laurent series associated to $\Phi_v$.

From Proposition~\ref{prop:classic-estim} (1) and Proposition~\ref{prop:growth Green}, we have 
\[\Delta^{l} \cdot g_{P,a,v}(t) = g_{P_t,v}(P_t^{lN}(a(t))) \le \log^+ |P^{lN}_t(a(t))| + G_{v}(P_t) + O(1)~.\]
Combining this with the previous two estimates, we get \[|P^{lN}_t(a(t))|_v = \beta |t|^{-b_l}  (1 + o(1))\] 
for some $\beta >0$, with $b_l := \Delta^l \cdot \ord_\mathfrak{c}(\mathsf{D}_{P,a}) - \ord_\mathfrak{c}(\mathsf{D}_{P})>0$. 
In other words, we may expand as a formal Laurent series 
$P^{lN}_t(a(t))= t^{-b_l} (\sum_{n\ge0} \beta_n t^n)$ with $\beta_0 = \beta\neq 0$.

We now apply Proposition~\ref{prop:expan bottcher} to the degree $d$ polynomial $P_t(z) = A(t) z^d + a_1(t) z^{d-1} + \cdots + a_d(t) \in \KK(\!(t)\!) [z]$, and write
\[
\varphi_{P_t}\left(z\right) = \alpha \left( z + \frac{a_1}{dA}\right) + \sum_{j \ge 1} \frac{\alpha_j}{z^j}
\]
where $\alpha^{d-1} =A$, and
$\alpha_j$ is  a polynomial of degree $j$ in the coefficients of $P_t$.
The latter property implies that $\ord_0(\alpha_j) \ge  - j \mu$ for any integer $j$ where 
\[\mu = \max \{0, - \ord_0(A), -\ord_0(a_j)\}\]
is the maximum of the order of poles of the Laurent series $A, a_1, \ldots, a_d$. 

From now on, we assume that $l$ is large enough so that $b_l > \mu$.
Then for each $j$, 
$\alpha_j(t)/(P_t^{lN}(a(t)))^j$ is a power series vanishing at $0$ up to order $\ge (b_l-\mu) j$, hence the series
$\sum_{j \ge 1} \frac{\alpha_j(t)}{(P_t^{NL}(a(t)))^j}$ converges formally in $\KK[\![t]\!]$.

Let $\tilde{\Phi}$ be the formal power series obtained by summing $\sum_{j \ge 1} \frac{\alpha_j(t)}{(P_t^{lN}(a(t)))^j}$ with 
$\alpha \left( P_t^{lN}(a(t)) + \frac{a_1}{dA}\right)$.

\smallskip

\noindent{\bf 4.}
We  show that $\tilde{\Phi}$ is the Laurent series expansion of $\Phi_v$ at $0$.
Observe that since all coefficients of $\tilde{\Phi}$ lie in a finite extension of $\KK$, and $v$ is an arbitrary place of $\KK$, this will show that 
$\tilde{\Phi}$ is an adelic series as required. 

We further increase $l$ so that $b_l > \ord_\mathfrak{c}(\mathsf{D}_{P})$, and
$\log |P^{lN}_t(a(t))|_v \gg G_v(P_t)$ for all $t$ small enough.  In this range, $P^{lN}_t(a(t))$ thus belongs to the domain of convergence of
the series $\sum_{j \ge 1} \frac{\alpha_j(t)}{z^j}$ by Proposition~\ref{prop:GreenBottcher} (1).

It follows that $f_p(t):= \sum_{j = 1}^p \frac{\alpha_j(t)}{(P^{lN}_t(a(t)))^j}$
forms a sequence of analytic functions which converges uniformly on compact subsets of some punctured disk $\D_\KK^*(0,r'_v)$. 
By the maximum principle, this sequence of analytic functions actually converges uniformly on the disk $\D_\KK(0,r'_v)$, and its
power series expansion is precisely the series  $\sum_{j \ge 1} \frac{\alpha_j(t)}{(P_t^{lN}(a(t)))^j}$ considered above. 
Adding the terms $\alpha \left( P_t^{lN}(a(t)) + \frac{a_1}{dA}\right)$,  we conclude 
that $\tilde{\Phi}$ is the Laurent series expansion of $\Phi_v$, hence is an adelic series.

\smallskip

\noindent{\bf 5.}
Similarly, the function $\Psi(t) = \varphi_{Q_t}\left(Q^{lM}_t(b(t))\right)^m$ is an adelic series. 
Let us consider an arbitrary Archimedean place $v_0 \in M_\KK$. Since $n \cdot g_{P,a,v_0} = m \cdot g_{Q,b,v_0}$,
we have $\log |\Phi/\Psi|_{v_0} =0$, hence  the meromorphic function $\Phi/\Psi$ has a constant modulus equal to $1$ in some 
punctured disk $\D_K^*(0,r''_{v_0})$. 
It is therefore equal to a constant $\zeta\in\U$.  In particular we have an equality of adelic series 
$\Phi = \zeta \cdot \Psi$, and $\zeta$ is necessarily an algebraic number. Now for any other place $v \in M_\KK$,  $n \cdot g_{P,a,v} = m \cdot g_{Q,b,v}$ holds, which implies
$|\zeta|_v = 1$. We conclude by Kronecker's theorem that $\zeta$ is a root of unity.
\end{proof}

\begin{lemma}\label{lem:image}
There exist two constants $L_0, L$ depending only on $\KK, d$ and $\delta$ such that 
\[ 
\zeta(\mathfrak{c}) := 
\zeta(\mathfrak{c}, l(\mathfrak{c}) +L_0)
=
\zeta(\mathfrak{c}, l(\mathfrak{c}) + L_0+ l \cdot L) 
\]
for all $\mathfrak{c}\in \mathcal{B}$ and all $l\in \N$.
\end{lemma}
\begin{proof}
Observe that we have  $\zeta(\mathfrak{c},l') =  \zeta(\mathfrak{c},l)^{\Delta^{l'-l}}$ for all $l'\ge l$, where $\Delta = d^N = \delta^M$.
The proof then follows since $\zeta(\mathfrak{c},l)$ is a root of unity belonging to the field $\KK$ of definition of our families.
\end{proof}

To simplify notation, we shall write $\ell(\mathfrak{c}) = l(\mathfrak{c}) + L_0$ in the sequel.

\begin{lemma}\label{lem:345}
Pick any $t \in C(\bar{\KK})$ and any place $v\in M_\KK$ such that $g_{P,a,v}(t) >0$. 
Then for all $l$ large enough, the B\"ottcher coordinates $\varphi_{P_t,v}$ and $\varphi_{Q_t,v}$
are well-defined at $P^{lN}_t(a(t))$ and $Q^{lM}_t(b(t))$ respectively. Moreover there exists a branch at infinity $\mathfrak{c}$
of $C$  such that 
\[\varphi_{P_t,v}\left(P^{lN}_t(a(t))\right)^n = \zeta({\mathfrak{c},l}) \times  \varphi_{Q_t,v}\left(Q^{lM}_t(b(t))\right)^m
~\]
for all $l\gg 1$.
\end{lemma}

\begin{proof}
For each branch $\mathfrak{c}\in \mathcal{B}$, fix a point $t_\mathfrak{c}$ which is sufficiently close to $\mathfrak{c}$ such that 
both adelic series $\varphi_{P_t}\left(P^{lN}_t(a(t))\right)$ and $\varphi_{Q_t}\left(Q^{lM}_t(b(t))\right)^m$ are well-defined in a neighborhood of 
$t_\mathfrak{c}$ and their quotient is equal to $\zeta(\mathfrak{c},l)$ as in the previous lemma.

Let $U$ be the connected component of $\{g_{P,a,v} >0\}\subset C^\an$ containing $t$. 
By the maximum principle, $U$ is unbounded hence contains some point $t_\mathfrak{c}$. 
For each $l$, let 
\[\Omega_l := \{ \tau, d^l \cdot g_{P,a,v}(\tau) > G_v(P_{\tau})\}~.\]
Then $\{U \cap \Omega_l\}_l$ forms an increasing sequence of open sets which cover $U$.
It follows from a purely topological argument that $t$ and $t_\mathfrak{c}$ belong to the same connected component $V$ of $U\cap \Omega_l$ for $l$ large enough. 
The analytic function
\[ \Psi(t)= \varphi_{P_t,v}\left(P^{lN}_t(a(t))\right)^n - \zeta({\mathfrak{c},l}) \times  \varphi_{Q_t,v}\left(Q^{lM}_t(b(t))\right)^m
\]
is well-defined on $V$ and constant equal to $0$ near $\mathfrak{c}$. It follows from the identity principle that $\Psi=0$ on $V$, hence
$\Psi(t) =0$.
\end{proof}

\paragraph*{(II) Construction of the semi-conjugacy at a fixed parameter}
\label{hypo-semi-conj}
We fix any Archimedean place $v_0 \in M_\KK$.
For each $\mathfrak{c} \in \mathcal{B}$, we also choose a connected neighborhood  $U_{v_0}(\mathfrak{c})$ of $\mathfrak{c}$ in $C^{\an}_{v_0}$ such that 
$d^{l(\mathfrak{c})} \cdot g_{P,a,v}(\tau)> G_v(P_{\tau})$, and 
$\delta^{l(\mathfrak{c})}\cdot g_{Q,b,v}(\tau)> G_v(Q_{\tau})$
 for all $\tau \in U_{v_0}(\mathfrak{c})$.

\smallskip

Fix any parameter $t \in C(\bar{\KK})$ which belongs to $U_{v_0}(\mathfrak{c})$ for some $\mathfrak{c} \in \mathcal{B}$, and such that 
neither $P_t$ nor $Q_t$ are integrable.  Since $t$ is fixed, we drop all references to $t$ to simplify notations in this paragraph. 
We also fix a finite extension $\KK'$ of $\KK$
such that $t \in C(\KK')$.

Consider the family of adelic branches at infinity $\mathfrak{s}_\zeta$ indexed by all roots of unity $\zeta \in \KK$  defined by 
the equation
\[\varphi_P(x)^n = \zeta \varphi_Q(y)^m~.\]
We claim that we may apply Theorem~\ref{thm:Junyi} to 
the sequence of points 
\[(a_l,b_l) = (P^{lN}(a), Q^{lM}(b))~.\] 
Indeed, observe first that $(a_l,b_l)$ all belong to $\KK'$. Fix a place $v \in M_{\KK'}$. 
If $g_{P,v}(a) = g_{Q,v}(b)=0$ then both sequences $|a_l|_v$ and $|b_l|_v$ are bounded so that 
\[
B_v = \sup_l \max \{ |a_l|_v, |b_l|_v\}<\infty~.
\]
Note also that $B_v=1$ at any place where both $P$ and $Q$ have good reduction. 
Now consider any place $v$ at which $g_{P,v}(a)>0$ and  $g_{Q,v}(b)>0$, and pick $R_v >0$ such that both series $\varphi_P$ and $\varphi_Q$
converge in the domain $\{ |x|_v > R_v\}$. Define
\[ C_v(\mathfrak{s}_\zeta) := \left\{ (x,y) \in \C_v^2, \, \min \{ |x|_v, |y|_v \} >R_v, \, \varphi_P(x)^n = \zeta \varphi_Q(y)^m\right\}~.\]
It follows from Lemma~\ref{lem:345},  that  $(a_l,b_l) \in \cup_\zeta C_v(\mathfrak{s}_\zeta)$ so that Theorem~\ref{thm:Junyi} applies as claimed.

We infer the existence of an algebraic curve $Z' \subset \A^2$ such that $(a_l, b_l) \in Z'$ for all $l$ large enough, 
and any branch at infinity of $Z'$ is contained in the set $\{ \mathfrak{s}_\zeta\}_{\zeta \in \U_\infty \cap \KK}$.

Observe that at the place $v_0$, we have
\[(a_{ \ell(\mathfrak{c}) + l \cdot L}, b_{ \ell(\mathfrak{c}) + l \cdot L}) \in C_{v_0}(\mathfrak{s}_{\zeta(\mathfrak{c})})\]
for all $l \ge0$. The Zariski closure of $\mathfrak{s}_{\zeta(\mathfrak{c})}$ is thus an irreducible component $Z$ of $Z'$
containing $(a_{ \ell(\mathfrak{c}) + l \cdot L}, b_{ \ell(\mathfrak{c}) + l \cdot L})$ for all $l\ge0$ and 
any branch at infinity of $Z$ is also contained in the set $\{ \mathfrak{s}_\zeta\}_{\zeta \in \U_\infty \cap \KK}$.

Consider now the map $\Phi(x,y) = (P^{N}(x),Q^{M}(y))$. 
Its iterate $\Phi^L$ stabilizes the set $\{(a_{ \ell(\mathfrak{c}) + l \cdot L}, b_{ \ell(\mathfrak{c}) + l \cdot L})\}_{l\ge0}$
hence fixes $\mathfrak{s}_{\zeta(\mathfrak{c})}$ and $Z$. 

Recall that by assumption neither $P$ nor $Q$ are integrable.
By Theorem~\ref{thm:inv-curve}, one can thus find two semi-conjugacies $u,v \in \KK'[T]$ whose degrees are \emph{coprime}, and a polynomial $ R\in \KK'[T]$ such that 
$Z = \{ u(x) = v(y)\}$ and $u \circ P^{N L} = R \circ u$,  $v \circ Q^{M L} = R \circ v$.
Since the degrees of $u$ and $v$ are coprime, $Z$ has a unique branch at infinity which is necessarily $\mathfrak{s}_{\zeta(\mathfrak{c})}$.

Now by~\eqref{eq008}, we have $\varphi_P(x)^n = \hat{P}_n(x) + o(1)$, and $\varphi_Q(y)^m = \hat{Q}_m(y) + o(1)$.  It follows that the restriction to $Z$
of the polynomial $\hat{P}_n(x) - \zeta(\mathfrak{c}) \cdot \hat{Q}_m(y)$ is actually vanishing at the branch at infinity $\mathfrak{s}_{\zeta(\mathfrak{c})}$, hence vanishes identically on $Z$ so that 
$Z = \{ \hat{P}_n(x) = \zeta \cdot \hat{Q}_m(y)\}$. 

Possibly conjugating $R$ by a suitable dilatation, we have thus obtained:
\begin{lemma}\label{lem:at a fixed parameter}
Fix any parameter $t  \in C(\bar{\KK}) \cap U_{v_0}(\mathfrak{c})$ such that neither $P_t$ nor $Q_t$ are integrable. 

Then the curve $\hat{P}_n(x)  = \zeta(\mathfrak{c}) \cdot \hat{Q}_m(y)$ is irreducible in $\A^2$, and 
has a unique branch at infinity given by
\[ \varphi_P(x)^n =  \zeta(\mathfrak{c}) \cdot \varphi_Q(y)^m~.\]
Moreover, one can find  a polynomial $R_t$ of degree $\Delta^{L}$ such that 
\begin{align}
&\hat{P}_{t,n} \circ P_t^{N L} = R_t \circ \hat{P}_{t,n}~; \text{ and }\label{eq:1shot}\\
&(\zeta(\mathfrak{c}) \cdot \hat{Q}_{t,m}) \circ Q_t^{M L} = R_t \circ (\zeta(\mathfrak{c}) \cdot \hat{Q}_{t,m})~. \label{eq:2shot}
\end{align}
Moreover we have
$\hat{P}_{t,n}(a_{\ell(\mathfrak{c})}(t)) = \zeta(\mathfrak{c})\cdot \hat{Q}_{t,m}(b_{\ell(\mathfrak{c})}(t))$.
\end{lemma}

\paragraph*{(III) End of proof when ($\diamond$) is satisfied}
Recall that $n$ and $m$ are determined by our data as the mininal integers such that $n \cdot \mathsf{D}_{P,a}= m \cdot \mathsf{D}_{Q,b}$. 
Also $N$ and $M$ are the least integers such that $\Delta := d^N = \delta^M$.

We fix any branch at infinity $\mathfrak{c} \in \mathcal{B}$, and any Archimedean place $v_0 \in M_\KK$.
Recall the definitions of $L, \zeta(\mathfrak{c})$ and $\ell(\mathfrak{c})$ from Lemma~\ref{lem:image}.

Consider the set $\mathcal{S}$ of all parameters $t \in C(\KK_{v_0})$ such that the following conditions hold:
\begin{itemize}
\item[($\mathcal{P}_1$)]
there exists a polynomial $R_t \in \KK_{v_0}[T]$ of degree $\Delta^L$
satisfying~\eqref{eq:1shot} and~\eqref{eq:2shot} above;
\item[($\mathcal{P}_2$)] 
the curve $\hat{P}_{t,n}(x)  = \zeta(\mathfrak{c}) \cdot \hat{Q}_{t,m}(y)$ is irreducible in $\A^2$ and has 
a unique branch at infinity given by $\varphi_{P_t}(x)^n = \zeta(\mathfrak{c}) \cdot \varphi_{Q_t}(y)^m$;
\item[($\mathcal{P}_3$)]
we have $\hat{P}_{t,n}(a_{\ell(\mathfrak{c})}(t)) = \zeta(\mathfrak{c})\cdot \hat{Q}_{t,m}(b_{\ell(\mathfrak{c})}(t))$.
\end{itemize}

Observe that Lemma~\ref{lem:at a fixed parameter} implies that $\mathcal{S}$ contains all points in $C(\bar{\KK})\cap U_{v_0}(\mathfrak{c})$.
We shall prove that  the set of $t\in C(\KK_{v_0})$ satisfying the three conditions ($\mathcal{P}_1$), ($\mathcal{P}_2$), and ($\mathcal{P}_3$)  is a Zariski-closed subset.

\smallskip

Condition ($\mathcal{P}_3$) states the equality of two regular functions on $C$. Since $C(\bar{\KK})\cap U_{v_0}(\mathfrak{c})$ is infinite,  Condition ($\mathcal{P}_3$) is satisfied for all $t \in C(\KK_{v_0})$
hence for all parameters  $t \in C(\bar{\KK})$.

\smallskip

Let us deal next with Condition ($\mathcal{P}_2$).
Pick any $t \in C'(\KK_{v_0})$. Observe that the affine curve $Z = \{ \hat{P}_{t,n}(x)  =  \zeta(\mathfrak{c})\cdot \hat{Q}_{t,m}(y)\}$ is always irreducible, and has a unique branch at infinity because $n = \deg(\hat{P}_{t,n})$ and $m = \deg(\hat{Q}_{t,m})$ are coprime. The fact that the branch at infinity of $Z$ is given by  $\varphi_{P_t}(x)^n =  \zeta(\mathfrak{c})\cdot\varphi_{Q_t}(y)^m$ is equivalent to the vanishing of the formal Laurent series $ \hat{P}_{t,n}(\varphi_{P_t}^{-1} (\xi \cdot t^m)) - \zeta \hat{Q}_{t,m}(\varphi_{Q_t}^{-1}(t^n)))$ where $\xi$ is any root of unity satisfying $\xi^n =  \zeta(\mathfrak{c})$. 
Since $\varphi_{P_t}^{-1}$ and $\varphi_{Q_t}^{-1}$ are formal Laurent series in $z^{-1}$ whose coefficients are regular functions on $C$ by~\eqref{eq-bott},  the set of $t\in C(\KK_{v_0})$ for which Condition ($\mathcal{P}_2$) holds is a Zariski closed subset of $C$.

\smallskip

To understand Condition ($\mathcal{P}_1$), recall that we wrote $P_t(z) = \alpha(t)^{d-1} z^d + o_t(z^d)$. To simplify notation we assume $NL =1$.
We remark that the equation~\eqref{eq:1shot} given by $\hat{P}_{n} \circ P = R \circ \hat{P}_{n}$ is equivalent to resolution of $n\times d+1$ linear equations
$\mathcal{L}_i$, $i=0, \ldots, n\times d$ obtained by identifying the coefficients in $z^i$  of both sides. 

Look at the equations $\mathcal{L}_{nd}, \mathcal{L}_{n(d-1)}, \ldots, \mathcal{L}_{0}$. 
We get a linear system of the form
\[\begin{cases}
a_0 \alpha^{dn} &= A_0
\\
a_1 \alpha^{(d-1)n} + a_0 B_{10} &= A_1
\\
a_2 \alpha^{(d-2)n} + a_1 B_{21} + a_0 B_{20} &= A_2
\\
\ldots
\\
a_d +  a_{d-1} B_{d,d-1}+ \cdots + a_0 B_{d0} & = A_d
\end{cases}\]
where $B_{ij}$ and $A_i$ are regular functions on $C$. Since $\alpha$ is invertible, we get a \emph{unique} polynomial $R_{P,n}$
satisfying all these equations whose coefficients are regular functions $a_i \in K[C]$. 

It follows from this discussion that~\eqref{eq:1shot} is solvable iff the coefficients of $R_{P,n}$ satisfy the linear equations
$\mathcal{L}_i$, $i=0, \ldots, n\times d$. 
In other words, the set of parameters $t$ such that~\eqref{eq:1shot} is solvable is the intersection of the zero locus of
finitely many regular functions on $C$ hence is Zariski closed.
A similar argument applies to~\eqref{eq:2shot}, and this concludes the proof.

\medskip

We have proved that the set of parameters for which the three conditions ($\mathcal{P}_1$), ($\mathcal{P}_2$), and ($\mathcal{P}_3$)
hold is equal to $C(\bar{\KK})$. We also observe  that the coefficients to the polynomial $R_t$ depends algebraically on $t$, since all coefficients
$\alpha$, $A_i$ and $B_{ij}$ are regular functions on $C$ in the linear system above. 
This completes the proof of (2) $\Rightarrow$ (3) of Theorem~\ref{tm:unlikely}.

\paragraph*{(IV) Base change and Condition ($\diamond$)}

Let us assume now that $P$ and $Q$ are arbitrary families parametrized by $C$. Pick any base change 
$C' \to C$ such that the induced families on $C'$  satisfy the condition ($\diamond$).
By what precedes, there exists a branch at infinity $\mathfrak{c}$ of $C'$, integers $N,M, n, m, L, l(\mathfrak{c})$, and a root of unity $\zeta(\mathfrak{c})$ such that 
for all $t \in C'(\bar{\KK})$, there exists a polynomial $R_t$ satisfying the three Conditions ($\mathcal{P}_1$), ($\mathcal{P}_2$), and ($\mathcal{P}_3$).

Since $R_t$ is uniquely determined (for instance by~\eqref{eq:1shot}), we have $R_t = R_{t'}$ whenever $\pi(t) = \pi(t')$.
It follows that the coefficients of $R_t$ are regular functions on $C'$ that are pull-back of regular functions on $C$, and
we obtain a family of polynomial $\tilde{R}$ parametrized by $C$ satisfying the three Conditions ($\mathcal{P}_1$), ($\mathcal{P}_2$), and ($\mathcal{P}_3$)
for the dynamical pairs $(P,a)$ and $(Q,b)$ as required.

\paragraph*{Another approach in the case of a single branch}

The previous proof is quite intricate. One difficulty lies in dealing with the existence of several branches of $C$ at infinity for which the roots of unity $\zeta(\mathfrak{c},l)$ might a priori be different. 

Assume that we may find an integer $l$ and $\xi$ such that  $\xi = \zeta(\mathfrak{c},l)$ for all $\mathfrak{c}$. Note that this is in particular the case when $C$ has a single place at infinity. 

The function $\hat{P}_n(P^{Nl}(a)) - \xi \hat{Q}_m(Q^{Ml}(b))$ is regular on $C$ and 
vanishes at all branches at infinity for $l$ large enough by Lemma~\ref{lem:345}. It is
hence equal to $0$.

This argument simplifies Step II and avoids to rely on Xie's theorem.


\subsection{More precise forms of Theorem~\ref{tm:unlikely}}
The proof that we have developed of Theorem~\ref{tm:unlikely} actually gives more. 
We first observe that the arguments of the previous section yields
\begin{theorem}\label{tm:precise intrication}
Pick any irreducible affine curve $C$ which is defined over a number field $\KK$. Assume that all
its branches at infinity are also defined over $\KK$.

Let $(P,a)$ and $(Q,b)$ be active non-integrable dynamical \emph{entangled} pairs 
parametrized by $C$ of respective degree $d,\delta\geq2$. 

Then there exist  coprime integers $n$ and $m$ such that $n \cdot \mathsf{D}_{P,a} = m \cdot \mathsf{D}_{Q,b}$
and coprime integers $N, M$ such that $\Delta := d^N = \delta^M$. 

Furthermore, let $\ell$ and $L$ be any integers such that 
\begin{itemize}
\item
$\Delta^{\ell} \times \ord_\mathfrak{c}(\mathsf{D}_{P,a}) >
\ord_\mathfrak{c}(\mathsf{D}_{P})$ for all branch at infinity of $C$;
\item
$M^{\ell + L}_\Delta(\xi) = M^{\ell}_\Delta(\xi)$ for any root of unity in $\xi \in \KK$.
\end{itemize}
Then one can find a root of unity $\zeta \in \KK$, 
and a family of polynomial $R_t$ of degree $\Delta^{L}$ such that 
\begin{align}
&\hat{P}_{t,n} \circ P_t^{N L} = R_t \circ \hat{P}_{t,n}~; \text{ and }\label{eq:1shot'}\\
&(\zeta \cdot \hat{Q}_{t,m}) \circ Q_t^{M L} = R_t \circ (\zeta \cdot \hat{Q}_{t,m})~. \label{eq:2shot'}
\end{align}
Moreover we have
$\hat{P}_{t,n}(P^{\ell N}(a(t))) = \zeta \cdot \hat{Q}_{t,m}(Q^{\ell M}(b(t)))$.
\end{theorem}
In the statement we wrote $M_\Delta (z) = z^\Delta$.
\begin{remark}
Observe that the integers $n,m, N$ and $M$ are uniquely determined by the condition to be coprime, and that $L$ can be bounded from above by a constant depending 
only on $[\KK:\Q]$.
\end{remark}

We also have the next result, compare with~\cite[Theorem 1.3]{BD}. 
\begin{theorem}\label{tm:unlikelyprecise}
Pick any irreducible affine curve $C$ defined over a number field $\KK$. Let $(P,a)$ and $(Q,b)$ be active, non-integrable monic centered dynamical pairs parametrized by $C$ of respective degree $d,\delta\geq2$. Then, the following are equivalent:
\begin{enumerate}
\item $\mathrm{Preper}(P,a,\bar{\KK})\cap\mathrm{Preper}(Q,b,\bar{\KK})$ is an infinite subset of $C(\bar{\KK})$;
\item $\mathrm{Preper}(P,a,\bar{\KK})=\mathrm{Preper}(Q,b,\bar{\KK})$;
\item the height functions $h_{P,a}$ and $h_{Q,b}$ are proportional;
\item there exists a non-repeating sequence $t_n\in C(\bar{\KK})$ such that
\[\lim_{n\to\infty}h_{P,a}(t_n)=\lim_{n\to\infty}h_{Q,b}(t_n)=0;\]
\item there exist integers $n,m\geq1$ such that for any place $v\in M_\KK$,
\[n\cdot \Delta g_{P,a,v}=m\cdot\Delta g_{Q,b,v},\]
 as positive measures on $C^{v,\an}$;
\item there exist integers $n,m\geq1$ such that for all places $v\in M_\KK$, we have
\[n\cdot g_{P,a,v} =  m\cdot g_{Q,b,v} \ \text{ on } C^{\an}_v;\]
\item  there exist integers $n,m \ge1$, $N,M\geq1$, $r,s\ge0$, and families $R,\tau$ and $\pi$ of polynomials parametrized by $C$ such that
\[\tau\circ P^N= R\circ\tau \ \text{ and } \ \pi\circ Q^M= R\circ\pi,\]
and $\tau (P^{r}(a))= \pi(Q^{s}(b))$.
\end{enumerate}
\end{theorem}

\begin{proof}
The fact that (1) $\Leftrightarrow$ (2) $\Leftrightarrow$ (3) $\Leftrightarrow$ (6)
is the content of Theorem~\ref{thm:partial1}.  The implication (3) $\Rightarrow$ (7) is the content of \S\ref{sec:biggest proof}, and (7) $\Rightarrow$ (2) is clear.
The implication (1) $\Rightarrow$ (4) is clear as we may take $t_n \in \mathrm{Preper}(P,a,\bar{\KK})\cap\mathrm{Preper}(Q,b,\bar{\KK})$
and $h_{P,a}(t_n) = h_{Q,b}(t_n) =0$.  The implication (4) $\Rightarrow$ (5) is a direct application of Thuillier-Yuan's Theorem~\ref{tm:yuan}, see \S \ref{sec:1on3}.

Let us prove (5) $\Rightarrow$ (6). The argument is contained in \S \ref{sec:1on3}, but we repeat it for the convenience of the reader.
Suppose $n\cdot \Delta g_{P,a,v}=m\cdot\Delta g_{Q,b,v}$ for all places $v\in M_\KK$.
If $v$ is archimedean, then $n\cdot g_{P,a,v} =  m\cdot g_{Q,b,v}$
by Theorem~\ref{thm:identical measures}, and $n \cdot \mathsf{D}_{P,a} = m \cdot \mathsf{D}_{Q,b}$ by Theorem~\ref{atoneplace}. 
Pick any place $v \in M_\KK$, then
$n\cdot g_{P,a,v} - m\cdot g_{Q,b,v}$ is a harmonic function which extends continuously to $\bar{C}^{\an}_v$ hence is constant by the maximum principle.
Since it vanishes on $\mathrm{Preper}(P,a,\bar{\KK})$ we get $n\cdot g_{P,a,v} = m\cdot g_{Q,b,v}$ as required.
\end{proof}


\section{Proof of Theorem~\ref{tm:unlikely-car0}}\label{sec:specialization}

Let us recall its statement. 

\unlikelyspecial*

We let $K$ be a field of characteristic $0$ over which $C$, $(P,a)$ and $(Q,b)$ are defined. 
As in \S \ref{sec:overviewB} the implication (2) $\Rightarrow$ (1) is easy. 
We thus assume that $\preper(P,a,\bar{K}) \cap \preper(Q,b,\bar{K})$ is infinite. To prove (2),
we shall use a specialization argument to reduce to the case $K$ is a number field which was treated in the previous section. 
We proceed in four steps.

\medskip

\noindent {\bf 1.}
Fix any embedding $ C \subset \A^N$ and pick polynomials $a_i, b_j \in K[z_1, \ldots, z_N]$, and
$V_1, \ldots, V_M \in K[z_1, \ldots, z_N]$
such that $ C$ is the scheme theoretic intersection of the hypersurfaces $\{ V_i = 0\}$, $P(T) = \sum_i a_i|_C T^i$, and
$Q(T) = \sum b_j|_C T^j$. The completion $\bar{C}$ 
of $C$ in $\p^N$ is defined as the vanishing of the homogeneous polynomials
$\bar{V}_i(z_0, \cdots, z_N) = z_0^{\deg(V_i)} V_i(\frac{z_1}{z_0},\cdots,\frac{z_N}{z_0})$. 
By choosing an adequate embedding we may suppose that $\bar{C}$
is smooth near any of its point lying on the hyperplane at infinity. 

Choose any finitely generated $\Q$-algebra $R$ contained in $K$ that contains all coefficients defining our data so that 
the polynomials $a_i, b_j$ and $V_l$ belong to $R[z_1, \cdots , z_N]$.
We may assume that $K= \mathrm{Frac}(R)$.
Then $\Lambda:=\spec (R)$ is an affine variety defined over $\Q$ which is in general neither reduced nor irreducible. 
We let $\mathfrak{C}$ (resp. $\bar{\mathfrak{C}}$) be the affine (resp. projective) $\Lambda$-scheme defined by the vanishing 
of all polynomials $V_i$ (resp. $\bar{V}_i$). 
Up to replace $\Lambda$ by a Zariski open dense subset (i.e. to replace $R$ by a larger ring), we may assume by generic flatness 
and~\cite[Theorem~9.9]{MR0463157}
that $\Lambda$ is irreducible, and $\bar{\mathfrak{C}} \to \Lambda$ is a flat family of curves 
such that $\bar{C} = \bar{\mathfrak{C}}_K$ and $C = \mathfrak{C}_K$.

The dynamical pairs $(P,a)$ and $(Q,b)$ induce 
regular maps $P_\mathfrak{C}, Q_\mathfrak{C} \colon \p^1_\mathfrak{C} \to  \p^1_\mathfrak{C}$
of degree $d$ and $\delta$ respectively, and regular functions $\mathfrak{a}, \mathfrak{b} \colon \mathfrak{C} \to \p^1_\mathfrak{C}$.

\medskip

\noindent {\bf 2.}
Any maximal ideal $s$ of $R$ corresponds to a closed point in $\Lambda$.
We denote by $\bar{C}_s$, $C_s$, $P_s$, $a_s$, $Q_s$ and $b_s$ the specializations of the corresponds objects at $s$. 
These are all defined over the residue field $\kappa(s)$ of $s$ which is a finite extension of $\Q$.

Note that the intersection of $\bar{C}$ with the hyperplane at infinity is a finite set of points defined over an algebraic extension of $K$. 
Enlarging $R$ if necessary, we may suppose that all these points are defined over $R$. It follows that the Zariski closure in $\bar{\mathfrak{C}}$ 
of a branch $\mathfrak{c}$ at infinity of $C$ defines a branch at infinity $\mathfrak{c}_s$ of $\mathfrak{C}_s$ for all $s$.

\begin{lemma}\label{lem:specialization divisor}
There exists a Zariski open dense subset $U$ of $\Lambda$ such that 
$D(P,a)_s= D(P_s,a_s)$ for all $s\in U$.
\end{lemma}
\begin{proof}
We may suppose that $D(P,a)$ is non zero and the family is given under the form~\eqref{eq:critical form} so that $P= P_{c,\alpha}$ 
for some $c,\alpha \in R$.

Pick any branch at infinity of  $C$ lying in the support of $D(P,a)$. By Proposition~\ref{prop:classic-estim}, 
for any $q$ such that \[ - \ord_\mathfrak{c} (P^q(a)) > - \ord_\mathfrak{c} (c, \alpha)~,\]
we have  $\ord_\mathfrak{c}(D(P,a)) = \frac1{d^q}\ord_\mathfrak{c} (P^q(a))$. 
Since for any $\phi \in K(C)$, one has $\ord_\mathfrak{c}(\phi) = \ord_{\mathfrak{c}_s} (\phi_s)$
for all $s$ in a Zariski open dense subset of $\Lambda$, we conclude that 
\[ 
\ord_\mathfrak{c}(D(P,a)) = \frac1{d^q}\ord_\mathfrak{c} (P^q(a)) = 
\frac1{d^q}\ord_{\mathfrak{c}_s} (P_s^q(a_s)) = \ord_{\mathfrak{c}_s}(D(P_s,a_s)) \]
for all $s$ in a Zariski open dense subset.

This concludes the proof when all branches at infinity lie in the support of $D(P,a)$.
Otherwise, let $\mathfrak{c}_1, \ldots, \mathfrak{c}_k$ be the set of branches at infinity of $C$
such that $\ord_{\mathfrak{c}_j}(D(P,a)) =0$. 
By Lemma~\ref{lem:divisor}, we may fix an integer $N\ge 1$, such that for any local parametrization $\theta$ of a branch $\mathfrak{c}_j$
the series  $t^N \times P_{\theta(t)}^q(a(\theta(t)))$ has no pole at $0$.

Pick any regular function $A \colon C\to \A^1$
whose poles are contained in $\mathfrak{c}_1, \ldots, \mathfrak{c}_k$ and have order $\ge N$, and whose zeroes are included in the support
of $D(P,a)$. Up to replace $N$ by a larger integer, such a function always exist. 
In the sequel, we identify this function to a family of dilatation $z \mapsto A z$ on $\A^1_C$.
Observe that the conjugated pair is defined by the family $\tilde{P} = A^{-1} \circ P\circ A$ which is still parametrized by $C$, and by 
the new marked point $\tilde{a} = A^{-1}_t(a(t))$. By construction, for any parametrization $\theta$ of a branch $\mathfrak{c}_j$, we have
\[ \tilde{P}^q_{\theta(t)} (\tilde{a}({\theta(t)})) =
A_{\theta(t)}^{-1} \circ P^q_{\theta(t)}(a({\theta(t)}))\] 
which is regular near $0$ for all $q$. Observe that this remains true by specialization at any point $s$
so that 
\[ 0 \le   \ord_{\mathfrak{c}_j}(\tilde{P}^q(\tilde{a})) \le  \ord_{\mathfrak{c}_{j,s}}(\tilde{P}_s^q(\tilde{a}_s)) \] for all $q$ , hence
\[ 0 = \ord_{\mathfrak{c}_j}(D(P,a)) =  \ord_{\mathfrak{c}_{j,s}}(D(P_s,a_s)) \]
for all $s$. 
\end{proof}

\medskip

\noindent {\bf 3.}
The key point of the proof is contained in the next result. 
\begin{proposition}\label{prop:specialization fiber}
If Condition ($\vartriangle$) is satisfied,  for any archimedean place $v$,
the set 
\[\{ s \in \Lambda^{\an}_v, \preper(P_s,a_s) \cap \preper(Q_s,b_s) \text{ is infinite}\}\] contains a non-empty open subset of $\Lambda^{\an}_v$.
\end{proposition}

\begin{proof}
In the whole proof, we argue in the analytifications  with respect to a fixed archimedean place. To simplify notation, we drop the reference to this place.  
We view $\pi \colon\mathfrak{C}\to\Lambda$ as a flat family of complex algebraic curves. The closure in $\mathfrak{C}$ 
of any closed point $a \in C$ determines an irreducible hypersurface $Z(a)$ such that $Z(a)\to\Lambda$ is finite-to-one onto a Zariski dense open subset of $\Lambda$.
In general this projection is not proper.

Let $x_i$ be a sequence of distincts points in $\preper(P,a,\bar{K}) \cap \preper(Q,b,\bar{K})$. 
Then $Z_i := Z(x_i)$ is a sequence of distinct irreducible hypersurfaces of $\mathfrak{C}$ included in 
the space $\preper(P_\mathfrak{C},\mathfrak{a}) \cap \preper(Q_\mathfrak{C},\mathfrak{b})$. 
\begin{lemma}\label{lem:321}
For each $i$, the projection map $Z_i \to \Lambda$ is finite-to-one, proper and surjective.
\end{lemma}

\begin{proof}[Proof of Lemma~\ref{lem:321}]
It is sufficient to prove that $Z_i \to \Lambda$ is proper. We fix any archimedean place $v$, and prove that 
$Z_{i,v}^{\an} \to \Lambda^{\an}_v$ is proper, see~\cite[Proposition~3.4.7]{berko-book}. Pick any point $\mathfrak{c} \in \bar{\mathfrak{C}}$
which defines a branch at infinity of some curve of $C_{s_0}$. 

Observe that near the point $\mathfrak{c}$, we have
\[\frac1{d^q} \log^+|P^q(a)| = \frac{1}{d^q}\ord_\mathfrak{c} (P^q(a))  \log |z|^{-1} + O(1).\]
The condition ($\vartriangle$) thus implies $\frac1{d^q} \log^+|P^q(a)| > G(P)$ in a neighborhood $U$ of $\mathfrak{c}$.
This implies $g_{P,a} >0$ on $U$ by Proposition~\ref{prop:GreenBottcher}, hence $Z_i \cap U = \varnothing$. 
\end{proof}

Consider the set $\mathcal{F} := \{ s \in \Lambda^{\an}, \preper(P_s,a_s) \cap \preper(Q_s,b_s) \text{ is finite}\}$. We need to 
exhibit an open euclidean subset of $\Lambda^{\an}_v$ which does not intersect the countable set $\mathcal{F}$.
Pick any $s_0 \in \mathcal{F}$. By the previous lemma, each $Z_i$ intersects the fiber $C_{s_0}$
so that we may find a closed point $t_0 \in C_{s_0}$ which belongs to infinitely many $Z_i$'s. 

By Theorem~\ref{thm:finite branches}, we may thus suppose that $0$ is a fixed point for the family $P$ which is super-attracting at $t_0$, whose local degree
is not locally constant, and $Z_i$ is an irreducible component of $\{P^{n_i}(a) = 0\}$ with $n_i \to \infty$.

The local degree function $t \mapsto \deg_0(P_t)$ is upper-semi continuous with respect to the Zariski topology on $\mathfrak{C}$. Restricting $\Lambda$ if necessary, 
we thus suppose that the locus $\mathcal{E} \subset \mathfrak{C}$ where $\deg_0(P_t)$ is not locally constant is a smooth hypersurface, 
which intersects each curve $C_s$ transversally, and whose projection onto $\Lambda$ is proper. 
We may also assume that $\deg_0(P_t) = \nu$ if $t$ belongs to $\mathcal{E}$, 
and $\deg_0(P_t) = \mu <\nu$ otherwise.

\begin{lemma}\label{lem:422}
Any point $t_1 \in \mathcal{E}$ admits a (euclidean) neighborhood $V$
such that 
$Z_i \cap \mathcal{E} \cap V= \{t_1\}$ for all $i$.
\end{lemma}

By definition $\mathcal{E} \cap C_{s_0}$ is finite, and we may apply the previous lemma
to each point in this set. By properness of $\mathcal{E} \to \Lambda$, we get a euclidean neighborhood $W$ of $s_0$ such that 
for any $s \in W$ distinct from $s_0$, for any $i$ the intersection $Z_i \cap C_s$ 
does not lie in $\mathcal{E}$. By Theorem~\ref{thm:finite branches}, no point in $C_s$ may thus contain infinitely many hypersurfaces $Z_i$. 
This implies that $C_s \cap \{ Z_i\}$ is infinite which shows that the euclidean open set $W \setminus \{s_0\}$ does not intersect $\mathcal{F}$.
\end{proof}

\begin{proof}[Proof of Lemma~\ref{lem:422}]
Locally near $t_1$, choose local analytic coordinates $(\lambda, t)$ 
such that $\mathcal{E} = \{ t=0\}$, $\pi (\lambda, t) = \lambda$, and
\[P_{\lambda,t}(z)= b_\mu(\lambda,t) z^\mu + \cdots + b_{\nu-1}(\lambda,t) z^{\nu-1} +  z^\nu ( 1+ h(\lambda, t,z))\]
where $ t | b_j$ for all $j = \mu, \cdots, \nu-1$, and  $h(0) = 0$.

The restriction of the family $P$ to $\mathcal{E}$ is thus equal to $P_{\lambda,0}(z) =z^\nu ( 1+ h(\lambda, 0,z))$
and we may find an analytic change of coordinates $w= z + O_\lambda(z)$ such that 
$P_{\lambda,0}(w) = w^\nu$, see e.g.~\cite[Theorem~1.3]{MR3186512}. In an open neighborhood $V$ of $t_1$, for any integer $n\ge 1$ the intersection of the hypersurface
$\{ P^n(a) =0\}$ with $\mathcal{E}$ is thus determined by the equation $\{a^\nu =0\} $.
By reducing $V$ if necessary, we thus have $Z_i \cap \mathcal{E} \cap V\subset \{ a =0\} = \{t_1\}$.
\end{proof}

\medskip

\noindent {\bf 4.} 
One can now conclude the proof as follows. Fix any archimedean place $v$. 
By Proposition~\ref{prop:specialization fiber} there exists an open set $U \subset \Lambda^{\an}_v$ such that any closed point 
$s\in \Lambda \cap U$ satisfies  $\preper(P_s,a_s) \cap \preper(Q_s,b_s)$ is infinite, and $D(P,a)_s= D(P_s,a_s)$.
 
 For any $s \in \Lambda \cap U$,  Theorem~\ref{tm:precise intrication} shows that for any $t \in C_s$ we have
\begin{align}
&\hat{P}_{t,n} \circ P_t^{N L} = R_t \circ \hat{P}_{t,n}~; \text{ and }\label{eq:21}\\
&(\zeta \cdot \hat{Q}_{t,m}) \circ Q_t^{M L} = R_t \circ (\zeta \cdot \hat{Q}_{t,m})\label{eq:22} \\
&\hat{P}_{t,n}(P^{\ell N}(a(t))) = \zeta \cdot \hat{Q}_{t,m}(Q^{\ell M}(b(t)))\label{eq:23}
\text{ for all } l \ge 1,
\end{align}
for some polynomial $R_t$, for some integers $n,m,N,M, \ell, L$ and for a root of unity $\zeta$.

By Lemma~\ref{lem:specialization divisor}, the four integers $n,m, N$ and $M$ are actually independent on $s$ so that $\Delta := d^N = \delta^M$ too.  
We need to argue that $\zeta$ may be chosen uniformly in $s$.

To see this, we go back to Lemma~\ref{lem:344}. For each branch at infinity $\mathfrak{c}$ such that 
$\ord_\mathfrak{c}(\mathsf{D}_{P,a})>0$ we pick an integer $l(\mathfrak{c})>0$ such that 
\[
\Delta^{l(\mathfrak{c})} \times \ord_\mathfrak{c}(\mathsf{D}_{P,a}) >
\ord_\mathfrak{c}(\mathsf{D}_{P})
\text{ and }
\Delta^{l(\mathfrak{c})} \times \ord_\mathfrak{c}(\mathsf{D}_{Q,b}) >
\ord_\mathfrak{c}(\mathsf{D}_{Q})
~.\]
Lemma~\ref{lem:344} then yields for each $l \ge l(\mathfrak{c})$ a root of unity $\zeta(\mathfrak{c},l,U)$
such that 
\[\varphi_{P_t}\left(P^{lN}_t(a(t))\right)^n = \zeta(\mathfrak{c},l,U) \times  \varphi_{Q_t}\left(Q^{lM}_t(b(t))\right)^m
\]
for all $s \in U$, and all $t$ in the branch $\mathfrak{c}$ on $C_s$.
In particular Lemma~\ref{lem:image} holds uniformly on $U$: 
we can find two integers $\ell$ and $L$ such that 
$\zeta(\mathfrak{c},\ell,U) = \zeta(\mathfrak{c},\ell + l\cdot L,U)$
for all $l \ge0$.  

It follows that for a Zariski dense set of points $s \in \Lambda\cap U$, 
there exist $\ell, L$ and $\zeta$ (independent on $s$) satisfying~\eqref{eq:21},~\eqref{eq:22}, and~\eqref{eq:23}. 
Since these equations are linear in the coefficients of $R_t$, one may choose these coefficients in the fraction field of $\Lambda$, and
we conclude that these relations are actually satisfied for all $s \in \Lambda$. 

This implies the next theorem and concludes the proof of Theorem~\ref{tm:unlikely-car0}.

\begin{theorem}\label{tm:precise intrication-car0}
Pick any irreducible affine curve $C$ defined over an algebraically closed field $K$ of characteristic $0$. 
Let $(P,a)$ and $(Q,b)$ be active non-integrable dynamical \emph{intricated} pairs 
parametrized by $C$ of respective degree $d,\delta\geq2$ such that $(P,a)$ satisfies the condition ($\vartriangle$).

Then there exist  coprime integers $n$ and $m$ such that $n \cdot \mathsf{D}_{P,a} = m \cdot \mathsf{D}_{Q,b}$
and coprime integers $N, M$ such that $\Delta := d^N = \delta^M$. 

Furthermore, there exists root of unity $\zeta$
satisfying  $M^{\ell + L}_\Delta(\zeta) = M^{\ell}_\Delta(\zeta)$ for some integers 
$\ell$ and $L$; 
and a family of polynomial $R_t$ of degree $\Delta^{L}$ such that 
\begin{align}
&\hat{P}_{t,n} \circ P_t^{N L} = R_t \circ \hat{P}_{t,n}~;\label{eq:1shot-car0}\\
&(\zeta \cdot \hat{Q}_{t,m}) \circ Q_t^{M L} = R_t \circ (\zeta \cdot \hat{Q}_{t,m})~;  \text{ and } \label{eq:2shot-car0}\\
&\hat{P}_{t,n}(P^{\ell N}(a(t)) = \zeta \cdot \hat{Q}_{t,m}(Q^{\ell M}(b(t)). \label{eq:3shot-car0}
\end{align}
\end{theorem}


\section{Further results and open problems} \label{sec:further-open}

\subsection{Effective versions of the theorem} 

Suppose that $(P,a)$ and $(Q,b)$ are algebraic dynamical pairs parametrized by a curve $C$ defined over a number field $\KK$. 
If the two pairs are entangled then we have seen that the two heights $h_{P,a}$ and $h_{Q,b}$ are proportional so that 
$\preper (P,a) \cap \preper(Q,b) = \{ h_{P,a} =0\}= \{h_{Q,b} =0\}$. 

When the two pairs are not entangled then $H:= h_{P,a} + h_{Q,b}$ is an Arakelov height whose essential minimum is positive (apply Theorem~\ref{tm:unlikelyprecise} (4)). 
In particular, there exists $\epsilon >0$ such that $\preper (P,a) \cap \preper(Q,b) \cap \{ H \le \epsilon\}$ is a finite subset of $C(\bar{\KK})$.

It would be very interesting to explore whether one can obtain uniform bounds on $\preper (P,a) \cap \preper(Q,b) $ when $P,Q,a,b$ vary in families. 

\begin{conj}
Let $P$ be any non-integrable family of polynomials of degree $d\ge 2$. Then there exists a constant $C = C(d,N)$ such that 
for any $a,b$ of degree $\le N$ the set  $\preper (P,a) \cap \preper(P,b)$ is either infinite or its cardinality is $\le C$.
\end{conj}

Using Zhang's pairing for heights, Fili \cite{Fili} gave an upper bound on the cardinality of $\preper(z^2+c,0)\cap\preper(z^2+c,1)$. The latter set was later computed exactly by Buff~\cite{Buff-PCF-unicritical}. DeMarco, Krieger and Ye~\cite{DKY-quadratic} have recently obtained uniform bounds for a similar problems: 
by estimating Zhang's pairings, they managed to prove the existence of a constant $B>1$ such that, for any two complex parameters $c_1\neq c_2$, the set $\preper (z^2+c_1) \cap \preper(z^2+c_2)$ has cardinality $\le B$.


\subsection{The integrable case} 

In all statements above, we have supposed that the families were not integrable. 
Suppose that $(P,a)$ is an active integrable pair parametrized by a curve $C$, and pick any any active pair $(Q,b)$ which is entangled with it.
We shall describe all such possibilities up to a base change. 

\smallskip

\noindent {\bf 1.}
The family $Q$ is also integrable. 

Observe that the bifurcation locus of $P$ is equal to $\{ t, \, a(t) \in J(P)\}$ by Proposition~\ref{prop:bifJstab}. It thus contains a smooth curve, and $Q$ is integrable by Theorem~\ref{tm:rigidaffine}.

\smallskip

\noindent {\bf 2.} 
Reduction to the monomial case.

Replacing $C$ by an open affine subset and doing a base change if necessary, one may conjugate the two families to constant families, so that 
$(P,Q) = (M_d, M_\delta), (M_d, \pm T_\delta)$, or $=(\pm T_d, \pm T_\delta)$ (up to a permutation).
Since $T_d$ and $M_d$ are semi-conjugate for all $d$, and $-T_d$ and $-M_d$ are semi-conjugate for all odd $d$, we may also
make a base change to reduce the situation to $(P,Q) = (M_d, M_\delta)$.

\smallskip

\noindent {\bf 3.} 
Apply Lang's result.

The entanglement of the pairs $(M_d,a)$ and $(M_\delta,b)$ says that the set of $t \in C$ such that 
both $a(t)$ and $b(t)$ are roots of unity is infinite. Fix any $t_0$ in this set, and write $\zeta = a(t_0)^{-1}, \xi = b(t_0)^{-1}$
so that $\zeta, \xi \in \mathbb{U}_\infty$.

Consider the algebraic curve $ D= \{(\zeta a(t), \xi b(t)), \, t \in C \}$ inside $\A^2$.
Then $D$ contains the point $(1,1)$ and infinitely many other points for which both coordinates are roots of unity. 
By~\cite{MR190146}, the curve $D$ is given by the equation $D = \{x^n y^m=1\}$ for some non-zero integers $n,m\in \Z$.
We conclude that $a(t)^n b(t)^m$ is a constant equal to some root of unity.  


\subsection{Algorithm}

In this section we discuss briefly the possibility of constructing an algorithm deciding whether two dynamical pairs
$(P,a)$ and $(Q,b)$ are entangled. Suppose these two pairs are defined over an algebraic curve $C$ defined over a number field
and have respective degree $d$ and $\delta$.

\smallskip

\noindent {\bf 1.} Decide whether the pairs are active or not.  

To see whether $(P,a)$ is active, we first make a base change and conjugate the family to polynomials of the form~\eqref{eq:defpoly}.
Under this parametrization, $(P,a)$ is passive iff $P$ is a constant family and $a$ is a constant.

If either $(P,a)$ or $(Q,b)$ is passive, then we stop. 
Otherwise we proceed to the next step.

\smallskip

\noindent {\bf 2.} Compare the divisors $\mathsf{D}_{P,a}$ and $\mathsf{D}_{Q,b}$.

We first compute $\mathsf{D}_{P,a}$. Pick any branch $\mathfrak{c}$ of $C$ at infinity. 
If $\ord_\mathfrak{c}(\mathsf{D}_{P,a}) >0$, then it is posible to compute $\ord_\mathfrak{c}(\mathsf{D}_{P,a})$ in finitely many steps.  Indeed, for all $q\geq1$ sufficiently large, we have $\ord_\mathfrak{c}(\mathsf{D}_{P,a})=-d^{-q}\ord_\mathfrak{c}(P^q(a))$ in this case.

We do not know however whether one can decide if $\mathfrak{c}$ belongs to the support of  $\mathsf{D}_{P,a}$ or not.
To continue our algorithm, we shall thus suppose that $C$ has a single place at infinity\footnote{this works when $C= \A^1$ for instance}. 

If $\mathsf{D}_{P,a}$ and $\mathsf{D}_{Q,b}$ are not proportional, then we stop. 
Otherwise we find the smallest integers $n$ and $m$ such that $n\cdot \mathsf{D}_{P,a}= m\cdot \mathsf{D}_{Q,b}$.

\smallskip

\noindent {\bf 3.} Solving a linear system. 

We first determine a field of definition $\KK$ for the two pairs. We then decompose $d$ and $\delta$ into prime factors, and pick coprime integers $N,M$ so that
$\Delta:= d^N = \delta^M$ (if there are no such integer then we stop). We next determine $\ell$ and $l$ such that the conditions
\begin{itemize}
\item
$\Delta^{\ell} \times \ord_\mathfrak{c}(\mathsf{D}_{P,a}) >
\ord_\mathfrak{c}(\mathsf{D}_{P})$ for all branch at infinity of $C$;
\item
$M^{\ell + L}_\Delta(\xi) = M^{\ell}_\Delta(\xi)$ for any root of unity in $\xi \in \KK$.
\end{itemize}
are both satisfied. Note that $\ell$ and $l$ are bounded solely in terms of $\mathsf{D}_{P,a}$, $\mathsf{D}_{P}$ (which can be computed as in Step 2), and $[\KK:\Q]$.

Then we solve the equations~\eqref{eq:1shot'} and~\eqref{eq:2shot'} with the family of polynomials $R_t$ being the unknowns.
Note that these equations are \emph{linear}: if they are not solvable then we stop our algorithm. Otherwise, we produce
$R_t$. If we have $\hat{P}_{t,n}(P^{\ell N}(a(t))) = \zeta \cdot \hat{Q}_{t,m}(Q^{\ell M}(b(t)))$, then we conclude that the two pairs are entangled.


\subsection{Application to Manin-Mumford's problem}

Let $f \colon X \to X$ be a surjective endomorphism of a projective variety (of dimension $ \ge 2$). 
The Manin-Mumford problem for $f$ concerns the classification of all $f$-invariant subvarieties of $X$,
and asks whether an irreducible subvariety $Z$ containing a Zariski dense subset of $\preper(f)$
is itself preperiodic. 

This problem has been explored for a polarized endomorphisms\footnote{i.e. whose action on the Neron-Severi space admits an ample class as an eigenvector}, see~\cite{GTZ,Pazuki,GNY,GNY2};
and for H\'enon maps~\cite{Favre-Dujardin}.

\smallskip

Let $(P,a)$ and $(Q,b)$ be two dynamical pairs with \emph{fixed} polynomials and marked points parametrized by the same affine curve $C$. 
Observe that  $f(z,w)= (P(z),Q(w))$ induces an endomorphism\footnote{it is polarized iff $\deg(P) = \deg(Q)$} of $\p^1 \times \p^1$. 
Denote by  $Z$ the closure in $\p^1 \times \p^1$ of the image of $C$ by the map $t \mapsto (a(t), b(t))$, so that 
$Z$ contains infinitely many preperiodic points iff $(P,a)$ and $(Q,b)$ are entangled.
Ghioca, Nguyen, and Ye~\cite{GNY} have proved that this occurs iff $\deg(P)^N = \deg(Q)^M$ for some integers $N, M$ and 
$Z$ is preperiodic for the endomorphism $(P^N,Q^M)$.\footnote{The rational case was further explored by Mimar~\cite{Mimar-13}}.

Using our results, we obtain the following extension of their theorem.

\begin{theorem}\label{thm:partialDMM}
Let $P$ and $Q$ be two algebraic families of non-integrable polynomials of degree $\ge 2$ defined over a number field $\KK$ 
parametrized by a smooth affine curve $C$.

Suppose that $Z$ is an irreducible curve  of $\A^2 \times C$ which projects surjectively onto $C$, and
contains infinitely many preperiodic points of the endomorphism $(z,w,t) \mapsto (P_t(z), Q_t(w),t)$. 

Then there exist two integers $N, M \ge 1$ such that $\deg(P)^N = \deg(Q)^M$, 
and an algebraic surface $S$ containing $Z$ which is preperiodic under 
$F(z,w,t) := (P^N_t(z), Q^M_t(w),t)$.
\end{theorem}

\begin{remark}
The statement may look awkward from the perspective of the dynamical Manin-Mumford conjecture. However when $\deg(P) > \deg(Q)$, the only irreducible surfaces $S$ which are invariant under
$(z,w,t) \mapsto (P_t(z),Q_t(w),t)$ and project onto $C$ are of the form $\A^1 \times \{ w_0\} \times C$ or $\{ z_0\} \times \A^1 \times C$. Indeed for a generic point $t_0$,   
the curve $S \cap \A^2 \times \{ t_0\}$ is fixed by  the endomorphism $(P_{t_0},Q_{t_0})$ which admits only the vertical and horizontal fibers of $\p^1 \times \p^1$
as invariant classes.
\end{remark}

\begin{remark}
Suppose conversely that  $Z$ is included in an algebraic surface $S$ which projects onto $C$ is fixed by $F$.
We obtain a family of affine curves $S_t, t \in C$ with an algebraic dynamical system $F_t \colon S_t \to S_t$. 
The latter cannot be invertible since $\deg(P), \deg(Q) \ge 2$. 
It follows that $F_t$ defines a family of polynomials and $Z$ determines a marked point. 
If the corresponding marked point is active, then $Z$ contains infinitely many preperiodic points for $F$. 
The same is true when $Z$ is stably preperiodic. By Theorem~\ref{thm:active-characterization1}, the only remaining case is when
the pair $(F_t, Z_t)$ is isotrivial. 
The projections $S \to \A^1 \times C$ induce semi-conjugacies between $F_t$ and the families $P_t$ and $Q_t$ hence
the latter are both isotrivial.

We have thus shown that the converse to Theorem~\ref{thm:partialDMM} holds except if $P$ and $Q$ are isotrivial 
and both sem-conjugated to the same polynomial.
\end{remark}

\begin{proof}
Let $\mathsf{n} \colon \hat{Z} \to Z$ be the normalization of $Z$ and let $a  = \pi_1 \circ \mathsf{n}$ and  $b  = \pi_2 \circ \mathsf{n}$.
Observe that we may assume by base change $\hat{Z} \to C$ that $Z = \{ (a(t), b(t),t), t \in \hat{Z}\}$, and that the two dynamical pairs $(P,a)$ and $(Q,b)$ are parametrized by $\hat{Z}$. Our
assumption implies that they are entangled.

By Theorem~\ref{tm:unlikely}, there exist integers $N,M\geq1$, $r,s\ge0$, and families $R,\tau$ and $\pi$ of polynomials parametrized by $C$ such that
\[\tau\circ P^N= R\circ\tau \ \text{ and } \ \pi\circ Q^M= R\circ\pi,\]
and $\tau (P^{r}(a))= \pi(Q^{s}(b))$.

We conclude observing that if $ p = (a(t), b(t), t) \in Z$ for some $t\in \hat{Z}$, then for any integer $l$, the point $F^l(p)= (P^{Nl}(a(t)), Q^{ML}(b(t)), t)$ is included in the surface
$\{ (z,w,t), \, \tau (P^{r}(z))= \pi(Q^{s}(w))\}$ which is fixed by $F$.
\end{proof}

%
%
%
%
%
%


\chapter{Entanglement of marked points}

We  specialize the results of the previous chapters to a single family of polynomials. 
Let $P$ be any algebraic family of polynomials of degree $d$ parameterized by a curve $C$
defined over an algebraically closed field $K$ of characteristic $0$.
Given any marked point $a \in K[C]$, we describe all dynamical pairs $(P,b)$
that are entangled with $(P,a)$, and obtain  Theorem~\ref{tm:unlikely-marked} from the introduction that we recall
for convenience. 

\unlikelymarked*

We  then proceed by proving a finiteness theorem when our data are defined over a number field (Theorem~\ref{tm:unlikely-finiteness} from the introduction). 

\unlikelyfiniteness*


\section{Proof of Theorem~\ref{tm:unlikely-marked}}
Suppose first that there exists $g\in\Sigma(P)$ and integers $r,s\geq0$ such that
\[P^r(b)=g\cdot P^s(a),\] as regular functions on $C$. Since $P$ is not integrable, $\Sigma(P)$ is a finite group and
there exists a morphism $\rho \colon \Sigma(P) \to \Sigma(P)$ such that
 $P (g \cdot x) = \rho(g) \cdot P(x)$ for all $x$. In particular, for any parameter $t$ such that $P_t^s(a(t))$ has a finite orbit, there exists $g_n \in \Sigma(P)$ such that 
 $P^{rn}_t(b(t)) = g_n \cdot P^{sn}_t(a(t))$ so that $b(t)$ is also preperiodic. We have proved that $\preper(P,a,\bar{K}) \subset \preper(P,b,\bar{K})$. 
 
Since $a$ is active, Remark~\ref{rem:zar dense} implies $\preper(P,a,\bar{K})$ to be infinite, hence $\preper(P,a,\bar{K}) \cap \preper(P,b,\bar{K})$ is infinite.

\medskip
 
 Suppose conversely that $\preper(P,a,\bar{K}) \cap \preper(P,b,\bar{K})$ is infinite. 
Application of Theorems~\ref{tm:unlikely} and~\ref{tm:unlikely-car0} do not quite imply what we are looking for, so that we go back
to  Section~\ref{sec:biggest proof} and follow the details of the
proof of (3) $\Rightarrow$ (2).

Let us suppose that $K=\KK$ is a number field. 
Following the arguments of \S\ref{sec:1on3} we get two coprime integers say $n \le m$ such that
$n \cdot h_{P,a} = m \cdot h_{P,b}$.

We now follow carefully the arguments in Section~\ref{sec:biggest proof} (I) \& (II) using the fact that $N=M=1$. 
Assume  $(\diamond)$, i.e. the leading term of $P$ admits a $(d-1)$-th root.
We get a constant $L\ge1$ such that the following holds. 
For any fixed Archimedean place $v_0 \in M_\KK$, and for each branch at infinity $\mathfrak{c}$ with $\ord_{\mathfrak{c}}(\mathsf{D}_{P,a}) >0$, there exists a connected neighborhood  
$U_{v_0}(\mathfrak{c})$ of the branch in $C^{\an,v_0}$ such that 
for any parameter $t \in C(\bar{\KK}) \cap U_{v_0}(\mathfrak{c})$ , the map 
$\Phi'_t(x,y) = (P^L_t(x),P^L_t(y))$ fixes the Zariski closure $Z$ of 
\[ C_{v_0}(\mathfrak{s}_\zeta)= \left\{ (x,y) \in \C_{v_0}^2, \, \min \{ |x|_{v_0}, |y|_{v_0} \} >R_{v_0}, \, \varphi_{P_t}(x)^n = \zeta \varphi_{P_t}(y)^m\right\}~.\]
Recall that $Z$ is an irreducible algebraic curve of $\A^2$.
By~\cite[Theorem~6.24]{medvedev-scanlon} (see also~\cite[Theorem~4.9]{pakovich}), $Z$
is necessarily the graph or the cograph of a polynomial $v_t$ which commutes with $P_t$. As $n\leq m$, this gives $n=1$ and $Z=\{x=v_t(y)\}$, 
and arguing as in the paragraph preceding the statement of Lemma~\ref{lem:at a fixed parameter}, we 
get the existence of a root of unity $\zeta(\mathfrak{c})\in \KK$, and of an integer $\ell(\mathfrak{c})$ satisfying~\eqref{eq:1shot} and~\eqref{eq:2shot}, i.e.
\begin{align*}
&\zeta(\mathfrak{c}) \cdot \hat{P}_{m,t} \circ P_t^L = P_t^L \circ  (\zeta(\mathfrak{c}) \hat{P}_{m,t})\\
 & a_l  = \zeta(\mathfrak{c}) \hat{P}_{m,t} (b_l), \text{ for } l \ge \ell(\mathfrak{c})~.
 \end{align*}
The Main Theorem of \cite{Schmidt-Steinmetz} applied to $P_t$ and $\zeta(\mathfrak{c}) \hat{P}_{m,t}$ now implies the existence of a polynomial
$R_t\in\C[z]$ and $\sigma_1,\sigma_2\in\Sigma(R_t)$ such that 
\[\zeta(\mathfrak{c})\cdot \hat{P}_{m,t}=\sigma_1\cdot R_t^{k_1} \ \text{ and } \ P_t=\sigma_2\cdot R_t^{k_2}~,\]
for some integers $k_1,k_2$. 

Since the family $P$ is primitive, we may reduce the neighborhood $U_{v_0}(\mathfrak{c})$ so that $P_t$ is primitive for all $t\in U_{v_0}(\mathfrak{c})$,
and assume $R_t= P_t$. Observe that $k_2 =1$ and $k_1=m$. 
We have obtained
\begin{lemma}\label{lem:at a fixed parameter2}
Fix any parameter $t  \in C(\bar{\KK}) \cap U_{v_0}(\mathfrak{c})$ such that $P_t$ is primitive, and not integrable. 
There exists $\sigma_t \in \Sigma(P_t)$ such that 
\[ a_l = \sigma_t(b_{m+l})\]
for all $l\ge \ell(\mathfrak{c})$, where $a_l =P_t^l(a(t))$ and $b_{m+l}=P_t^{m+l}(b(t))$.
\end{lemma}

We continue arguing as in  Section~\ref{sec:biggest proof} (III). It is clear that the set of parameters $t$ in $C$ such that there exists $\sigma_t \in \Sigma(P_t)$ satisfying
\[ a_{\ell(\mathfrak{c})} = \sigma_t(b_{m+\ell(\mathfrak{c})})\]
is Zariski closed. It follows that this set is to $C$ by the preceding Lemma. Since $\sigma_t$ is uniquely determined by the data (except for those finitely many parameters
for which $b_t$ is the fixed point of a non-trivial symmetry of $P_t$), we conclude to the existence of $\sigma \in \Sigma(P)$ such that 
$a_{\ell(\mathfrak{c})} = \sigma (b_{m+\ell(\mathfrak{c})})$ as required. 

The case when $(\diamond)$ is not satisfied is taken care of as in  Section~\ref{sec:biggest proof} (IV), and the same specialization argument as in Section~\ref{sec:specialization} yields the theorem over any field of characteristic zero if condition $(\vartriangle)$ holds.


\section{Proof of Theorem~\ref{tm:unlikely-finiteness}}
We assume that $P$ is a non-integrable and primitive algebraic family of degree $d\ge2$ polynomials parameterized by an affine curve which is
defined over a number field $K$. We also fix a marked point $a\in K[C]$ such that $(P,a)$ is active.

\smallskip

Up to an affine transformation, we shall assume that $P$  has a reduced presentation  $P(z) = z^\mu P_0(z^m)$ so that the center of $P$ is $0$.
Recall that the group of dynamical symmetries $\Sigma(P) = \U_m$ of a degree $d$ polynomial $P$ comes equipped with a morphism
$\rho: \Sigma(P) \to \Sigma(P)$ such that $P(g\cdot z) = \rho(g) \cdot P(z)$. Recall that
we denoted by $\Sigma_0(P)$ the union of the kernels all $\rho^n$ for all $n\ge 1$.
\begin{lemma}
Fix any $g \in \Sigma(P)$ and any integer $n\ge0$.

Then the point $g\cdot P^n(a)$ belongs to the grand orbit of $a$ iff
either $g\in \Sigma_0(P)$, or $P^m(a)$ is the center of $P$ for some $m\ge n$.
\end{lemma}
\begin{proof}
If $g$ belongs to $\Sigma_0(P)$ and $\rho^m(g) = \id$, then
we have
$P^m(g \cdot P^n(a)) = \rho^{m}(g) \cdot P^{nm}(a) =P^{nm}(a)$. 

If $P^m(a) = 0$ with $m\ge n$, then $P^{m-n}(g \cdot P^n(a)) = \rho^{m-n}(g) \cdot P^m(a) = 0$
hence $g \cdot P^n(a)$ lies in the grand orbit of $a$.

This proves one implication.

\smallskip

Suppose conversely that $g \cdot P^n(a)$ belongs to the grand orbit of $a$. 
Then there exist two integers $l$ and $q$ such that $P^l(a) = P^q(g \cdot P^n(a))= \rho^q(g) \cdot P^{qn}(a)$.
Since $a$ is active the divisor at infinity $\mathsf{D}:= \mathsf{D}_{P,a}$ is non zero, and 
$\mathsf{D}_{P,P^l(a)}= d^l \mathsf{D}$ whereas $\mathsf{D}_{P,\rho^q(g) \cdot P^{qn}(a)} = d^{qn}\mathsf{D}$. 
Indeed we may always assume that $K$ is a subfield of $\C$, and apply
Propositions~\ref{prop:sameSigma} and~\ref{prop:behave green}.
We conclude that  $l = qn$ hence $P^l(a) = \rho^q(g) \cdot P^{l}(a)$. 
We thus have either $P^l(a) =0$ with $l \ge n$, or $g \in \Sigma_0(P)$ as required.
\end{proof}

Denote by $\mathcal{E}(a)$ the set of marked points $b \in \bar{K}[C]$ such that the pairs $(P,a)$ and $(P,b)$ are entangled.
Observe that 
\[
\mathcal{E}(a) \supset \mathcal{A}_0 := \{ g \cdot P^n(a)\,\text{ with } n \ge 0  \text{ and } g \in \Sigma_0(P) \}~. \]
Our objective is to show that the complement of the right hand side in $\mathcal{E}(a)$ is a finite set. 
Since by Theorem~\ref{tm:unlikely-marked} we have the inclusion $\mathcal{E}(a)\subset \mathcal{A}$ with 
\[
\mathcal{A}: =\{ b \in \bar{K}[C], \text{ there exists } n,m \ge0 \text{ and } g\in \Sigma(P) \text{ s.t. } P^n(b) = g \cdot P^m(a) \}~,
\]
we are reduced to proving that
\begin{theorem}\label{thm:pfThmC}
The set $\mathcal{A} \setminus \mathcal{A}_0$ is finite.
\end{theorem}

We begin with the following

\begin{proposition}\label{prop:bdd preimage}
For any integer $D\ge1$, the set of points $b\in \bar{K}[C]$ such that $P^n(b) = P^m(a)$ with
$n \ge m -D$ is finite. 
\end{proposition}

\begin{remark}
By a theorem of Benedetto ~\cite{Benedetto}, (see also \cite{Baker-functionfield} for the case of rational maps), the set of points on $K(C)$ of sufficiently small height is bounded when $P$ is not isotrivial. This implies that one can assume $|n-m|$ bounded when the family is not isotrivial.  Beware that it is not true that for all $c>0$ the set of points of height $\le c$ is bounded.

Indeed, as remarked in \cite{demarco}, if $P(z)=z^2+t\in \mathbb{C}(t)[z]$, then, over the field $\C(t)$ of rational functions on $\p^1$, any point $b\in\mathbb{C}$ has canonical height $h_P(b)=1/2$.
\end{remark}

\begin{proof}
We proceed by contradiction and pick an infinite sequence $b_l \in \bar{K}[C]$ in the grand orbit of $a$ such that 
$P^{n_l}(b_l) = P^{m_l}(a)$ with $n_l \ge m_l -D$. Replacing $a$ by $P^D(a)$, and $b_l$ by a suitable iterate, 
we may always assume that $n_l = m_l$. 

\medskip

\noindent{${\bf 1}$.} We may interpret the family $P$ as a degree $d$ polynomial
\[\mathsf{P}:\mathbb{A}^1_{\bar{K}(C)}\to\mathbb{A}^1_{\bar{K}(C)}\]and $a$ and $b_l$ as points $\mathsf{a},\mathsf{b}_l\in\mathbb{A}^1_{\bar{K}(C)}$.

\smallskip

Fix any branch $\mathfrak{c}$ of $C$ at infinity. To this branch, we can associate a non-Archimedean norm $|\cdot|_\mathfrak{c}$ on $\bar{K}(C)$ by setting 
\[|f|_\mathfrak{c}:=\exp(-\mathrm{ord}_\mathfrak{c}(f)), \text{ for all } \ f\in \bar{K}(C).\]
We infer
\[0\leq g_{\mathsf{P},\mathfrak{c}} (\mathsf{b}_l) = d^{m_l - n_l}g_{\mathsf{P},\mathfrak{c}} (\mathsf{a})\leq d^Dg_{\mathsf{P},\mathfrak{c}}(\mathsf{a}), \]
 so that 
 \[ \max \{ 0 , - \mathrm{ord}_\mathfrak{c}(\mathsf{b}_l) \}=  \log^+|\mathsf{b}_l|_\mathfrak{c} < d^Dg_{\mathsf{P},\mathfrak{c}}(\mathsf{a}) + \sup |g_{\mathsf{P},\mathfrak{c}}(z)  -\log^+ |z|  | < \infty~.\]
 In particular the sequence $ \mathrm{ord}_\mathfrak{c}(\mathsf{b}_l)$ is bounded from below.

 \smallskip
 
Since the degree of the rational map $b_l \colon C \to \p^1$ is equal to 
\[
\deg(b_l) = \sum_\mathfrak{c} \max \{ 0 , - \mathrm{ord}_\mathfrak{c}(\mathsf{b}_l) \},\]
we conclude that the graph $\Gamma_l$ of $\mathsf{b}_l$ belongs to a bounded family of curves in the projective surface $\bar{C} \times \p^1$. 
In other words there exists an irreducible variety $W$, a regular map $\Gamma \colon W \times C \to \A^1$, and a Zariski-dense subset $w_l$ of $W$ such that 
$\Gamma(w_l, \cdot) = b_l$ for all $l$.

\medskip

\noindent{${\bf 2}$.}
We fix any embedding of $\bar{K}$ into $\C$, and work with the euclidean topology on $C^\an$.
Pick any connected component $\Omega$ of  $\{ g_{P,a} >0 \}$ in $C^\an$. 

By the maximum principle one can find a branch at infinity $\mathfrak{c}$
lying in the closure of $\Omega$ (in $\bar{C}$), as well as in the support of $\mathsf{D}_{P,a}$. 
As usual we fix a local parameterization $t \mapsto \theta(t)$ of $\mathfrak{c}$ in $\bar{C}$, and drop the reference to $\theta$
to simplify notation. Replacing $a$ by an iterate if necessary,
we can evaluate the B\"ottcher coordinate of $P_t$ at $a(t)$ for any $t$ sufficiently small. Since $P^{n_l}(b_l) = P^{n_l}(a)$, 
we have $g_{P,a} = g_{P,b_l}$ so that the B\"ottcher coordinate is also defined 
at $b_l(t)$ for $t$ small, and there exists a root of unity $\zeta_l$ such that 
\[ 
\varphi_t(b_l(t)) = \zeta_l \varphi_t(a(t))~.\]
We claim that for all $w \in W$ there exists a constant $\zeta(w) \in \C$ such that 
\[ \varphi_t(\Gamma(w,t)) = \zeta(w) \varphi_t(a(t))~,\]
for all $t$ small enough. 

\smallskip

To see this recall from \S\ref{sec:bottcher} that we have an expansion of the B\"ottcher coordinate of the form
\[
\varphi_t(z) = z + \sum_{k\ge 1}\frac{\alpha_k(t)}{z^k}
\]
where $\alpha_k(t)$ are analytic so that $\varphi_t(a(t)) = \sum_{k \ge -k_0} a_k t^k$ for some $a_k \in \C$ (we may take a further iterate of $a$ 
as in Step 3 on p.\pageref{step3} so that the series formally converges).

Observe that we can write $\Gamma(w,t) = \sum_{k \ge -k_0} h_k(w) t^k$ where $h_k$ is a regular function on $W$.  
It follows that the equation $\varphi_t(\Gamma(w,t)) = c \varphi_t(a(t))$
is equivalent to a series of equations of the form
\[
H_k := h_k + P_k (h_{-n}, \cdots, h_{k-1}) = c a_k~, \]
where $P_k$ is a polynomial in $n +k$ variables. 
For each integer $N\ge-n$, we obtain
\[ 
[H_{-n}(w) : \cdots : H_N(w)] = [a_{-n} : \cdots : a_N] \in \p^{n+N-1}\]
for all $w = w_l$. Since $\{w_l\}$ is Zariski-dense the equality holds for all $w \in W$.
Letting $N \to \infty$, we get our claim.

\medskip

\noindent{${\bf 3}$.}
Note that by construction the map  $w \mapsto \zeta(w)$ is algebraic. 
We claim that it is not constant. Indeed, for all $t$ close enough to the branch at infinity $\mathfrak{c}$, 
the B\"ottcher coordinate $\varphi_t$ is a isomorphism on $\{g_t > \frac12 g_t(a(t)) \}$ so that 
$\zeta_l = \zeta_{l'}$ implies
$\varphi_t(b_l(t)) = \varphi_t(b_{l'}(t))$, hence
$b_l(t) = b_{l'}(t)$, from which we infer $b_l = b_{l'}$. 

\smallskip

Since $W$ is irreducible, we may pick $w \in W$ such that $|\zeta(w)|<1$. 
Recall from the previous step that  $\Omega$ is a connected component of  $\{g_{P,a}>0\}$, and that
$\mathfrak{c}$ is a branch at infinity lying in its closure.
Near that branch, we have
$g_{P,\Gamma(w,\cdot)} =  g_{P,a} - \log |\zeta(w)|$ hence by analytic continuation
$g_{P,\Gamma(w,\cdot)} =  g_{P,a} - \log |\zeta(w)|$ on $\Omega$. 

Pick any point $t_*$ on the boundary of $\Omega$. Then $g_{P,a}(t_*) =0$, and $g_{P,\Gamma(w,\cdot)}$ is harmonic
near $t_*$. It follows that in a neighborhood of $t_*$ the boundary of $\Omega$ is contained in 
$g_{P,\Gamma(w,\cdot)} =  - \log |\zeta(w)|$ hence is locally real-analytic. 

\smallskip

Since $\Omega$ was an arbitrary component of $\{g_{P,a}>0\}$, we have proved that 
the bifurcation locus of the pair $(P,a)$ is real-analytic. By Theorem~\ref{tm:rigidaffine}, the family is integrable
which contradicts our assumption.  
\end{proof}

\begin{lemma}\label{lem:key-finite}
Suppose there exists an integer $n\ge 0$, and a sequence  $b_l \in \bar{K}[C]$ such that 
$P^{n}(b_l) = P^{m_l}(a)$ with $m_l\to\infty$. 

Then for all $l$ large enough, we have $b_l \in\mathcal{A}_0 = \{ g \cdot P^n(a), \, g\in \Sigma_0(P),\, n\ge 0\}$.
\end{lemma}

\begin{remark}
The previous proposition was valid for any field of characteristic zero. Our proof of Lemma~\ref{lem:key-finite} uses our standing assumption that $K$ is a
number field.
\end{remark}

\begin{proof}
The B\"ottcher coordinates $\varphi_P(z)$ is a formal Laurent series in $z^{-1}$ 
whose coefficients belong to $K[C]$ by Proposition~\ref{prop:bottcher}.
Since $a$ is active and $m_l -n \to \infty$, we may suppose that 
for some branch $\mathfrak{c}$ at infinity of $C$ (any branch in the support of $\mathsf{D}_{P,a}$ works), the series
$\varphi_t(P_t^{m_l-n}(a(t)))$ is well-defined in a neighborhood of $\mathfrak{c}$
for all $l$. It follows that for any $l$ one can write 
$$
\varphi_t(b_l(t)) = \zeta_l \varphi_t(P_t^{m_l-n}(a(t)))
$$
for some $d^n$-th root of unity $\zeta_l$ and all $|t| \ll 1$.

\smallskip

Let us fix some $d^n$-th root of unity $\zeta$. We claim that the set of indices
$l$ such that $\zeta_l = \zeta$ and $b_l \notin \mathcal{A}_0$ is finite. 
Note that this implies the lemma.

To simplify notation, we shall assume that $\zeta_l = \zeta$ for all $l$. 
Fix any Archimedean place $v_0$ and a sufficiently small (euclidean) neighborhood $U_{v_0}(\mathfrak{c})\subset C^{\an,v_0}$ of the branch at infinity
$\mathfrak{c}$, as in the first paragraph of (II) on p.\pageref{hypo-semi-conj}.

Pick any closed point $t\in U_{v_0}(\mathfrak{c})\cap C(\bar{K})$. Observe that $a(t)$ is not preperiodic, and choose any finite extension $L/K$ such that 
$t \in C(L)$.
For any $l$ we have $P^{n}_t(b_l(t)) = P^{m_l}_t(a(t))\in L$ hence $b_l(t)$ belongs to a fixed
finite extension of $L$ (the one in which $P^n_t$ splits). 

Consider the adelic series at infinity $\mathfrak{s}_\zeta$
given by
\[\mathfrak{s}_\zeta= \{ (x,y) \in \A^1 \times \A^1, \, \varphi_t(x) = \zeta \, \varphi_t(y) \},
\]
and observe that, by our assumption,
for each integer $l$ and for each place $v\in M_L$, either the point $(b_l(t), P^{m_l-n}_t(a(t)))$ belongs to $Z^v(\mathfrak{s}_\zeta)$ (e.g. when $v = v_0$) or 
has bounded norm. 

By Theorem~\ref{thm:Junyi}, one can find thus  an irreducible algebraic curve $Z\subset  \A^1 \times \A^1$ such that 
 $Z$ has a single branch at infinity included in $\mathfrak{s}_\zeta$, and $(b_l(t), P^{m_l-n}_t(a(t))) \in Z$ for all $l$.

Since $Z$ has a single branch which is smooth and transverse to the fibrations induced by the two projections $ \A^1 \times \A^1 \to \A^1$, 
it is the graph of an automorphism of $\A^1$, say $z \mapsto g(z)$ with $g$ affine. 
It follows that for all $z$ large enough one has 
$$\varphi_t(P^n_t(g(z)))=\varphi_t(g(z))^{d^n} = (\zeta\, \varphi_t(z))^{d^n} = \varphi_t(P^n_t(z)))$$
hence $g$ belongs to $\Sigma_0(P^n_t)= \Sigma_0(P_t)$. Because $P_t$ is monic and centered, $g$ is necessarily linear, and since
$\varphi_{t}$ is tangent to the identity at infinity, we conclude that $g(z) = \zeta z$.

\smallskip

It follows that $b_l(t) = \zeta P^{m_l -n}(a(t))$ for all $l$ and infinitely many parameters $t$, so that 
$b_l = \zeta P^{m_l -n}(a)$, and $b_l \in \mathcal{A}_0$ as was to be shown.
\end{proof}

We may now prove Theorem~\ref{thm:pfThmC}.

\begin{proof}[Proof of Theorem~\ref{thm:pfThmC}]
Suppose by contradiction that we can find a sequence of distinct points $b_l \in \bar{K}[C]$ such that 
$P^{n_l}(b_l) = P^{m_l}(a)$, $P^{n_l-1}(b_l) \neq P^{m_l-1}(a)$ and $b_l \notin \mathcal{A}_0$.

By Proposition~\ref{prop:bdd preimage}, we may suppose that $m_l -n_l \to \infty$.
When $n_l$ does not tend to $\infty$, then we can extract a subsequence such that $n_l = n$
and apply the preceding lemma. This shows that $n_l \to\infty$. 

Observe that since $P$ is not integrable, the group of dynamical symmetries of $P$ is finite, 
and there exists an integer $N$ such that for any $g\in \Sigma_0(P)$
we have $\rho^N(g) = 1$ (where $\rho \colon \Sigma(P) \to \Sigma(P)$ is the morphism arising in Proposition~\ref{prop:Sigma}).

Pick any integer $n>N$. We can then apply Lemma~\ref{lem:key-finite}
to the sequence of points $P^{n_l -n}(b_l)$, and we obtain that  $P^{n_l -n}(b_l)$ belongs to $\mathcal{A}_0$
for all $l$ large enough. But then  we get $P^{n_l -n}(b_l) = g\cdot P^{m'_l}(a)$ for some integer $m'_l\ge0$ and some $g\in \Sigma_0(P)$, hence
 $P^{n_l -n+N}(b_l) =  P^{m'_l+N}(a)$ which contradicts the minimality of $n_l$. 
\end{proof}

%
%
%
%
%
%



\chapter{The unicritical family} \label{chapter:unicritical}

Note that the original papers by Baker-DeMarco were mainly focused on the unicritical family $P_t(z)= z^d+t$\index{polynomial!unicritical}. 
In this short chapter, we propose to extend some of their results, and to illustrate the theorems proved in the previous chapters
on this special family. 

We also introduce the set $\mathbb{M}$ of those $\lambda\in \C^*$ such that 
the bifurcation locus of the pair $(P_t, \lambda^{-1}t)$ is connected. We prove this set is compact, and perfect.

\section{General facts} 

In this section, we gather several facts about the unicritical family and make some computations that will be useful
in the sequel. To simplify the discussion we work over the field of complex numbers.

\medskip

We fix $d\ge 2$ and consider the family $P_t(z) = z^d +t$ parametrized by the affine line $t \in \A^1$. 
Observe that $P_t$ is integrable iff $t=0$ or $d=2$ and $t=-2$,
and that $\Sigma(P_t) = \U_d$ when $t \neq 0$.
The family is not isotrivial, and primitive. Beware though that $P_t$ is decomposable as soon as $d$ is not a prime. 

\smallskip

We pick any marked point $a \in \C[t]$, which we write as
$a (t) = \alpha t^\kappa + o(t^\kappa)  \in \C[t]$, for some $\alpha \neq 0$, $\kappa \ge0$.
We do not exclude the case $a$ is a constant. Define \[\cal{M}(d,a) = \{ t \in \C, \, a(t) \in K(P_t)\}=
\{ t \in \C, \, g_{P_t}(a(t))=0\},\] so that the bifurcation locus of $(P,a)$ is equal to the boundary of $\cal{M}(d,a)$. 
When $a=0$,  the boundary of $\cal{M}(d,0)$ is the bifurcation locus of the unicritical family, and $\cal{M}(2,0)$ is the Mandelbrot set.

\medskip

Recall the definition of the (logarithmic) capacity of a compact set $K$ in the plane, see~\cite{ransford,tsuji} or~\cite[\S A.8]{Sibony}.
First one defines the Green function $g_K$ of $K$ as the upper-semi-continuous regularization of 
the supremum of all subharmonic functions $u$ of the plane such that $u|_K \le0$, and  
$u (z) = \log|z| + O(1)$ at infinity. When $g_K$ is not identically $+\infty$, it is subharmonic and harmonic on $\C \setminus K$. 
The measure $\mu_K := \Delta g_K$ is the harmonic measure of $K$. 
Near infinity, we have the expansion $g_K(z) = \log|z| + V + o(1)$, and the constant $\mathrm{cap}(K) := e^{-V}>0$ is called the  capacity of $K$.
When $g_K$ is identically $+\infty$, we set $\mathrm{cap}(K)  =0$. Recall that one can define the energy of any probability measure $\mu$ which is compactly supported by
\[\mathcal{E}(\mu) := \int\log|x-y|^{-1} \, d\mu(x) d\mu(y) \in (-\infty, +\infty]~.\]
When $\mathrm{cap}(K)>0$, the harmonic measure of $K$ is the unique measure such that
$\mathcal{E}(\mu_K) = \inf\{ \mathcal{E}(\mu), \, \supp(\mu) \subset K \}$, and $\mathcal{E}(\mu_K)= -\log \mathrm{cap}(K)$.

\medskip

To simplify notations, we write $g_t = g_{P_t} = \lim_n \frac1{d^n} \log^+|P^n_t|$ and $g_a(t) = g_{P_t}(a(t))$, 
and we let $\varphi_{d,t}$ be the B\"ottcher coordinate of $P_t$
which is defined in the open set $\{ g_t > g_t(0)\}$, see Proposition~\ref{prop:GreenBottcher}.

An easy induction shows
\begin{align}
P^n_t(z) &= z^{d^n} + d^{n-1} t z^{d^n -d} + O_t(z^{d^n -d-1}) \text{ for all } n\ge1, \label{eq:dvpPn}
\end{align}
and by Proposition~\ref{prop:expan bottcher}, we have
\begin{align}
\varphi_{d,t}(z) &= z + \frac{t}{d z^{d-1}} + \sum\limits_{j=0}^\infty\frac{\alpha_j(t)}{z^{d+j}}~,\label{eq:dvpbott}
\end{align}
where $\alpha_j\in\Z[t]$ satisfies $\deg_t(\alpha_j)\leq (1+d+j)/d$.

Let us first treat the case when the marked point is constant.
\begin{prop}\label{prop:unicritical cst}
Suppose $a$ is a constant function. 
Then for all $n \in \N^*$ we have
\begin{align*}
P^{n+1}_t(a) & = t^{d^n}+ a^d t^{d^n-1} + o(t^{d^n-1})~,\text{ and } \\
g_{a}(t) &= \frac1d \log|t| + o(1),
\end{align*} so that $\mathrm{cap}(\cal{M}(d,a)) =1$.
Moreover, the inequality $g_{t}(a^d+t) > g_{d,t}(0)$ holds  for all $t$ large enough, and
\begin{equation}\label{eq:610}
g_{a}(t) = \frac1d \log| \varphi_{d,t}(a^d+t)|~.
\end{equation}
\end{prop}
This result is proved in~\cite[Lemma~3.2 $\&$ Proposition~3.3]{BDM}. 
\begin{proof}
The first equality is obtained by induction on $n$.
The second follows from Proposition~\ref{prop:classic-estim} which\footnote{beware of the change of parametrization} gives
\[
\log^+ |z| - C_1 \le 
g_t (z)  \le \log^+\max \{|z|, |t|^d\} + C_1
\]
for some constant $C_1$, and for all $|z| \ge C_2 \max \{ 1, |t|^d\}$.
For all $n$, and for $t$ large enough we get
\[
\left|
g_a(t) - \frac1{d^n} \log^+|P^n_t(a)| 
\right| \le \frac{C_3}{d^n}\]
which implies
$g_a(t) = \frac1{d} \log|t| + o(1)$ by the previous computations.
The Green function of $\cal{M}(d,a)$ is  $d\times g_a$ since the latter is subharmonic, is equal to $0$ exactly on $\cal{M}(d,a)$ and
has the expansion at infinity $= \log |t| + o(1)$. We also get $\mathrm{cap}(\cal{M}(d,a)) =1$.

Since $g_t(a^d+t) = \log|t| +o(1)$ and $g_t(0) = g_0(t)= \frac1d \log|t| +o(1)$, we get 
$g_{t}(a^d+t) > g_{d,t}(0)$ holds  for all $t$ large enough so that~\eqref{eq:610} holds.
\end{proof}

When the marked point is not constant, the previous proposition needs to be modified as follows.

\begin{prop}\label{prop:unicritical pos}
Suppose $a$ is a complex polynomial of degree $\kappa \ge1$ whose dominant term is equal to $\alpha$.
Then for all $n \in \N^*$ we have
\begin{align*}
P^n_t(a(t))&= \alpha^{d^n} t^{\kappa d^n} + o(t^{\kappa d^n}) \text{ and } \\
g_{a}(t) &=\kappa \log|t| + \log|\alpha| + o(1),
\end{align*} so that $\mathrm{cap}(\cal{M}(d,a)) =|\alpha|^{-1/\kappa}$.
Moreover for all $t$ large enough, we have
\begin{equation}\label{eq:611}
\begin{cases}
g_{a}(t) = \frac1\kappa \log| \varphi_{d,t}(a(t))|
& 
\text{ when } \kappa \ge 2, \text{ or } \kappa =1 \text{ and } |\alpha| \ge 1;
\\
g_{a}(t) = \frac1{d\kappa} \log| \varphi_{d,t}(a^d(t)+t)|
&
\text{ when } \kappa =1 \text{ and } |\alpha|<1.
\end{cases}
\end{equation}
\end{prop}

The proof is identical to the previous one and is left to the reader.

When the marked point is $0$, i.e. for the classical Multibrot sets, we have
\begin{lemma}\label{lm:containsdisk}
For any $d\geq2$, the central hyperbolic component of the set $\cal{M}(d,0)$ contains $\D(0,1/4)$.
\end{lemma}

\begin{proof}
Let $x_n$ be the sequence defined by $x_0=0$ and $x_{n+1}=x_n^2+1/4$. The sequence $(x_n)$ is strictly increasing and converges to $1/2$, which is the only fixed piont of $z^2+1/4$.

Let now $d\geq2$ and pick $t\in\D(0,1/4)$. An easy induction shows
\[|P_t^n(0)|\leq P_{1/4}^n(0)\leq x_n\leq \frac{1}{2}\]
for all $n\geq1$. This ends the proof.
\end{proof}

\section{Unlikely intersection in the unicritical family}

Recall from the introduction our strengthening of the original Baker and DeMarco's result which deals with families of unicritical polynomials of possibly different degrees.

\samemandel*

\begin{remark}
With similar techniques, it is possible to treat the case $d = \delta$ and $\deg(a) \neq \deg(b)$, but the general case remains elusive. 
\end{remark}

\begin{proof}
Observe that since the parameter space is the affine line, the divisors at infinity of the two pairs are supported at a single point, and 
Propositions~\ref{prop:unicritical cst} and~\ref{prop:unicritical pos} 
imply $\mathsf{D}_{P,a} = \mathsf{D}_{P,b} = [\infty]$ (when $\deg(a) =0$), 
and $=\kappa [\infty]$ (otherwise). Condition ($\vartriangle$) is therefore satisfied and we may apply Theorem~\ref{tm:precise intrication-car0}.

Observe that $n=m=1$ and $d^N = \delta^M$. It follows that 
there exists a root of unity $\zeta$, and an integer $L\ge 1$ such that 
\[\zeta (z^{\delta} +t)^{\circ ML} = ((\zeta z)^{d} +t)^{\circ NL}\]
and
\[(z^{d} +t)^{\circ NL}(a(t)) = \zeta (z^{\delta} +t)^{\circ ML}(b(t)) \]
Expanding the first equation, and using~\eqref{eq:dvpPn} yields
\[
\zeta (z^{\delta^{ML}} + t \delta^{ML}z^{\delta^{ML}-\delta} + \text{l.o.t}) =  (\zeta z)^{d^{NL}} + t d^{NL} (\zeta z)^{d^{NL}-d} + \text{l.o.t} \]
which implies $d=\delta$, and $N=M=1$. 
Note that  $\zeta (z^{d} +t)^{\circ L} = ((\zeta z)^{d} +t)^{\circ L}$ hence $\zeta \in \Sigma(z^d +t)^{\circ L}=  \Sigma(z^d +t)= \U_d$, and $\zeta =1$.
We have thus proved that
\[(z^{d} +t)^{\circ L}(a(t)) = (z^{d} +t)^{\circ L}(b(t))~. \]
We now exploit the equality $g_a(t) = g_b(t)$ together with~\eqref{eq:610} and~\eqref{eq:611}.

\smallskip

When $\kappa =0$, we follow the arguments of Baker and DeMarco. The two analytic functions 
$\varphi_{d,t}(\alpha^d+t)$ and $\varphi_{d,t}(\beta^d+t)$ have the same modulus near $\infty$, and are
both tangent to the identity. They are hence equal. By injectivity of $\varphi_{d,t}$ near infinity, we get $\alpha^d= \beta^d$.

\smallskip

When $\kappa \ge2$, then $\varphi_{d,t}(a(t)) = a(t) + O(t^{-1})$ and $\varphi_{d,t}(b(t))= b(t) + O(t^{-1})$. Since these functions have the same modulus near $\infty$, 
we get $b(t) = \xi a(t)$ with $|\xi| =1$. 
We claim that in fact $\xi \in \U_d$.

Write $a (t)= t^j(\alpha_0  + \alpha_1 t + O(t^2))$ with $\alpha_0\neq 0$, and $j\ge0$. When
$j\ge 1$, we get for all $n\ge1$
\[
P^n_t(a(t)) = t + Q_n(t) + d^{n-1} \alpha_0^{d} t^{(n-1)(d-1) +dj} + O( t^{n(d-1) +dj+1})\]
where $Q_n$ is a polynomial with constant (integral) terms vanishing up to order at least $2$ at $0$.
And
\[
0 = P^L_t(a(t)) - P^L_t(\xi a(t)) =  (\xi^d -1) d^{n-1} \alpha_0^{d} t^{(n-1)(d-1) +dj} + O( t^{n(d-1) +dj+1})
\]
which implies $\xi^d=1$.
Otherwise $j=0$, and  we have
\[
P^n_t(a(t)) = \alpha_0^{d^n} + t (d^n \alpha_0^{d^n-1} \alpha_1 + R_n(\alpha_0)) + O(t^2)\]
with $R_1(T) = 1$, and $R_{n+1} (T) = d T^{d^n (d-1)}R_n(T) +1$. We have 
$P^L_t(a(t)) = P^L_t(\xi a(t))$ hence $\xi^{d^L} =1$, 
and $R_L(\xi T) = R_L(T)$. 
This implies $R_{L-1}(\xi T) = \xi^{d^{L-1}} R_{L-1}(T)$ so that 
\[d (\xi T)^{d^{L-1} (d-1)}R_{L-2}(\xi T) +1 = \xi^{d^{L-1}} (d T^{d^{L-1} (d-1)}R_{L-2}(T) +1),\]  
hence $\xi^{d^{L-1}} =1$.
By induction we get $\xi^d =1$ as required.

When $\kappa =1$, we replace $a$ and $b$ by $P(a)$ and $P(b)$ respectively, and by the previous argument we obtain
$P^2(a) = P^2(b)$ which implies $P(a) = P(b)$. Details are left to the reader.
\end{proof}


\section{Archimedean rigidity}

We expect that under suitable conditions two complex dynamical pairs
parametrized by the same algebraic curve and having the same (complex)
bifurcation measures are entangled, see (Q1) from the Introduction. 

We explore here this problem in the unicritical family, and obtain

\begin{theorem}
Fix $d \ge2$, and pick  $a , b \in \C[t]$ of the same degree. Then  $\cal{M}(d,a)= \cal{M}(d,b)$ iff
 $a^d = b^d$.
\end{theorem}

\begin{remark}
It would be interesting to characterize dynamical pairs $(z^d+t, a(t))$ and $(z^\delta +t, b(t))$
with $d \neq \delta$ having the same bifurcation locus. It is not clear however how to show 
 that $d$ and $\delta$ are multiplicatively dependent. 
\end{remark}

\begin{proof}
Write $a = \alpha t^\kappa + o(t^\kappa)$, and $b = \beta t^\kappa + o(t^\kappa)$ with $\kappa \ge0$, and $\alpha \beta \neq 0$.
Define $\tilde{g}_{d,a} := d g_{d,a}$ if $a$ is constant, $\tilde{g}_{d,a} := d\kappa g_{d,a}$ if $\kappa =1$ and $|\alpha| <1$,  
and $\tilde{g}_{d,a} := \kappa g_{d,a}$ in the remaining cases. 

Observe that $\tilde{g}_{d,a}$ is a subharmonic function on the complex plane, equal to $0$ on $\cal{M}(d,a)$, harmonic
on its complement and $\tilde{g}_{d,a}(t) = \log|t| + O(1)$ at infinity. It follows that $\tilde{g}_{d,a}$ is the Green function of $\cal{M}(d,a)$. 
Since by assumption  $\cal{M}(d,a)= \cal{M}(d,b)$, the two functions $g_{d,a}$ and $g_{d,b}$ are proportional, and even equal since
$\deg(a) = \deg(b)$ and the capacity of $\cal{M}(d,a)$ and  $\cal{M}(d,b)$ are equal.
We deduce from this that for all $t$ large enough, one has
\[\varphi_{d,t}(a^d(t) +t) = \zeta\, \varphi_{d,t}(b^d(t) +t) \text{ for some } |\zeta| =1~. \] 
If $a$ and $b$ are constants, then looking at the expansion of $\varphi_{d,t}$ yields $\zeta =1$, and $a^d =b^d$.
When $\kappa \ge1$, we get
$a^d(t) +t  = \zeta (b^d(t) +t )$, and $(a^d(t) +t)^d +t  = \zeta^d ((b^d(t) +t )^d +t)$.
Looking at the order $1$ terms, $a(t) = \alpha_0 + t \alpha_1 + O(t^2)$, $b(t) = \beta_0 + t \beta_1 + O(t^2)$
we get
\[
\alpha_0^d + (1 + d \alpha_0^{d-1} \alpha_1) t = \zeta \beta_0^d + \zeta (1 + d \beta_0^{d-1} \beta_1) t \]
\[
\alpha_0^{d^2} + \left(1 + d \alpha_0^{d(d-1)}(1+d \alpha_0^{d-1} \alpha_1)\right) t =\zeta^d \beta_0^{d^2} + \zeta^d \left(1 + d \beta_0^{d(d-1)}(1+d \beta_0^{d-1} \beta_1)\right) t~. \]
This implies $\zeta ^d =1$ hence $P^2(a) = P^2(b)$. We may then repeat the proof of Theorem~\ref{thm:same-mandel} starting from "When $\kappa\ge1$", and we conclude that  $P(a)=P(b)$.
\end{proof}


\section{Connectedness of the bifurcation locus}

We explore in this section the connectedness of $\mathcal{M}(d,a)$ under suitable assumptions on the marked point. 

\begin{theorem}\label{thm:car-mandel}
Assume $a\in\C[t]$ is either a constant or that its degree is a power of $d$ and its leading coefficient lies in the closed unit disk. The following assertions are equivalent:
\begin{enumerate}
\item the Mandelbrot set $\cal{M}(d,a)$ is connected,
\item $a(t)= \zeta P_t^n(0)$ for some $n\ge 0$ and some $\zeta \in \U_d$. 
\end{enumerate}
\end{theorem}
\begin{proof}
When (2) is satisfied, the connectedness of $\cal{M}(d,a)$ is a famous theorem of Douady-Hubbard-Sibony, see e.g.~\cite[Chapter VIII, Theorem~1.2 page 124]{carleson-gamelin}.

Suppose (2) is not satisfied, and observe that our assumptions imply the estimate $\mathrm{cap}(\cal{M}(d,a)) \ge 1$. 
We claim that $\cal{M}(d,a)\setminus \cal{M}(d,0)$ is non empty.

Suppose by contradiction that $\cal{M}(d,a)\subset \cal{M}(d,0)$. 
This implies that the series of inequalities 
\[
0 \ge 
 - \log \mathrm{cap}(\cal{M}(d,a)) =
\mathcal{E}\left(\mu_{\cal{M}(d,a)}\right) \ge \mathcal{E}\left(\mu_{\cal{M}(d,0)}\right)= - \log \mathrm{cap}(\cal{M}(d,0))=0,\]
so that $\mathcal{E}\left(\mu_{\cal{M}(d,a)}\right) = \mathcal{E}\left(\mu_{\cal{M}(d,0)}\right)$ which implies
$\mu_{\cal{M}(d,a)} = \mu_{\cal{M}(d,0)}$.
Since the support of these measures are $\cal{M}(d,a)$ and $\cal{M}(d,0)$ respectively, we obtain $\cal{M}(d,a)= \cal{M}(d,0)$.
Since $\deg(a) = d^n$ for some integer $n$, we may apply the previous theorem to $P^n(0)$ and $a$, and we get 
$a= P^n(0)$ or $P(a) = P^n(0)$ which contradicts our standing assumption. 

We conclude using the next lemma.
\end{proof}

\begin{lemma}\label{lem:total-dis-mandel}
The set $\cal{M}(d,a)$ is totally disconnected  in a neighborhood of any point
$t_0 \in \cal{M}(d,a)\setminus \cal{M}(d,0)$. 
\end{lemma}

\begin{proof}
It is sufficient to prove the statement outside finitely many point, so that we may suppose that $a'(t_0) \neq 0$.

Recall that the family $P_t(z) = z^d+t$ is stable in a neighborhood of $t_0$ and that the Julia set of $P_{t_0}$ is a Cantor set. 
Thus there exists an open disk $U$ centered at $t_0$ and a holomorphic motion $h \colon U \times J(P_{t_0})  \to \C$
conjugating the dynamics. By Theorem~\ref{thm:full hol motion} (see~\cite{McMS}), we may reduce $U$ and extend the holomorphic motion 
to the full complex plane $h \colon U \times \C  \to \C$ so that $P_t(h(t,z)) = h(t,P_{t_0}(z))$ remains valid.

By Proposition~\ref{prop:bifJstab}, we have
$\partial \mathcal{M}(d,a)\cap U = \{t\in U, \, \ a(t)\in J(P_t)\}$.
Since $a'(t_0) \neq 0$, we may reduce $U$ and find a polydisk $W = U \times V  \subset \C^2$ containing 
$\Gamma := \{ (t, a(t)), \, t \in U\}$
 such that  the intersection of $\Gamma$ with  $\{(t,h(t,a(t_0))), \, t \in U\}$ is transversal and reduced to $(t_0, a(t_0))$. 
By Rouch\'e's theorem, it follows that any curve $\{(t,h(t,z)), \, t \in U\}$ intersects transversally $\{ (t, a(t)), \, t \in U\}$ at a single point $(t, H(z)) \in U$, 
for any $z$ close enough to $a(t_0)$.
By continuity of the roots the map $z \mapsto H(z) $ is a homeomorphism, and 
$\partial \mathcal{M}(d,a)\cap \D = H^{-1} ( J(P_{t_0}))$ is totally discontinuous as required. 
\end{proof}

\section{Some experiments}

Theorem~\ref{thm:car-mandel} leaves open the characterization of those marked points for which $\mathcal{M}(d,a)$ is connected. 
We propose to investigate the situation when the marked point is given by $a(t) = \lambda^{-1} t$ for some $\lambda\in \C^*$ (we shall see below
the reason of choosing $\lambda^{-1}$ instead of $\lambda$). 
In particular, we are interested in the description of the set $\mathbb{M}$ of complex numbers $\lambda\in \C^*$ such that 
$M_\lambda =  \{ t\in\C, \, \lambda^{-1} t \in K(z^d + t) \}$ is connected. This set may be viewed as some kind of "higher" Mandelbrot set. 

We sum up in the next theorem what we know about this set. 

\begin{theorem}~
\begin{enumerate}
\item
A point $\lambda \in \C^*$ belongs to $\mathbb{M}$ iff $M_\lambda \subset \cal{M}(d,0)$.
\item
$\mathbb{M}\cap \{ |\lambda| \ge 1 \} = \U_d$.
\item
$\mathbb{M} \supset \{0< |\lambda| \le1/8\}$ and $\partial M_\lambda$ is a quasi-circle for all $0< |\lambda| \le1/8$.
\item
The set $\mathbb{M} \cup \{0 \}$ is closed and perfect.
\end{enumerate}
\end{theorem}
\begin{proof}
If $M_\lambda$ is not included in $\cal{M}(d,0)$, then Lemma~\ref{lem:total-dis-mandel} implies that $\lambda \notin \mathbb{M}$. 
Suppose now that $M_\lambda \subset \cal{M}(d,0)$ so that $g_{M_\lambda} \ge g_{\cal{M}(d,0)}$. It follows from 
Propositions~\ref{prop:unicritical cst} and~\ref{prop:unicritical pos} that 
\[ g_{P_t}(\lambda^{-1} t) \ge d g_{P_t}(0)\]
so that we may evaluate the B\"ottcher coordinate at $\lambda^{-1} t$ for any $t\notin M_\lambda$, and 
$t \mapsto \varphi_{d,t}(\lambda^{-1} t)$ defines a conformal map $ \C \setminus M_\lambda \to \C \setminus \bar{\D}(0,1)$
which is tangent to $\lambda^{-1} t$ at infinity, and tends to $1$ in modulus when $t \to \partial M_\lambda$. 
It thus defines a conformal isomorphism, hence $M_\lambda$ is connected.

\smallskip

Point (2) follows from 
Theorem~\ref{thm:car-mandel}. 

\smallskip

For (3), we start by observing  that if $|z| \ge \max \{ 2, |t|\}$, then
$|P_t(z)|  = |z^d +t| \ge |z|^d - |t| > 2|z| - |t| \ge |z|$ so that 
by induction  the sequence $|P^n_t(z)|$ is strictly increasing to infinity. 
If $|\lambda| \ge 8$ and $|t| \ge 1/4$, then $|\lambda t| \ge \max \{ 2, |t|\}$, 
so that $M_\lambda \subset \D (0, 1/4)$. 
But $\D (0, 1/4)$ is included in the central hyperbolic component $\heartsuit$ of $\cal{M}(d,0)$
consisting of those parameters $t$ for which $0$ converges to an attracting fixed point of $P_t$. In particular, $M_\lambda\subset \cal{M}(d,0)$ in this case, whence $\lambda\in \mathbb{M}$ by the first point. 
It remains to prove that, when $0<|\lambda|\leq1/8$, the set $\partial M_\lambda$ is a quasi-circle. Indeed, the component $\heartsuit$ is a $J$-stable family, whence there exists a holomorphic motion $h:\heartsuit\times \mathbb{S}^1\to\C$ of the Julia set. According to Proposition~\ref{prop:bifJstab}, 
\[\partial M_\lambda=\{t\in\heartsuit\, : \ \lambda^{-1}t\in h_t(\mathbb{S}^1)\}\Subset \heartsuit,\]
i.e. $t$ satisfies $t=\lambda h_t(e^{i\theta})$ for some $\theta\in\R$. Recall that $h_t$ is quasi-conformal with quasi-conformality constant $K_t:=(1+|\rho(t)|)/(1-|\rho(t)|)$, where $\rho(t)\in\D$ is the multiplier of the attracting fixed point of $P_t$. In particular, 
$\partial M_\lambda=\tilde{h}(\mathbb{S}^1)$, where $\tilde{h}$ is the quasi-conformal extension of $t\mapsto h_t^{-1}(\lambda^{-1}t)$.

\smallskip

Let us now argue why $\mathbb{M}$ is closed in $\C^*$. Suppose that $M_{\lambda_0}$ is not connected. 
Write  $\mu_\lambda = \Delta g_{P_t}(\lambda^{-1}t)$. 
By what precedes, we can find a small disk $U$
whose closure does not intersect $\cal{M}(d,0)$ and
such that  $\mu_{\lambda_0} (U) >0$. 
Since the Green function $g_{P_t}(z)$ is continuous in both variables, it follows that $(\lambda, t) \mapsto g_{P_t}(\lambda^{-1} t)$ is also continuous
hence the measures $\lambda \mapsto \mu_{\lambda}$ also varies continuously. For any $\lambda$ close enough to $\lambda_0$, 
we get $\mu_{\lambda}(U) >0$, hence $M_\lambda$ is not included in $\cal{M}(d,0)$, which implies the complement of $\mathbb{M}$ to be open. 

\smallskip

Define
\[ 
G_\mathbb{M}(\lambda) := \int_\C g_{\cal{M}(d,0)}   \, d\mu_\lambda~. \]
Since $ g_{P_t}(\lambda^{-1}t) = \log |t| - \log|\lambda| + o(1)$, an integration by parts yields
\[ 
G_\mathbb{M}(\lambda) = 
\int_\C \left(g_{P_t}(\lambda^{-1}t) +  \log|\lambda|\right) \, \Delta( g_{\cal{M}(d,0)})
= \log |\lambda| +  \int_\C g_{P_t}(\lambda^{-1}t) \, \Delta( g_{\cal{M}(d,0)})
~. \]
Since the support of $ \Delta( g_{\cal{M}(d,0)})$ is compact, and $g_{P_t}(\lambda^{-1}t) $
is continuous in both variables and subharmonic, 
$G_\mathbb{M}$ is continuous and subharmonic on $\C$.
When $\lambda \neq0$, observe that $G_\mathbb{M}(\lambda) = 0$ iff $g_{\cal{M}(d,0)} \equiv 0$ on $M_\lambda$
and this is equivalent to $M_\lambda \subset \cal{M}(d,0)$. 
In other words, $\{G_\mathbb{M}(\lambda) = 0\}=\mathbb{M}$. 
It follows from the maximum principle applied to $G_\mathbb{M}$ that  $\mathbb{M}$ has no isolated point. 
\end{proof}

\begin{remark}
Observe that $g_{P_t}(\lambda^{-1}t) \to d\times g_{\cal{M}(d,0)}$ uniformly on compact subsets of $\C^*$
as $\lambda \to \infty$, hence 
\[
\lim_{\lambda \to \infty} G_\mathbb{M}(\lambda) -\log |\lambda|
= \int_\C  g_{\cal{M}(d,0)}\, \Delta( g_{\cal{M}(d,0)}) = 0~.\]
Since the capacity of $\mathbb{M}$ being $<1$, the function
 $G_\mathbb{M}$ is however not the Green function of  $\mathbb{M}$, and is
 not harmonic on $\C\setminus \mathbb{M}$.
 Note that the boundary of $\mathbb{M}$ is included in $\supp (\Delta G_\mathbb{M})$.
 \end{remark}

Let us ask a couple of natural questions related to the geometry of $\mathbb{M}$.

\begin{question}
Is $\mathbb{M}$ connected? What is its logarithmic capacity?
\end{question}

%
%
%
%
%
%


\chapter[Special curves]{Special curves in the parameter space of polynomials}\label{chapter:special}


We collect all informations from the previous chapters, and prove in \S\ref{sec:proof thmG}
Baker-DeMarco's conjecture characterizing curves in the moduli space of complex polynomials containing an infinite set of PCF parameters 
(Theorem~\ref{thm:specialclas} from the Introduction). \index{curve!special}

In the subsequent sections, we investigate a combinatorial classification of special curves inspired by the classification of
PCF polynomials in terms of Hubbard trees. More precisely, we seek a one-to-one correspondence between special curves and 
decorated graphs, but due to the presence of symmetries, this task turns out to be delicate to achieve. We have thus contended ourselves 
to present a partial correspondence
encoding a large class of special curves by a combinatorial gadget that we call critically marked dynamical graphs. 

We define and study this notion in \S\ref{sec:classif special}. To any polynomial $P$ is attached a natural marked dynamical 
graph $\Gamma(P)$ that encodes its critical relations and its symmetries. In \S\ref{sec:special graphs}, we show how to associate
a marked dynamical graph to an irreducible curve, and define the category of special graphs.

In Section \S\ref{sec:realization}, we give quite general sufficient conditions on a special graph $\Gamma$ in order to ensure the existence of a special curve $C$
whose dynamical graph is isomorphic to $\Gamma$ (Theorem~\ref{thm:realizability}). Its proof is strongly inspired by previous works by DeMarco and McMullen. Although the overall strategy that we follow is similar to the proof of~\cite[Theorem~1.2]{DeMarco-McMullen}, our approach is somewhat more involved since we need to control when orbits of critical points merge whereas DeMarco and McMullen  only quantified
when the Green functions at critical points agree. 

In \S\ref{sec:correspondence}, we exhibit a correspondence between special curves and the class of marked dynamical graphs that we have defined (Theorem~\ref{thm:correspondence}).

The applicability of this theorem relies on our ability to realize PCF polynomials whose critical points have prescribed period and preperiod. 
In \S\ref{sec:real-PCF}, we state and prove a series of results on the realizability of some combinatorics by PCF polynomials. Our results are complete for unicritical polynomials, and fairly optimal in any degree for strictly PCF polynomials (Theorem~\ref{thm:realizationPCF}).

We conclude this chapter by discussing the classification of special curves in low degrees (\S\ref{sec:class-petite}), and
by a series of questions and open problems on the geometry of special curves (\S\ref{sec:classic}).

\section{Special curves in the moduli space of polynomials}\label{sec:proof thmG}

In this section we give a proof of Theorem~\ref{thm:specialclas}.
Let us first recall its statement. 

\specialclas*

\begin{proof}
If the family is not primitive, then by definition one can find a finite ramified cover $C' \to C$,  a family of polynomials $Q$ parameterized by $C'$, and a
symmetry $\sigma \in \Sigma(P)$  such that $P = \sigma Q^n$ for some $n \ge2$. We are thus in Case 1. 

From now on, we suppose that the family is primitive. Since all critical points are marked and the family is not isotrivial, we can make a base change, and 
suppose that $C$ is an algebraic curve in $\A^{d-1}$ so that $P$ can be written under the form
\begin{equation*}
P_{c,a}(z)\pe \frac{1}{d}z^d+\sum_{j=2}^{d-1}(-1)^{d-j}\sigma_{d-j}(c)\frac{z^j}{j}+a^d,
\end{equation*}
where $\sigma_j(c)$ is the monic symmetric polynomial in $(c_1,\ldots,c_{d-2})$ of degree $j$,
so that $c_0 = 0, c_1, \cdots, c_{d-2}$ are the critical points of $P$ (and $a^d$ is a critical value). 

Recall that PCF polynomials of the form $P_{c,a}$ have algebraic coefficients by Corollary~\ref{cor:PCF}. 
Since $C$ contains a Zariski-dense subset of $\bar{\Q}$-points, it can be defined by equations with coefficients lying in a number field (see~\cite[\S4.1 page 384]{specialcubic} for a proof). 

By Theorem~\ref{thm:active-characterization1}, for any passive critical points $c_i$ there exist integers $n_i>m_i$ such that 
$P^{n_i}(c_i) = P^{m_i}(c_i)$. 
Let $\mathsf{A}$ be the set of active critical points. It is non-empty since otherwise $P_{c,a}$ would be PCF for all $(c,a) \in C$ by Theorem~\ref{thm:demarco-stab}.
The family being defined over a number field, we may apply Theorem~\ref{tm:unlikely-marked} to each pair of active critical points 
$(P,c_i)$ and $(P,c_j)$ which implies the result.
\end{proof}

One can complement the previous result by the following theorem. 
Recall the definition of the critical heights $h_\bif$ and $\widetilde{h_\bif}$ on p.\pageref{def:bif-height}.
\begin{theorem}\label{tm:crit}
Pick any non-isotrivial  family $P$ of polynomials of degree $d\ge2$ with marked critical points which is parameterized by an algebraic curve $C$
containing infinitely many PCF parameters and defined over a number field $\KK$. 
Let $c$ be any active critical point. 

Then $\mathrm{Preper}(P,c)=\mathrm{PCF}(P)$, and there exists 
$\alpha \in\Q_+^*$ such that for any place $v\in M_\KK$ we have
\[\alpha \cdot g_{P,v}(c) = G_v(P); \text{ and }
\alpha \cdot \mu_{P,c,v}=\mu_{\bif,v}.
\]
In particular, the height functions $h_{P,c}$, $h_{\bif}$ and $\widetilde{h_\bif}$ are proportional on $C(\bar{\KK})$.
\end{theorem}

\begin{proof}
All passive critical points are persistently preperiodic so that $\mathrm{PCF}(P)$
is the intersection of the loci $\mathrm{Preper}(P,c)$ for all active critical points $c$. But the previous theorem implies
that these loci are all equal, hence $\mathrm{Preper}(P,c)=\mathrm{PCF}(P)$.

Let $\mathsf{A}$ be the set of active critical points as above. 
Observe that for any place $v \in M_\KK$, we have
\[
G_v(P) = \max \{ g_{P,v}(c_i), \, P'(c_i) =0 \} = \max \{ g_{P,v}(c_i), \, c_i \in \mathsf{A} \}~.\]
Apply Theorem~\ref{tm:unlikelyprecise} to all dynamical pairs $(P,c)$ with $c\in \mathsf{A}$. Item (6)
implies that $g_{P,c,v}$ are all proportional which leads to $\alpha \cdot g_{P,v}(c) = G_v(P)$, for some $\alpha\in\Q_+^*$ independent of $v$.

Taking the Laplacian on both sides yields $\alpha' \cdot \mu_{P,c_i,v}=\mu_{\bif,v}$ for some constant $\alpha'\in\Q_+^*$.
Similarly, we have 
\[\alpha \cdot h_P(c) = \widetilde{h_\bif}
\]
for any $c \in \mathsf{A}$, 
hence  $ h_{P,c}$ and $h_{\bif}$ are proportional, as claimed.
\end{proof}


\section{Marked dynamical graphs} \label{sec:classif special}

We define and study the notion of marked dynamical graphs, which encode 
the dynamical relations between critical orbits of a polynomial.
 
\subsection{Definition}
A vector field $\xi$ on a graph $\Gamma$ is an orientation of each edge of $\Gamma$ such that for each vertex $v$ there exists a unique
outgoing edge at $v$. A flow on a graph is a (continuous) graph map $\mu\colon \Gamma \to \Gamma$ such that there exists a vector field $\xi$
for which the unique outgoing edge at a vertex $v$ is the edge $[v , \mu(v)]$.

Suppose $\Gamma$ is a connected graph endowed with a vector field. If it  is not a tree, then it is the union of a single loop with finitely many trees attached to it, 
and the induced flow $\mu$ maps a point in any of the trees into the loop after finitely many iterations, whereas the restriction of $\mu$ to the loop is a periodic rotation. 
In the case $\Gamma$ is a tree, there exists a height function $H \colon V(\Gamma) \to \Z$ such that $H(\mu(z)) = \mu(H(z)) + 1$
which is uniquely defined once the value at one point is fixed.

\begin{figure}[h]
\centering
 \def\svgwidth{6cm}
\begingroup%
  \makeatletter%
  \providecommand\color[2][]{%
    \errmessage{(Inkscape) Color is used for the text in Inkscape, but the package 'color.sty' is not loaded}%
    \renewcommand\color[2][]{}%
  }%
  \providecommand\transparent[1]{%
    \errmessage{(Inkscape) Transparency is used (non-zero) for the text in Inkscape, but the package 'transparent.sty' is not loaded}%
    \renewcommand\transparent[1]{}%
  }%
  \providecommand\rotatebox[2]{#2}%
  \newcommand*\fsize{\dimexpr\f@size pt\relax}%
  \newcommand*\lineheight[1]{\fontsize{\fsize}{#1\fsize}\selectfont}%
  \ifx\svgwidth\undefined%
    \setlength{\unitlength}{476.5375682bp}%
    \ifx\svgscale\undefined%
      \relax%
    \else%
      \setlength{\unitlength}{\unitlength * \real{\svgscale}}%
    \fi%
  \else%
    \setlength{\unitlength}{\svgwidth}%
  \fi%
  \global\let\svgwidth\undefined%
  \global\let\svgscale\undefined%
  \makeatother%
  \begin{picture}(1,0.47453526)%
    \lineheight{1}%
    \setlength\tabcolsep{0pt}%
    \put(0,0){\includegraphics[width=\unitlength,page=1]{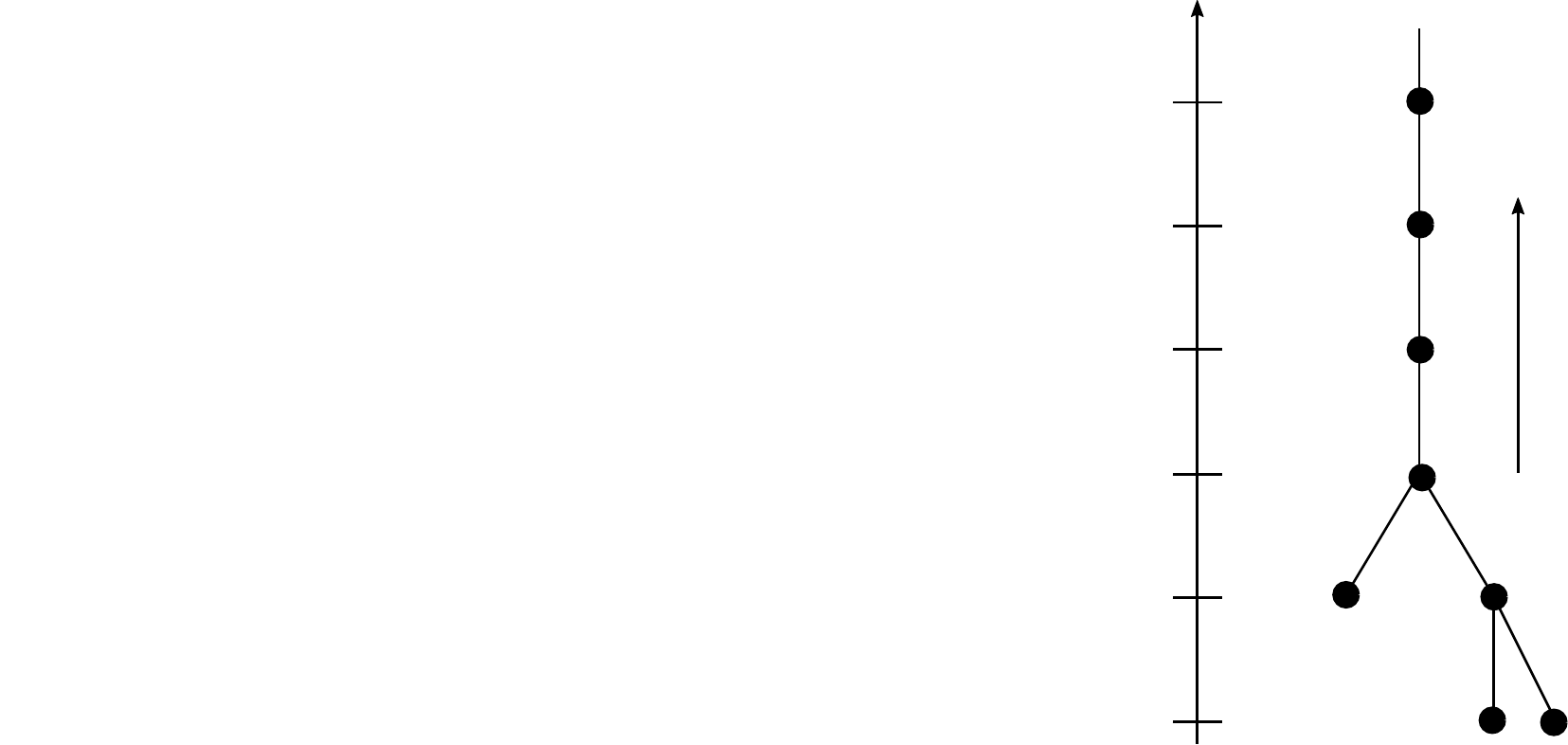}}%
    \put(0.68467587,0.44200711){\makebox(0,0)[lt]{\lineheight{1.25}\smash{\begin{tabular}[t]{l}$H$\end{tabular}}}}%
    \put(0.68507878,0.2518165){\makebox(0,0)[lt]{\lineheight{1.25}\smash{\begin{tabular}[t]{l}$0$\end{tabular}}}}%
    \put(0.66006834,0.1678105){\makebox(0,0)[lt]{\lineheight{1.25}\smash{\begin{tabular}[t]{l}$-1$\end{tabular}}}}%
    \put(0.68320011,0.33050909){\makebox(0,0)[lt]{\lineheight{1.25}\smash{\begin{tabular}[t]{l}$1$\end{tabular}}}}%
    \put(0,0){\includegraphics[width=\unitlength,page=2]{dynamical-graph.pdf}}%
  \end{picture}%
\endgroup%
  
  \caption{Flows on graphs}
 \label{fig:flow on graphs}
\end{figure}

A  finite or infinite graph $\Gamma$ with vertex\index{graph!dynamical} set $V(\Gamma)$ is a \emph{dynamical graph of degree $\le d$ marked by a finite set $\mathcal{F}$}
when it is endowed with the following data:
\begin{itemize}
\item[(G1)]
a  map $\mu \colon \mathcal{F} \to \Gamma$;
\item[(G2)]
a function $d_\pi \colon V(\Gamma) \to \N^*$ such that  $\sum_{\pi(w) =v} d_\pi(w) \le d$ for all $v\in V(\Gamma)$;
\item[(G3)]
a flow $\pi \colon \Gamma \to \Gamma$;
\item[(G4)]
an action of $\U_k$ on $V(\Gamma)$ 
such that: 
\begin{itemize}
\item
the action is free on the complement of at most one point;
\item
$\pi( g \cdot v) = \rho(g) \cdot \pi(v)$ for some morphism $\rho \colon \U_k \to \U_k$;
\item 
$d_\pi( g \cdot v)  = d_\pi(v)$;
\end{itemize}
\item[(G5)]
an action of $\U_k$ on $ \mathcal{F}$ such that $ \mu(g \cdot i) = g \cdot \mu(i)$.
\end{itemize}
We also impose the following minimality condition:
\begin{itemize}
\item[(G6)]
$\Gamma$ is the orbit by the flow 
of  the set $\{g \cdot \mu(i), \, g \in \U_k, \, i \in   \mathcal{F}\}$.
\end{itemize}

The group $\U_k$ is called the symmetry group of the marked dynamical graph. 
One can always replace $\mathcal{F}$ by its image in $V(\Gamma)$ and suppose that $\mu$ is injective. 
For our purposes it is however convenient not to do so. 

\begin{remark}
When $k=1$, we say that $\Gamma$ has no symmetry. In this case, $\Gamma$ is completely determined by a finite graph, 
e.g. by the union of its finite components and the convex hull of the points $\mu(i)$ lying in infinite tree components. 
\end{remark}

\begin{remark}
When $k\ge 2$, observe that the action of $\U_k$ does not extend in general
to a continuous action on the edges of $\Gamma$ (except if $\rho = \id$). 
\end{remark}

Marked dynamical graphs encodes the relations between iterates of a finite set of points. 

\begin{example}\label{ex:mark-dyn}
Let  $\Phi \colon U \to U$ be any self-map on a set $U$, and 
let $\mathcal{F} \subset U$
be any finite collection of points. We define a marked dynamical graph $G(\mathcal{F},\Phi)$ as follows.
Vertices of $G(\mathcal{F},\Phi)$ are points in $U$ lying in the forward orbit of
at least one point in $\mathcal{F}$; an edge
joins two vertices $v$ and $v'$ iff $v=\Phi(v')$ or $v'=\Phi(v)$; 
the flow is defined by $\pi_\mathcal{F}(z) = \Phi(z)$;  the marking
$\mu_\mathcal{F} \colon \mathcal{F} \to G(\mathcal{F},\Phi)$ is given
by the natural inclusion; and $d_\pi (v) = \# \Phi^{-1}(v)$.
\end{example}

\subsection{Critically marked dynamical graphs}\label{sec:criticallymarkedgraph}
We say that $\Gamma$ is a \emph{critically marked dynamical graph}\index{graph!critically marked dynamical} of degree $d$
when $\mathcal{F} = \{ 0, \cdots, d-2\}$,  $d_\pi (v) = 1+ \Card \mu^{-1} (v)$, 
and the action of $\U_k$ on $V(\Gamma)$ 
satisfies the extra condition (G4'):
\begin{itemize}
\item
either it is free on $\mu(\mathcal{F})$ and $\pi( g \cdot v) = g \cdot \pi(v)$;
\item
or there exists a vertex $v_* \in \mu(\mathcal{F})$ which is fixed by both $\U_k$
and $\pi$, and 
$\pi( g \cdot v) = g^{d_\pi(v_*)} \cdot \pi(v)$ for all $v$;
\item
or there exists a vertex $v_* \in \mu(\mathcal{F})$ which is fixed by $\U_k$, 
but not by $\pi$,  $k =d_\pi(v_*)$, and
$\pi( g \cdot v) = \pi(v)$ for all $v$.
\end{itemize}

\begin{lemma}\label{lem:basic-count}
For any critically marked dynamical graph of degree $d$ having a symmetry group of order $k$, 
we have $d -1 = \sum_v (d_\pi(v) -1)$ and $k \le d$. 
\end{lemma}
\begin{proof}
The first equality follows from $\Card(\mathcal{F}) = d-1$. For the second inequality, 
suppose first that $\U_k$ acts freely on $\mu(\mathcal{F})$. Then by (G5), we have
$d_\pi(g \cdot v) = d_\pi(v)$ hence $k$ divides $d-1$. 
When $\U_k$ fixes a point $v_* \in \mu(\mathcal{F})$, then
$d = d_\pi(v_*) + \sum_{v\neq v_*} (d_\pi(v) -1)$ hence
$k$ divides $d-d_\pi(v_*)$, or $k=1$.
\end{proof}

\begin{remark}
Given any critically marked dynamical graph $\Gamma$ and any subgroup $G$ of its symmetry group $\U_k$
such that $G \cdot \mu(\mathcal{F}) = \U_k \cdot \mu(\mathcal{F})$, 
we can build a new marked dynamical graph by replacing $\U_k$ by $G$. 
\end{remark}
This remark leads to the following notion. 

\begin{definition}
A critically marked dynamical graph $\Gamma$ of degree $d$ is said to be asymmetric if it cannot be embedded 
into a critically marked dynamical graph $\Gamma'$ of degree $d$ having a non-trivial group of symmetry.
\end{definition}

There are simple criteria detecting whether a graph is asymmetric or not. We refer to \S \ref{sec:asymmetric graph} for results
in that direction. 

\smallskip

Let us discuss graphs with $\U_2$-symmetries, and let $\Gamma$ be any critically marked dynamical graph having
this group of symmetries. Two possibilities arise. 

Either $\rho (-1) =-1$ so that the action commutes with $\pi$, 
and $\U_2$ acts continuously on the graph $\Gamma$.  In that case the action of $\U_2$ on $\mathcal{F}$ is free by (G4').
It may or may not have a fixed point in $V(\Gamma) \setminus \mathcal{F}$.

Or $\rho (-1) =+1$, and we have $\pi( - v) = \pi(v)$. In this case, 
(G4') implies there exists a vertex $v_*\in\mu(\mathcal{F})$ which is fixed by $\mathbb{U}_2$ but not by $\pi$.
Observe that $\pi ( - \pi^n(v)) = \pi^{n+1}(v)$ for all vertices.

\begin{figure}[h]
\centering
 \def\svgwidth{10cm}
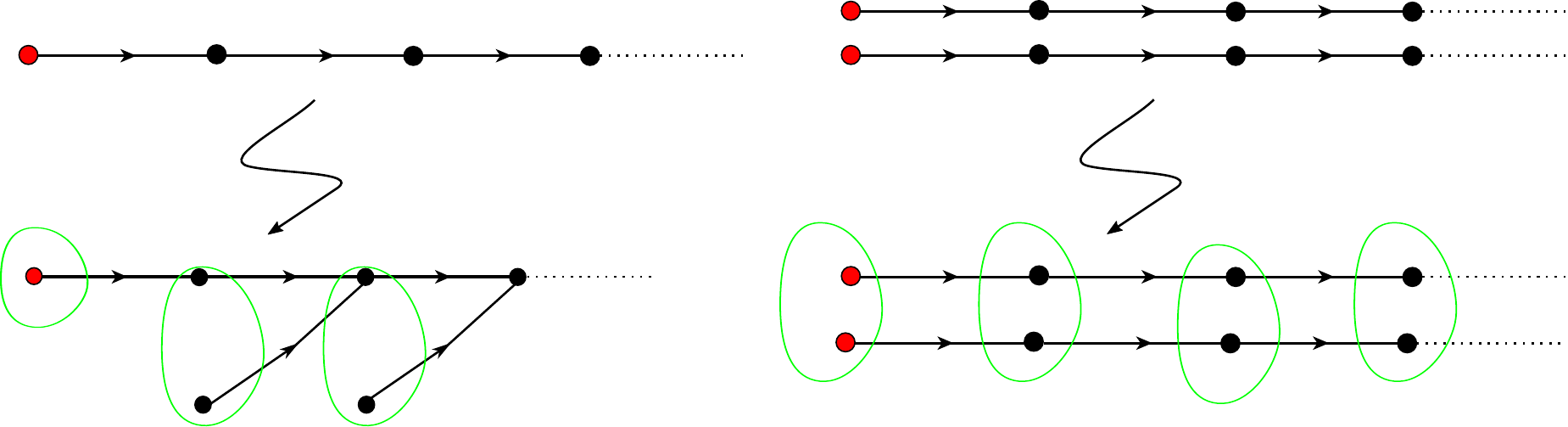
  \caption{Two examples of non-asymmetric graphs}
 \label{fig:non-asymmetric}
\end{figure}

\subsection{The critical graph of a polynomial}
To any polynomial $P$ of degree $d\ge 2$ with marked critical points $c_0, \ldots c_{d-2}$, we can attach a (possibly disconnected) critically marked dynamical graph
$\Gamma(P)$ which encodes the critical relations as follows.  Vertices of $\Gamma$
are given by $\sigma P^n(c_i)$, with $i \in \{ 0, \cdots, d-2\}$, $\sigma \in \Sigma(P)$ and $n\ge0$, and we draw an oriented 
edge between any point $z$ and its image $P(z)$ so that $\Gamma$ carries a canonical flow $\pi \colon \Gamma \to \Gamma$ sending $z$ to $P(z)$. 
The marking is given by the map $\mu \colon \{ 0, \cdots, d-2\} \to \Gamma$ which sends $i$ to $c_i$. It satisfies $\pi \circ \mu = \mu \circ P$. 
Observe that for any vertex $v$ of $\Gamma$, the degree $d(v)$ equals the local degree of $P$ at $v$, so that
(G2) holds since the number of preimages of a point by $P$ counted with multiplicity equals $d$.

Finally recall that $\Sigma(P)$ is canonically isomorphic to $\U_k$ for some $1 \le k\le d$. We let
the action of $\U_k$ on the set of vertices of $\Gamma$ be the one induced by $\Sigma(P)$ on the orbits of the critical points. 

Since $\U_k$ acts by rotation on the complex plane, its action is free off the origin which is fixed so that 
the action of $\U_k$ on $\Gamma$ satisfies (G4). If $P$ is a monic and centered polynomial and if we write its minimal decomposition 
$P(z) = z^m Q(z^k)$ with $Q(0) \neq 0$, then either $m=1$ so that the fixed point of the action by $\U_k$ is not a critical point; 
or $m\ge 2$ and $0$ is a critical point or order $m$; or $m=0$ and $0$ is a critical point which is not fixed by $P$. 
This implies (G4') and $\Gamma(P)$ is a critically marked dynamical graph
in the sense above.

\medskip

For any polynomial outside a countable union of  algebraic subvarieties in the moduli space, the marked dynamical graph is a union of 
$d-1$ rays whose extremities are $\mu(i), i =0, \cdots, d-2$. The marked dynamical graph of a PCF polynomial is a finite union of
finite graphs (having a single loop with finitely many trees attached to it).

The marked dynamical graph of a polynomial encodes the dynamical relations between critical points in the following sense. 
\begin{lemma}\label{lem:how-to-prove-gamma}
Let $P$ and $Q$ be two polynomials with marked critical points of the same degree $d \ge2$.
Then the two critically marked dynamical graphs $\Gamma(P)$ and $\Gamma(Q)$ are equal iff the following two properties hold.
\begin{enumerate}
\item
There exists $k\le d$ such that $\U_k= \Sigma(P) = \Sigma(Q)$ and the two morphisms $\rho_P \colon \Sigma(P)\to \Sigma(P)$
and $\rho_Q \colon \Sigma(Q)\to \Sigma(Q)$ are identical;
\item
for each pair of integers $i, j \in \{0, \cdots, d-2\}$ and for each $\sigma \in \U_k$, 
the two sets $\{ (n,m), \, P^n(c_i) = \sigma P^m(c_j)\}$ and $\{ (n,m), \, Q^n(c_i) = \sigma Q^m(c_j)\}$
are equal.
\end{enumerate}
\end{lemma}
\begin{proof}
Suppose that $P$ and $Q$ have identical critically  marked dynamical graphs.
Then we have $\Sigma(P) = \Sigma(Q)$. Pick any vertex $v$ of the graph $\Gamma(P)$ whose image by $\pi$ is not fixed by $\U_k$. 
For any $g \in \U_k$, we have $\pi( g \cdot v)  = \rho_P(g) \cdot \pi(v) = \rho_Q(g) \cdot \pi(v)$. 
Since $\U_k$ acts freely on the orbit of $\pi(v)$, we conclude that $\rho_P = \rho_Q$ hence (1) is satisfied. 
Choose $i, j \in \{0, \cdots, d-2\}$ and $\sigma \in \U_k$. 
Since
\begin{equation}\label{eq:identify graph}
\left\{ (n,m), \, P^n(c_i) = \sigma P^m(c_j)\right\} = \left\{ (n,m), \, \pi^n (\mu(i)) = \sigma \cdot \pi^m (\mu(j))\right\} 
\end{equation}
 condition (2) holds.

\smallskip

Conversely suppose that $P$ and $Q$ satisfy (1) and (2). 
We first show that one can recover the marked dynamical graph of $P$ from $\rho =\rho_P$ and the sets
$\{ (n,m), \, P^n(c_i) = \sigma P^m(c_j)\}$ where $i, j$ range over all pairs in $\{0, \cdots, d-2\}$ and $\sigma$ over all elements of $\U_k = \Sigma(P)$.
We first build the infinite graph $\hat{\Gamma}$ whose vertices are triple $(i,n,\sigma)$ with  $i \in \{0, \cdots, d-2\}$, $n \ge0$ and $\sigma \in \U_k$; 
and edges join $(i,n,\sigma)$ to $(i,n+1,\rho(\sigma))$. On this graph the map $\hat{\pi}(i,n,\sigma) = (i,n+1,\rho(\sigma))$ is a flow, and we have a marking $\hat{\mu}(i) = (i,0,1)$.
We also have a natural action by $\U_k$ given by composition on the last factor $\sigma' \cdot (i,n,\sigma) = (i,n, \sigma' \sigma)$.

We observe now that the map sending $(i,n,\sigma)$ to $\sigma P^n(c_i)$ identifies $\Gamma(P)$ 
as the quotient of $\hat{\Gamma}$ by the relation identifying $(i_1,n_1,\sigma_1)$ and $(i_2,n_2,\sigma_2)$ iff
$(n_1, n_2) \in \{ (n,m), \, P^n(c_i) = \sigma_1^{-1} \sigma_2 P^m(c_j)\}$. Moreover the flow on $\Gamma(P)$ is induced by $\hat{\pi}$ and similarly the marking is induced by
$\hat{\mu}$. 

This shows that the two graphs $\Gamma(P)$ and $\Gamma(Q)$ are isomorphic, that their markings, their flows, and their corresponding actions of $\U_k$ coincide.
\end{proof}

\subsection{The critical graph of an irreducible subvariety in the moduli space of polynomials}
We observe that one can attach to any irreducible subvariety of the moduli space of critically marked polynomials a marked dynamical graph.

\begin{proposition}
Let $V$ be any irreducible subvariety in the moduli space of critically marked polynomials. 
Then there exists a  marked dynamical graph $\Gamma(V)$ such that $\Gamma(P) = \Gamma(V)$
for all $P\in V$ outside a countable union of subvarieties. 
\end{proposition}

\begin{remark}
It is in general not possible to get equality $\Gamma(P) = \Gamma(V)$ on a Zariski dense open subset. When $V$ is a special curve, 
then the set of PCF polynomials in $V$ is infinite countable (hence Zariski dense) and for each of these polynomials the marked dynamical graph is
finite, although $\Gamma(V)$ is not.

However we expect that $\Gamma(P) = \Gamma(V)$ for a (euclidean) dense subset of $P \in V$ (see Question~\ref{qst:wringing} below). 
\end{remark}

\begin{proof}
By Proposition~\ref{prop:sameZariskiopen}, there exists a open Zariski dense subset $V^\circ \subset V$
such that $\Sigma(P) = \U_k$ for all $ P \in V^\circ$. Reducing $V^\circ$ if necessary, we may also assume
that the morphism $\rho_P \colon \Sigma(P) \to \Sigma(P)$ is induced by the same morphism $\rho \colon \U_k \to \U_k$
for all $P\in V^\circ$. 

For any $i,j \in \{0, \cdots, d-2\}$,  and $\sigma \in \U_k$, observe that 
the set $Z(n,m,i,j,\sigma) = \{ t \in V^\circ, \,  P^n_t(c_i) = \sigma P^m_t(c_j)\} $ is Zariski closed. 
The union $\mathcal{Z}$ of all sets $Z(n,m,i,j,\sigma)$ which have empty interior is thus a countable
union of strict subvarieties of $V$.

 All polynomials $P_t$ with $t \in V^\circ \setminus \mathcal{Z}$ have the same group of symmetries, 
 and the same sets $\{ (n,m), \, P^n(c_i) = \sigma P^m(c_j)\}$
so that the marked dynamical graph $\Gamma(P)$ is constant on this set by the previous lemma.
\end{proof}

\section{Dynamical graphs attached to special curves}\label{sec:special graphs}

We aim at characterizing marked dynamical graphs attached to special curves. 
Before doing so let us begin with the following observation. 
\begin{lemma}
 Let $v$ be any vertex of a marked dynamical graph $\Gamma$ and pick any symmetry $g$ of $\Gamma$.
Then $v$ has a finite $\pi$-orbit iff $g\cdot v$ has.
\end{lemma}

\begin{proof}
This result follows from the fact that the symmetry group of $\Gamma$ is a finite group and $\pi^n (g \cdot v) = \rho^n(g) \cdot \pi^n(v)$ for any integer $n$.
\end{proof}

A geometric consequence of the previous lemma is the following. 
For any symmetry $g$, a vertex $v$ belongs to an infinite 
connected component of $\Gamma$ iff $g \cdot v$ does. 

\begin{definition}
A marked dynamical graph is said to be \emph{special}\index{graph!special dynamical} if it  contains exactly one infinite connected component
up to symmetry. In other words, for any two infinite connected components
$T$ and $T'$ of the dynamical graph, there exists a symmetry $\sigma$ such that $\sigma(T) \cap T'\neq \varnothing$.
\end{definition}

\begin{theorem}\label{tm:marked graph is special}
Let $C$ be any irreducible curve in the moduli space of critically marked polynomials. 
If the family of polynomials induced by $C$ is primitive, then 
the following are equivalent: 
\begin{itemize}
\item
the curve $C$ is special (i.e. contains infinitely many PCF polynomials);
\item
the critically marked dynamical graph $\Gamma(C)$ is special.
\end{itemize}
\end{theorem}

For the proof, we rely on the next lemma.

\begin{lemma}\label{lm:partition-Crit}
Let $\Gamma$ be any special marked dynamical graph. 
Then there exists a partition $\{0, \cdots, d-2\} = \mathsf{A} \sqcup \mathsf{P}$ such that
\begin{itemize}
\item
the $\pi$-orbit of a point $\mu(i)$, $i \in \{0, \cdots, d-2\}$ is finite if and only if $i \in \mathsf{P}$;
\item
for any $i, j \in \mathsf{A}$, there exist $n,m \ge 0$, and a symmetry $\sigma$ such that
$\pi^n(\mu(i)) = \sigma \cdot \pi^m(\mu(j))$.
\end{itemize}
\end{lemma}

\begin{proof}
Let $\mathsf{P}$ (resp. $\mathsf{A}$) be the set of indices $i$ such that the critical point $c_i$ is passive (resp. active) on $C$.
By Theorem~\ref{thm:active-characterization1}, a critical point $c_i$ is passive if and only if it is stably preperiodic, hence the first point follows from the definition of $\mu(i)$. The second point follows from Theorem~\ref{thm:specialclas}.
\end{proof}

\begin{proof}[Proof of Theorem~\ref{tm:marked graph is special}]
Suppose first that the curve $C$ is special. Observe that at least one critical point $c$ is active on $C$ so that 
for all $t \in C$ outside a countable set, the orbit of $c$ under $P_t$ is infinite. 
Pick any $t \in C$ such that $\Gamma(P_t) = \Gamma(C)$ and $\{ P_t^n(c)\}_{n\ge0}$ is infinite.
By assumption the family $\{P_t\}_{t\in C}$ is primitive, and Theorem~\ref{thm:specialclas} implies that $\Gamma(P)$ has at most one infinite component (up to symmetry). 
Since $c$ is not preperiodic, it has exactly one infinite component, and $\Gamma(C)$ is special. 

\smallskip

Conversely, when the marked dynamical graph $\Gamma(C)$ is special, we let $\mathsf{A}$ be the set of indices $i\in \{0, \cdots, d-2\}$ 
such that $\mu(i)$ falls into the infinite component of $\Gamma(C)$. Observe that any critical point $c_i$ for which  $i \notin \mathsf{A}$ is stably preperiodic on $C$. 
Pick any $i \in \mathsf{A}$. Since $C$ is a curve inside the moduli space of critically marked polynomials, 
its image in the moduli space of polynomials remains a curve, hence the family is not isotrivial. Theorem~\ref{thm:active-characterization1} implies
 the critical point $c_i$ to be active.

It follows that the set of parameters $t\in C$ such that $P_t$ is PCF coincides with the set $\{ t \in C, c_i \text{ is preperiodic}\}$. 
The latter set is infinite countable by Montel's theorem (see,~e.g., \cite[Lemma~2.3]{favredujardin}) which shows that $C$ is special.
\end{proof}

The next two figures describe all special critically marked dynamical graphs of degree $2$ and $3$.

\begin{figure}[h]
\centering
 \def\svgwidth{10cm}
\begingroup%
  \makeatletter%
  \providecommand\color[2][]{%
    \errmessage{(Inkscape) Color is used for the text in Inkscape, but the package 'color.sty' is not loaded}%
    \renewcommand\color[2][]{}%
  }%
  \providecommand\transparent[1]{%
    \errmessage{(Inkscape) Transparency is used (non-zero) for the text in Inkscape, but the package 'transparent.sty' is not loaded}%
    \renewcommand\transparent[1]{}%
  }%
  \providecommand\rotatebox[2]{#2}%
  \newcommand*\fsize{\dimexpr\f@size pt\relax}%
  \newcommand*\lineheight[1]{\fontsize{\fsize}{#1\fsize}\selectfont}%
  \ifx\svgwidth\undefined%
    \setlength{\unitlength}{573.71052713bp}%
    \ifx\svgscale\undefined%
      \relax%
    \else%
      \setlength{\unitlength}{\unitlength * \real{\svgscale}}%
    \fi%
  \else%
    \setlength{\unitlength}{\svgwidth}%
  \fi%
  \global\let\svgwidth\undefined%
  \global\let\svgscale\undefined%
  \makeatother%
  \begin{picture}(1,0.6058455)%
    \lineheight{1}%
    \setlength\tabcolsep{0pt}%
    \put(0,0){\includegraphics[width=\unitlength,page=1]{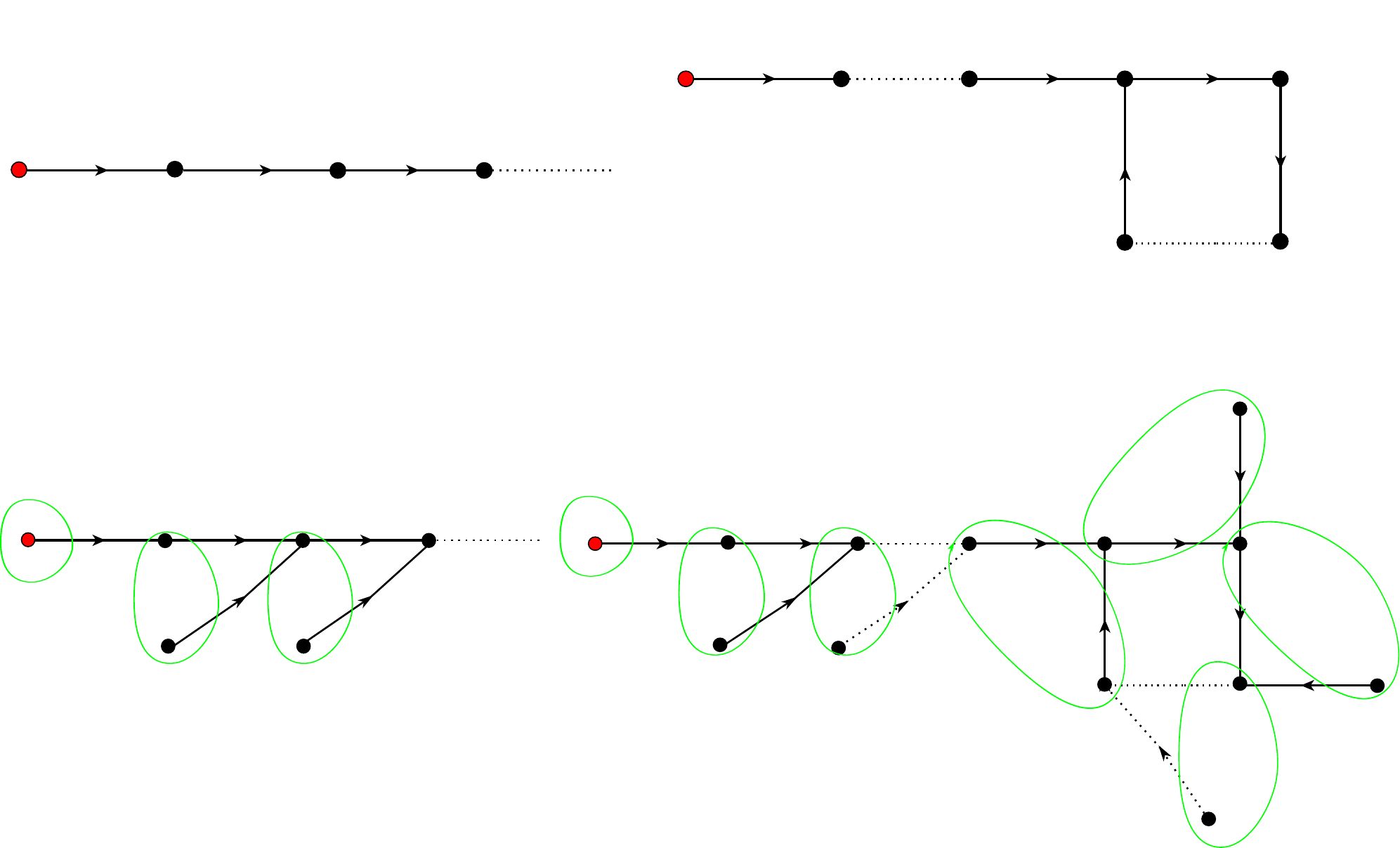}}%
    \put(0.0915752,0.59525106){\makebox(0,0)[lt]{\lineheight{1.25}\smash{\begin{tabular}[t]{l}With no symmetry.\\\end{tabular}}}}%
    \put(0.09516614,0.34152476){\makebox(0,0)[lt]{\lineheight{1.25}\smash{\begin{tabular}[t]{l}With $\mathbb{U}_2$-symmetry group\end{tabular}}}}%
  \end{picture}%
\endgroup%

  \caption{Special critically marked dynamical graphs of degree $2$: the critical point is marked in red, and $\mathbb{U}_2$-orbits in green}
 \label{fig:marked-degree2}
\end{figure}

\begin{figure}[h]
\centering
 \def\svgwidth{10cm}
\begingroup%
  \makeatletter%
  \providecommand\color[2][]{%
    \errmessage{(Inkscape) Color is used for the text in Inkscape, but the package 'color.sty' is not loaded}%
    \renewcommand\color[2][]{}%
  }%
  \providecommand\transparent[1]{%
    \errmessage{(Inkscape) Transparency is used (non-zero) for the text in Inkscape, but the package 'transparent.sty' is not loaded}%
    \renewcommand\transparent[1]{}%
  }%
  \providecommand\rotatebox[2]{#2}%
  \newcommand*\fsize{\dimexpr\f@size pt\relax}%
  \newcommand*\lineheight[1]{\fontsize{\fsize}{#1\fsize}\selectfont}%
  \ifx\svgwidth\undefined%
    \setlength{\unitlength}{553.05309233bp}%
    \ifx\svgscale\undefined%
      \relax%
    \else%
      \setlength{\unitlength}{\unitlength * \real{\svgscale}}%
    \fi%
  \else%
    \setlength{\unitlength}{\svgwidth}%
  \fi%
  \global\let\svgwidth\undefined%
  \global\let\svgscale\undefined%
  \makeatother%
  \begin{picture}(1,0.55505153)%
    \lineheight{1}%
    \setlength\tabcolsep{0pt}%
    \put(0,0){\includegraphics[width=\unitlength,page=1]{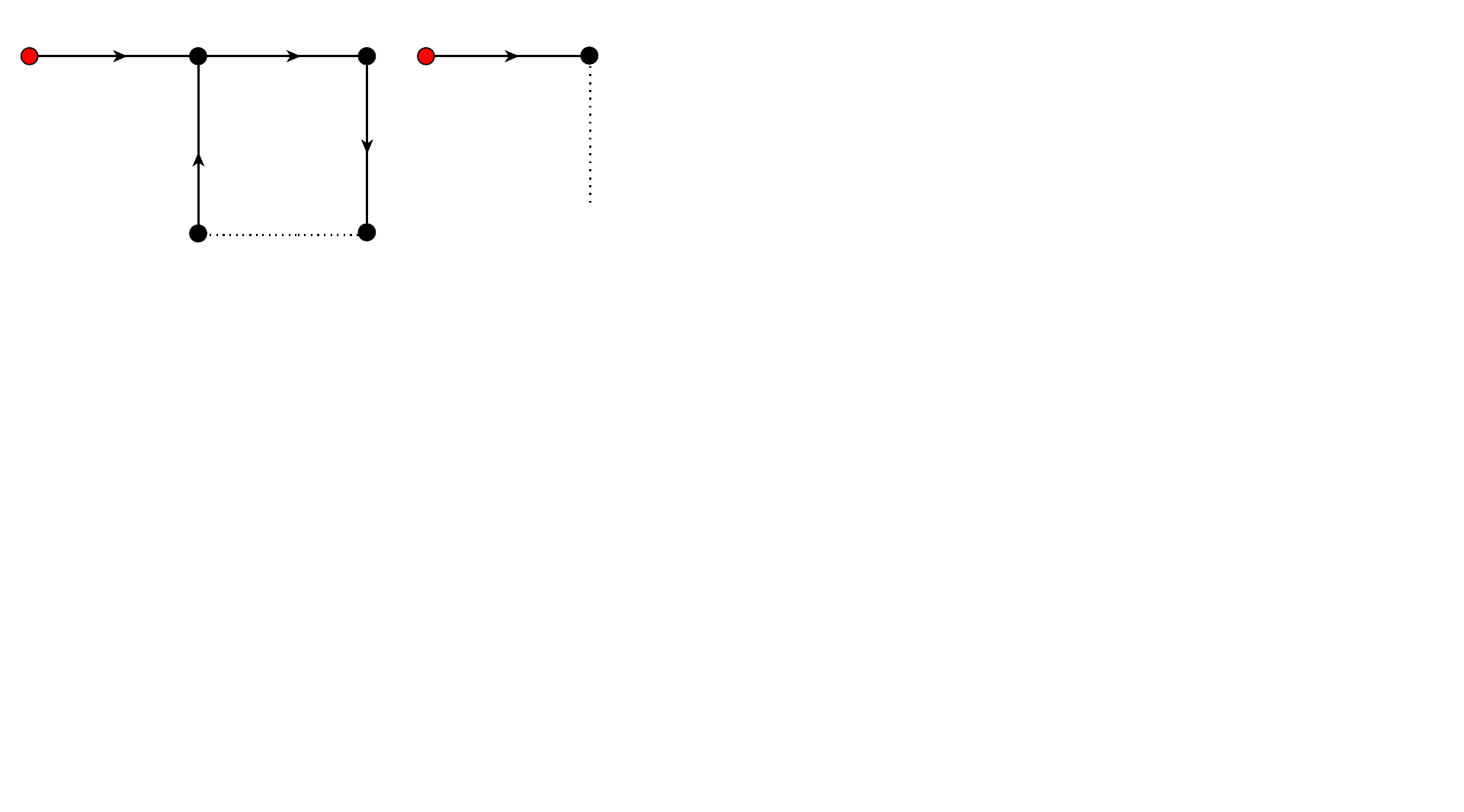}}%
    \put(0.30763618,0.29876963){\makebox(0,0)[lt]{\lineheight{1.25}\smash{\begin{tabular}[t]{l}Not realizable\end{tabular}}}}%
    \put(0,0){\includegraphics[width=\unitlength,page=2]{marked-graph-degree3.pdf}}%
    \put(0.3100988,0.54446397){\makebox(0,0)[lt]{\lineheight{1.25}\smash{\begin{tabular}[t]{l}Asymmetric\end{tabular}}}}%
    \put(0.30534707,0.13551588){\makebox(0,0)[lt]{\lineheight{1.25}\smash{\begin{tabular}[t]{l}With $\mathbb{U}_2$-symmetry group\end{tabular}}}}%
    \put(0,0){\includegraphics[width=\unitlength,page=3]{marked-graph-degree3.pdf}}%
  \end{picture}%
\endgroup%

  \caption{Special critically marked dynamical graphs of degree $3$. Circled red dots are multiple critical critical points}
 \label{fig:marked-degree3}
\end{figure}
 

\section{Realization theorem}\label{sec:realization}
This section is devoted to the proof of the following result which forms the bulk of the proof of our correspondence theorem
to be stated in the next section. 

Observe that if $\Gamma$ is critically marked dynamical  graph, then the union of all its finite components
is a finite critically marked dynamical  graph $\Gamma_\fin$ (of smaller degree). 

We shall say that a finite critically marked dynamical graph $\Gamma_0$ without symmetry is realizable by a PCF polynomial if there exists a PCF polynomial $P$
such that $\Gamma_0$ is equal to $\Gamma(P)$ with the action of $\Sigma(P)$ removed. 

\begin{theorem}\label{thm:realizability}
Let $\Gamma$ be any special asymmetric critically marked dynamical  graph such that:
\begin{itemize}
\item[(R1)]
two distinct marked points have different images;
\item[(R2)]
the finite dynamically marked subgraph $\Gamma_\fin$ is 
realizable by a PCF polynomial.
\end{itemize}
Then we can find a primitive polynomial $P$ with disconnected Julia set
such that $\Gamma(P) = \Gamma$.
\end{theorem}

We shall give in \S \ref{sec:real-PCF} quite general conditions on a graph to be realizable by a PCF polynomial. 

\medskip

During the whole construction, we fix a constant $\rho>1$, and write $M_d(z) = z^d$.
A brief explanation of our strategy is given at the end of \S \ref{sec:trunca}.

\subsection{Asymmetric special graphs}\label{sec:asymmetric graph}
Before embarking on the proof of the theorem above, we note the ubiquity of asymmetric graphs among special ones. 

\begin{proposition}
Let $\Gamma$ be any special dynamical graph without any symmetry. 
Suppose that $\Gamma$ is not asymmetric. 
Then
\begin{enumerate}
\item
either there exists a vertex $v_* \in \mu(\mathcal{F})$ which is fixed by $\pi$, 
an integer $k\ge2$ which divides a power of $d_\pi(v_*)$,
and a free action of $\U_k$ on $\mu(\mathcal{F})\setminus \{ v_*\}$ which preserves $d_\pi$;
\item 
or there exists an integer $k\ge2$, a vertex $v_* \in \mu(\mathcal{F})$ with infinite $\pi$-orbit and $d_\pi(v_*) =k$, 
and a free action of $\U_k$ on $\mu(\mathcal{F})\setminus \{ v_*\}$ which preserves $d_\pi$, and 
such that $\pi( g \cdot v) = \pi(v)$ for all $v \neq v_*$.
\end{enumerate}
\end{proposition}

\begin{remark}
Suppose $\Gamma$ is a special dynamical graph without any symmetry. 
Then 2.  implies $\Gamma$ to be non-asymmetric. 
It is however not true that 1. implies $\Gamma$ to be non-asymmetric.
\end{remark}

\begin{proof}
Since $\Gamma$ is special and has no symmetry, it has a unique infinite component $\Gamma_{\esc}$. 
Suppose that $\Gamma$ is not asymmetric, embed $\Gamma \subset \Gamma'$ where $\Gamma'$ is a critically
marked dynamical graph with symmetry group $\U_k$ with $k\ge2$.
By (G4'), we have three possibilities.

 If $\U_k$ acts freely on $\mu(\mathcal{F})$ then 
$\pi(g \cdot v) = g \cdot \pi(v)$, and $g \cdot \Gamma_{\esc} \cap \Gamma_{\esc} = \varnothing$
if $g \neq 1$. But points in $\mu(\mathcal{F})$ having an infinite $\pi$-orbit are permuted by $\U_k$
hence all belong to $\Gamma_{\esc}$ which is impossible. 

The second option is that $\U_k$ fixes a vertex $v_*$ which is also fixed by $\pi$, and that
$\pi( g \cdot v) = g^{d_\pi(v_*)} \cdot \pi(v)$. Pick any vertex $v\in \mu(\mathcal{F}) \cap \Gamma_0$, and 
$g$ a generator of $\U_k$. Since $\Gamma_0$ is connected, we have $\pi^n(g \cdot v) = \pi^n(v)$ for some $n$, 
hence $k$ divides $d_\pi(v_*)^n$. We thus fall into Case 1. 

 The last option is when $\U_k$ fixes a vertex $v_*$ which has infinite $\pi$-orbit. Then $\pi(g \cdot v) = \pi(v)$
 and the fact that $d_\pi$ is $\U_k$-invariant implies the result.
\end{proof}

\subsection{Truncated marked dynamical graphs}\label{sec:trunca}
We thus fix once and for all a special marked dynamical graph $\Gamma$ satisfying the assumptions of Theorem~\ref{thm:realizability}. 
As in Lemma~\ref{lm:partition-Crit}, we define $\mathsf{A}$ to be the set of $i\in \{0, \cdots, d-2\}$
such that the $\pi$-orbit of $\mu(i)$ is infinite; and let $\mathsf{P}$ be its complement.

We let $H \colon V(\Gamma) \to \Z\cup \{ -\infty \}$ be the unique (height) function such that $H(\pi(v)) = H(v) +1$, which we 
normalize by the condition $\max \{ H(\mu(i)), i = 0, \cdots, d-2\} = 0$.
By convention we set $H|_{\Gamma_\fin} = -\infty$.

Write $\Gamma_\esc = \Gamma\setminus \Gamma_\fin$.
We first build from $\Gamma_\esc$ a sequence of marked dynamical graphs as follows. 
For any $n\in \Z$, let $\Gamma_n  = \Gamma_\esc \cap \{H \ge 1- n\}$. Observe that $\Gamma_n$ is connected since 
for any two vertices $v, v'$ of $\Gamma_\esc$, one may find integers $m, m'$ such that $\pi^m(v) =\pi^{m'}(v')$.
It is also naturally a marked dynamical graph of $\Gamma$ as follows. 
We have a canonical (injective) marking
$\mu_n \colon \partial\Gamma_n \to \Gamma_n$, and the flow $\pi$ preserves $\Gamma_n$, so that the data
$(\Gamma_n, \mu_n, \pi)$ determines a marked dynamical graph as defined in \S\ref{sec:classif special}. 

Note that $\partial \Gamma_n$ contains $\{ H = 1-n\}$ but might be strictly larger. Note also
that $d_\pi(v) =1$ for all vertices of $\Gamma_n$, $n\le 0$, whereas $d_\pi(v) \ge2$ for at least one vertex in $\Gamma_1$.
Finally, we have $\Gamma_n = \Gamma_\esc$ for all $n$ sufficiently large ($n \ge 1 - \min_{\Gamma_\esc} H$).

Recall the construction of marked dynamical graphs from Example~\ref{ex:mark-dyn}.

\begin{lemma}\label{lem:first realization}
There exists a finite set $\mathcal{F} \subset \{ |z| = \rho\}$ such that $G(\mathcal{F},M_d) = \Gamma_0$.
\end{lemma}
The domain of definition of $M_d$ can be chosen to be $\{|z| >\rho^{1/d}\}$.

\begin{proof}
Observe first that any point in  $\{ H =1\}$ lies in the orbit of at least one point 
$\mu(i)$ so that the cardinality $n$ of $\{ H=1 \}$ is at most $d-1$.

We build a map $\theta \colon V(\Gamma_0) \to \C$ whose image will be the vertices of $G(\mathcal{F},M_d)$.
Let $v_1$ be the branched point of $\Gamma_0$ of maximal height equal to $H(v_*) = H_*$. 
This point is uniquely determined and we let $\theta(v_*)$ be any complex number of modulus $\rho^{H_*}$.

Since $\pi$ is a flow, the number of preimages by $\pi$ of $v_*$ is equal to the number of branches of $\Gamma$ at $v_1$
hence is $\le (d-1)$. We may thus find an injective map $\theta \colon \pi^{-1}(v_*) \to \{ |z| = \rho^{H_* -1}\}$
such that $\theta (\pi(v)) = M_d(\theta(v))$. Applying the same argument to each point of height $H_* -2, H_* -3$, etc.
we construct by induction an injective map $\theta \colon V(\Gamma_0) \to \C$ such that $\theta (\pi(v)) = M_d(\theta(v))$. 
We conclude by setting $\mathcal{F} = \theta(\{ H =1 \})$.
\end{proof}

\begin{remark}\label{rem:def mu0}
From the previous lemma, we get a canonical injective map $\mu_0 \colon \Gamma_0 \to \{ |z| \ge \rho\}$
such that $\mu_0(\pi(v)) = M_d(\mu_0(v))$. Observe that $\mathcal{F}$ is in bijection with $\{ H=1\}$. 
\end{remark}
We can now explain our strategy for the proof of Theorem~\ref{thm:realizability}. We shall construct by induction on $n$, 
an increasing sequence 
of Riemann surfaces $S_n$ and finite map $\Phi_n\colon S_n\to S_{n}$ 
such that $\Phi_n|_{S_{n-1}} = \Phi_{n-1}$ in such a way that the dynamical 
graph associated to the critical points of $\Phi_n$ equals $\Gamma_n$. 
The construction of the sequence $S_n$ is given in \S \ref{sec:construct-sequence-Riemann}.
At step $n$, we shall need to pick a polynomial whose critical points satisfy some constraints forced by
the geometry of the graph $\{ n \le H \le n-1\}$. We discuss in the next section the construction of such a
polynomial. 

When $\Gamma_\fin =\emptyset$, then the union of $S_n$ is a planar domain and it follows from an argument of McMullen
that $\Phi_n$ extends through the complement of $\cup S_n$ in $\C$ thus defining the required polynomial. 
When $\Gamma_{\fin}\not=\emptyset$, the construction is more involved as we have to patch a suitable disk containing
the filled-in Julia of a PCF polynomial realizing $\Gamma_\fin$ with $S_n$ for some large $n$. Once this is done, the argument
proceeds as in the former case. 
\begin{remark}
Our assumption (R1) is only used to get condition (D2) in Theorem~\ref{thm:real portrait} below, and it is likely that (R1) is in fact superfluous. 
\end{remark}

\subsection{Polynomials with a fixed portrait}
We now discuss the construction of polynomials with prescribed ramification locus. We shall need a more precise result
than~\cite[Proposition~7.3]{DeMarco-McMullen}. Our treatment is slightly different from op. cit. and more topological in nature. 

\smallskip

To avoid a statement with too many assumptions, we first describe our setup.  
Let $\gamma$ be any  simple closed curve in the complex plane.
We fix a finite set of points $ \mathcal{G} \subset \gamma$ and for each $p \in \mathcal{G}$
a non-empty finite set of positive integers $\mathcal{D}(p) = \{ n_{i,p}\}$. We also 
fix a (possibly empty) subset $\mathcal{D}_0(p) \subset \mathcal{D}(p)$. 

\begin{theorem}\label{thm:real portrait}
Suppose $\gamma, \mathcal{G}, \mathcal{D}(p)$ and $\mathcal{D}_0(p)$ are given as above, and pick any two positive integers
$d'$ and $N$ such that 
\begin{itemize}
\item[(D1)]
$(d'-1) = (N-1) + \sum_\mathcal{G} \sum_{i\in \mathcal{D}(p)} (n_{i,p} -1)$;
\item[(D2)]
for all $p\in \mathcal{G}$, $d' \ge \sum_{i\in \mathcal{D}(p)} n_{i,p}$;
\item[(D3)]
for all $p\in \mathcal{G}$, $N \ge \sum_{p\in \mathcal{G}} \Card \mathcal{D}_0(p)$.
\end{itemize}
Then for any point $z$ lying in the bounded connected component of $\C \setminus \gamma$,  
there exist a polynomial $Q$ of degree $d'$, and a simple closed curve $\gamma' \subset Q^{-1}(\gamma)$, 
 such that:
\begin{itemize}
\item[(R1)]
the bounded component of $\C \setminus \gamma'$ contains a unique point $z' \in Q^{-1}(z)$
and $\deg_{z'}(Q) =N$;
\item[(R2)]
for each $p\notin Q^{-1}(\mathcal{G}) \cup \{z\}$, $\deg_p(Q) =1$;
\item[(R3)]
for  each $p \in \mathcal{G}$, there exists a finite set  $\mathcal{Q}(p) \subset Q^{-1}(p)$ such that 
the function $\delta_p(q) := \deg_q(Q)$ defines a bijective map $\delta_p\colon \mathcal{Q}(p) \to \mathcal{D}(p)$;
\item[(R4)]
for each $p \in \mathcal{G}$, the set  $\delta_p (\mathcal{Q}(p) \cap \gamma')$ contains $\mathcal{D}_0(p)$.
\end{itemize}

\end{theorem}

In plain words, it is always possible to find a polynomial with a prescribed branched portrait (determined by $\mathcal{D}$) , with critical values $\mathcal{G}$ in $\gamma$ and 
branched locus in a fixed closed disk (the bounded component of $\C \setminus \gamma'$). 

We include below two examples. The first one (Figure~\ref{fig:portrait1}) shows that  in general $\delta_p (\mathcal{Q}(p) \cap \gamma')$
might exceed $\mathcal{D}_0(p)$. The second one (Figure~\ref{fig:portrait2}) proves that
 the polynomial $Q$ satisfying the conditions above may not be unique (up to a composition by an affine transformation).

\begin{figure}[h]
\centering
 \def\svgwidth{10cm}
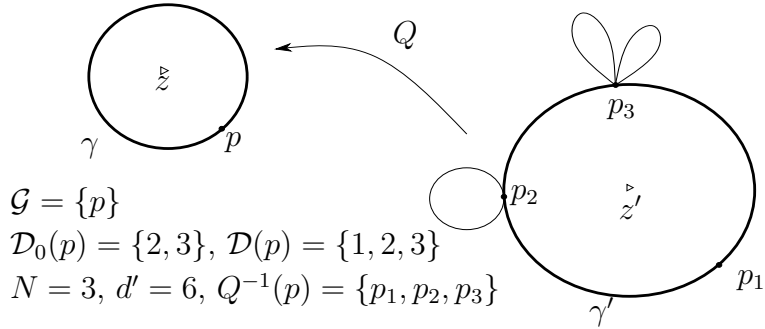
  \caption{Example  where $\mathcal{Q}(p)\subset \gamma'$ but $\mathcal{D}_0(p)\subsetneq \mathcal{D}(p)$}
 \label{fig:portrait1}
\end{figure}

\medskip

\begin{figure}[h]
\centering
 \def\svgwidth{\linewidth}
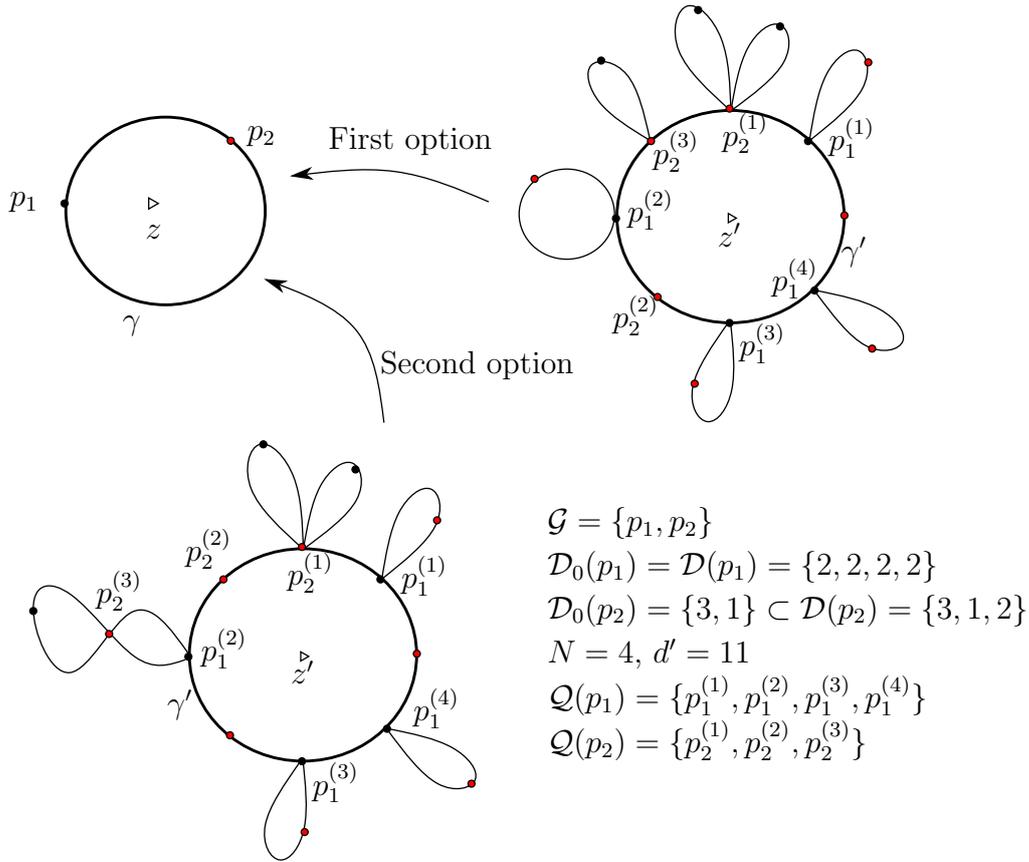  
\caption{Non-uniqueness of the realization of a critical portrait }
 \label{fig:portrait2}
\end{figure}

\begin{proof}
Recall that any topological branched cover over a complex domain admits a canonical structure of Riemann surface
for which the cover is holomorphic. In order to prove the theorem, it is thus only necessary to produce a topological branched cover of the Riemann
sphere having the prescribed branch portrait. 

To simplify the discussion, we shall use the following convenient terminology. A disk is a domain of the complex plane
homeomorphic to the unit disk whose boundary is a  simple closed curve. 
Given a disk $D$ and a point $p \in \partial D$, we say that we attach $k\ge 1$ disks to $D$ at $p$, when we choose $k$ disjoint disks $D_1, \ldots, D_k$ such that 
$\overline{D_i} \cap \overline{D_j} = \overline{D_i} \cap \overline{D} = \{ p\}$ for all $i \neq j$. 
Observe that the pull-back by $z \mapsto (1-z)^{k+1}$ of the unit disk, gives a (topological) disk containing $0$ in its interior with $k$ disks attached to it at $1$.
When the boundary of $D$ is a general closed simple curve, one can use Jordan-Schoenflies theorem to reduce the situation to the unit disk and attach disks to $D$. 
\smallskip

Let $D$ be the bounded connected component of $\C \setminus \gamma$ which is a disk by Jordan's theorem. 
Fix any other disk $D_0$, and take any branched cover  $h \colon \overline{D}_0 \to \bar{D}$
of degree $N$ which is totally ramified only at one point $z'$ which is mapped to $z$. 
Observe that each point $p \in \mathcal{G}$ has exactly $N$ preimages on $\gamma':= \partial D_0$.

\smallskip

Suppose first $\mathcal{G}$ is a singleton. By (D1) and (D2), we have
$d' \le  N + d' - \Card (\mathcal{D}(p))$ so that $N \ge\Card (\mathcal{D}(p))$. 
We may thus  select $\mathcal{Q}(p) \subset h^{-1}(p) \subset \gamma'$,  
together with a bijection  $\delta_p\colon\mathcal{Q}(p) \to \mathcal{D}(p)$.
To each point $q \in \mathcal{Q}(p)$, we attach $\delta_p(q)-1$ disks, and
extend $h$ to be a homeomorphism from the closure of each of the attached disks onto $\bar{D}$, and mapping 
$q$ to $p$. At this point, we have a union of open disks $U$ whose boundary $\Gamma$ is
a union of simple closed curves whose two-by-two intersections are either empty or reduced to a point in $\mathcal{Q}(p)$.  
The number of connected components of $U$ equals $1 + \sum (n_{i,p} -1) = d' - N +1$. 
The map $h \colon \Gamma \to \gamma$ is a branched cover of degree $(d'-N) + N = d'$. 
We may thus extend $h$ to a branched cover of the Riemann sphere of degree $d'$ which leaves $\infty$ totally invariant, 
and is unramified over $\C \setminus \overline{U}$. This concludes the proof in this case.

\smallskip

The case  $\Card (\mathcal{G}) \ge 2$ is harder to treat 
since it may happen that $\Card (\mathcal{D}(p)) \ge N$ for some $p$. 
Observe though that $\Card (\mathcal{D}_0(p)) \le N$ for all $p$ by (D3). 
For each $p\in \mathcal{G}$, we may thus select $\mathcal{Q}_0(p)\subset h^{-1}(p) \subset \gamma'$
and a bijective map $\delta_p\colon\mathcal{Q}_0(p) \to \mathcal{D}_0(p)$. We attach $\delta_p(q) -1$ disks 
$D_i(q)$ to each point $q \in \mathcal{Q}_0(p)$, and extend $h$ as a homeomorphism from each $\overline{D_i(q)}$ onto $\bar{D}$ as above. 
Denote by $U_0=\cup_{i,q} D_i(q)$ the union of these disks. 

The proof now proceeds  as follows. We attach one step at a time $(n_{i,p} -1)$ disks at a well-chosen point in $h^{-1}(p)$, and for some $n_{i,p}$, 
and extend $h$ as a homeomorphism
from each of these disks onto $\bar{D}$. At each step $k\ge0$, we build a domain $U_k$ which is a union of open disks,
whose boundary is a union of circles whose two-by-two intersection is either a point or empty, and we get a finite branched cover
$h\colon\overline{U}_k \to \bar{D}$.

We also have a sequence of sets $\mathcal{Q}_{k}(p) \subset h^{-1}(p)$, and 
injective maps $\delta^{(k)}_p \colon \mathcal{Q}_k(p) \to \mathcal{D}(p)$ such that  $\delta^{(0)}_p = \delta_p$
is the map defined on $\mathcal{Q}_0(p)$ above. 
We shall say that $n_{i,p} \in \mathcal{D}(p)$ has been \emph{allocated} at step $k$, when it belongs to the image of $\delta^{(k)}(p)$.
The goal is to reach a situation where $\delta^{(k)}_p$ is surjective for all $p$ (i.e. all elements of $\mathcal{D}(p)$ have been allocated for all $p$). 

Introduce the sets
\[
\mathcal{D}_1(p) = \{i\notin \mathcal{D}_0(p), n_{i,p} =1\}
\text{ and }
\mathcal{D}_+(p) = \{i \notin \mathcal{D}_0(p), n_{i,p} \ge2\} 
~.\]
We claim that there exists a procedure such that after finitely many steps, all elements of $\mathcal{D}_+(p)$ have been allocated for all $p$.
Grant this claim, and let $k$ be the number of steps needed to allocate all elements of $\mathcal{D}_+(p)$. Note that in this case, the number of points in $h^{-1}(p)$ which are not in $\mathcal{Q}_{k}(p)$ is equal to 
\begin{align*}
\mu & :=
N - \Card(\mathcal{D}_0(p)) - \Card(\mathcal{D}_+ (p)) + \sum_{q \neq p} \sum_{\mathcal{D}(q)} (n_{i,q} -1)
\\
&\mathop{=}\limits^{(D1)}
d'- \sum_{\mathcal{D}(p)} (n_{i,p} -1) -   \Card(\mathcal{D}_0(p))- \Card(\mathcal{D}_+ (p))
\\
&=  \left(d'- \sum_{\mathcal{D}(p)} n_{i,p}\right) + \Card(\mathcal{D}_1(p))
\mathop{\ge}\limits^{(D2)}
\Card(\mathcal{D}_1(p))~.
\end{align*}
Since the number of points  in $\mathcal{D}(p)$ which remains to be allocated is equal to 
$\Card(\mathcal{D}_1(p))$, we can extend the function $\delta_p$ to $\mathcal{D}(p)$ injectively as required. 

\medskip

To prove our claim, we need to allocate elements in $\mathcal{D}_+(p)$.
For each $p$ and at each step $k \ge0$ of the construction, we let $\Delta_k(p)$ be the number of elements of $\mathcal{D}_+(p)$
which have not been allocated, and let $F_k(p)$ be the number of preimages of $h^{-1}(p)$ that are free in the sense 
that they do not belong to  $\mathcal{Q}_{k}(p)$. 

At step 0, we have $\Delta_0(p) = \Card (\mathcal{D}(p)) - \Card (\mathcal{D}_0(p))$, and 
\begin{align*} 
F_0(p) 
&=  \sum_{p'\neq p} \sum_{i\in \mathcal{D}_0(p')}(n_{i,p'} -1) + N - \Card (\mathcal{Q}_{0}(p))
\\ &=  \sum_{p'\neq p} \sum_{i\in \mathcal{D}_0(p')}(n_{i,p'} -1) + N - \Card (\mathcal{D}_{0}(p)).
\end{align*}
Order the points in $\mathcal{G}$ such that 
\[
\Card(\mathcal{D}_+(p_1)) \ge \cdots \ge \Card(\mathcal{D}_+(p_s))> 0 = \Card(\mathcal{D}_+(p_j))\] 
for all $j\ge s+1$, and let us first suppose that $\Card (\mathcal{D}_{0}(p)) <N$ for all $p$.
At step $1$, we may thus allocate one (randomly chosen) element $n_{i_1,p_1}, \ldots, n_{i_s,p_s}$ of each set $\mathcal{D}_+(p_1), \ldots, \mathcal{D}_+(p_s)$. 
Observe that 
\[
\begin{cases}
\Delta_{1}(p_i) &= \Delta_0(p_i) -1,\text{ and }
\\
F_{1}(p_i) &= F_0(p_i) -1 + \sum_{q\neq p_i} (n_{i,q}-1) \ge F_0(p_i).
\end{cases}
\] 
for all $1\le i\le s$ since $\Card (\mathcal{G}) \ge2$. If $\Delta_1(p_i) =0$ for all $i$, then we are done. Otherwise, 
$\Delta_1(p_1), \ldots, \Delta_1(p_{s_1}) \ge1$ and $\Delta_1(p_{s_1 +1}), \ldots, \Delta_1 (p_s) = 0$.
At step $2$, we allocate one element of each set $\mathcal{D}_+(p_1), \ldots, \mathcal{D}_+(p_{s_1})$. We may continue
in this way until $\Delta_k(p_i)=0$ for all $i\ge2$ and $\Delta_k(p_1) >0$.
At this step, we get 
\begin{align*} 
F_k(p_1)
&=  N - \Card (\mathcal{D}_0(p_1))- \Card (\mathcal{D}_+(p_2)) + \sum_{p\neq p_1} \sum_{i\in \mathcal{D}(p)}(n_{i,p} -1)
\\
& =  d' - \Card (\mathcal{D}_0(p_1))- \Card (\mathcal{D}_+(p_2)) - \sum_{i\in \mathcal{D}(p_1)}(n_{i,p_1} -1)
\\
&
= \left(d' - \sum_{i\in \mathcal{D}(p_1)} n_{i,p_1} \right) +   \Card (\mathcal{D}(p_1)) -  \Card (\mathcal{D}_0(p_1)) - \Card (\mathcal{D}_+(p_2)) 
\\
& \ge 
 \Card (\mathcal{D}(p_1)) -  \Card (\mathcal{D}_0(p_1))
- \Card (\mathcal{D}_+(p_2)) 
= \Delta_k(p_1)~,
\end{align*}
so that we may allocate all remaining points in $\mathcal{D}_+(p_1)$. This proves the claim when $\Card (\mathcal{D}_{0}(p)) <N$ for all $p$.

\smallskip

Suppose finally that $\Card (\mathcal{D}_{0}(p_*)) =N$. Then (D3) implies that $\mathcal{D}_{0}(p) =\varnothing$ for all $p\neq p_*$. 
Write  $\varepsilon:= \sum_{q \neq p_*} \sum_i (n_{i,q} -1)$, so that
\begin{align*}
\Card (\mathcal{D}_+(p_*)) +  \Card (\mathcal{D}_1(p_*))
&= 
 \sum_i n_{i,p_*} - \sum_i (n_{i,p_*} -1) - \Card (\mathcal{D}_0(p_*))
\\
&\mathop{=}\limits^{(D1)} \sum_i n_{i,p_*} - d' + \varepsilon \mathop{\le}\limits^{(D2)} \varepsilon~.
\end{align*}
One then removes $p_*$ from the set $\{p_1, \ldots, p_s\}$, and apply the sequence of steps above. 
The number of free points in $h^{-1}(p_*)$ is  equal to $\varepsilon$, and one can thus allocate
all elements of  $\mathcal{D}_+(p_*) \cup  \mathcal{D}_1(p_*)$.
This concludes the proof. 
\end{proof}

\subsection{Construction of a suitable sequence of Riemann surfaces}\label{sec:construct-sequence-Riemann}
This part is the key to the proof of Theorem~\ref{thm:realizability}. It is strongly inspired by the approach of~\cite{DeMarco-McMullen}. 

Recall that $\Gamma$ is a fixed marked dynamical graph of degree $d$. 
We let $\Delta := 1 + \sum (d_\pi(\mu(i)) -1)$ where the sum is taken over all indices $i$ such that $\mu(i)$ lies 
in a bounded component of $\Gamma$ (i.e. in $\Gamma_\fin$). 

To alleviate notations we identify a graph and its set of vertices. 

Choose $\mathcal{F} \subset \gamma_0 := \{ |z| = \rho\}$ as in Lemma~\ref{lem:first realization} so that $G(\mathcal{F}, M_d) = \Gamma_0$.
By definition $\mathcal{F} = \partial \Gamma_0$ and we have an injective map $\mu_0 \colon \Gamma_0 \to \{ |z| = \rho\}$.

\smallskip

We shall build by induction a sequence of objects $S_n, \Phi_n, G_n, \mu_n,\gamma_n$ indexed by $n\in \N$
which satisfy the following conditions: 
\begin{itemize}
\item[(C0)]
$S_0 = \{ |z| > \rho^{1/d}\}$, $G_0 = \log^+|z|$,  $\Phi_0 = M_d$, and 
$\mu_0 \colon \Gamma_0 \to \{ |z| \ge \rho\}$ is the map above
such that $G(\mathcal{F},\Phi_0)= \Gamma_0$;
\item[(C1)]
$S_0 \subset S_{n-1} \subset S_{n}$ is an increasing sequence of connected open Riemann surfaces such that $\partial S_{n-1}$ is a real analytic curve in $S_{n}$, and $S_{n} \setminus \overline{S_{n-1}}$
is a finite union of conformal annuli of finite modulus; 
\item[(C2)]
$G_n \colon S_n \to \R_+$ is a harmonic function 
such that $G_n |_{S_{n-1}} = G_{n-1}$, and 
$G_n(z_k) \to  \frac1{d^{n+1}} \log \rho$ for any diverging sequence $z_k \in S_n\setminus S_{n-1}$ (i.e. eventually leaving any compact subset  
of $S_n \setminus S_{n-1}$);
\item[(C3)]
$\Phi_n \colon S_n \to S_n$ is a finite proper holomorphic map
such that $\Phi_n(S_n)=S_{n-1}$, $\Phi_n|_{S_{n-1}} = \Phi_{n-1}$, and $G_n \circ \Phi_n = d\, G_n$;
\item[(C4)]
 $\mu_n \colon \Gamma_{n} \to S_n$ is an injective map such that 
$\mu_n|_{\Gamma_{n-1}} = \mu_{n-1}$, and $\mu_n ( \pi(v))= \Phi_n(\mu_n(v))$;
\item[(C5)]
$d_\pi (v) = \deg_{\mu_n(v)} (\Phi_n)$ for all $v \in \Gamma_{n}$, and  $\deg_{z} (\Phi_n)=1$ if $z \notin \mu_n(\Gamma_{n})$;
\item[(C6)]
$\gamma_n$ is a simple closed curve included in $\{G_n = \frac1{d^n} \log \rho\}$,
which contains $\mu_n ( \partial \Gamma_{n} \setminus \partial \Gamma_{n-1})$ and bounds a connected component $A_{\star,n}$
of $S_{n} \setminus \overline{S_{n-1}}$;
\item[(C7)]
the restriction of $\Phi_n$ to any connected component $\neq A_{\star,n}$ of $S_{n} \setminus \overline{S_{n-1}}$ induces a conformal isomorphism onto its image;
\item[(C8)]
the restriction of $\Phi_n$ to $A_{\star,n}$  induces a covering map onto its image of degree
\[
d_\star(n) := \Delta + \sum_{H(v) \le -n}  (d_\pi(v) -1)
~,
\]
and $\gamma_n$ is included in the boundary of  $\overline{\Phi_n(A_{\star,n})}$ in $S_n$.
\end{itemize}

\begin{figure}
\centering
 \def\svgwidth{\linewidth}
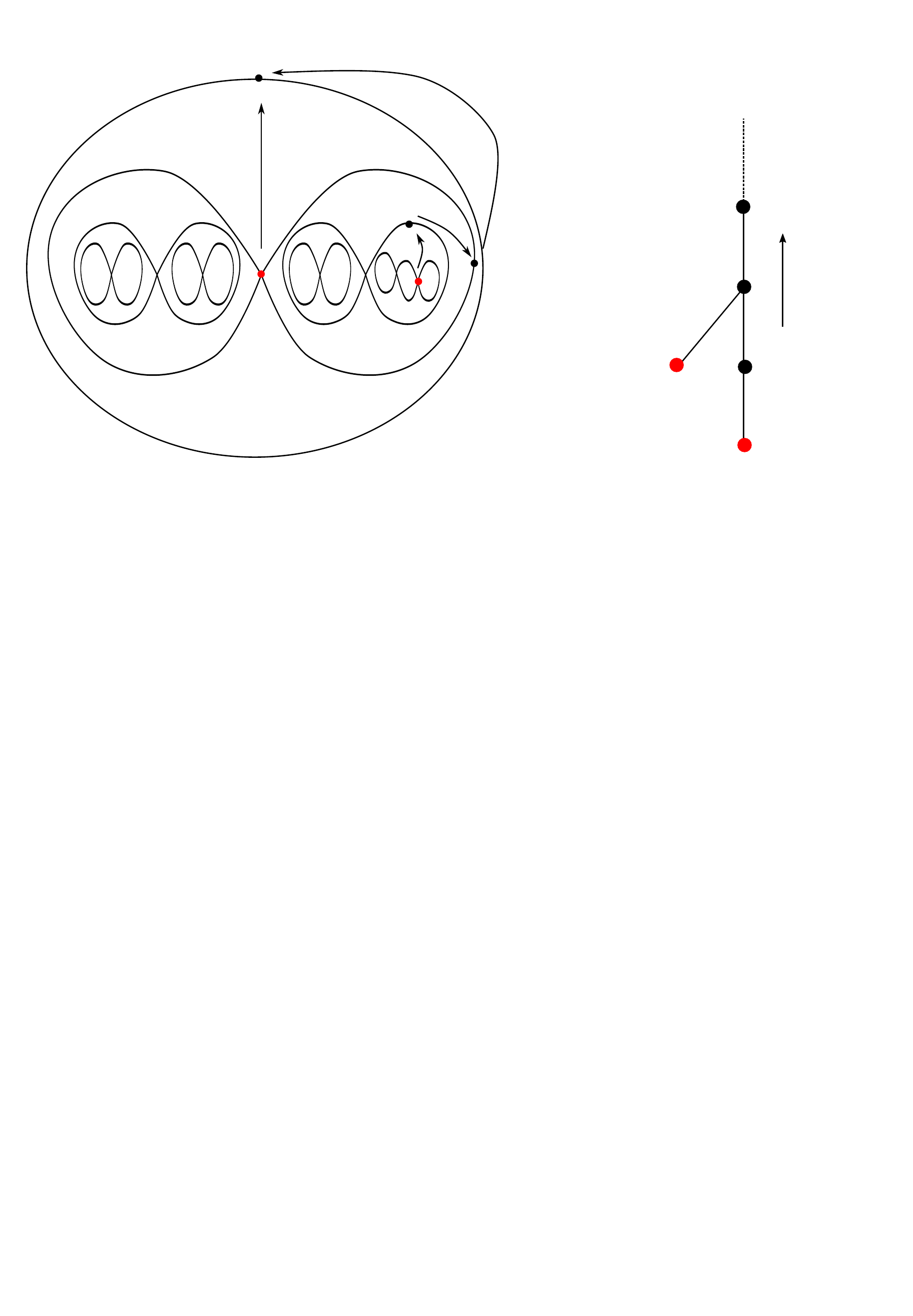  
\caption{Constructing the surfaces $S_n$}
 \label{fig:picture-Sn}
\end{figure}

\smallskip

We impose (C0). 
Since $G(\mathcal{F}, \Phi_0) = \Gamma_0$
we can extend $\mu_0$ in a unique way such that (C4) is satisfied. Observe that (C2) and (C3) are satisfied since $G_0(z) \to \frac1d \log \rho$ when $|z| \to \rho^{1/d}$, and
$G_0 \circ M_d = \log |z^d| = d G_0$. Similarly (C5) -- (C7) hold, and (C8) holds with $d_\star(0) = d$  
by Hurwitz formula
since on the one hand
\[\Delta  = 1+\sum (d_\pi(\mu(i)) -1)\] where the sum ranges over all $\mu(i)$ lying in bounded components of $\Gamma$, and on the other hand
$H(\mu(i)) \le 0$ for all $i$ in the unbounded component of $\Gamma$. 

\smallskip

Assume now we have constructed $S_n, \Phi_n, G_n, \mu_n,\gamma_n$ for some $n\ge0$. Before we explain 
our construction of $S_{n+1}$, we begin with some simple observations.

Adding $\infty$ to the surface $S_0$ yields a Riemann surface which we denote by $\hat{S}_0$. 
By (C1), the boundary of $S_0$ in $S_1$ is a curve, hence the natural inclusion $\imath \colon S_0 \to S_1$
extends to an injective holomorphic map $\hat{S}_0 \to \hat{S}_1 := S_1 \cup \{\infty\}$.
By induction, we get an increasing sequence of Riemann surfaces $\hat{S}_n = S_n \cup \{\infty\}$, so that 
$\hat{S}_0 \subset \hat{S}_{n-1}  \subset \hat{S}_n$.

The function $G_n$ extends continuously (as a $\R_+ \cup{\infty}$-valued function) to $\hat{S}_n$ by setting $G_n(\infty) = +\infty$. It is superharmonic and
$\Delta G_n = - \delta_\infty$.
\begin{lemma}\label{lem:Sn-use}
In $S_n$, we have 
\[S_{n-1} = \left\{ G_n >\frac1{d^n} \log \rho \right\}, \text{ and } \partial S_{n-1} =  \left\{ G_n = \frac1{d^{n}} \log \rho \right\}.\] 
More precisely, for any connected component $A$ of $S_n \setminus \overline{S_{n-1}}$ there exists a conformal
isomorphism $\psi \colon \{ \rho_{n,A} < |z| < \rho^{1/d^n} \}\to A$ for some $\rho_{n,A} >0$ such that 
$G_n( \psi(z))$ is proportional to $\log |z|$. 
\end{lemma}
\begin{proof}
By (C2), we know that $G_n \to  \frac1{d^n} \log \rho$
when approaching the boundary of $S_{n-1}$, hence $\partial S_{n-1 } \subset  \{ G_n = \frac1{d^n} \log \rho \}$.
By the minimum principle applied to $\hat{G}_n$, we get the inclusion  $S_{n-1} \subset \{ G_n >\frac1{d^n} \log \rho \}$.

Now look at $S_n \setminus \overline{S_{n-1}}$. By (C1), it is a finite collection of annuli. 
Choose one of them say $A$: it has at least one boundary component included in $\partial S_{n-1}$ over which one has
 $G_n = \frac1{d^n} \log \rho$. Since $S_n$ is connected and $G_n (z)\to \infty$ as $z \to \infty$, $G_n|_A$
 is harmonic and not constant, so that the other boundary component of $A$ cannot be included in  $\partial S_{n-1}$.
 It follows from (C2) that $G_n \to \frac1{d^{n+1}} \log \rho$ when approaching this boundary. 

Pick any conformal isomorphism $\psi \colon \{ \rho_{n,A} < |z| < \rho^{1/d^n} \} \to A$ with $\rho_{n,A} <   \rho^{1/d^n}$ sending
 $\{ z= \rho^{1/d^n}\}$ to the boundary of $A$ included in $\partial S_{n-1}$.
Then $G_n \circ \psi$ is harmonic, equal to $\frac1{d^n} \log \rho$ on $\rho^{1/d^n}$, and tends to 
$\frac1{d^{n+1}} \log \rho$
when one approaches the circle $\{ |z| = \rho_{n,A}\}$. By circular symmetry, it follows that $G_n \circ \psi (z) = \lambda \log|z|$ for some $\lambda \neq 0$
and  $\lambda \log \rho_{n,A} = \frac1{d^{n+1}} \log \rho$.

This implies the inclusions $ \{ G_n >\frac1{d^n} \log \rho \}\subset S_{n-1}$, $\left\{ G_n = \frac1{d^{n-1}} \log \rho \right\} \subset \partial S_{n-1}$
and the second part of the lemma.
\end{proof}

We now need to analyze the ramification locus of $\Phi_n$. 
\begin{lemma}\label{lem:imu}
The image of $\mu_n$ is included in $\cup_{l \le n} \{ G_n = \frac1{d^l} \log \rho \}$. 
\end{lemma}
\begin{proof}
Observe that $\mu_0(\Gamma_0)$ is the union of the orbits of points in $\mathcal{F}$ under $M_d$, hence
by construction it is included in $\cup_{j \ge 0} \{ |z| = \rho^j \} = \cup_{ l \le 0} \{ G_n = \frac1{d^l} \log \rho \}$.

Now the complement of $\Gamma_0$ in $\Gamma_1$ is precisely the set  $\partial \Gamma_1 \setminus \partial \Gamma_0$, 
and by (C6) this set is mapped by $\mu_1$ inside $\{G_1 = \frac1d \log \rho\}$. This implies the lemma in the case $n=1$, 
and a simple induction on $n$ allows one to conclude.
\end{proof}

We shall construct $S_{n+1}$ by patching an open Riemann surface
along each connected component of $S_n \setminus \overline{S_{n-1}}$. 
More precisely, for each component $A$, we shall find a planar domain $V_A$ containing an annulus $B_A$ and build 
a conformal isomorphism $\Phi_A \colon B_A \to A$. 
The surface $S_{n+1}$ will be obtained as the disjoint union of $S_n$ and all domains of the form $V_A$, each patched to $S_n$ along the annulus
$B_A$ using $\Phi_A$. We refer the reader to the figures~\ref{fig:map1} and~\ref{fig:map2} for a schematic view on the patching procedure. 

\begin{figure}[h]
\centering
 \def\svgwidth{\linewidth}
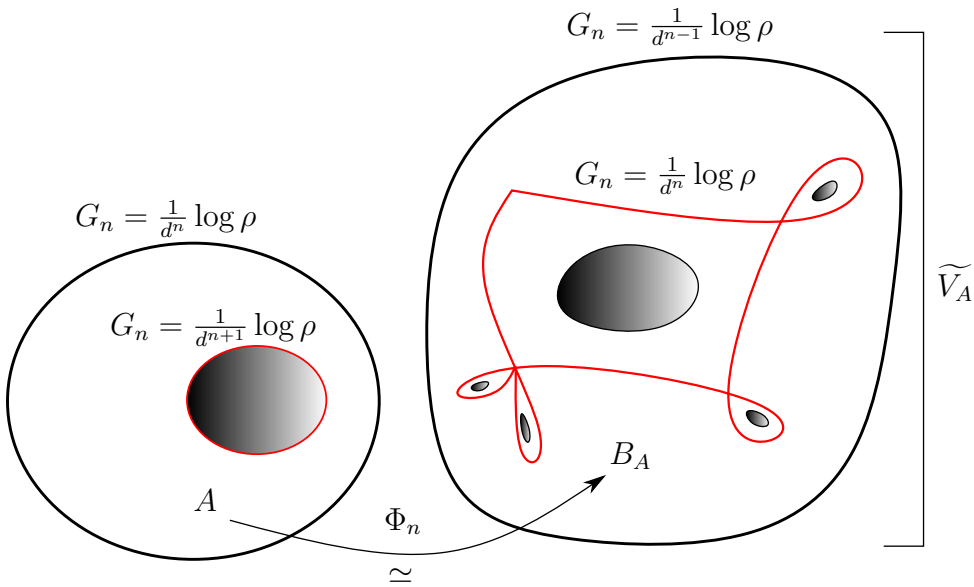  
\caption{The patching procedure when $A\neq A_\star$}
 \label{fig:map1}
\end{figure}

\begin{figure}[h]
\centering
 \def\svgwidth{\linewidth}
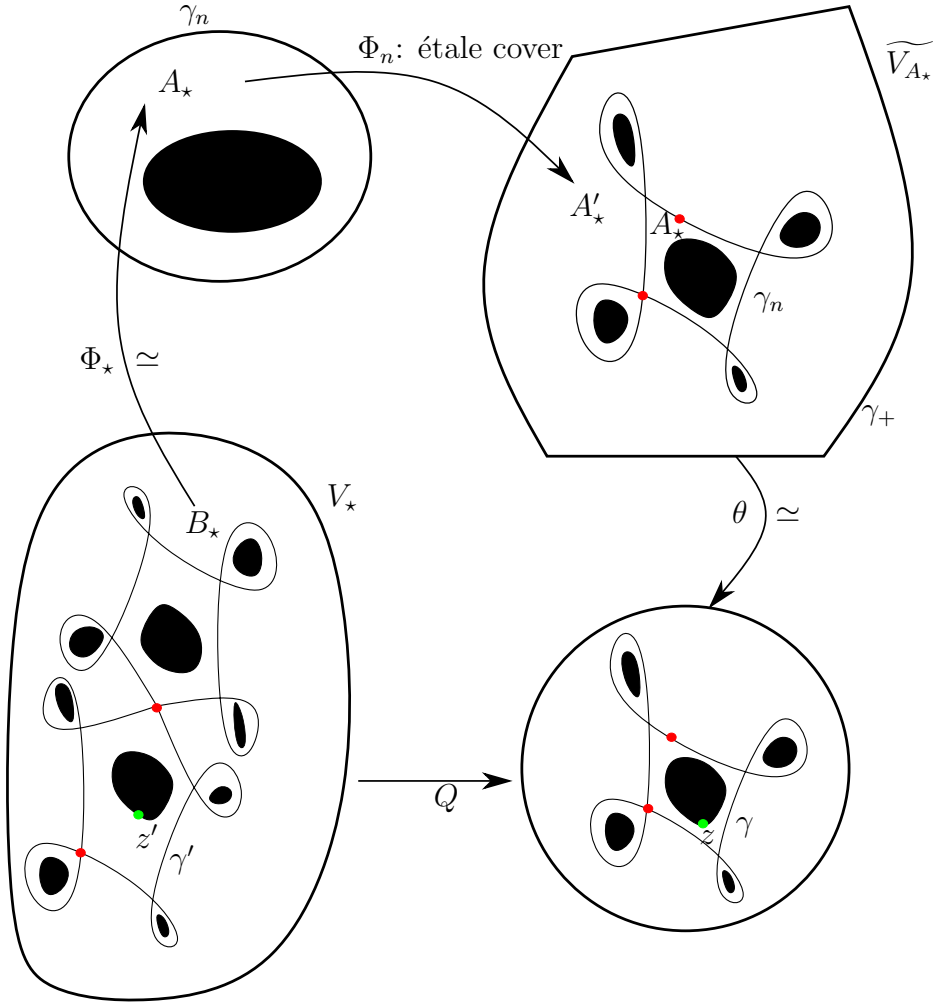  
\caption{The patching procedure when $A= A_\star$}
 \label{fig:map2}
\end{figure}

\paragraph*{Construction of $V_A$, $B_A$ and $\Phi_A$} 
Recall from (C6), that $\gamma_n$ is a simple closed curve bounding a connected component $A_\star (= A_{\star,n})$ of $S_n \setminus \overline{S_{n-1}}$. 
We begin with the following general lemma.
\begin{lemma}\label{lem:planar1}
Let $A$ be any connected component of $S_n\setminus\overline{S_{n-1}}$, and let
$\widetilde{V_A}$ be the connected component of $\{G_n < \frac1{d^{n-1}} \log \rho \}$ in $S_n$ which contains $\Phi_n(A)$.
Then the following holds: 
\begin{enumerate}
\item
the boundary of $\widetilde{V_A}$ in $S_n$ is a real-analytic simple closed curve $\gamma'_A$; 
\item
$\widetilde{V_A}\cap \{G_n <\frac1{d^n}\log\rho\}$ is a finite union of conformal annuli of finite modulus; 
\item 
there exists a univalent holomorphic map $\theta \colon \widetilde{V_A} \to \D$ 
which extends continuously to $\widetilde{V_A}\cup \{\gamma'_A\} \to \overline{\D}$ 
whose image is the complement of finitely many connected full compact subsets of $\D$.
\item
$\widetilde{V_A}$ is a planar domain. 
\end{enumerate}
Moreover, when $A=A_\star$, then $\gamma'_{A_\star} = \Phi_n(\gamma_n)$ and $\widetilde{V_{A_\star}}$ contains $A_\star$.
\end{lemma}

For any connected component $A \neq A_\star$ of $S_n \setminus \overline{S_{n-1}}$, set $V_A:= \widetilde{V_A}$.
By (C3) and (C7), $\Phi_n$ induces a conformal isomorphism
of $A$ onto an annulus $B_A$ which is a component of $\{ \frac1{d^{n}} \log \rho < G_n <  \frac1{d^{n-1}} \log \rho\}$, and we define
 $\Phi_A := \Phi_n \colon A \to B_A$. 

\smallskip

Let us consider now the component $A_\star$ which is bounded by $\gamma_n$.
By (C8), $\Phi_n$ induces an unramified cover  of degree $d_\star(n) := \Delta + \sum_{H(v) \le -n} (d_\pi(v) -1)$
from $A_\star$ onto an annulus $A'_\star:=\Phi_n(A_\star)$ bounded by two real analytic connected curves, 
$\Phi_n(\gamma_n) \subset \{G_n = \frac1{d^{n-1}} \log \rho\}$, and $\gamma_{A_\star}\subset \{ G_n = \frac1{d^n} \log \rho\}$ which contains
$\gamma_n$ by the previous lemma. 

Define $\mathcal{G}$ as the image under $\theta$ of $\mu_n \{ H = 1-n\} = \mu_n ( \partial \Gamma_{n} \setminus \partial \Gamma_{n-1})$. 
Since $\theta$ is injective, any point $p\in \mathcal{G}$ has a unique preimage $p=\theta(v)$, and we may set 
\begin{align}
\mathcal{D}(p) &:= \{ d_\pi(v'),\, (v,v') \in \mathcal{B}(v)\}\label{eq:defD}
\\
\mathcal{D}_0(p) &:= \{ d_\pi(v'), \, (v,v') \in \mathcal{B}(v) \text{ and } v' \notin \partial \Gamma_{n+2}\}\label{eq:defD0}
\end{align}
where $\mathcal{B}(v)$ is the set of edges $(v,v')$ of $\Gamma$ with $H(v') < H(v)$.
In other words, $\mathcal{D}(p)$ encodes all multiplicities of those vertices $w$ mapped by $\pi$ to $v$; and $\mathcal{D}_0(p)$
consists of multiplicities of those vertices $w\in\pi^{-1}(v)$ for which $\pi(w') = w$ for at least one $w'\in \Gamma_{n+1}$.
Finally set  
\begin{align*}
\gamma &:= \theta(\gamma_n), \text{ and }
N := \Delta + \sum_{H(v)\le -n-1} (d_\pi(v) -1).
\end{align*}

\begin{lemma}\label{lem:4321}
The three conditions (D1) -- (D3) hold for the collection of objects $\mathcal{G}$, $\mathcal{D}(p)$, $\mathcal{D}_0(p)$, $\gamma$, and $d':= d_\star(n)$. 
\end{lemma}
Pick any point $z \in \partial \theta(A_{\star}) \cap \partial \theta(\widetilde{V_{A_\star}})$, and
apply Theorem~\ref{thm:real portrait} to $\gamma, \mathcal{G}$, and
the two collections of integers  $\mathcal{D}_0$ and  $\mathcal{D}$ defined above. 
We obtain a polynomial $Q$ of degree $d'$, a real analytic curve $\gamma' \subset Q^{-1}(\gamma)$ and a point $z'$
in the bounded component  of $\C \setminus \gamma'$ such that (R1) -- (R3) hold.

By construction, all critical values of  $Q$ are included in $\mathcal{G} \cup \{z\} \subset \gamma$ hence in $\D$.  
It follows that $Q \colon Q^{-1}(\D) \to \D$ is a ramified cover of degree $d' = \deg(Q)$, hence $Q^{-1}(\D)$ is connected. 
We also get that $A'_\star$ contains no critical values of $\tilde{Q}:= \theta^{-1}\circ Q$, hence $B_\star = \tilde{Q}^{-1}(A'_\star)$ is an annulus, 
and $\tilde{Q} \colon B_\star \to A'_\star$ is an unramified covering map of degree $d'$.
 
By (C8), $\Phi_n$ is a covering map of degree $d'$ from $A_\star$ onto $A'_\star$,  so that we may find a conformal isomorphism 
$\Phi_\star \colon B_\star \to A_\star$ such that $\Phi_n \circ \Phi_\star = \tilde{Q}$.

We let $V_\star := \tilde{Q}^{-1}(\widetilde{V_{A_\star}})$, and we patch $V_\star$ to $S_n$ using the conformal isomorphism $\Phi_\star$.

\begin{proof}[Proof of Lemma~\ref{lem:planar1}]
By (C1), the boundary $\gamma_A$ of $A$ in $S_n$ is real analytic. It is a simple closed curve  since $S_{n-1}$ is connected.
By (C7) and (C8), $\Phi_n$ induces a finite cover of $A$ onto its image so that $A':= \Phi_n(A)$ is a conformal annulus which is a component
of $S_{n-1}\setminus\overline{S_{n-2}}$. 
The boundary of $A'$ in $S_n$ has two connected components, one of which is the image under $\Phi_n$
of the boundary of $A$ in $S_n$. It follows that $\partial A'$ is the disjoint union of $\gamma'_A=\Phi_n(\gamma_A)$ which is a simple closed curve proving 1., and a closed subset $\gamma$ of $S_n$. 
By (C2) and (C3), $G_n|_\gamma = \frac1{d^n} \log \rho$, hence $\gamma$ is a compact real-analytic curve since $G_n$ is harmonic. 

We claim that $\gamma$ may be written as the union of
finitely simple closed curves $\gamma = \ell_1 \cup \cdots \cup \ell_k$.

Grant this claim. Recall that $\widetilde{V_A}$ is the connected component of $ \{G_n <\frac1{d^{n-1}}\log \rho\}$
containing $A'_\infty$.

Since $\ell_i\subset \{G_n =\frac1{d^n}\log \rho\}$, this curve is included in $\partial S_{n-1}$ hence
by (C1) it bounds an annulus of finite modulus in $S_n$. The union of these annuli is equal to $V_A\cap \{G_n <\frac1{d^n}\log\rho\}$ proving 2.

We may thus attach to each $\ell_i$ a closed conformal disk $\bar{\D}_i$, 
and the union of $A'$ and these disks contains $\widetilde{V_A}$ and is simply connected (any path is homotopic to one in $A'$, and the path generating the fundamental group of
$A'$ is homotopic to the union of the $\ell_i$'s). This shows the existence of a univalent holomorphic map $\theta\colon \widetilde{V_A}\to \D$ 
satisfying condition 3. (the complement of the image of $V_A$ in $\bar{\D}_i$ is the decreasing intersection of closed disks hence is connected and full). 

When $A=A_\star$, the boundary of $A_\star$ in $S_n$ is $\gamma_n$ by (C8), hence $A_\star$ is included in $\widetilde{V_{A_\star}}$.

\smallskip

To prove the claim, observe that the singular locus of $\gamma$ is included in the intersection $\mathcal{S}$ of $\gamma$ with the critical locus of $G_n$: the latter is a finite (possibly empty) set.  The gradient flow of $G_n$ induces a continuous map from $S^1$ to $\gamma$ which is a local diffeomorphism
onto $\gamma\setminus\mathcal{S}$ and displays $\gamma$ as the quotient of the circle by equivalence relation of the following form:
there exists a finite set $F\subset S^1$ such that all equivalence class of points $\notin F$ are trivial.
An induction on the cardinality of $F$ shows that any quotient of the circle is a union of circles as claimed.
\end{proof}
\begin{remark}
Although we shall not use it, note that in our situation  the equivalence relation is induced by a family of finite subsets of $S^1$ that are unlinked (see \S\ref{sec:combin} below for a definition). This implies $\gamma$ to be a tree of simple closed curves, see Figure~\ref{fig:tree-circle}
\end{remark}

\begin{figure}[h]
\centering
\def\svgwidth{8cm}
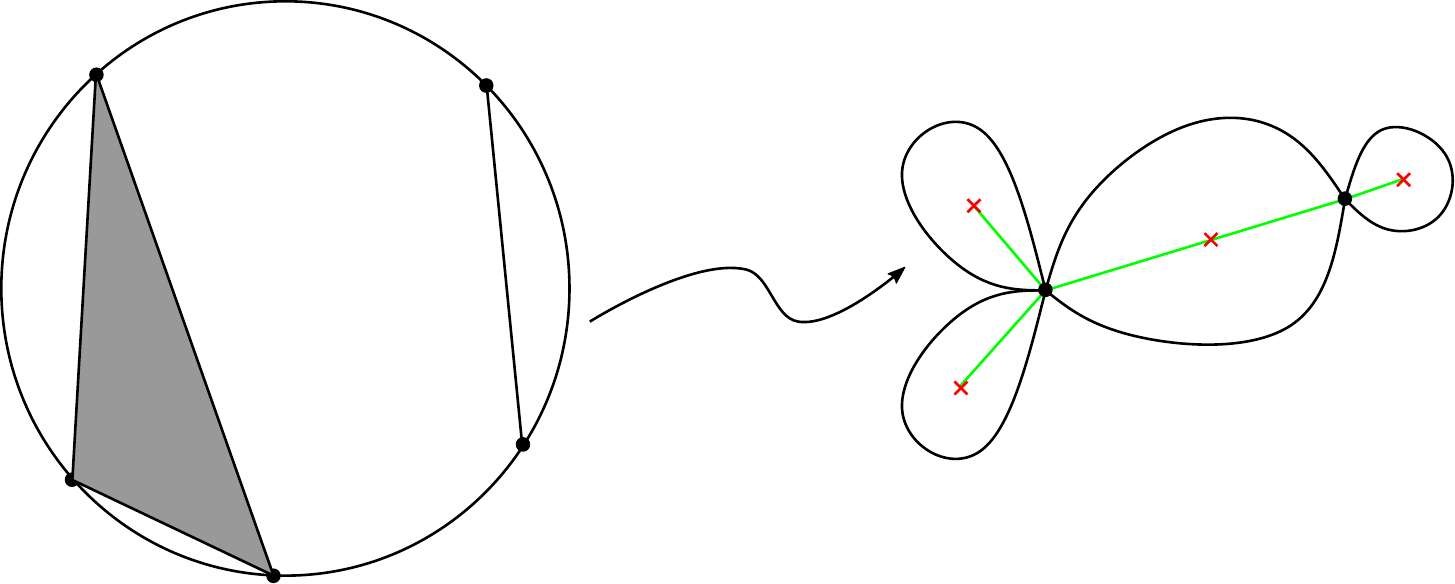
\caption{A tree of circles}
\label{fig:tree-circle}
\end{figure}

\begin{remark}\label{rem:compl-disk}
In fact $\widetilde{V_A}$ is the complement in a conformal disk of finitely many compact domains with real-analytic boundary. 
Indeed, this is clear if $n=1$, and by induction on $n$, $\widetilde{V_A}$ is isomorphic to the pull-back by a polynomial of a region with real-analytic boundary.
\end{remark}

\begin{proof}[Proof of Lemma~\ref{lem:4321}]
Recall that the collection of integers $\{n_{i,p}\}$ with $i\in \mathcal{D}(p)$ is in bijection
with the set $\{d_\pi(v')\}$ with $\pi(v')= v$ and $p=\theta(v)$.
Condition (D1) is an easy check: 
\begin{align*}
d' -1 = \sum_{H \le -n} (d_\pi(v)-1)
&= 
\left(
1+ \sum_{H \le -n-1} (d_\pi(v)-1)
\right)
-1 + \sum_{H = -n} (d_\pi(v)-1)
\\
&= 
N -1 + \sum_\mathcal{G} \sum_{\mathcal{D}(p)}(n_{i,p} -1)~.
\end{align*}
For the proof of (D2) we use our standing assumption (R2). Pick any point $p\in \mathcal{G}$, say $p= \theta(v)$ with $H(v) = -n+1$. 
By (R2), $\mathcal{B}(v)$ contains at most one vertex $v_0$ such that $d_\pi(v_0) \ge 2$.
By the minimality condition (G6), for any other vertex $v'\in \mathcal{B}(v)$ there exists a vertex $w= w(v')$ such that $d_\pi(w)\ge2$ and
$\pi^s(w) = v'$ for some $s\ge1$.
We infer
\begin{align*}
\sum n_{i,p}
&= 
\sum_{\mathcal{B}(v)} d_{\pi}(v')
= 
d_\pi(v_0)  -1 + \Card(\mathcal{B}(v))
\\
&\le 
d_\pi(v_0)  -1 + \sum_{H\le -n-1} (d_{\pi}(w) -1) \le d'
\end{align*}
which implies (D2).

\smallskip

Observe finally that  $\bigcup_{\mathcal{G}} \mathcal{D}_0(p)$ is in bijection with the set of edges of $\Gamma_{n+1}$ that
connect a vertex in $\{H = -n-1\}$ to a vertex in $\{H = -n\}$. By the minimality condition (G6), for each of these edge $e$ there exists a 
vertex $v \in \Gamma$ and an integer $s\ge1$ such that $e = (\pi^s(v), \pi^{s+1}(v))$ and $d_\pi(v) \ge 2$. Since we have $H(v)= -s + H(\pi^{s+1}(v)) = -s -n $, 
we obtain
\[
\sum_{\mathcal{G}} \Card \left(\mathcal{D}_0(p) \right)
= 
\Card \left( \bigcup_{\mathcal{G}} \mathcal{D}_0(p) \right)
\le
\sum_{H \le -n-1} 
(d_\pi(v) -1)
\le N
\]
which implies (D3).
\end{proof}

\paragraph*{Construction of $\Phi_{n+1}$, $G_{n+1}$, and $\mu_{n+1}$} 
Now that we defined $S_{n+1}$ as the union of $S_n$ and annuli $V_A$ attached to any component $A$ of
$\{\frac1{d^n}< G_n< \frac1{d^{n-1}}\log \rho\}$. To avoid confusion, denote by $I_A\colon V_A\to\widetilde{V_A}$
the canonical isomorphism between the open subset $V_A$ of $S_{n+1}$ and $\widetilde{V_A}\subset S_n$.
We define the map:
\[
\Phi_{n+1}
=
\begin{cases}
\Phi_n \text{ on } S_n,
\\
I_A  \text{ on } V_A \text{ for } A \neq A_\star,
\\
 \tilde{Q} \text{ on }  V_\star. 
\end{cases}
\]
Observe that by construction $\Phi_{n+1}$ is a well-defined holomorphic map
such that  $\Phi_{n+1}(S_{n+1}) \subset S_n$.
Similarly we set 
\[
G_{n+1} = 
\begin{cases}
G_n \text{ on } S_n,
\\ 
 \frac1d G_n \text{ on }  V_A \text{ for } A \neq A_\star,
 \\
 \frac1d G_n \circ \tilde{Q} \text{ on }  V_\star. 
\end{cases}
\]
When $A \neq A_\star$, it follows from (C3) that $G_{n+1}$ is well-defined and harmonic on $V_A$. 
On $V_\star$, we observe that $\frac1d G_n \circ \tilde{Q}$ is a harmonic function on $B_\star$
whose restriction to the outer (resp. inner) boundary of $V_\star$, i.e. on $\tilde{Q}^{-1}(\gamma_+)$ (resp. on $\tilde{Q}^{-1}(\gamma_-)$) is constant equal to $\frac1{d^n} \log \rho$
(resp. to $\frac1{d^{n+1}} \log \rho$) so that $\frac1d G_n \circ \tilde{Q}= G_n\circ \Phi_\star$ on $B_\star$.

We conclude that $G_{n+1}$ is a well-defined harmonic function on $S_{n+1}$ satisfying $G_{n+1} \circ \Phi_{n+1} = dG_{n+1}$. 

\medskip

Let us now define the map $\mu_{n+1}$. We set $\mu_{n+1} = \mu_n$ on $\Gamma_n$.  Observe that 
\[
T:= \Gamma_{n+1} \setminus \Gamma_n=\partial \Gamma_{n+1} \setminus \partial \Gamma_n= \{ H=-n\}~.\] 
Recall that we defined $\mathcal{G} := \theta( \mu_n(\{ H=1-n\}))$. 
We attach to a vertex $w\in T$ a point $p(w)\in \mathcal{G}$ by setting $p(w) =  (\theta\circ \mu_n) (\pi(w))$.

From the definition of $\mathcal{D}(p)$ and $\mathcal{D}_0(p)$, see~\eqref{eq:defD} and~\eqref{eq:defD0}, and since $\mu_n$ is injective on $\{ H=1-n\}$, and $\theta$ is a conformal isomorphism, we also get a 
canonical bijection 
\[\alpha\colon T \to \bigcup_{p\in\mathcal{G}} \mathcal{D}(p)\]
such that: for all $w\in T$ we have $\alpha(w) \in \mathcal{D}(p(w))$; and $\alpha(w) \in \mathcal{D}_0(p(w))$ iff $w\notin \partial \Gamma_{n+2}$ (i.e. 
$w=\pi(w')$ for some $w'\in \Gamma_{n+2}$).

We now observe that the polynomial $Q$ obtained by applying Theorem~\ref{thm:real portrait} comes with a family of bijections
$\delta_p\colon\mathcal{Q}(p)\to\mathcal{D}(p)$ where $\mathcal{Q}(p)$ is a subset of $Q^{-1}(p)$ (see condition (R3)), and that
$\delta^{-1}_p(\mathcal{D}_0(p))$ is included in $\gamma'$.
For any $w\in T$, we may thus set $\mu_{n+1}(w)= \delta^{-1}_{p(w)}(\alpha(w))$. The latter point naturally belongs to $V_\star$
hence to $S_{n+1}$ since the latter is obtained by patching $V_\star$ to $S_n$ using the map $\Phi_\star$.

\paragraph*{Verification that all conditions (C1) -- (C8) are satisfied} 

By construction $S_{n+1}$ is the union $S_n \cup\bigcup_A V_A$ where $A$ ranges over all connected components of $S_{n}\setminus \overline{S_{n-1}}$.
We further have a conformal isomorphism $\Phi_A\colon B_A\to A$ from an annulus $B_A\subset V_A$, and
$V_A$ is patched with $S_n$ using
this biholomorphism. To check (C1) we need to prove that $V_A\setminus B_A$ is a finite union of conformal annuli of finite modulus, 
and that the boundary of $B_A$ inside $V_A$ is real analytic.

When $A\neq A_\star$, we have 
\[V_A\setminus B_A = \widetilde{V_A}\setminus \Phi_n(A) = \widetilde{V_A}\cap \left\{G_n < \frac1{d^n}\log \rho\right\}\]
which is a union of annuli by Lemma~\ref{lem:planar1} (2). Note also that $\partial B_A = \{G_n = \frac1{d^n}\log \rho\}$,
hence  $\partial B_A$ is real-analytic since $G_n$ is harmonic.

Otherwise $A = A_\star$, and
\[V_\star\setminus B_\star = \tilde{Q}^{-1}(\widetilde{V_{A_\star}}\setminus A'_\star)~,\]
with $A'_\star= \Phi_n(A_\star)$. The latter equality implies that $\partial A'_\star = \{G_n = \frac1{d^n}\log \rho\}$ as before
so that  $\partial A'_\star$ is real-analytic. By Lemma~\ref{lem:imu}, the image of $\mu_n$ is included in  $\bigcup_{l \le n} \{ G_n = \frac1{d^l} \log \rho \}$,
hence $\mu_n(\Gamma_n) \cap \widetilde{V_{A_\star}} \subset \{G_n = \frac1{d^n}\log \rho\}$.
Now by  Lemma~\ref{lem:planar1} (2), $\widetilde{V_{A_\star}}\setminus A'_\star$ is a finite union of annuli, and since $\tilde{Q}$ is ramified only
over the image of $\mu_n$, it follows that $\tilde{Q}$ induces a finite cover from $V_\star\setminus B_\star$ onto  $\widetilde{V_{A_\star}}\setminus A'_\star$.
This proves (C1).

\medskip

The function $G_{n+1}$ is a composition of a harmonic function and a holomorphic map, hence is harmonic. By construction, we also have $G_{n+1}|_{S_n}=G_n$.
If $z_k$ is a diverging sequence in $S_{n+1}\setminus S_n$, then after extraction it belongs to some open set $V_A$ for some $A$.
If $A\neq A_\star$, then the sequence $z_k$ can be identified with a sequence $\widetilde{z_k}$ 
diverging in $\widetilde{V_A}\setminus B_A$, so that 
$G_{n+1}(z_k) =\frac1d G_n(\widetilde{z_k}) \to \frac1{d^{n+2}}\log\rho$.
When $z_k\in A_\star$, then $\tilde{Q}(z_k)$ diverges in $\widetilde{V_{A_\star}}\setminus A'_\star$, hence
again $G_{n+1}(z_k) =\frac1d G_n(\tilde{Q}(z_k)) \to \frac1{d^{n+2}}\log\rho$ proving (C2).

\medskip

We have $\Phi_{n+1}|_{S_n} = \Phi_n$, $\Phi_{n+1}(V_A)\subset S_n$ for all $A$, and $G_{n+1}\circ \Phi_{n+1}=d G_{n+1}$ by the very definitions of $\Phi_{n+1}$ and $G_{n+1}$. The properness is a consequence of the properness of $\Phi_n$ and of (C7) and (C8) to be proved below. 

Let $W$ be any connected component of $S_n\setminus\overline{S_{n-1}}$, and consider the component $B$ of 
$S_{n}\setminus\overline{S_{n-2}}$ containing it. By induction we know that there exists a component $W'$ of 
$S_n\setminus\overline{S_{n-1}}$ whose image by $\Phi_n$ is included in $B$. Let $V'$ be the component of $S_{n+1}\setminus\overline{S_{n-1}}$ 
containing $W'$. Its image by $\Phi_{n+1}$ is a connected component of $S_{n}\setminus\overline{S_{n-2}}$ by the properness of $\Phi_{n+1}$, 
which contains $B$ hence it is $W$. This proves $\Phi_{n+1}$ is surjective.

\medskip

Observe that $\mu_{n+1}=\mu_n$ on $\Gamma_n$ by definition, and that $\mu_{n+1}(w)= \delta^{-1}_{p(w)}(\alpha(w))$
for any $w\in \{H= -n\}\subset \Gamma_{n+1}$. From the previous section, we also infer 
\begin{align*}
\Phi_{n+1} (\mu_{n+1}(w)) &= \theta^{-1} \left(Q \left(\delta^{-1}_{p(w)}(\alpha(w))\right)\right) 
\\& = \theta^{-1}(p(w))= \mu_n(\pi(w))= \mu_{n+1}(\pi(w))
\end{align*}
which completes the proof of (C4).

\medskip

Pick any $v\in \Gamma_{n+1}$.  If $v\in\Gamma_n$, then 
$d_\pi (v) = \deg_{\mu_n(v)} (\Phi_n)= \deg_{\mu_{n+1}(v)} (\Phi_{n+1})$. Otherwise $v\in T$, so that 
$\alpha(v)= d_\pi(v)\in\mathcal{D}(p(v))$, and
\[
\deg_{\mu_{n+1}(v)}(\Phi_{n+1})
=\deg_{\mu_{n+1}(v)}(\tilde{Q})
=
\deg_{\mu_{n+1}(v)}(Q)
\mathop{=}\limits^{(R3)}
\delta_{p(v)}(\mu_{n+1}(v))= 
d_\pi(v)
\]
proving (C5).

\medskip

We now define the curve $\gamma_{n+1}$. Recall that Theorem~\ref{thm:real portrait} yields a polynomial $Q$ and a real-analytic curve $\gamma'$
bounding a disk $D'$ such that $Q(\gamma')=\gamma$. In particular, $Q(D')$ is the disk $D$ bounded by the simple closed curve $\gamma$.
Note further that the restriction of $Q$ to $D'$ has a single ramification point $z'$ of degree $N = \Delta + \sum_{H(v)\le -n-1} (d_\pi(v) -1)$, 
and that $Q(z')$ is a point $z$ which was fixed on $\partial \theta(A_{\star}) \cap \partial \theta(\widetilde{V_{A_\star}})$.
We let $\gamma_{n+1}$ be the image of $\gamma'\subset V_\star$ in $S_{n+1}$. 

Before proving (C6) -- (C8), we discuss first the structure of $S_{n+1} \setminus \overline{S_{n}}$. This set is equal to
$\{G_{n+1} < \frac1{d^{n+1}} \log \rho\}$, hence is included in the union $\cup V_A$ where $A$ ranges over all connected components of 
 $S_{n} \setminus \overline{S_{n-1}}$. By definition of $G_{n+1}$ and by Lemma~\ref{lem:planar1} (2), it follows that $S_{n+1} \setminus \overline{S_{n}}$ is a finite
 union of annuli, each of finite moduli. 
 
 Observe that $G_{n+1}|_{\gamma_{n+1}}=  \frac1{d^{n+1}} \log \rho$, hence $\gamma_{n+1}$ bounds a unique component of $S_{n+1} \setminus \overline{S_{n}}$
 that we define to be $A_{\star,n+1}$. But $\gamma'$ also contains $\bigcup_{\mathcal{G}} \delta_p^{-1}(\mathcal{D}_0(p))$, and the latter set is precisely $\mu_{n+1} ( \partial \Gamma_{n+1} \setminus \partial \Gamma_{n})$ proving (C6).

\medskip

Pick any component $A'$ of $S_{n+1} \setminus \overline{S_{n}}$
different from $A_{\star,n+1}$.  Then either $A'$ is contained in a component of 
$S_{n+1} \setminus \overline{S_{n-1}}$ that does not contain $A_{\star,n}$ and $\Phi_{n+1}$ is an isomorphism on $A$ by definition.
Or we may view $A'$ inside $V_\star$ where is it a component of $V_\star\setminus \overline{B_\star}$ (see Figure~\ref{fig:map2}).
Critical points of $Q$ outside $\{z'\}$ are mapped to $\gamma$, hence are included in the boundary of $B_\star$ (inside $V_\star$).
It follows that $Q$ has no critical point in the smallest conformal disk in  $Q^{-1}(\D)$ which contains $A'$. 
This proves that $Q|_{A'}$ is a conformal isomorphism onto its image concluding (C7).

\medskip

Finally, the smallest conformal disk in $Q^{-1}(\D)$ containing $A_{\star,n+1}$ contains a single critical point $z'$ for $Q$ 
of multiplicity $d_\star(n+1) =N$, hence $\Phi_{n+1}|_{A_{\star,n+1}}$  
induces a covering map onto its image of degree $d_\star(n+1)$.
When viewed in $V_\star$, then the set $\overline{\Phi_{n+1}(A_{\star,n+1})}$  equals $B_\star$
whose boundary contains $\gamma'$ so that
 $\gamma_{n+1}$ is included in the boundary of  $\overline{\Phi_{n+1}(A_{\star,n+1})}$
concluding the proof of (C8).

\paragraph*{Addendum} 
Let us include here the following important remark. 
\begin{lemma}\label{lem:planar2}
Each surface $S_n$ is a planar domain. 
\end{lemma}
\begin{proof}
Observe that $\hat{S}_n\setminus \overline{\hat{S}_{n-1}}$ is a finite union of annuli of finite modulus so that
we may attach to $\hat{S}_n$ finitely many disks to obtain a compact Riemann surface $\overline{S}_n$. 
We prove by induction on $n$ that $\overline{S}_n$ is conformally equivalent to the Riemann sphere.

Pick any connected component $A$ of $S_n\setminus \overline{S_{n-1}}$, and let $B$ be the connected component of
$\{G_n<\frac1{d^{n-1}}\log \rho\}$ containing it. Since $\Phi_n$ is surjective onto $S_{n-1}$, one can find a component $\tilde{A}$
of $S_n\setminus \overline{S_{n-1}}$ whose image by $\Phi_n$ is included in $B$, and Lemma~\ref{lem:planar1} proves that 
$B$ is conformally equivalent to the complement of finitely many connected sets $K_1, \cdots, K_r$ in the unit disk.
We know moreover that the function $G_n$ is harmonic on $B$, tends to $\frac1{d^{n-1}}\log \rho$ when $z \to \partial B$, 
(resp. to $\frac1{d^{n+1}}\log \rho$ when $z\to K_i$), that $B_+= \{z\in B, G_n> \frac1{d^{n}}\log \rho\}$ is an annulus, and that
$\{z\in B, \frac1{d^{n+1}}\log \rho< G_n<  \frac1{d^{n}}\log \rho\}$ is a union of annuli $A_1, \cdots, A_r$.

Note that $\overline{S_{n-1}}$ is obtained by attaching a disk to $B_+$, whereas $\overline{S_{n}}$ is obtained by attaching $r$ disks
to each of the annuli $A_i$, $i=1,\ldots, r$. In both cases, we obtain a disk with boundary $\partial B$. 
We may thus patch together diffeomorphisms of the disk being the identity on its boundary to obtain
a smooth diffeomorphism between $\overline{S_{n-1}}$ and $\overline{S_{n}}$. By induction the latter is 
thus compact and simply connected. 
\end{proof}

\subsection{End of the proof of Theorem~\ref{thm:realizability}}
Let us first suppose that $\mathsf{P} = \varnothing$. In this case, one can use geometric arguments based on estimating moduli of annuli.
Apply the construction of \S\ref{sec:construct-sequence-Riemann} with $\Delta =0$: we obtain
an open Riemann surface $\hat{S} = \cup_{n\ge0} \hat{S_n}$ with 
a marked point $\infty \in \hat{S}$, and a holomorphic map  $\Phi \colon \hat{S} \to \hat{S}$ leaving $\infty$ totally invariant, and such that 
$\Phi|_{S_n} = \Phi_n$ for all $n$. 
By Lemma~\ref{lem:planar2}, $S_n = \hat{S}_n\setminus \{\infty\}$ is a planar domain for each $n$, hence we may find a sequence of univalent maps
$\kappa_n \colon \hat{S}_n \to \C \cup \{ \infty \}$ such that $\kappa_n(\infty) = \infty$. We normalize them 
so that in the chart $S_0 = \{ |z| >\rho^{1/d} \}$ we have the expansion $\kappa_n(z) = z + O(1)$. 

By Koebe's $1/4$-theorem, $\psi_n$ forms a normal family, so that one can extract a subsequence and get a univalent map 
$\psi \colon S \to \C$.  We thus get a domain $U = \psi(S) \subset \C$ and a holomorphic map $\Phi \colon U \to U$.

Observe that for $n$ large enough (larger than $ 1- \min_\Gamma H$), then $d_\star(n) =1$, and 
$S_n \setminus \overline{S_{n-1}}$ is a finite union of annuli on which the restriction of $\Phi_n$ is a biholomorphism by (C7) and (C8).
We thus get a system of nested annuli of modulus bounded from below by a positive constant, 
and it follows from~\cite[\S 2.8]{MR1312365} (or~\cite[\S 8D]{MR0264064}) that  $K := \C \setminus U$ has absolute area zero in the sense of Ahlfors,
so that $\Phi$ extends through
$K$ (see e.g.~\cite[\S 4]{MR3429163}). 

By (C0) we get a polynomial map $\Phi$  of degree $d$, whose marked dynamical graph is equal to $\Gamma$ by (C4) as required. 
 
\medskip

To handle the general case $\mathsf{P} \neq \varnothing$, we take a slightly different approach using quasi-conformal deformation
arguments. 

The graph $\Gamma$ is the union of an infinite connected marked dynamical graph $\Gamma_\esc$
and a finite marked dynamical graph $\Gamma_\fin$. By assumption the latter is realizable by a PCF polynomial so that we may find 
a PCF polynomial $P_0$  whose marked dynamical graph (with the action of the symmetry group removed)
is isomorphic to $\Gamma_\fin$.

Apply the construction of \S\ref{sec:construct-sequence-Riemann} with $\Delta =\deg(P_0)$.
Fix any integer $n_0$ larger than the depth of the infinite part of the marked dynamical graph, i.e. $n_0\ge \min_{\Gamma_\esc} H$.
Build a Riemann surface $S$ by patching a conformal disc $\D_A=\{|z| <1\}$ to each annuli component $A$ of $S_{n_0}\setminus \overline{S_{n_0-1}}$.
By Lemma~\ref{lem:planar2}, $S$ is a conformal plane which contains $S_{n_0}\supset S_{n_0-1}$ as open subsets, and $S_{n_0}\setminus \overline{S_{n_0-1}}$
is a finite union of annuli of finite modulus.

Recall that for anyone of these components $A$, the map $\Phi_{n_0}|_A$ induces an unramified cover onto its image, 
whose degree is equal to $1$ when $A\neq A_\star$ and to $\Delta$ when $A= A_\star$.

We may thus find a smooth map $\Phi\colon S \to S$ such that:
\begin{itemize}
\item
$\Phi$ is an orientation preserving finite branched cover;
\item
$\Phi = \Phi_{n_0}$ on a neighborhood of $S_{n_0-1}$;
\item for each $A\neq A_\star$, $\Phi$ is conformal and univalent on a conformal disk $\hat{\D}_A \Subset\D_A$
such that $\overline{\D_A\setminus \hat{\D}_A}$ is a compact subset of $A$;
\item
there exists a conformal disk $\hat{\D}_\star \Subset\D_\star$ such that
 $\overline{\D_\star\setminus \hat{\D}_\star}$ is a compact subset of $A$, and
$\Phi|_{\hat{\D}_\star}$ is conformally conjugated to $P_0$ in a neighborhood of its Julia set;
\item
critical points of $\Phi$ in $S\setminus S_{n_0}$ are 
contained in $\hat{\D}_\star$.
\end{itemize}
Observe that the set $K(\Phi)$ of points $x\in S$ such that $\Phi^n(x)\in \cup_A \D_A$ for all $n$
is compact and that $\Phi$ is conformal in a neighborhood of $K(\Phi)$.

We now modify the complex structure on $S$ as follows.
Let $\sigma_0$ be the standard complex structure, and set $\sigma_n:=\Phi^{n*} \sigma_0$ with corresponding Beltrami form $\mu_n$. Note that $\mu_0\equiv0$. Moreover, as $\Phi$ is a finite branched cover, we have
\[\mu_1(z)=\frac{\partial_{\bar{z}}\Phi(z)}{\partial_z\Phi(z)}\frac{d\bar{z}}{dz}, \quad \text{for a.e} \ z\in S,\]
and $\mu_1$ is non-zero exactly where $\Phi$ is not conformal, i.e. on a finite union of conformal annuli and $\|\mu_1\|_{L^\infty}<1$.
Now, outside a neighborhood 
of $K(\Phi)$, all the maps $\Phi^n$ being conformal, we have $\mu_n=\mu_0$. Finally,
in any disk $\D_A$, we may write
\[
\Phi^* \left(
\mu_n \frac{d\bar{z}}{dz} \right)
= 
\mu_n \circ \Phi
\, 
\frac{\overline{\Phi_z}}{\Phi_z}
\,
\frac{d\bar{z}}{dz}
\]
so that $\mu_n$ converges in a neighborhood of $K(\Phi)$
to an $L^\infty$ Beltrami form $\mu_\infty$ of sup-norm $<1$ since $\Phi$ is orientation preserving, 
and equal to $0$ on $K(\Phi)$.

We may thus apply the Ahlfors-Bers theorem which yields a quasi-conformal homeomorphism
$\psi \colon S \to S$ such that $\psi^*\mu_\infty$ corresponds to the standard complex structure.
The map $P:= \psi^{-1}\circ \Phi \circ \psi$ is a conformal map of the complex plane
which is conjugated near $\infty$ to $\Phi$ hence is a polynomial of degree $d$.

Since critical points and their multiplicities are topological invariant, it follows that the marked critical graph
of $P$ is isomorphic to $\Gamma$ which completes the proof.

\begin{remark}
By construction $P$ admits a renormalization which is hybrid equivalent to $P_0$.
\end{remark}


\section[Special curves and dynamical graphs]{Special curves and special critically marked dynamical graphs}\label{sec:correspondence}

\begin{theorem}\label{thm:correspondence}
Let $\Gamma$ be any special asymmetric critically marked dynamical  graph such that 
\begin{itemize}
\item[(R1)]
two distinct marked points have different images;
\item[(R2)]
the finite dynamically marked a subgraph $\Gamma_\fin$ is 
realizable by a PCF polynomial.
\end{itemize}
Then the Zariski closure in $\mpoly_{\rm crit}^d$  of the set of primitive polynomials
$P$ with disconnected Julia set and such that $\Gamma(P) = \Gamma$ is a (possibly reducible) special curve $C(\Gamma)$.

Moreover for any irreducible component $C_i$ of $C(\Gamma)$, we have $\Gamma (C_i) = \Gamma$.
\end{theorem}

In other words, there is a natural one-to-one correspondence between large classes of critically marked dynamical  graphs and
of special curves.

\subsection{Wringing deformations and marked dynamical graphs}
Let us review the construction of wringing deformations by Branner and Hubbard. We refer to~\cite[Chapter~II]{BH}
for details on the construction which relies on quasiconformal deformation techniques.

Let $P$ be any monic and centered complex polynomial of degree $d\ge2$.
For any $\tau:= \rho+i\theta$ with $\rho>0$ (i.e. $i\tau\in \H$), we define the analytic diffeomorphism $\ell_\tau:\C\setminus \overline{\D}\to\C\setminus \overline{\D}$ by setting
\[\ell_\tau(e^{r+i\psi}):=\exp(r\rho+i(\theta r+\psi)),\]
 so that
 \[\ell_\tau^{-1}(e^{r+i\psi})=\exp(r\rho^{-1}+i(-\theta\rho^{-1} r+\psi)).\]
Denote by $\sigma_0$ the standard complex structure on $\mathbb{C}$.
We let  $\sigma_{\tau}$ be the unique measurable almost complex structure on the complex plane
satisfying the following conditions:
\begin{itemize}
\item
$\sigma_{\tau}=\sigma_0$ on  $K(P)$;
\item
$\sigma_{\tau}= \varphi_P^* \ell_\tau^*(\sigma_0)$ on $\{g_P> G(P)\}$;
\item
$\sigma_{\tau}$ is invariant under $P$.
\end{itemize}
It follows from~\cite[Proposition~6.1]{BH} that there exists a unique quasiconformal map
$h_{\tau}:\p^1\to \p^1$ solving the Beltrami equation
\[\frac{\partial h_\tau}{\partial\bar{z}}=\mu_{\tau}\frac{\partial h_\tau}{\partial z}\]
and such that 
\begin{itemize}
\item
$h_{\tau}(\infty)=\infty$;
\item
$P_\tau :=h_{\tau} \circ P \circ h_{\tau}^{-1}$ is a monic and centered polynomial of degree $d$;
\item
the map $\ell_\tau \circ \varphi_P\circ h_{\tau}^{-1}$ is conformal in a neighborhood of infinity and satisfies
$\ell_\tau \circ \varphi_P\circ h_{\tau}^{-1} (z) = z + O(1)$.
\end{itemize}
We infer from~\cite[Proposition~7.2]{BH} that for each $z$ the map $\tau \mapsto h_{\tau}(z)$ is analytic in $\tau$, and
that the family of polynomials $P_\tau$ is also analytic.
When $K(P)$ is connected, then $P_\tau =P$ for all $\tau$, see~\cite[Proposition~8.3]{BH}.
When $K_P$ is disconnected, then the map  $\tau \mapsto P_\tau$ is not constant in the space of monic centered polynomials of degree $d$, see~\cite[Proposition~8.4]{BH}.

Let us now mark the critical points of $P$ so that $\crit(P) = \{ c_0, \ldots, c_{d-2}\}$.
Since being critical is a purely topological property, 
the critical points of $P_\tau$ are given by $c_i(\lambda):= h_\tau(c_i)$
so that $P_\tau$ is also critically marked.
We may thus talk about the critically marked dynamical graph $\Gamma(P_\tau)$ for all $\tau\in -i\H$.

Recall the definition of the symmetry group $\Sigma(P)$ from \S\ref{sec:symmetry} and its associated morphism $\rho_P\colon\Sigma(P) \to \Sigma(P)$
such that $P \circ g= \rho_P(g)\circ P$.
\begin{proposition}\label{prop:wring}
For any critically marked monic and centered polynomial $P$, and
for any $\tau\in -i\H$, we have $\Sigma(P)=\Sigma(P_\tau)$, $\rho_P=\rho_{P_\tau}$, and $\Gamma(P) = \Gamma(P_\tau)$.
\end{proposition}

\begin{proof}
Since $P$ and $P_\tau$ are conjugated by a homeomorphism of the plane, condition (2) from Lemma~\ref{lem:how-to-prove-gamma}  
is satisfied so that it is sufficient to prove that $\Sigma(P)=\Sigma(P_\tau)$ and $\rho_P=\rho_{P_\tau}$.

For any $g \in \Sigma(P)$ and any $\tau \in -i\H$, we set $g_\tau =  h_\tau \circ g \circ h_\tau^{-1}$.
We claim that $g_\tau$ is holomorphic. This shows that $g\mapsto g_\tau$ defines an injective morphism
 $\Sigma(P)\to\Sigma(P_\tau)$. Reversing the argument we get that the morphism is bijective, 
 and it is clear that $P_\tau \circ g_\tau = g_{\rho(\tau)} \circ P_\tau$.
 
  It thus remains to prove the claim. 
 Let $N$ be the minimal integer so that  
 $\Sigma(P) \supset \rho(\Sigma(P)) \supset \cdots \supset \rho^{N}(\Sigma(P)) = \rho^{N+1}(\Sigma(P))$. 
 
 Pick any $g\in \rho^{N}(\Sigma(P))$.
 Observe that $\rho$ induces a group isomorphism of $\rho^{N}(\Sigma(P))$. Replacing $P$ by a suitable iterate we may thus suppose that 
 $P \circ g = g \circ P$, and write $g(z)= \zeta z$ for some $\zeta^d =\zeta$. 
 
Observe that $\ell_\tau$ and $M_d$ commute so that $\varphi_{\tau} := \ell_\tau \circ \varphi_P \circ h_\tau^{-1}$ satisfies 
$\varphi_{\tau} \circ P_\tau = (\varphi_{\tau})^d$, $\varphi_{\tau}(z) = z + O(1)$ hence is the B\"ottcher coordinate of $P_\tau$.
In fact, since $P_\tau$ is also monic and centered then  $\varphi_{\tau}(z) = z + O(z^{-1})$, and
$h_\tau(z)=\ell_\tau(\varphi_P(z)) + O(z^{-1})=\ell_\tau(z) + O(z^{-1})$.

Let us look at $g'_\tau= \varphi_{\tau} \circ g_\tau \circ \varphi_{\tau}^{-1}$.
Then the equation $P_\tau g_\tau = g_\tau P_\tau$ translates as 
\[
(\varphi_\tau g_\tau \varphi_{\tau}^{-1})^{d}(z) = \varphi_\tau g_\tau P_\tau\varphi_{\tau}^{-1}= \varphi_\tau g_\tau  \varphi_\tau^{-1} (z^{d}) \]
so that
\begin{align*}
(g'_\tau(z))^{d} &= g'_\tau(z^{d}) =
h_\tau (\zeta h_\tau^{-1}(z^{d}))
=
\ell_\tau (\zeta \ell_\tau^{-1}(z^{d}))
=\zeta z^d = (\zeta z)^d~. 
\end{align*}
which implies $g'_\tau$ and $g_\tau$ to be holomorphic  near infinity. Using $P_\tau g_\tau = g_\tau P_\tau$ again, one obtains
that $g_\tau$ is holomorphic outside the filled-in Julia set. Since by construction the Beltrami coefficient of $h_\tau$ is also
$0$ on $K(P_\tau)$, it follows that $g_\tau$ is holomorphic everywhere so that the claim is proved for any $g\in \rho^N(\Sigma(P))$.

One now prove by decreasing induction on $n$ that $g_\tau$ is holomorphic for all $g\in \rho^n(\Sigma(P))$.
Suppose this is proved for some $n\le N$ and pick $g\in \rho^{n-1}(\Sigma(P))$. 
Then we have $P_\tau \circ g_\tau= g_{\rho(\tau)} \circ P_\tau$, and by the inductive hypothesis  $g_{\rho(\tau)}$
is holomorphic. Outside finitely many points, one may locally find an inverse branch of $P_\tau$ and write
$g_\tau= P_\tau^{-1} \circ g_{\rho(\tau)} \circ P_\tau$ so that $g_\tau$ is holomorphic. 
Since it is continuous, it is holomorphic everywhere
which concludes the proof.
\end{proof}

Observe that the previous proof uses quasi-conformal conjugacies in families in a fixed wringing plaque.
We conjecture the following more general result.
\begin{conj}\label{qst:wringing}
Let $P$ and $Q$ be degree $d\geq2$ polynomials. Assume there exists a quasi-conformal homeomorphism $h:\p^1\to\p^1$ which conjugates $P$ to $Q$. Then we have equalities $\Sigma(P)=\Sigma(Q)$, $\rho_P=\rho_Q$ and $\Gamma(P)=\Gamma(Q)$.
\end{conj}

\subsection{Proof of Theorem~\ref{thm:correspondence}}

We shall work in the affine space $\mathcal{M}:= \C^{d-1}$ with critically marked polynomials of the form
$P_{c,a}(z)= \frac{1}{d}z^d+\sum_{j=2}^{d-1}(-1)^{d-j}\sigma_{d-j}(c)\frac{z^j}{j}+a^d$.
There is a natural finite ramified cover from this space onto $\mpoly_{\rm crit}^d$.

Pick any special asymmetric critically marked dynamical graph $\Gamma$.
Let us denote by $\mathcal{Z}$ the set of parameters $(c,a)\in \mathcal{M}$ such that $P_{c,a}$ has degree $d$, $\Gamma(P_{c,a}) = \Gamma$  and $J_{P_{c,a}}$ is disconnected, 
and let $Z$ be the Zariski closure of $\mathcal{Z}$ in $\mathcal{M}$.

By  Theorem~\ref{thm:realizability}, $\mathcal{Z}$ is non-empty, therefore $Z$ is a non-empty algebraic variety. 

We start with the following observation. 
\begin{lemma} 
The algebraic subvariety $Z$ is of pure dimension $1$, and 
$\Gamma(W)$ is special for any irreducible component $W$ of $Z$.
\end{lemma}

\begin{proof}
Pick any irreducible component $W$ of $Z$. Let $\mathsf{A}$ (resp. $\mathsf{P}$) be the set of active (resp. passive) critical points of the family induced by $W$.
Since $W$ is a component of $Z$, one can find an infinite set of polynomials $P_k\in W$ such that $\Gamma(P_k) = \Gamma$. 
Note that a priori $\Gamma(P_k)$ may be different from $\Gamma(W)$.

Since $P_k$ have all the same critically marked dynamical graph, 
the set $\mathsf{A}_\star$ of critical points of $P_k$ with infinite orbit does not depend on $n$, and $\mathsf{A}_\star \subset \mathsf{A}$.
Moreover for each $i\notin\mathsf{A}_\star$ one can find integers $m > n\ge 0$ such that $P_k^n(c_i) = P_k^{m}(c_i)$ for all $k$, hence
$\mathsf{A}= \mathsf{A}_\star$. 
Now since $\Gamma$ is special, $\mathsf{A}_\star$ is non-empty and for each $i,j \in \mathsf{A}_\star$, there exist integers $n$ and $m$ such that 
$P^n(c_i) = P^m(c_j)$. This implies $\Gamma(W)$ to have at most one infinite connected component so that this graph is special. 
Note that these arguments prove in particular that $W$ cannot have dimension $0$.

Fix any critical point, say $c_0$ such that  $P^n(c_0) = P^m(c_j)$ implies $m\ge n$. 
 It follows that the two functions $G:= G(P_{c,a})$ and $g_0(c,a) := g_{P_{c,a}}(c_0)$ coincide on $W$.
 Suppose by contradiction that $\dim(W)=l \ge 2$. Since $G$ is continuous, psh and its growth at infinity is
 $G(c,a) = \log^+\max\{|c|,|a|\} + O(1)$ and $G(c,a) - \log^+\max\{|c|,|a|\}$ extends continuously at infinity in $\bar{\mathcal{M}}=\p^{d-1}$, it induces a continuous semi-positive metrization on $\cO(1)_{\bar{W}}$ 
 where $\bar{W}$ is the closure of $W$ in $\mathbb{P}^{d-1}_\C$. It follows that 
the mass of   $(dd^c G|_W)^l$ is positive equal to  $\mathrm{deg}(W)$,
 so that the positive closed $(2,2)$-current $(dd^c G|_W)^2$ is non-zero.
On the other hand, we have 
\[
 (dd^c G|_W)^2 = \lim_{\epsilon \to0} dd^c g_0 \wedge dd^c \max\{g_0,\epsilon\} =0\]
 since the current $dd^c g_0$ is supported on $\{g_0 =0\}$ by~\cite[Proposition~6.9]{favredujardin}, which gives the required contradiction.
 This proves $\dim(W) =1$ which concludes the proof.
  \end{proof}

Note that by Theorem~\ref{tm:marked graph is special} the curve $W$ is special.
It remains to prove that $\Gamma(W) = \Gamma$. Pick any polynomial $P\in W$ for which $\Gamma(P) = \Gamma$, and
which is a smooth point of $Z \supset W$. 
Note that any polynomial $P_{c,a}(z)= \frac1d z^d - \sigma_1(c)\frac{z^{d-1}}{d-1} + o(z^{d-1})$ 
is conjugated by $z \mapsto \delta z +\frac{  \sigma_1(c)}{d-1} $ for some fixed $\delta^{d-1} =d$ to 
a monic and centered polynomial, so that we have a finite ramified cover from $\mathcal{M}$ onto the space
of monic and centered polynomials. It follows from Proposition~\ref{prop:wring} that the lift of the wringing plaque to $\mathcal{M}$ yields
a holomorphic disk $\D_P$  which contains $P$ and such that $\Gamma(Q) = \Gamma$ for all $Q \in \D_P$.
We get $\D_P \subset Z$, and since $P$ is a regular point of $Z$, we conclude that $\D_P \subset W$ which implies 
$\Gamma(W) = \Gamma$ as required.


\section{Realizability of PCF maps}\label{sec:real-PCF}
The applicability of Theorem~\ref{thm:correspondence} for a graph relies on our ability to realize
a given finite critically marked dynamical graph as the marked graph of a PCF polynomial (up to symmetry). 

Recall that a finite critically marked dynamical graph $\Gamma$ without symmetry is realizable\index{graph!realizable dynamical} by a PCF polynomial if there exists a PCF polynomial $P$
such that $\Gamma$ is equal to $\Gamma(P)$ with the action of $\Sigma(P)$ removed.

In this section we give various conditions on $\Gamma$ to ensure its realizability.

In the case of a single marked points, the realizability amounts to the existence
of a solution to a polynomial equation which have been studied in details by Buff~\cite{MR3890968}.
A simple count then gives: 
\begin{theorem}\label{thm:realization-unicritical}
Any finite critically marked dynamical graph of degree $d\ge2$ with a single marked point
is realizable by a PCF polynomial. 
\end{theorem}

The situation is much harder in the presence of several marked points. 
Using transversality technics developed by A. Epstein, see e.g.~\cite{buffepstein}, 
the approach of~\cite{favregauthier} yields the following result. 
\begin{theorem}\label{thm:realizationPCF3}
Any finite critically marked dynamical graph of degree $d$
which is a union of $(d-1)$ cycles of different length 
is realizable by a PCF polynomial. 
\end{theorem}

Let us introduce the following numerical invariant attached to a critically marked dynamical graph.
Enumerate the marked vertices of $\Gamma$ such that  if $v_j$ lies in the orbit of $v_i$
then $j\ge i$. For each $i$, let $n_i$ be the minimum integer such that 
either $\pi^{n_i}(v_i)$ is periodic, or belongs to the orbit of $v_j$ for some $j<i$.
We set $\delta(\Gamma)$ to be the minimum of those positive numbers of the form
$\max \{ d^{-n_1}, \sum_{i\ge2} d^{-n_i}\}$ over all orderings of marked vertices satisfying the condition above.

\begin{theorem}\label{thm:realizationPCF}
Suppose that $\Gamma$ is a finite critically marked dynamical graph such that no two distinct cycles have the same period.
If $\delta(\Gamma) \le \frac1{2d}$, then $\Gamma$ is realizable by a PCF polynomial. 
\end{theorem}

The condition on  $\delta(\Gamma)$ requires all marked vertices to be strictly pre-periodic (so that 
the polynomial realizing $\Gamma$ is strictly PCF).
Note that the condition is satisfied when 
we have $n_i\ge3$ for all integers defined above.

The proof of this result is based on the realization theorem of Bielefeld, Fisher, Hubbard~\cite{BFH}
of strictly PCF polynomials with prescribed combinatorics.

\smallskip

It is likely that \emph{any} finite critically marked dynamical graph is realizable. 
Note however that the combinatorial complexity becomes huge 
already in degree $3$ when the conditions in the previous theorems are dropped. 
We refer to the work of Poirier~\cite{Poirier}
for a combinatorial description of all PCF polynomials.  


\subsection{Proof of Theorem~\ref{thm:realization-unicritical}}
Write $P_c(z) = c z^d+1$. We need to show that for any $n\ge0$ and $k\neq1$ there exists one $c\in \C$
such that $0$ is mapped by $P^k_c$ to a periodic point of exact period $n$ and $P^{k-1}_c(0)$ is not periodic.

\smallskip

Let us first treat the case $k=0$.
By~\cite[Lemma~3]{MR3890968} the following holds. 
For any integer $n$, the polynomial $Q_n(c):= P^n_c(0)$ has  degree $\frac{d^{n-1}-1}{d-1}$ and its roots are simple. 
Observe that $Q_n^{-1}(0)$ is the set of $c\in \C$ such that the critical point $0$ is periodic for 
$P_c$ of period divisible by $n$. By M\"obius inversion formula, it follows that the cardinality $\delta(n)$ of the set
$c\in \C$ such that the critical point $0$ is periodic for 
$P_c$ of exact period $n$ is equal to 
\begin{align*}
\delta(n) 
&:= 
\sum_{m\le n} \mu \left( \frac{n}{m} \right)\, \frac{d^{m-1}-1}{d-1}
~.
\end{align*}
Since $\sum_{m\le n} \mu \left( \frac{n}{m} \right)=0$ for all $n\ge2$, we infer $\delta(n)>0$ from the next lemma
\footnote{our proof actually yields $\delta(n) \ge c d^n$ for some positive constant $c>0$, see~\cite{favregauthier}.}. 

\begin{lemma}
For any $\rho\ge2$ and for any $n\ge1$, we have
\[\sum_{m| n} \mu \left( \frac{n}m \right) \rho^m
>0~.\]
\end{lemma}

\begin{proof}
We may assume $n\ge2$, so that
\begin{align*}
\sum_{m | n} \mu \left( \frac{n}m \right) \rho^m
&\ge
\rho^n - \sum_{1 \le m \le n/2} \rho^m
\ge
\rho^n 
- \left( 
\frac{\rho^{1+n/2} -1}{\rho-1}
\right)
\\
&>
\frac{\rho^n (\rho-1)  - \rho^{1+n/2}}{\rho-1} 
\ge 0~,
\end{align*}
since $n\ge 1+n/2$.\end{proof}

Suppose now that $k$ is positive. Observe that $k\ge2$ since $P_c$ is unicritical.
Let $Q_{k,n}$ be the polynomial with simple roots whose zero locus is the set of $c\in\C$ such that $0$ is mapped by $P^k_c$ to a periodic point of exact period $n$ and $P^{k-1}_c(0)$ is not periodic. We need to show that $\delta_{k,n} := \deg(Q_{k,n}) >0$.

By~\cite[Lemma~10]{MR3890968}, we have 
\[
\frac{Q^d_{k+n-1}(c) - Q_{k-1}^d(c)}{Q_{k+n-1}(c) - Q_{k-1}(c)}
 = Q^{d-1}_{\gcd\{k-1,n\}} \prod_{m| n} Q_{k,m}(c)\] so that 
\begin{align*}
\delta_{k,n}
&= 
\sum_{m|n} \mu \left( \frac{n}{m} \right)\, 
\left( d^{m+k-2} - d^{-1 + \gcd\{k-1,m\}}
\right)
~.
\end{align*}
Write $l = k-1$. By the lemma below, we have
\begin{align*}
d\times \delta_{k,n}
=
\sum_{m|n} \mu \left( \frac{n}{m} \right)\, 
\left( d^{m+l} - d^{\gcd\{l,m\}}
\right)
\ge 
d^l \left(d^n 
- \left( 
\frac{d^{1+n/2} -1}{d-1}
\right)\right)
 -d^l>0
 ~.
\end{align*}
This concludes the proof.
\begin{lemma}
For any $\rho\ge2$, for any integer $l\ge 1$, and for any $n\ge1$, we have
\[0\le \sum_{m| n} \mu \left( \frac{n}m \right) \rho^{\gcd\{l,m\}}
\le \rho^l~.\]
\end{lemma}

\begin{proof}
Set $\Delta(n)= \sum_{m| n} \mu \left( \frac{n}m \right) \rho^{\gcd\{l,m\}}$ so that 
$\rho^{\gcd\{l,n\}}= \sum_{m| n} \Delta(m)$. It is sufficient to show that $\Delta(n)\ge0$ for all $n$.
Observe that for each divisor $n$ of $l$ we have
$\Delta(n)= \sum_{m| n} \mu \left( \frac{n}m \right) \rho^{m}$ so that $\Delta(n)\ge0$ follows from the previous lemma.
We claim that by induction $\Delta(n) =0$ if $n$ is not a divisor of $l$. Indeed write $\gcd\{n,l\} = L$. Then
\begin{align*}
\rho^L
&= 
\sum_{l'|L} 
\left(\Delta (l') + \left(
\sum_{\substack{l' < m, \, m|n \\ \gcd\{m,l\} =l'}} \Delta(m) 
\right)
\right)
= 
\sum_{l'|L} \Delta (l') + \Delta(n)
\end{align*}
hence $\Delta(n) =0$ as claimed.
\end{proof}


\subsection{Proof of Theorem~\ref{thm:realizationPCF3}}
Let $m_1 \neq m_2 \neq \cdots \neq m_{d-1}$ be the length of each cycle of $\Gamma$. 
The locus $\mathrm{Per}(m_j,0)$ of polynomials $P_{c,a}$ having a critical orbit of exact period $m_j$
is an algebraic hypersurface, see e.g.~\cite[\S 2.1]{BB2}. It follows from~\cite[Corollary~4.9]{favregauthier}
that the $(d-1)$ hypersurfaces  $\mathrm{Per}(m_j,0)$ intersect transversally. 
In the parameterization by $P_{c,a}$ these hypersurfaces have no intersection at infinity by Bassanelli and Berteloot, see e.g.~\cite[\S 4]{BB2},
so that  the number of PCF polynomials realizing $\Gamma$ is equal to the product
$\prod_{i=1}^{d-1} \deg(\mathrm{Per}(m_j,0))$ which is positive (in fact $\gtrsim d^{\sum m_j}$), see~\cite[Lemma~6.3]{favregauthier}.


\subsection{Combinatorics of strictly PCF polynomials}\label{sec:combin}
We review briefly the classification of PCF polynomial in terms of critical portraits. 
A PCF polynomial is said to be strict when none of its critical point is periodic. 

\paragraph*{Critical portrait}
The notion of a critical portrait was introduced in~\cite{BFH} which encodes the combinatorics of a given polynomial. 
Recall that two finite disjoint subsets $F, F'$ of $\R/\Z$ are said to be unlinked when 
$F$  lies in a single connected component of $(\R /\Z)\setminus F'$. 
Two subsets are unlinked iff there exist disjoint open segments $I$ and $I'$ such that 
$F \subset I$ and $F'\subset I'$.

\begin{definition}
A critical portrait is a collection $(\Theta_1, \cdots , \Theta_N)$ of disjoint finite sets in $\R/\Z$ satisfying the following three conditions:\index{critical portrait}
\begin{itemize}
\item[(CP1)]
for any fixed $i$, $\Theta_i = \{\theta_{i,1},\cdots,\theta_{i,d(i)}\}$ with $d(i)\ge2$ and $d\theta_{i,j} = d\theta_{i,1}$ for all $j$;
\item[(CP2)]
$\sum_i (\Card  (\Theta_i) -1) = d  - 1$;
\item[(CP3)]
for any $i\neq j$, the sets $\Theta_i$ and $\Theta_j$ are unlinked.
\end{itemize}
\end{definition}

\paragraph*{Landing map}
Let $P$ be any polynomial of degree $d\ge2$ with connected  Julia set so that 
all its critical points have a bounded orbit.
Recall from Proposition~\ref{prop:GreenBottcher} that the B\"ottcher coordinate $\varphi_P$ 
defines a univalent map from $\{g_P>0\}=\C \setminus K_P$ onto $\{ |z| > 1\}$. 
The external ray $R_\theta$ of angle $\theta\in \R/\Z$ 
is the image under $\varphi_P^{-1}$ of the ray $\{te^{i\pi\theta}\}_{t>1}$. Equivalently external rays are gradient lines
of the Green function. One says that the external ray $R_\theta$ lands at a point $z\in K_P$ if 
$z$ is the unique boundary point of the ray.  

When $K_P$ is locally connected, it follows from Caratheodory theorem that all external rays land at a point, so that we get a continuous surjective map
$\mathsf{e}\colon \R/\Z \to J_P$ which semi-conjugates $M_d$ to $P$.

\paragraph*{Critical portrait of PCF maps}
When $P$ is a strictly PCF polynomial with critical points $c_1, \ldots , c_N$, its Julia set is connected and locally connected (see~\cite[Theorem~17.5]{Milnor4}).
For each $i$, one may thus choose a ray $R_{\eta_i}$ landing at $P(c_i)$, and set 
$\Theta_i = M_d^{-1} (\eta_i)$. One can prove that the collection of sets $\{\Theta_i\}_{1\le i\le N}$ defines a critical portrait, see~\cite[Proposition~2.10]{BFH}.
Note that the critical portrait is only well-defined once a choice of external rays landing at critical values is made.

To state the next result we recall the definition the $\Theta$-equivalence relation given in~\cite{BFH}:
\begin{itemize}
\item
$x, y\in \R/\Z$ are $\Theta$-unlinked if $\{x,y\}$ and $\Theta_i$ are unlinked for all $i$;
\item
$x$ and $y$ are $\Theta$-unlinkable if one can find  a pair $\{x', y'\}$ arbitrarily close to $\{x,y\}$
which is $\Theta$-unlinked;  
\item
$x$ and $y$ are $\Theta$-equivalent if for all $n\ge0$ the points $d^nx$ and $d^ny$ are $\Theta$-unlinkable. 
\end{itemize}
We shall write $x \sim_\Theta y$ whenever $x$ and $y$ are $\Theta$-equivalent.\index{critical portrait!of a PCF polynomial}
\begin{theorem}\label{thm:equivalence}
Suppose $\Theta = \{ \Theta_1, \cdots , \Theta_N\}$ is a critical portrait such that 
any $\theta \in \bigcup_i \Theta_i$ is strictly preperiodic for $M_d$. 
Then the landing map $\mathsf{e} \colon \R/\Z \to J_P$ induces a conjugacy
$(\R/\Z / \sim_\Theta, M_d) \to (J_P, P)$.
\end{theorem}

\paragraph*{Realization of a critical portrait}
The next result follows from~\cite[Theorem~II]{BFH}, see also~\cite[Corollary~5.3]{kiwi-portrait}.

\begin{theorem}\label{thm:kiwi}
Suppose $\Theta = \{ \Theta_1, \cdots , \Theta_N\}$ is a critical portrait such that 
any $\theta \in \bigcup_i \Theta_i$ is strictly preperiodic for $M_d$. 

Then there exist a strictly PCF polynomial $P_\Theta$
and a choice of external rays landing at its critical values
for which the critical portrait is equal to $\Theta$.

Moreover, for each $i$ all external rays $R_\theta$ with $\theta\in\Theta_i$ land at the same critical point, and $c_i$
which has multiplicity $\Card(\Theta_i)$. 
\end{theorem}

We now explain how to understand the critically marked dynamical graph of $P_\Theta$ from its critical portrait.  
Pick an angle $\theta_i \in \Theta_i$ for each $i$, and consider the dynamical graph $\Gamma_\Theta:= G(\{\theta_1, \ldots, \theta_N\},M_d)$
as follows (see Example~\ref{ex:mark-dyn}).
Vertices are $\{M_d^n(\theta_i)\}_{i,n\ge0}$ and an edge joins two vertices $v$ and $v'$ when $M_d(v)= v'$ or 
$M_d(v')= v$. The map $M_d$ induces a natural flow on $\Gamma$.
By (CP2), we may partition $\{ 0, \cdots, d-2\}$ into $N$ subsets $I_1, \ldots, I_N$ 
such that each set $I_j$ has the same cardinality as $\Theta_j$.
We define the marking $\mu \colon \{ 0, \cdots, d-2\} \to \Gamma$ sending
each integer $i\in I_j$ to $\theta_j$. 
We obtain in this way a critically marked dynamical graph (with no symmetry) that we denote by $\Gamma_\Theta$.

Since the landing map  is a semi-conjugacy between $M_d$ and $P_\Theta$, Theorem~\ref{thm:kiwi} implies
 $\mathsf{e}$ to induce a canonical map $\mathsf{e}\colon \Gamma_\Theta\to \Gamma(P_\Theta)$ 
 which is surjective by construction and semi-conjugates the flows.
\begin{prop}\label{prop:easy obs graph}
If $\Gamma_\Theta$ and  $\Gamma(P_\Theta)$ have the same number of cycles of any given period, 
then $\mathsf{e}$ induces an isomorphism
between $\Gamma_\Theta$ and the critically dynamical graph without symmetry obtained from $\Gamma(P_\Theta)$ 
by forgetting the action of $\Sigma(P_\Theta)$.
\end{prop}

\begin{proof}
We show that $\mathsf{e}\colon \Gamma_\Theta\to \Gamma(P_\Theta)$  induces a conjugacy. 
For each $n$ denote by $\Gamma_{\Theta,n},  \Gamma(P_\Theta)_n$ the dynamical subgraphs
of $\Gamma_\Theta$ and $\Gamma(P_\Theta)$ respectively of those points at (graph) distance $\le n$
from the periodic cycles.

Since both graphs are finite, for $n$ sufficiently large we have $\Gamma_{\Theta,n}= \Gamma_\Theta$ and $\Gamma(P_\Theta)_n =\Gamma(P_\Theta)$.
Observe that $\mathsf{e}$ maps $\Gamma_{\Theta,n}$ onto  $\Gamma(P_\Theta)_n$. Also $\Gamma_{\Theta,0}$ and  $\Gamma(P_\Theta)_0$ are unions of periodic cycles, and by assumption $\mathsf{e}\colon \Gamma_{\Theta,0}\to \Gamma(P_\Theta)_0$ is a conjugacy. 

We assume by induction that $\mathsf{e}\colon \Gamma_{\Theta,n}\to \Gamma(P_\Theta)_n$ is a conjugacy. 
Pick $\theta,\theta' \in \Gamma_{\Theta,n+1}$ such that $\mathsf{e}(\theta) = \mathsf{e}(\theta')$. By the inductive assumption $M_d(\theta)=M_d(\theta')$, 
so that the two angles differ by a multiple of $\frac1d$. We conclude using the next lemma.
\end{proof}

\begin{lemma}\label{lem:easy obs graph}
Let $\theta\neq\theta'\in \R/\Z$ be two angles such that $\theta- \theta' \in \frac{\Z}d$, and 
$\theta\sim_\Theta\theta'$. Then $\{\theta,\theta'\} \subset  \Theta_i$ for some $i$.
\end{lemma}

\begin{proof}
Consider the equivalence relation on $(\R/\Z)\setminus \bigcup_i\Theta_i$  defined by $x\equiv_\Theta y$ iff $\{x,y\}$ is $\Theta$-unlinked.
Equivalent classes for $\equiv_\Theta$ are in bijection with connected components of the complement in the unit disk of the union of the convex hulls
of $\Theta_i$ (see Figure below). 
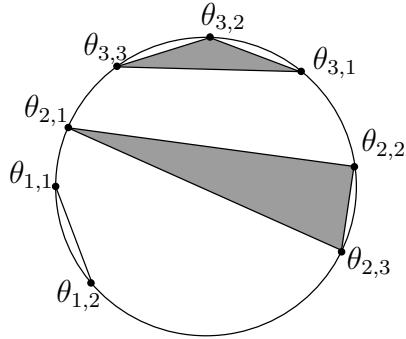
\begin{figure}[h]
\centering
\def\svgwidth{6cm}
\begingroup%
  \makeatletter%
  \providecommand\color[2][]{%
    \errmessage{(Inkscape) Color is used for the text in Inkscape, but the package 'color.sty' is not loaded}%
    \renewcommand\color[2][]{}%
  }%
  \providecommand\transparent[1]{%
    \errmessage{(Inkscape) Transparency is used (non-zero) for the text in Inkscape, but the package 'transparent.sty' is not loaded}%
    \renewcommand\transparent[1]{}%
  }%
  \providecommand\rotatebox[2]{#2}%
  \newcommand*\fsize{\dimexpr\f@size pt\relax}%
  \newcommand*\lineheight[1]{\fontsize{\fsize}{#1\fsize}\selectfont}%
  \ifx\svgwidth\undefined%
    \setlength{\unitlength}{272.96080656bp}%
    \ifx\svgscale\undefined%
      \relax%
    \else%
      \setlength{\unitlength}{\unitlength * \real{\svgscale}}%
    \fi%
  \else%
    \setlength{\unitlength}{\svgwidth}%
  \fi%
  \global\let\svgwidth\undefined%
  \global\let\svgscale\undefined%
  \makeatother%
  \begin{picture}(1,0.71165699)%
    \lineheight{1}%
    \setlength\tabcolsep{0pt}%
    \put(0,0){\includegraphics[width=\unitlength,page=1]{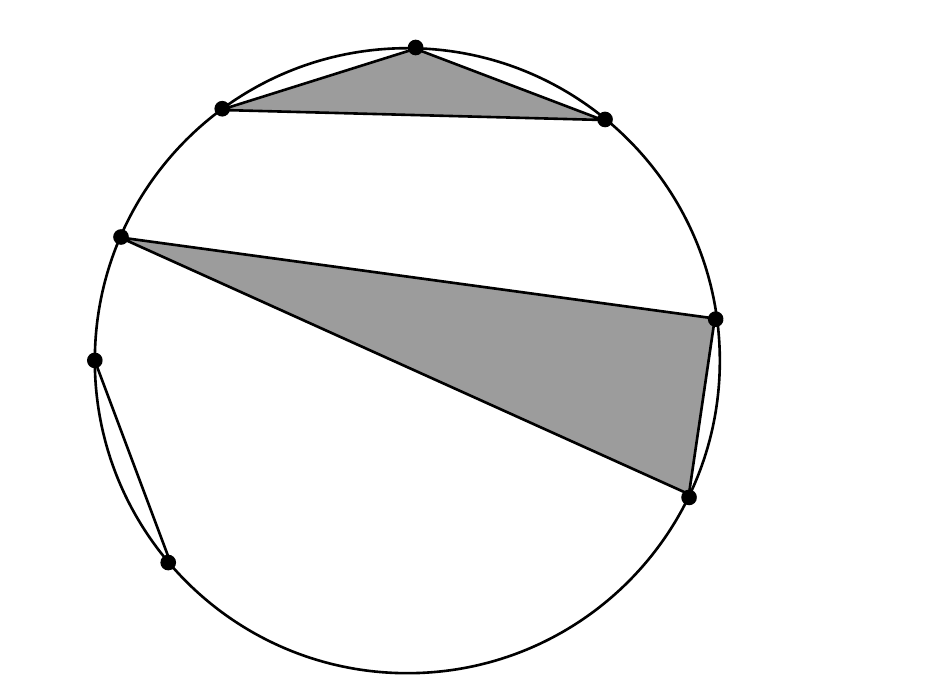}}%
    \put(-0.00221816,0.33903628){\makebox(0,0)[lt]{\lineheight{1.25}\smash{\begin{tabular}[t]{l}$\theta_{1,1}$\end{tabular}}}}%
    \put(0.10207983,0.07596735){\makebox(0,0)[lt]{\lineheight{1.25}\smash{\begin{tabular}[t]{l}$\theta_{1,2}$\end{tabular}}}}%
    \put(0.03081191,0.48469147){\makebox(0,0)[lt]{\lineheight{1.25}\smash{\begin{tabular}[t]{l}$\theta_{2,1}$\end{tabular}}}}%
    \put(0.77079979,0.40994489){\makebox(0,0)[lt]{\lineheight{1.25}\smash{\begin{tabular}[t]{l}$\theta_{2,2}$\end{tabular}}}}%
    \put(0.73837531,0.14891787){\makebox(0,0)[lt]{\lineheight{1.25}\smash{\begin{tabular}[t]{l}$\theta_{2,3}$\end{tabular}}}}%
    \put(0.66416741,0.58988709){\makebox(0,0)[lt]{\lineheight{1.25}\smash{\begin{tabular}[t]{l}$\theta_{3,1}$\end{tabular}}}}%
    \put(0.41815214,0.68910332){\makebox(0,0)[lt]{\lineheight{1.25}\smash{\begin{tabular}[t]{l}$\theta_{3,2}$\end{tabular}}}}%
    \put(0.16130894,0.62242471){\makebox(0,0)[lt]{\lineheight{1.25}\smash{\begin{tabular}[t]{l}$\theta_{3,3}$\end{tabular}}}}%
  \end{picture}%
\endgroup%

\caption{A critical portrait}
\label{fig:topView}
\end{figure}

In particular, there is exactly $(d-1)$ equivalent classes, and each of them is a finite union of open 
segments $\cup_{i=1}^k I_i$ that we may order cyclically clockwise. We observe now that there is a unique translation $T_2$ by a multiple of 
$\frac1d$ such that $T_2(I_2)$ has one common boundary point with  $I_1$. Proceeding inductively we find translations $T_j$ such that 
$T_j(I_i)$ has one common boundary point with $T_{j-1}(I_{j-1})$, so that the union $\overline{I_1} \cup \overline{T_j(I_j)}$ is a segment of length $\frac1d$.

If  $\theta$ and $\theta'$ differ by a multiple of $\frac1d$ and none of them belong to $\bigcup_i \Theta_i$, then $\theta$ and $\theta'$ cannot belong to the
same $\equiv_\Theta$-equivalence class so that in particular we have
$\theta \not\sim_\Theta \theta'$. In fact if $\theta\in \Theta_i$ and $\theta \sim_\Theta \theta'$, then
$\theta'\in \Theta_i$ too.
\end{proof}

\subsection{Proof of Theorem~\ref{thm:realizationPCF}}
Let $\Gamma$ be a critically marked dynamical finite graph as in the statement of the theorem. We may suppose $d\ge 3$
by Theorem~\ref{thm:realization-unicritical}.
Recall the definition of a marked dynamical graph attached to an arbitrary dynamical system from \S\ref{sec:criticallymarkedgraph}.

\medskip
\noindent {\bf Step 1}. Choosing adapted periodic orbits for $M_d$.

It is convenient to work in $d$-ary base for points in $\R/\Z$. 
If $p = \sum_{k\ge1}\frac{\epsilon_k}{d^k}$ with $\epsilon_k \in \{0, \ldots, d-1\}$ then we have
$M_d(p) =  \sum_{k\ge1}\frac{\epsilon_{k+1}}{d^k}$.

Let $n_1, \ldots, n_r$ be the period of each periodic cycle of $\Gamma$. By assumption, we have $n_i\neq n_j$ for $i\neq j$.
For  each $i = 1, \ldots, r$, define
\[p_i := \left(\sum_{k=1}^{n_i}\frac{\epsilon^{(i)}_k}{d^k} \right) \cdot \frac{d}{d-1}~,\]
where $\epsilon^{(i)}_k = 2$ if $(k-1) |n_i$ and $\epsilon^{(i)}_k = 0$ otherwise. 

Observe that $\frac{d}{d-1}= \sum_{k\ge0} \frac1{d^k}$ hence $p_i$ is periodic of period divisible by $n_i$. 
By construction the first digit of $M_d^k(p_i)$ when $k|n_i$ and $k<n_i$ is $2$, whereas the first digit of $p_i$ is $0$,
hence the exact period of $p_i$ is $n_i$.

\begin{lemma}\label{lem-estim-distance}
For each $i$ and for each each divisor $m$ of $n_i$ such that $m<n_i$, we have
$d_{\R/\Z}( M_d^m(p_i),p_i) \ge \frac1d$.
\end{lemma}

\begin{proof}
Recall that for any two points $z, w \in [0,1]$, we have $d_{\R/\Z}(z,w)= \min\{ |z-w|, 1- |z-w|\}$.
To simplify notation we drop the index $i$, and estimate the distance between $p$ and $M^m_d(p)$
for $m$ dividing $n$.
We have
\begin{align*}
\tau:= d^m p -p \mod 1
&=  \left(\sum_{k=1}^n\frac{\epsilon_{k+m} - \epsilon_{k}}{d^k} \right) \cdot \frac{d}{d-1}
\\
&=  \left(\frac2d + \sum_{k=2}^n\frac{\epsilon_{k+m} - \epsilon_{k}}{d^k} \right) \cdot \frac{d}{d-1}~.
\end{align*}
We thus get the lower bound
\[
\tau \ge \left(\frac2d - \frac1d\right) \cdot \frac{d}{d-1} =  \frac1{d-1}
~.
\]
To get an upper bound, we let $r$ be the smallest integer $\ge2$
not dividing $n$. 
Then 
\[
\tau \le  \frac2d - \frac2{d^r} + \frac1{d^r}\le \frac2{d}
\]
hence\[
d_{\R/\Z}( M_d^m(p),p) =  \min\{ \tau, 1- \tau\}
\ge \min \left\{\frac1{d-1}, 1 - \frac2{d}\right\}
\]
which concludes the proof.
\end{proof}

\medskip
\noindent {\bf Step 2}. Find a critical portrait $\Theta$ such that $\Gamma_\Theta = \Gamma$.

We enumerate the marked points $\mu(\mathcal{F})=\{ v_1, \ldots, v_r\}$ in such a way that 
$\pi^n(v_i) = v_j$ for some $n\ge0$ implies $i\ge j$.

Choose the first angle $\theta_1$ with the same combinatorics as $v_1$. 
Let $n_1\ge2$ be the minimum integer such that 
$\pi^{n_1}(p_1)$ is periodic. Then there are $(d-1)$ possibilities 
for $\pi^{n_1-1}(p_1)$ at distance either $1/d$ or $2/d$ from one to the other. 
We thus have $(d-1)d^{n_1-1}$ possibilities for
$\theta_1$ at distance either $\frac1{d^{n_1}}$ or  $\frac2{d^{n_1}}$. 
We may thus choose the angle such that 
$0 <\theta_1 < \frac2{d^{n_1}}$, and
set $\Theta_1 = \{ \theta_1, \ldots, \theta_1 + \frac{d_1-1}d\}$. 

In the same way, we choose the angle $\theta_2$ realizing the combinatorics of $p_2$, and such that 
$\theta_1 + \frac{d_1-1}d<\theta_2 < \theta_1 + \frac{d_1-1}d + \frac{2}{d^{n_2}}$, and set
$\Theta_2 = \{ \theta_2, \ldots, \theta_2 + \frac{d_2-1}d\}$. 

Observe that 
\[
\theta_2 +  \frac{d_2-1}d
< 
\theta_1 + \frac{d_1-1}d + \frac{2}{d^{n_2}} +  \frac{d_2-1}d
\le
\theta_1 + \frac{d-1}d + \frac{1}{d} \le \theta_1 +1
\]
so that $\Theta_1$ and $\Theta_2$ are unlinked. 
We may thus define inductively a critical portrait
$\Theta_1, \ldots, \Theta_N$ such that 
$\Theta_j = \{ \theta_j, \ldots, \theta_j + \frac{d_j-1}d\}$, 
$\theta_{j-1} + \frac{d_j-1}d<\theta_j < \theta_{j-1} + \frac{d_j-1}d + \frac{2}{d^{n_j}}$
hence
\[
\theta_N +  \frac{d_N-1}d
< 
\theta_1 + \sum_{j=1}^N \frac{d_j-1}d + \sum_{j=2}^N \frac{2}{d^{n_j}} 
\mathop{\le}\limits^{\delta(\Gamma)\le \frac1{2d}}
\theta_1 + \frac{d-1}d + \frac{1}{d} \le \theta_1 +1~.
\]
By construction, we have $\Gamma_\Theta = \Gamma$.

\medskip
\noindent {\bf Step 3}. Proof that $\Gamma_\Theta = \Gamma (P_\Theta)$.

Recall that $P_\Theta$ is a PCF polynomial obtained by applying Theorem~\ref{thm:kiwi}. 
We claim that $\Gamma_\Theta$ and $\Gamma (P_\Theta)$
have the same number of periodic cycles of any given period. 

Then by Proposition~\ref{lem:easy obs graph}, $\Gamma_\Theta$ is isomorphic to $\Gamma (P_\Theta)$ with the action of
the symmetry group removed, and the proof is complete. 
It thus remains to prove our claim. 

Pick any periodic point as in Step 1 of the proof. We drop the index $i$ for simplicity. Then $p$ has period $n$
and its $d$-ary expansion is of the form $p=\frac2{d^2} + \sum_{k\ge3}\frac{\epsilon_k}{d^k}$. Note that 
for any $m$ dividing $n$, we have $M_d^m(p) =\frac2{d} + \sum_{k\ge3}\frac{\epsilon_{k+m}}{d^k} \mod 1$.
It follows that $p$ belongs to the segment $(2/d^2, 3/d^2)$. Since $\theta_1 \in (0 , \frac2{d^{n_1}})$ and $n_1\ge2$,
we conclude that $p\in (\theta_1, \theta_1+ \frac1d)$. Since $d_{\R/\Z}(M^m_d(p),p)\ge \frac1d$ by
Lemma~\ref{lem-estim-distance}, $M^m_d(p)$ does not belong to the segment $ (\theta_1, \theta_1+ \frac1d)$
hence $\{M^m_d(p),p\}$ is not $\Theta$-unlinked. In particular, $\mathsf{e}(M^m_d(p)) \neq \mathsf{e}(p)$
which implies the point $\mathsf{e}(p)$ to have exact period $n$. 

This concludes the proof.


\section{Special curves in low degrees}\label{sec:class-petite}

We discuss special curves of degree $d\leq 5$.
Observe that (up to finite branched cover) the only non-isotrivial one-parameter family of degree $2$ polynomials is a curve: the family $P_t(z)=z^2+t$, and it forms a special family. 

\medskip

Recall that we denoted by $\Sigma(d,k,\mu)$ the set of monic centered polynomials of degree $d\geq2$ which can be written as $z^\mu Q(z^k)$ with $k\geq2$ maximal and $Q(0)\neq0$.

\paragraph*{Classification in degree $3$}
The special curves of cubic polynomials can be classified in the following way:

\begin{itemize}
\item either one critical point is preperiodic 
\item or the two distinct critical points belong to the same grand orbit. Remark that the unicritical family $\Sigma(3,3,0)$ is a particular example where this happens,
\item or the curve is the Zariski (or Euclidean) closure of $\Sigma(3,2,1)$: for a general polynomial $P$ in the curve, $\Sigma(P)=\mathrm{Aut}(P)=\mathbb{U}_2$. this curve can be parametrized by $P_t(z)=z(z^2+t)$, $t\in\mathbb{A}^1$.
\end{itemize}

One has a combinatorial classification in degree $3$: if the curve is is not  $\Sigma(3,2,1)$, then the dynamical graph is asymetrical 
and we have a correspondence special curves/dynamical graphs. We refer to Figure~\ref{fig:marked-degree3} for the description of special marked dynamical graphs of degree $3$.

\paragraph*{Classification in degree $4$}
The special curves of degree $4$ polynomials can be classified in the following way:

\begin{itemize}
\item either the curve is non-primitive: it can be parametrized as $P(z)= z^4 +az^2 +c$, with $4c = a^2 - 2a\zeta$ and $\zeta^3 = -1$.
\item or two critical points are preperiodic,
\item or one critical point is periodic, and the other two lie in the same grand orbit,
\item or the three distinct critical points belong to the same grand orbit. Remark that the unicritical family $\Sigma(4,4,0)$ is a particular example where this happens,
\item or the curve is the Zariski (or Euclidean) closure of a curve in $\Sigma(4,2,0)$ such that $\Sigma(P)=\mathrm{Aut}(P)=\mathbb{U}_2$, two critical points are permutted by a symmetry and
\begin{itemize}
\item either the third critical point is preperiodic,
\item or the third critical point shares the same grand orbit as one of the swapped critical points.
\end{itemize}
\item or the curve is the Zariski closure of $\Sigma(4,3,1)$: for a general polynomial $P$ in the curve, the three critical points are permuted by $\U_3$. This curve can be parametrized by $P_t(z)=z(z^3+t)$, $t\in \A^1$,
\item or the curve is the Zariski closure of $\Sigma(4,2,2)$: for a general polynomial $P$ in the curve, one critical point is fixed and the other two are swapped by the $\U_2$.This curve can be parametrized by $P_t(z)=z^2(z^2+t)$, $t\in \A^1$.
\end{itemize}

\begin{figure}[h]
\centering
\def\svgwidth{12cm}
\begingroup%
  \makeatletter%
  \providecommand\color[2][]{%
    \errmessage{(Inkscape) Color is used for the text in Inkscape, but the package 'color.sty' is not loaded}%
    \renewcommand\color[2][]{}%
  }%
  \providecommand\transparent[1]{%
    \errmessage{(Inkscape) Transparency is used (non-zero) for the text in Inkscape, but the package 'transparent.sty' is not loaded}%
    \renewcommand\transparent[1]{}%
  }%
  \providecommand\rotatebox[2]{#2}%
  \newcommand*\fsize{\dimexpr\f@size pt\relax}%
  \newcommand*\lineheight[1]{\fontsize{\fsize}{#1\fsize}\selectfont}%
  \ifx\svgwidth\undefined%
    \setlength{\unitlength}{435.75118143bp}%
    \ifx\svgscale\undefined%
      \relax%
    \else%
      \setlength{\unitlength}{\unitlength * \real{\svgscale}}%
    \fi%
  \else%
    \setlength{\unitlength}{\svgwidth}%
  \fi%
  \global\let\svgwidth\undefined%
  \global\let\svgscale\undefined%
  \makeatother%
  \begin{picture}(1,0.49947178)%
    \lineheight{1}%
    \setlength\tabcolsep{0pt}%
    \put(0,0){\includegraphics[width=\unitlength,page=1]{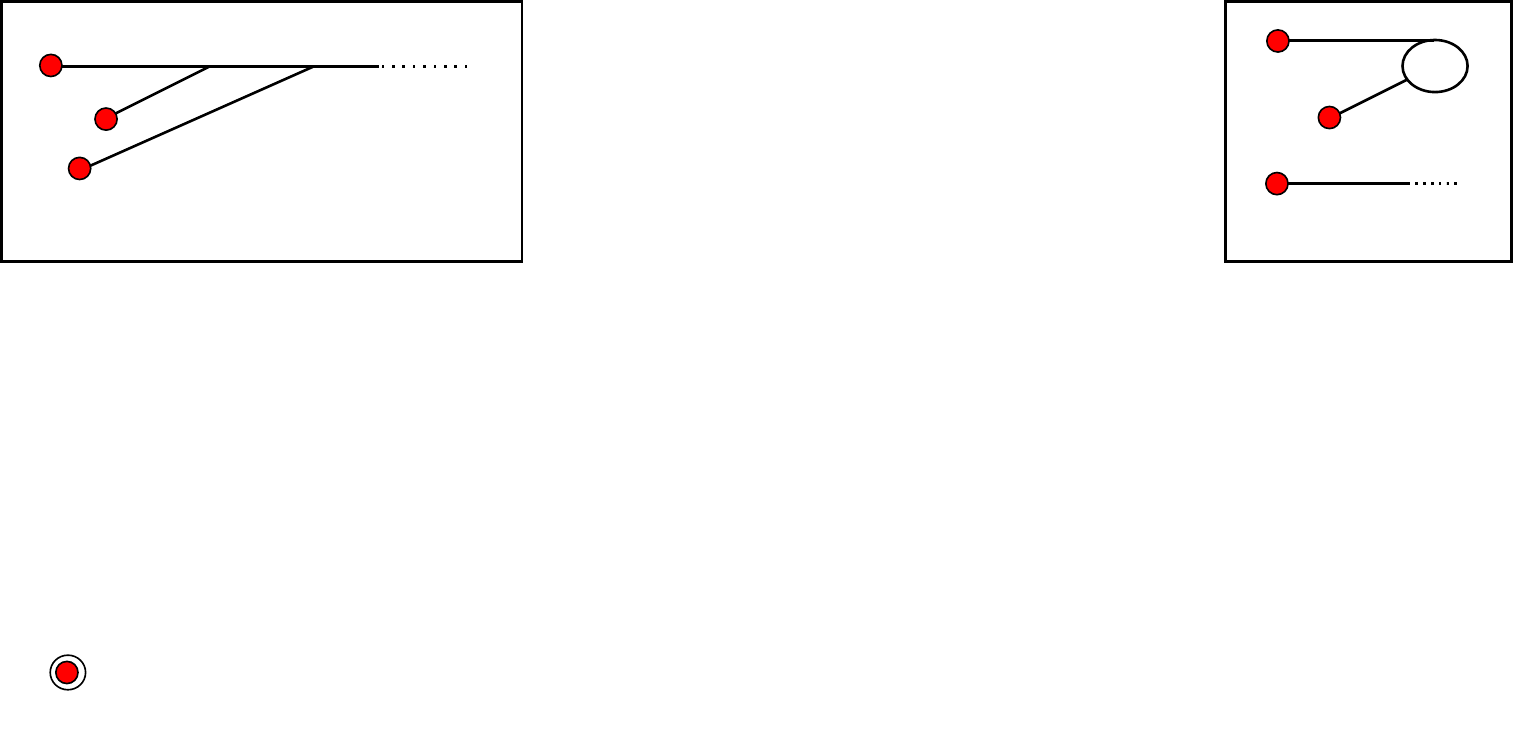}}%
    \put(0.07753771,0.04633411){\makebox(0,0)[lt]{\lineheight{1.25}\smash{\begin{tabular}[t]{l}multiple critical point \end{tabular}}}}%
    \put(0,0){\includegraphics[width=\unitlength,page=2]{marked-graph-degree4.pdf}}%
    \put(0.08043145,0.00369335){\makebox(0,0)[lt]{\lineheight{1.25}\smash{\begin{tabular}[t]{l}simple critical point \end{tabular}}}}%
    \put(0,0){\includegraphics[width=\unitlength,page=3]{marked-graph-degree4.pdf}}%
    \put(0.67822965,0.03389513){\makebox(0,0)[lt]{\lineheight{1.25}\smash{\begin{tabular}[t]{l}Direction of the flow\end{tabular}}}}%
  \end{picture}%
\endgroup%

  \caption{Asymmetrical special critically marked dynamical graphs of degree $4$}
\label{fig:clas_graph4}
\end{figure}

\paragraph*{Classification in degree $5$}
Since $5$ is prime, all special curves are primitive.
The special curves of degree $5$ polynomials can be classified in the following way:

\begin{itemize}
\item three critical points are preperiodic,
\item two critical points are preperiodic and the other two lie in the same grand orbit,
\item one critical point is preperiodic and the other three lie in the same grand orbit,
\item all four critical points lie in the same grand orbit. A particular example is the closure of $\Sigma(5,5,0)$ which is the unicritical family and can be parametrized by $P_t(z)=z^5+t$, $t\in\A^1$,
\item the curve is the closure of $\Sigma(5,2,3)$: for a general $P$ in the curve, $\Sigma(P)=\mathrm{Aut}(P)=\U_2$, two critical point are fixed and equal and the other two critical points are permuted by a symmetry,
\item the curve is the closure of $\Sigma(5,3,2)$: for a general $P$ in the curve, $\Sigma(P)=\U_3$ and $\mathrm{Aut}(P)=\U_1$, one critical point is fixed and the other three critical points are permuted by a symmetry,
\item the curve is the closure of $\Sigma(5,4,1)$: for a general $P$ in the curve, $\Sigma(P)=\mathrm{Aut}(P)=\U_4$ and all four critical points are permuted by a symmetry,
\item the curve is the closure of a curve contained in $\Sigma(5,2,1)$ such that, for a general $P$ in the curve, $\Sigma(P)=\mathrm{Aut}(P)=\U_2$, critical points are permuted by a symmetry by pairs and one critical point in each pair share the same grand orbit. 
\end{itemize}

\section{Open questions on the geometry of special curves}\label{sec:classic}

We know very little about the geometry of special curves. 

\begin{itemize}
\item[(Q1)]
Are special curves smooth in the space of monic and centered polynomials with marked critical points?  Beware that this space is a finite cover of the moduli space of critically marked polynomials, and is isomorphic to $\A^{d-1}$ (it is the quotient of $\{P_{c,a}\}$ by the action of $\U_d$ defined by $\zeta\cdot (c,a) = (c,\zeta a)$).  

In degree $3$, Milnor~\cite[\S 5]{Milnor-cubic} proved that the curve $\mathcal{S}_p$ for which the marked critical point has exact period $p$ is smooth.

\item[(Q2)] 
What is the Euler characteristic and the genus of special curves? Again in degree $3$, for curves for which one critical point is periodic, Bonifant, Kiwi and Milnor~\cite{Bonifant-Kiwi-Milnor} proved that the Euler characteristic of the curve $\mathcal{S}_p$ satisfies
\[\chi(\mathcal{S}_p)=\deg(\mathcal{S}_p)\cdot (2-p)+N_p\]
where $N_p$ is the number of branches at infinity of $\mathcal{S}_p$.
Dujardin~\cite{dujardin2} subsequently showed that
$3^{-p}\chi(\mathcal{S}_p)\to-\infty$,
as $p\to+\infty$. DeMarco-Schiff~\cite{DeMarco-Schiff} gave an algorithm to compute $N_p$ for all $p\ge1$, and implemented it for $p\leq 26$.

\item[(Q3)] 
Milnor~\cite{Milnor-cubic} also conjectured that in degree $3$, the curves $\mathcal{S}_p$ are connected (or equivalently irreducible since they are smooth). Pick any integer $d\geq3$ and any special marked dynamical graph $\Gamma$. Is the curve $C(\Gamma)$ connected (or irreducible)?

\item[(Q4)]
Estimate (or better compute) the degree of $C(\Gamma)$ in terms of the geometry of their special marked dynamical graph $\Gamma$? Note that, in degree $3$, for any integer $p$, the graph $\Gamma(\mathcal{S}_p)$ consists of a single loop of length $p$, together with an infinite half-line, so that their degrees satisfy
$\deg(\mathcal{S}_p)\sim \alpha 3^p$ for some constant $\alpha>0$.

\item[(Q5)]
How are special curves distributed in the moduli space of critically marked degree $d$ polynomials (or equivalently in $\{P_{c,a}\}$)? 

For each $i=0,\cdots, d-2$, define the bifurcation current of the $i$-th critical point by $T_{\bif, i} = dd^c g_i(c,a)$. For each multi-index $I=(i_0, \cdots, i_{d-2})$
define \[T_{\bif}^I =T_{\bif,0}^{i_0} \wedge\cdots \wedge T_{\bif,d-2}^{i_{d-2}}~.\]  
Recall that   $T_{\bif}^I =0$ iff $i_j\ge2$ for some $j$ or $|I| \ge d$.
Recall also that an algebraic subvariety is said to be special if it contains a Zariski dense subset of PCF polynomials.

Now suppose that $\Gamma_k$ is a sequence of special marked dynamical graphs such that the special curve $C_k=C(\Gamma_k)$ is well-defined
(by Theorem~\ref{thm:correspondence}).
Is any weak limit of the sequence of closed positive $(d-2,d-2)$-currents
\[\frac{1}{\deg(C_k)}[C_k]\]
 equal to some bifurcation current $T_{\bif}^I\wedge [Z]$ for some multi-index $I$ and some special algebraic subvariety $Z$ such that 
 $|I| + \codim(Z) = d-2$? 

This is proved in degree $3$ by Dujardin-Favre~\cite{favredujardin} for curves for which one critical point is preperiodic.

\item[(Q6)]
Is any finite critically marked graph realizable by a PCF polynomial (using e.g. the work of Poirier~\cite{Poirier})?
Can we remove conditions (R1) and (R2) from Theorems~\ref{thm:realizability} and~\ref{thm:correspondence}? In other words, 
are all asymmetric special marked dynamical graph realizable?
using e.g. the work of Poirier~\cite{Poirier}? Can one extend the correspondence to graphs with no-trivial symmetries?

\item[(Q7)]
Let $\Gamma$ be any special marked dynamical graph. Is it true that 
the euclidean closure of the set of polynomials having $\Gamma$ as marked dynamical graph
is a special curve?
\end{itemize}

%
%
%
%
%
%

\backmatter

\printindex

\bibliographystyle{plain}
\bibliography{biblio}

\end{document}